\documentclass[11pt, a4paper]{amsart}

\usepackage[toc]{appendix}
\usepackage{enumitem}
\usepackage{natbib}
\usepackage{tikz-cd}
\usepackage[utf8]{inputenc}

\usepackage{amsfonts}
\usepackage{fullpage,amsthm,amsfonts,amssymb,stmaryrd, mathrsfs}
\usepackage{amsmath}
\usepackage{mathrsfs}
\usepackage{multirow}
\usepackage{hyperref}
\usepackage{mathtools}
\usepackage{fullpage}

\usepackage[all]{xy}
\usepackage{xcolor}

\hypersetup{
    colorlinks=true,
    linkcolor=purple,     
    citecolor=purple,      
    urlcolor=black,    
}

\theoremstyle{plain}
\newtheorem{theorem}{Theorem}[section]
\newtheorem{lemma}[theorem]{Lemma}

\newtheorem{corollary}[theorem]{Corollary}

\newtheorem{proposition}[theorem]{Proposition}
\newtheorem{fact}[theorem]{Fact}

\theoremstyle{definition}

\newtheorem{construction}[theorem]{Construction}
\newtheorem{definition}[theorem]{Definition}
\newtheorem{definition-proposition}[theorem]{Definition-Proposition}

\newtheorem{remark}[theorem]{Remark}
\newtheorem{convention}[theorem]{Convention}
\newtheorem{notation}[theorem]{Notation}

\numberwithin{equation}{section}

\def\calG{\mathcal{G}}

\def\AAA{\mathbb{A}}

\def\CC{\mathbb{C}}

\def\QQ{\mathbb{Q}}
\def\RR{\mathbb{R}}

\def\ZZ{\mathbb{Z}}

\DeclareMathOperator{\Gal}{Gal}
\DeclareMathOperator{\Res}{Res}
\DeclareMathOperator{\Hom}{Hom}
\DeclareMathOperator{\Ext}{Ext}

\DeclareMathOperator{\rank}{rank}

\DeclareMathOperator{\Spec}{Spec}

\DeclareMathOperator{\Sym}{Sym}

\DeclareMathOperator{\GU}{GU}

\DeclareMathOperator{\Ind}{Ind}

\DeclareMathOperator{\sgn}{sgn}
\DeclareMathOperator{\Tr}{Tr}

\DeclareMathOperator{\SL}{\mathrm{SL}}

\DeclareMathOperator{\GL}{\mathrm{GL}}
\DeclareMathOperator{\Gm}{\mathrm{G_{m}}}

\DeclareMathOperator{\SU}{\mathrm{SU}}
\DeclareMathOperator{\U}{\mathrm{U}}

\DeclareMathOperator{\lieg}{\mathfrak g}
\DeclareMathOperator{\liek}{\mathfrak k}
\DeclareMathOperator{\lieh}{\mathfrak h}
\DeclareMathOperator{\liep}{\mathfrak p}
\DeclareMathOperator{\liem}{\mathfrak m}
\DeclareMathOperator{\lieq}{\mathfrak q}
\DeclareMathOperator{\lien}{\mathfrak n}

\DeclareMathOperator{\liegl}{\mathfrak gl}

\newcommand{\dR}{\mathrm{dR}}

\newcommand{\Qp}{\QQ_p}

\newcommand{\Zp}{\ZZ_p}

\newcommand{\Sh}{\mathrm{Sh}}

\newcommand{\K}{\mathrm{K}}

\begin{document}

\title{On higher regulators of Picard modular surfaces}

\author{Linli Shi}
\address{
\begin{tabbing}
\indent \= Linli Shi \\ \> Department of Mathematics, ETH Zürich, Rämistrasse 101, 8092 Zürich, Switzerland \\ \> UniDistance Suisse, Schinerstrasse 18, 3900 Brig, Switzerland
\end{tabbing}
}
\email{\href{mailto: linli.shi@math.ethz.ch}{linli.shi@math.ethz.ch}}

\date{\today}

\begin{abstract}
We prove the motivic classes in the motivic cohomology groups of Picard modular surfaces with non-trivial coefficients constructed in a paper of Loeffler\textendash Skinner\textendash Zerbes are in the motivic cohomology groups of the interior motives. Then we establish a relation between the motivic classes and non-critical values of the motivic $L$-functions associated to cuspidal automorphic representations of $\GU(2,1)$, thus deducing non-triviality of the motivic classes and providing evidence for Beilinson's conjectures.  
\end{abstract}

\maketitle

\numberwithin{equation}{section}
\setcounter{tocdepth}{1}

\tableofcontents


\section{Introduction}
\subsection{Motivation}
\label{SS: motivation}
The analytic class number formula is a deep relationship between an analytic invariant and an arithmetic invariant of a number field $F$. It relates the leading Taylor coefficient at $0$ of the Dedekind zeta-function $\zeta_{F}(s)$ of $F$ to the Dirichlet regulator 
\begin{equation*}
    r\colon O_F^{\times} \rightarrow \RR^{r_1 + r_2},
\end{equation*}
where $O_F^{\times}$ is the unit group and $r_1$ (resp., $r_2$) is the number of real (resp., complex) places of the number field $F$.  In the 1980s, Beilinson made deep conjectures about special values of motivic $L$-functions generalizing the classical analytic class formula \cite{Beilinson85}. In order to do so, Beilinson replaced the units $O_{K}^{\times}$ by the motivic cohomology $\mathrm{H}^{i}_{M}(X, \QQ(n))$ where $X$ is a smooth $\QQ$-scheme and replaced the Dirichlet regulator by the higher regulator map 
\begin{equation*}
    r_{H}\colon \mathrm{H}^{i}_{M}(X, \QQ(n)) \rightarrow \mathrm{H}_{H}^{i}(X, \RR(n))
\end{equation*}
from motivic cohomology to absolute Hodge cohomology. For an introduction to Beilinson's conjectures, the interested reader may consult \cite{Nekovar94}. 

\par In this paper, we establish a connection between motivic classes in motivic cohomology with non-trivial coefficients of Picard modular surfaces and the motivic $L$-functions of some cuspidal automorphic representations of the unitary similitude group $\GU(2, 1)$, through which we show that the motivic classes are non-trivial. It answers a question raised in \cite{LSZ22}, in which the authors constructed an Euler system for $\GU(2, 1)$ from the motivic classes. Along the way, we also prove a vanishing on the boundary result for the motivic classes. We hope this work could shed some light on proving more general relationships between special values of motivic $L$-functions associated to automorphic representations and mixed motives. 

\subsection{Statement of the main results}

Now we introduce the main results. These can be divided into two parts. In the first part, we construct motivic classes $\mathcal{E}is_{M}(\phi_f)$ in the motivic cohomology with non-trivial coefficients of Picard modular surfaces $\mathrm{H}^{3}_{M}(S, V(2))$ and then we prove the classes are in a ``nice" subspace of the cohomology.  Along the way, we also prove the regulators of the classes $\mathcal{E}is_{H}(\phi_f)$ are in some ``nice" subspace of the absolute Hodge cohomology $\mathrm{H}^{3}_{M}(S, V(2))$. In the second part, we prove a relation between the image of the motivic class and a non-critical value of the motivic $L$-function associated to some ``nice" irreducible cuspidal automorphic representation $\pi$. The relation resembles Beilinson's conjectures on special values of $L$-functions.

\par Now, we state the results more precisely. Let $G = \GU(2, 1)$ and let $H$ be a subgroup of $G$ that can be embedded into $G$ and is an extension of the group $\GL_2$. You can find the precise definition of $H$ in Definition \ref{defn:G_H}. The relation between the groups $\GL_2$, $H$ and $G$ can be summarized in  following diagram: 
\[
    \begin{tikzcd}
        H \ar[r, hook] \ar[d, twoheadrightarrow] & G  \\
        \GL_2  &  
    \end{tikzcd}. 
\]
These maps of algebraic groups induce morphisms of the corresponding Shimura varieties\footnote{In the introduction, we ignore the level structures. Rigorously speaking, $S$ (resp., $M$) is defined to be the Weil restriction to $\QQ$ of $\Sh_{G}$ (resp., $\Sh_{H}$) (see Convention \ref{Convention: M_{Q} and S_{Q}} for more details). But in the introduction, we ignore this technical detail.} : 
\[
    \begin{tikzcd}
        M = \Sh_H \ar[r, hook, "\iota"] \ar[d, rightarrow, "p"] & S=\Sh_G  \\
        \Sh_{\GL_2}  &  
    \end{tikzcd}. 
\]
Let $V$ be an algebraic representation of $G$ (see Definition \ref{def:rep_G}), $V_2$ be the standard representation of $\GL_{2}$ and $W$ be an algebraic representation of $H$ (see Definition \ref{def:repn_H}). They correspond to motivic sheaves $V$ on $\Sh_{G}$, $V_2$ on $\Sh_{\GL_2}$ and $W$ on $\Sh_{H}$ (see Lemma \ref{lemma: Ancona} and Proposition \ref{Prop: Jin}). Then we have the corresponding morphisms of motivic cohomologies: 
\[
    \begin{tikzcd}
        \mathrm{H}_{M}^{1}(M, W(1)) \ar[r, "\iota^{*}"] & H^{3}_{M}(S, V(2))  \\
        \mathrm{H}_{M}^{1}(\Sh_{\GL_2}, \Sym^{n}V_2(1))  \ar[u, "p^{*}"] &  
    \end{tikzcd}. 
\]

Beilinson constructed the \emph{Eisenstein symbol} \cite[\S 3]{Beilinson_Modular_Curve}. For each nonnegative integer $n$, the Eisenstein Symbol is a $\QQ$-linear map
\[
    Eis^{n}_{M} \colon \mathcal{B}_{n} \rightarrow \mathrm{H}^{1}_{M}(\Sh_{\GL_2},\mathrm{Sym}^{n} V_2(1)), 
\]
where the source $\mathcal{B}_n$ is the space of locally constant $\QQ$-valued functions on $\GL_2(\AAA_{f})$ satisfying some equivariant properties (see Section \ref{SS: Eis symbol and Hodge realization}). 

\par If we compose the Eisenstein symbol with the pullback $p^{*}\colon \mathrm{H}^{1}_{M}(M,\mathrm{Sym}^{n} V_2(1)) \rightarrow \mathrm{H}^{1}_{M}(M, W(1))$ and then with the Gysin map $\iota_{*}\colon \mathrm{H}^{1}_{M}(M, W(1)) \rightarrow \mathrm{H}^{3}_{M}(S, V(2))$,  we get the map $\mathcal{E}is^{n}_{M}\colon \mathcal{B}_{n} \rightarrow \mathrm{H}^{3}_{M}(S, V(2))$. Hence, for $\phi_{f} \in \mathcal{B}_{n}$, we have a motivic class $\mathcal{E}is_{M}^{n}(\phi_f) \in \mathrm{H}^{3}_{M}(S, V(2))$. In other words, we have the following diagram: 
\[
    \begin{tikzcd}[row sep = 0]
        \mathcal{B}_n \ar[r, "{Eis^n_M}"] & \mathrm{H}^{1}_{M}(\Sh_{\GL_2}, \mathrm{Sym}^{n}V_2(1))  \ar[r, "{p^{*}}"] & \mathrm{H}^{1}_{M}(M, W(1)) \ar[r, "{\iota_{*}}"] & \mathrm{H}^{3}_{M}(S,V(2)) \\
        \phi_f \ar[r, mapsto] & Eis^n_M(\phi_f) \ar[rr, mapsto] & & \mathcal{E}is^{n}_{M}(\phi_f) 
    \end{tikzcd}. 
\]
This construction has been given in \cite[Definition 9.2.3]{LSZ22}. 
\par Applying the Beilinson regulator $r_{H}$, we have the following diagram of absolute Hodge cohomology: 
\begin{equation*}
         \begin{tikzcd}[row sep = 0]
        \mathcal{B}_{n, \RR} \ar[r, "{Eis^n_H}"] & \mathrm{H}^{1}_{H}(\Sh_{\GL_2}, \mathrm{Sym}^{n}V_2(1))  \ar[r, "{p^{*}}"] & \mathrm{H}^{1}_{H}(M, W(1)) \ar[r, "{\iota_{*}}"] & \mathrm{H}^{3}_{H}(S,V(2)) \\
        \phi_f \ar[r, mapsto] & Eis^n_H(\phi_f) \ar[rr] & & \mathcal{E}is^{n}_{H}(\phi_f) 
    \end{tikzcd}.
\end{equation*}
\subsubsection{Geometric results}
Our first geometric theorem states that the Hodge class $\mathcal{E}is^{n}_{H}(\phi_f)$ is in a ``nice" subspace of $\mathrm{H}^{3}_{H}(S,V(2))$, whose proof follows from the method used in \cite{LemmaI15}: 
\begin{theorem}[Theorem \ref{Thm: Hdg vanish on the boudary}]
\label{Thm: intro_Hodge_vanish_on_the_boundary}
For $V = V^{a, b}\{r, s\}$ as in Definition \ref{def:rep_G}, under certain conditions (see Theorem \ref{Thm: Hdg vanish on the boudary}), 
the map $\mathcal{E}is_{H}^n\colon \mathcal{B}_{n,\RR} \rightarrow \mathrm{H}^{3}_{H}(S, V(2))$ factors through the inclusion 
\[
    \Ext^{1}_{\mathrm{MHS}_{\RR}^{+}}(\mathbf{1}, \mathrm{H}^{2}_{B,!}(S, V(2))_{\RR}) \hookrightarrow \mathrm{H}^{3}_{H}(S, V(2))_{\RR}, 
\]
where $\mathrm{MHS}_{\RR}^{+}$ is the abelian category of mixed $\RR$-Hodge structures, ${\mathbf{1}}$ denotes the trivial Hodge structure that is the unit of $\mathrm{MHS}_{\RR}^{+}$ and $\mathrm{H}^{2}_{B,!}(S, V(2))$ is the interior Betti cohomology with coefficients in $V$ defined by
\begin{equation*}
    \mathrm{H}^{2}_{B,!}(S, V(2))_{\RR} := \mathrm{Im}(\mathrm{H}^{2}_{B,c}(S, V(2))_{\RR} \rightarrow \mathrm{H}^{2}_{B}(S, V(2))_{\RR}). 
\end{equation*}
\end{theorem}

\begin{remark}
    The interior Betti cohomology $\mathrm{H}^{2}_{B,!}(S, V(2))_{\CC}$ is the cohomology group containing the ``cohomological" cuspidal automorphic representations $\pi$ of $G(\AAA)$ (see Proposition \ref{Prop:coh_identity} and its proof). Hence, Theorem \ref{Thm: intro_Hodge_vanish_on_the_boundary} tells us that the regulator of the motivic class $\mathcal{E}is^{n}_{H}(\phi_f)$ is in a ``nice" subspace that is related to automorphic representations. It gives one hope that the class $\mathcal{E}is^{n}_{H}(\phi_f)$ can be related to the motivic $L$-function of some automorphic representation. 
\end{remark}

Our second geometric theorem is a ``motivic lifting'' of the first theorem that says that the motivic class $\mathcal{E}is^{n}_{M}(\phi_f)$ is in a ``nice" subspace of $\mathrm{H}^{3}_{M}(S,V(2))$: 
\begin{theorem}[Theorem \ref{Thm: mot vanish on the boudary}]
\label{Thm: intro_motivic_vanish_on_the_boundary}
Under the same conditions of the previous theorem, 
the map $\mathcal{E}is_{M}^n \colon \mathcal{B}_{n} \rightarrow \mathrm{H}^{3}_{M}(S, V(2))$ factors through the inclusion 
\[
    \mathrm{H}_{M}^{3}(\mathrm{Gr_{0}}M_{gm}(V(2)), \QQ(0)) \hookrightarrow \mathrm{H}^{3}_{M}(S, V(2)), 
\] 
where $\mathrm{H}_{M}^{3}(\mathrm{Gr_{0}}M_{gm}(V(2)), \QQ(0))$ (see Definition \ref{def: motivic cohomology of interior motives}) is the motivic cohomology of the interior motive $\mathrm{Gr_{0}}M_{gm}(S, V(2))$ (see Definition \ref{def: intersection motive and interior motive}).  
\end{theorem}

\begin{remark}
    \begin{itemize}
        \item The interior Betti cohomology is the realization of the interior motive $\mathrm{Gr_{0}M_{gm}}(S, V(2))$ through the Betti realization functor. 
        \item Under the Beilinson regulator $r_{H}$, the image of $\mathrm{H}_{M}^{3}(\mathrm{Gr_{0}}M_{gm}(V(2)), \QQ(0))$ lies in 
            \begin{equation*}
                \Ext^{1}_{\mathrm{MHS}_{\RR}^{+}}(\mathbf{1}, \mathrm{H}^{2}_{B,!}(S, V(2))_{\RR}).
            \end{equation*}
        Hence, Theorem \ref{Thm: intro_motivic_vanish_on_the_boundary} can be viewed as a ``motivic lifting'' of Theorem \ref{Thm: intro_Hodge_vanish_on_the_boundary}. 
        \item Theorem \ref{Thm: intro_motivic_vanish_on_the_boundary} tells us that the motivic class $\mathcal{E}is^{n}_{M}(\phi_f)$ is in a ``nice" subspace that is related to the Grothendick motive associated to some automorphic representation \cite[Theorem 5.6]{Wild_15}. It gives one hope to relate the motivic class $\mathcal{E}is^{n}_{M}(\phi_f)$ to the motivic $L$-function of some automorphic representation.
        \item In \cite{Kings98}, the author raised the question of proving the motivic version of vanishing on the boundary for Eisenstein class on Shimura varieties. In the case of Hilbert modular surfaces, a result of motivic vanishing on the boundary has been proved in \cite[Corollary 3.14]{Wild_09_HMS} using a different formalism from us. 
    \end{itemize}
\end{remark}

\subsubsection{Results on special values of $L$-functions}

Let $\pi = \pi_f \otimes \pi_{\infty}$ be some ``cohomological" irreducible cuspidal automorphic representation $\pi$ of $G(\AAA)$, where we denote by $E(\pi_f)$ its rational field (see Section \ref{Sec: main result}). 
Let $M(\pi_f, V(2))$ be the Grothendieck motive associated to $\pi$ \cite[Theorem 5.6]{Wild_15} after a Tate twist and $M_{B}(\pi_f, V(2))$ be its Betti realization. Then the absolute Hodge cohomology of $M_{B}(\pi_f, V(2))$ is a rank-$1$ $E(\pi_f) \otimes_{\QQ} \RR$-module. In it, we define two rational structures $\mathcal{K}(\pi_f, V(2))$ and $\mathcal{D}(\pi_f, V(2))$, where the rational structure $\mathcal{K}(\pi_f, V(2))$ (maybe trivial) is defined using the motivic class $\mathcal{E}is_{M}(\phi_f)$ and $\mathcal{D}(\pi_f, V(2))$ (non-trivial) is defined using the comparison between the Betti and de Rham realizations (see Section \ref{SS: Poincare duality}). Our theorem says that the difference between the two rational structures can be measured using a non-critical value of the motivic $L$-function $L(M_{\text{\'et}}(\pi_f, V(2)), s)$ (see Definition \ref{Def: mot L-function}) of $\pi$: 
\begin{theorem}[Theorem \ref{Them_mot_L_function}]
\label{Thm:intro_mot_L_value}
    Under the conditions in the statement of Theorem \ref{Them_mot_L_function}, 
    the relation between $\mathcal{K}(\pi_f,V(2))$ and $\mathcal{D}(\pi_f,V(2))$ is as follows: 
    \begin{equation*}
            \mathcal{K}(\pi_f,V(2)) = C \cdot L(M_{\text{\'et}}(\pi_f, V(2)), 0)\mathcal{D}(\pi_f, V(2)), 
    \end{equation*}
    where $C \in (E(\pi_f) \otimes_{\QQ} \CC)^{\times}$ and $C$ can be expressed explicitly in terms of a comparison of Whittaker periods (see Defintion \ref{def: Whittaker period}) and Deligne periods (see Definition \ref{def: Deligne period}).
\end{theorem}

\begin{remark}
    \begin{itemize}
        \item That the motivic $L$-function $L(M(\pi_f, V(2)), s)$ is being evaluated at $s = 0$ is exactly what Beilinson conjectured. \cite[Conjecture (6.1)]{Nekovar94} 
        \item The $C$ in Theorem \ref{Thm:intro_mot_L_value} should be in $(E(\pi_f) \otimes_{\QQ} \overline{\QQ})^{\times}$ according to Beilinson's conjectures.
        \item Theorem \ref{Thm:intro_mot_L_value} is a weight $w \le -3$ counterpart of the main theorem proved in \cite{PS18}, where the authors constructed a motivic class $\xi$ in the motivic cohomology with trivial coefficients $\mathrm{H}_{M}^{3}(S, \QQ(2))_{\overline{\QQ}}$ and proved a relation between $\xi$ and a non-critical $L$-value of some cuspidal automorphic representation $\pi$\footnote{There is a gap in the proof of \cite{PS18}, which seems could be fixed by the tool of tempered currents developed recently in \cite{BCLRJ24}}. 
        \item A direct corollary of Theorem \ref{Thm:intro_mot_L_value}  is that the map $\mathcal{E}is_{M}$ is non-trivial (see Corollary \ref{corollary: nontrivial_class}). This answers a question raised in \cite{LSZ22}, where the authors needed the non-triviality of the motivic classes $\mathcal{E}is(\phi_f)$ for $\phi_f \in \mathcal{B}_{n}$ as an input to construct an Euler system for $\GU(2,1)$. 
        \item Theorem \ref{Thm:intro_mot_L_value} is deduced from an automorphic version of the theorem (see Theorem \ref{Thm: auto_L-value}) and a comparison between automorphic $L$-functions and motivic $L$-functions after many normalizations (see Proposition \ref{Prop:rel_L_ftn}). 
    \end{itemize}
\end{remark}

\subsection{An overview of the proof}

\subsubsection{An outline of the proof of the geometric results}

We first give an outline of the proof of vanishing on the boundary of absolute Hodge cohomology (Theorem \ref{Thm: intro_Hodge_vanish_on_the_boundary}, Theorem \ref{Thm: Hdg vanish on the boudary}). 

\par For each of the two Shimura varieties $X = \{M, S\}$, we denote by $X^{*}$ its Baily{\textendash}Borel compactification (see Section \ref{SS: Baily-Borel compactification of PMF} for more details) and let $\partial{X} = X^{*} - X$ be the cusps of the compactification. Hence, we have the following commutative diagram: 
\[
    \begin{tikzcd}
        M \ar[r, "j^{\prime}"] \ar[d, "{\iota}"] & M^{*} \ar[d, "p"] & \partial{M} \ar[l, "i^{\prime}" above] \ar[d, "q"] \\
        S \ar[r, "j"] & S^{*} & \partial{S} \ar[l, "i" above], 
    \end{tikzcd}
\]
where $j$ and $j^{\prime}$ are open immersions, $i$ and $i^{\prime}$ are  closed immersions and $\iota$, $p$ and $q$ are closed immersions\footnote{It holds after a careful choice of level structures.}. 

\par By properties of the derived category of real algebraic mixed Hodge modules (see \cite[Definition A.2.4]{HW98}), we have the following exact sequence (see Proposition \ref{Prop: Hodge exact sequence}):

\begin{equation}
\label{Intro: Hodge exact sequence}
    0 \rightarrow \Ext^{1}_{\mathrm{MHS}_{\RR}^{+}}(\mathbf{1}, \mathrm{H}^{2}_{B,!}(S, V(2))_{\RR}) \rightarrow \mathrm{H}^{3}_{H}(S,V(2)) \rightarrow \mathrm{H}^{1}_{H}(\partial{S}, i^{*}j_{*}V(2)).
\end{equation}

By the functoriality, we can get the following commutative diagram (see Proposition \ref{Prop: Hodge commutative diagram}): 
\begin{equation}
\label{Intro: Hodge commutative diagram}
    \begin{tikzcd}
        \mathcal{B}_{n, \RR} \ar[d] & \\
        \mathrm{H}^{1}_{H}(M,W(1)) \ar[r] \ar[d] & \mathrm{H}^{0}_{H}(\partial{M}, i^{\prime*}j^{\prime}_{*}W(1)) \ar[d, "\theta^{H}"] \\
        \mathrm{H}^{3}_{H}(S, V(2)) \ar[r] & \mathrm{H}^{1}_{H}(\partial{S}, i^{*}j_{*}V(2)). 
    \end{tikzcd}
\end{equation}

By exact sequence (\ref{Intro: Hodge exact sequence}) and commutative diagram (\ref{Intro: Hodge commutative diagram}), in order to prove that the map $\mathcal{E}is_{H}^n \colon \mathcal{B}_{n,\RR} \rightarrow \mathrm{H}^{3}_{H}(S, V(2))$ factors through the inclusion 
\[
    \Ext^{1}_{\mathrm{MHS}_{\RR}^{+}}(\mathbf{1}, \mathrm{H}^{2}_{B,!}(S, V(2))_{\RR}) \hookrightarrow \mathrm{H}^{3}_{H}(S, V(2)), 
\]
it suffices to prove that the map $\theta^{H}$ is zero, which is carried out in Proposition \ref{Prop: Hdg theta is zero}. The key input of the proof of that fact is the explicit computation of degeneration of Hodge structures over $M$ and $S$ (see Lemma \ref{lemma: BW_H}, Lemma \ref{lemma: BW_G}), which is based on a theorem of Burgos-Wildeshaus (see Theorem \ref{theorem: BW04}) and Kostant's theorem (see Theorem \ref{thm: Kostant_theorem}). Finally, we remark that the degeneration of Hodge structures over $S$ has been computed in \cite{Anc17}. \\

\par We now outline the proof of vanishing on the boundary of motivic cohomology (Theorem \ref{Thm: intro_motivic_vanish_on_the_boundary}, Theorem \ref{Thm: mot vanish on the boudary}). The proof is parallel to the proof in the absolute Hodge cohomology case. However, we need to translate everything into the language of motives. 

\par First, we have a motivic version of the exact sequence (\ref{Intro: Hodge exact sequence}) (see Proposition \ref{Prop: mot exact sequence}): 
\begin{equation}
\label{Intro: Mot exact sequence}
    0 \rightarrow   \mathrm{H}_{M}^{3}(\mathrm{Gr_{0}}M_{gm}(V(2)), \QQ(0)) \rightarrow \mathrm{H}^{3}_{M}(S,V(2)) \rightarrow \mathrm{H}^{3}_{M}(\partial{S}, i^{*}j_{*}V(2)).
\end{equation}

\par Second, by functoriality of the triangulated category of mixed motives, we have the motivic version of the commutative diagram (\ref{Intro: Hodge exact sequence}) (see Proposition \ref{Prop: mot commutative diagram}): 
\begin{equation}
\label{Intro: Mot commutative diagram}
    \begin{tikzcd}
        \mathcal{B}_n \ar[d] & \\
        \mathrm{H}^{1}_{M}(M,W(1)) \ar[r] \ar[d] & \mathrm{H}^{1}_{M}(\partial{M}, i^{\prime*}j^{\prime}_{*}W(1)) \ar[d, "\theta^{M}"] \\
        \mathrm{H}^{3}_{M}(S, V(2)) \ar[r] & \mathrm{H}^{3}_{M}(\partial{S}, i^{*}j_{*}V(2)). 
    \end{tikzcd}
\end{equation}

\par Finally, similar to the absolute Hodge cohomology case, in order to prove that the map $\mathcal{E}is_{M}^n\colon \mathcal{B}_{n} \rightarrow \mathrm{H}^{3}_{M}(S, V(2))$ factors through the inclusion 
\[
     \mathrm{H}_{M}^{3}(\mathrm{Gr_{0}}M_{gm}(V(2)), \QQ(0)) \hookrightarrow \mathrm{H}^{3}_{M}(S, V(2)), 
\] 
it suffices to prove that $\theta^{M}$ is zero, whose proof is based on the fact that $\theta^{H}$ is zero and weight conservativity of Hodge realization functor for Voevodsky motives of Abelian type \cite[Theorem 1.12.]{Wild_15} (see Proposition \ref{Prop: mot theta is zero}). 

\subsubsection{An outline of the proof of the $L$-value results}

In order to prove the relation 
\begin{equation*}
    \mathcal{K}(\pi_f, V(2)) = C^{\prime} \cdot \mathcal{D}(\pi_f, V(2)), 
\end{equation*}
in the rank $1$ $E(\pi_f) \otimes_{\QQ} \RR$-module $\Ext^{1}_{\mathrm{MHS_{\RR}^{+}}}(\RR(0), M_{B}(\pi_f, V(2))_{\RR})$ for some $C^{\prime} \in (E(\pi_f) \otimes_{\QQ} \RR)^{\times}$, it suffices to find an $E(\pi_f) \otimes_{\QQ} \RR$-linear form $\psi \colon M_{B}(\pi_f, V(2))^{-}_{\RR}(-1) \rightarrow E(\pi_f) \otimes_{\QQ} \CC$ which is trivial on $F^{0}M_{dR}(\pi_f, V(2))_{\RR}$ such that 
\begin{equation*}
    \psi(\tilde{v}_{K}) = C^{\prime} \cdot \psi(\tilde{v}_{D}), 
\end{equation*}
where $v_{K}$ (resp., $v_{D}$) is a $E(\pi_f)$-generator of $\mathcal{K}(\pi_f, V(2))$ (resp., $\mathcal{D}(\pi_f, V(2))$) and $\tilde{v}_{K}$ (resp., $\tilde{v}_{D}$) is a lifting of $v_{K}$ (resp., $v_{D}$) through the exact sequence (\ref{exact seq: Ext^1}) (see Lemma \ref{Lemma: linear form and rational structure}).

\par By the Poincar\'{e} duality pairing 
\begin{equation*}
        \langle \cdot, \cdot \rangle_{B} \colon \mathrm{H}^{2}_{B, !}(S, V(2)) \otimes \mathrm{H}^{2}_{B, !}(S, D) \rightarrow \QQ(0),
\end{equation*}
where $D$ is the contragredient representation of $V$ and the construction of a differential form $\Omega \in M(\tilde{\pi}_f|\mu|^{-2}, D)_{\CC}^{+}$, which is carried out in Section \ref{SS: the test vector}, we construct a linear form $\Psi$ and we get 
\begin{equation*}
    \mathcal{K}(\pi_{f}, V(2)) = \frac{\langle \Omega, \tilde{v}_{K} \rangle_{B}}{\langle \Omega, \tilde{v}_{D} \rangle_{B}} \mathcal{D}(\pi_f, V(2)). 
\end{equation*}
Hence, we are left to compute $\langle \Omega, \tilde{v}_{K} \rangle_{B}$ and $\langle \Omega, \tilde{v}_{D} \rangle_{B}$.  

\par The way to compute the pairing $\langle \Omega, \tilde{v}_{K} \rangle_{B}$ is as follows. We first map the class $v_{K}$ to Deligne{\textendash}Beilinson cohomology through the natural map $r_{H \rightarrow D}$ from absolute Hodge cohomology to Deligne{\textendash}Beilinson cohomology (see definition of Deligne{\textendash}Beilinson cohomology in Proposition \ref{Prop: DB-coh_temperd currents}), in which the class will be represented by a pair of tempered currents (see Remark \ref{remark: Eis_D class})
\begin{equation*}
    (\iota_{*}T_{p^{*}Eis_{H}^{n}(\phi_f)}, \iota_{*} T_{p^{*}Eis_{B}^{n}(\phi_f)})
\end{equation*}
and this can be used to compute the pairing $\langle \Omega, \tilde{v}_{K} \rangle_{B}$ (see Lemma \ref{Lemma: tempered current associated to Eis classes} for details). Finally, we express the pairing in terms of an integral of differential forms over Shimura varieties $\Sh_{H}$ (see Proposition \ref{Prop: pairing to integration of diff form}), which is expressed explicitly in Propositin \ref{prop: explicit_pairing}. 

\begin{remark}
    \begin{itemize}
        \item The tool of tempered currents defined in \cite{BCLRJ24} is a key ingredient of this step. Hence, we recall the key definitions in Section \ref{SS: DB cohomology}. 
        \item In order to apply the tool of Deligne{\textendash}Beilinson cohomology, we need to apply ``Liebermann’s trick" (see Definition \ref{Def: DB-coh with coefficients}) because we are working with non-trivial coefficients. 
    \end{itemize}
\end{remark}

\par By a careful choice of test vectors (see Proposition \ref{Proposition_comp_pair_w_zeta_int}), we can express the above explicit integral as a zeta integral $I(\varphi, \Phi, \nu, s)$ (see Definition \ref{def: Zeta_integral}) constructed by Gelbart and Piatetski-Shapiro, which gives an integral repersentation of the twisted standard automorphic $L$-function of the contragredient $\tilde{\pi}$ of $\pi$. We then prove the algebraicity of the nonarchimedean local Zeta integral at rational points in Propositon \ref{Prop:algebracity} and compute the archimedean Zeta integral explicitly in Proposition \ref{Prop:Arch_Zeta}. Combining all of these with an identity (see Lemma \ref{Lemma: relate L-function of contra of pi to L-function of pi}) between automorphic $L$-functions of an automorphic representation and it congragredient, we can express the pairing $\langle \Omega, \tilde{v}_{K} \rangle_{B}$ in terms of an automorphic $L$-value of $L(s, \pi, \mathrm{std})$ (see Definition \ref{def: L-function}) at non-critical points and Whittaker periods (see Definition \ref{def: Whittaker period}).  

\par By the definition of Deligne periods and Deligne $E(\pi_f)$-structure (see Definition \ref{def: Delign-rational structure}), we express the pairing $\langle \Omega, \tilde{v}_{D} \rangle_{B}$
in terms of Deligne periods (see Proposition \ref{prop: periods}). 

\par Combining the above computation of $\langle \Omega, \tilde{v}_{K} \rangle_{B}$ and $\langle \Omega, \tilde{v}_{D} \rangle_{B}$, we prove a relation 
\begin{equation*}
    \mathcal{K}(\pi_f, V(2)) = C^{\prime} \cdot \mathcal{D}(\pi_f, V(2)), 
\end{equation*}
where $C^{\prime}$ is expressed in terms of a non-critical value of $L(s, \pi, \mathrm{std})$ at $s = 0$, Whittaker periods and Deligne periods (see Theorem \ref{Thm: auto_L-value}). Finally, we relate the automorphic $L$-function $$L(s, \pi, \mathrm{std})$$ to the motivic $L$-function $L(M_{\text{\'et}}(\pi_f, V(2)), s)$ (see Proposition \ref{Prop:rel_L_ftn}). The fact that $C^{\prime} \in (E(\pi_f) \otimes_{\QQ} \CC)^{\times}$ follows from $L(M_{\text{\'et}}(\pi_f, V(2)), 0) \neq 0$ (see the proof of Corollary \ref{corollary: nontrivial_class}). 

\subsection{Structure of the paper}
In Section \ref{SS:algebraic groups}-Section \ref{SS: Lie groups}, we define certain algebraic groups, Lie groups and their representations. In Section \ref{SS:Shimura varieties}-Section \ref{SS: cohomology of Picard modular surfaces}, we define the corresponding Shimura varieties and recall the needed properties of cohomology of Picard modular surfaces. In Section \ref{SS:discrete series}-Section \ref{SS: Hodge decomposition}, we compute the $L$-packet of discrete series at the archimedean place and compute the Hodge decomposition of corresponding motives associated to automorphic representations. In Section \ref{SS: relative motive}-Section \ref{SS: Beilinson regulator}, we recall the definition of motivic cohomology, absolute Hodge cohomology and Beilinson regulators. In Section \ref{SS: motivic class}, we recall the construction of the motivic classes in \cite{LSZ22} and their image under the Beilinson regulator, which we call the Hodge classes. In Section \ref{SS: Baily-Borel compactification of PMF}-Section \ref{SS: the proof of the Hodge vanishing}, we prove that the Hodge classes belong to a ``nice'' subspace of absolute Hodge cohomology. In Section \ref{SS: weight structures}-Section \ref{Sec: motivic vanishing}, we prove that the motivic classes belong to a ``nice'' subspace of motivic cohomology. In Section \ref{SS: Poincare duality}-Section \ref{SS: res of the diff form}, we explain how to associate a differential form to the contragredient automorphic representation and use it and Poincar\'{e} duality to construct a linear form.  In Section \ref{SS: DB cohomology}, we recall the definition of Deligne{\textendash}Beilinson cohomology using tempered currents. In Section \ref{SS: DB-cohomology with coeffients}-Section \ref{SS the pairing}, we give an explicit expression of the classes in Deligne{\textendash}Beilinson cohomology and an explicit expression of its pairing with the above differential form. In Section \ref{SS: global zeta integral}-Section \ref{SS: comparison of integral with zeta integral}, we express the above integral in terms of a zeta integral. In Section \ref{SS: unfolding}-Section \ref{SS: arch local integral}, we unfold the zeta integral and recall an algebraicity result of nonarchimedean zeta integrals and compute the archimedean zeta integral explicitly. In Section \ref{SS: motivic periods}, we compute the pairing of the Deligne rational structure with the diffferential form. In Section \ref{SS: Whittaker period}, we recall the definition of the Whittaker period. In Section \ref{SS: connection to auto L-function}-Section \ref{SS: conenction to motivic L-function}, we prove the main results of the paper. 

\subsection*{Acknowledgments}
This paper is based on the author’s doctoral dissertation at the University of Connecticut. He would like to thank his Ph.D. advisor, Liang Xiao, for introducing him to the subject and for many insightful discussions during his graduate studies. He is especially grateful to his postdoctoral mentors, David Loeffler and Sarah Zerbes, for their continued interest in this work and for many valuable suggestions and comments that have improved the paper. He also thanks Keith Conrad for carefully correcting typographical and linguistic errors. Part of this work was carried out during a visit to the Morningside Center of Mathematics in Beijing in 2023–2024, generously hosted by Yongquan Hu, whose hospitality is warmly appreciated.


\section{Notations and Conventions}
\subsection*{General notations}

\begin{itemize}
    \item Let $\QQ \subset \RR \subset \CC$ be the set of rational numbers, real numbers and complex numbers. 
    \item Let $\QQ^{\times}_{+}$ be the set of non-zero positive rational numbers. 
    \item We use $\overline{\QQ}$ to denote the set of algebraic numbers and use $\overline{\QQ}^{\times}$ to denote the set of non-zero algebraic numbers. 
    \item $=_{\overline{\QQ}^{\times}}$: for $A, B \in \CC$, $A =_{\overline{\QQ}^{\times}} B$ means that there exists a $c \in \overline{\QQ}^{\times}$ such that $A = c \cdot B$. 
    \item For a number field $F$, we use $\AAA_{F}$ to denote the ring of adeles and use $\AAA_{F, f}$ to denote the ring of finite adeles. We use $\AAA$ to denote the ring of adeles of $\QQ$ and use $\AAA_{f}$ to denote the ring of finite adeles of $\QQ$.
    \item We use the notation $\hat{\ZZ}$ to denote the ring $\varprojlim \ZZ / N \ZZ$ and use $\hat{\ZZ}^{\times}$ to denote the abelian group $\varprojlim (\ZZ / N \ZZ)^{\times}$, where $N$ runs over positive integers.
    \item We use $i$ to denote the complex number $\sqrt{-1}$.
\end{itemize}

\subsection*{Algebraic geometry}

\begin{itemize}
    \item $\mathcal{O}_{X}$: When $X$ is a scheme (resp., a complex manifold), we denote by $\mathcal{O}_{X}$ the structure sheaf (resp.,  structure sheaf of holomorphic functions) of $X$. 
    \item Let $\mathbb{S} = \mathrm{Res}_{\CC/\RR} \Gm$ be the Deligne torus. Following Deligne and Pink, we use the following convention for the equivalence of categories between algebraic representations of $\mathbb{S}$ in finite-dimensional $\RR$-vector spaces and (semisimple) $\RR$-mixed Hodge structures. Let  $(\rho, V)$ be such a representation of $\mathbb{S}$. The summand $V^{p, q} \subset V$ of Hodge type $(p, q)$ is the summand of $V$ such that $\rho(z_1, z_2)$ acts by $z_1^{-p}z_2^{-q}$ for $(z_1, z_2) \in \mathbb{S}(\CC)$. In particular, any algebraic representation $V$ of $\mathbb{S}$ with central character $c$ corresponds to a pure Hodge structure of weight $-c$. 
    \item For a $\CC$-scheme $X$, we use the notation $X^{an}$ to denote the complex analytic variety associated to $X$ through GAGA. 
    \item For a field extension $E$ of $F$ and a scheme over the field $E$, we use the notation $\mathrm{Res}_{E/F}(X)$ to denote the Weil restriction of $X$ to $F$. To be more precise, for any $F$-scheme $S$, the $S$-points of the scheme $\mathrm{Res}_{E/F}(X)$ are the $S \times_{\Spec (F)} \Spec (E)$-points of $X$. 
    \item When we say Betti cohomology of an algebraic variety over $\QQ$, we always mean Betti cohomology of the $\CC$-points of the algebraic variety. In order to simplify notation, we avoid any notation for taking $\CC$-points.
\end{itemize}

\subsection*{Group Theory}

\begin{itemize}
    \item For a positive integer $n$, we use $S_{n}$ to denote the permuation group of $n$ elements. 
    \item For an affine group scheme $G$ over a ring $k$, sometimes we write $g \in G$ to replace 
          \begin{equation*}
              g \in G(R),
          \end{equation*}
          for a $k$-algebra $R$. 
    \item For a reductive group $G$ over $\QQ$, we denote by $A_{G}$ the connected component of the identity of the real points of the maximal $\QQ$-split center of $G$.  
    \item For a reductive group $G$ over $\QQ$, let $K$ be a maximal compact group of $G(\RR)$ and denote $K_{G} = A_{G} K$. 
    \item For a Lie group $G$, we denote by $G^{+}$ the connected component of the identity in $G$. 
\end{itemize}

\subsection*{Representation theory}

\begin{itemize}
    \item For a representation $(\pi, V)$ of a group $G$, sometimes we use $\pi$ to denote the space $V$ when there is no confusion. 
    \item A representation of a $p$-adic group is always a smooth admissible $\CC$-representation.
    \item For a representation $V$ of a group (algebraic group, real Lie group, $p$-adic group or adelic group), we use $V^{\vee}$ to denote the contragredient representation of $V$. 
    \item We use the notation $\Ind$ to the denote unnormalized induction functor. 
    \item For a representation $\pi$ of a group $G$ over a field $E$, which is an extension of a field $F$, we use the notation $\mathrm{Res}_{E/F} \pi$ to denote the restriction of the representation to the field $F$. 
\end{itemize}

\subsection*{Automorphic Forms}

\begin{itemize}
    \item For a reductive group $G$ over $\QQ$ and a quasi-character $\xi$ (not necessarily unitary) of $A_{G}$,  let $\mathrm{L}^{2}(G(\QQ) \backslash G(\AAA), \xi)$ be the space of functions $\phi \colon G(\QQ) \backslash G(\AAA) \rightarrow \CC$ such that 
    \begin{equation*}
        \phi(zx) = \xi(z) \phi(x), \, \, \, z \in A_{G}, x \in G(\AAA), 
    \end{equation*}
    and which are square integrable modulo $A_{G}$. 
    \item Let $\mathrm{C}^{\infty}(G(\QQ) \backslash G(\AAA), \xi)$, respectively $\mathrm{C}_{c}^{\infty}(G(\QQ) \backslash G(\AAA), \xi)$ be the space of functions $\phi \colon G(\QQ) \backslash G(\AAA) \rightarrow \CC$ such that 
        \begin{itemize}
            \item $\phi(zx) = \xi(z) \phi(x), \, \, \, z \in A_{G}, x \in G(\AAA)$, 
            \item the restriction of $\phi$ to $G(\RR)$ is smooth, respectively, smooth and compactly supported modulo $A_{G}$, 
            \item the restriction of $\phi$ to $G(\AAA_f)$ is locally constant and compactly supported. 
        \end{itemize}
    \item Let $\mathrm{C}_{(2)}^{\infty}(G(\QQ) \backslash G(\AAA), \xi)$ be the space 
          \begin{equation*}
              \mathrm{L}^{2}(G(\QQ) \backslash G(\AAA), \xi) \cap \mathrm{C}^{\infty}(G(\QQ) \backslash G(\AAA), \xi). 
          \end{equation*}
          We have inclusions of  $((\lieg_{\CC}, K_{G}) \times G(\AAA_f))$-modules 
          \begin{equation*}
              \mathrm{C}_{c}^{\infty}(G(\QQ) \backslash G(\AAA), \xi) \subset \mathrm{C}_{(2)}^{\infty}(G(\QQ) \backslash G(\AAA), \xi) \subset \mathrm{C}^{\infty}(G(\QQ) \backslash G(\AAA), \xi). 
          \end{equation*}
    \item Let $\mathrm{C}_{\mathrm{cusp}}^{\infty}(G(\QQ) \backslash G(\AAA), \xi) = \mathrm{L}^{2}_{\mathrm{cusp}}(G(\QQ) \backslash G(\AAA), \xi)$ be the space of cusp forms in 
    \begin{equation*}
        \mathrm{C}^{\infty}(G(\QQ) \backslash G(\AAA), \xi). 
    \end{equation*}
    Smooth truncation to a large compact set modulo $A_{G}$ induces a map of $((\lieg_{\CC}, K_{G}) \times G(\AAA_f))$-modules 
    \begin{equation*}
        \mathrm{C}_{\mathrm{cusp}}^{\infty}(G(\QQ) \backslash G(\AAA), \xi) \rightarrow \mathrm{C}_{c}^{\infty}(G(\QQ) \backslash G(\AAA), \xi). 
    \end{equation*}
    \item An irreducible cuspidal automorphic representation $\pi = \pi_{f} \otimes \pi_{\infty}$ is an irreducible $((\lieg_{\CC}, K_{G}) \times G(\AAA_f))$-submodule of $\mathrm{C}_{\mathrm{cusp}}^{\infty}(G(\QQ) \backslash G(\AAA), \xi)$. 
\end{itemize}

\section{Preliminaries}
\label{Sec:preliminaries}

\subsection{Algebraic groups, algebraic representations and branching laws} 
\label{SS:algebraic groups}
In this subsection, we define the reductive group $\GU(2,1)$ and a subgroup $H$ of $G$, together with their representations. 

\subsubsection{Algebraic groups}
\begin{notation}
    Let $E$ be an imaginary quadratic field of discriminant $-D$, and let $x \mapsto \bar{x}$ be the non-trivial Galois automorphism of $E$ over $\QQ$. Let $\mathcal{O}$ be the ring of integers of $E$. We fix an identification of $E \otimes_{\QQ} \RR$ with $\CC$ such that the imaginary part of $\delta = \sqrt{-D}$ is positive.
\end{notation}
\begin{definition}
\label{defn:G_H}
\begin{enumerate}
    \item Let $J \in \GL_3(E)$ be the Hermitian matrix 
        \[
            J = \begin{pmatrix} 
                0 & 0 & \delta^{-1} \\
                0 & 1  & 0 \\
                -\delta^{-1} & 0 & 0 
                \end{pmatrix}, \quad \mathrm{where} \ \delta = \sqrt{-D}, 
        \]
        and let $G = \GU(J)$ be the group scheme over $\mathbb{Z}$ such that for all $\mathbb{Z}$-algebras $R$, we have 
        \[
            G(R) = \{(g, \mu) \in \GL_3(\mathcal{O} \otimes R) \times R^{\times} | \prescript{t}{}{\bar{g}}Jg = \mu J\}. 
        \]
        This is a unitary similitude group and we denote by $\mu \colon G \rightarrow \Gm$ the similitude character. The center $Z_{G}$ of the group $G$ can be identified with $\Res_{\mathcal{O}/\ZZ}(\Gm)$ via 
        \begin{equation}
        \label{eq: Z_{G}}
           z \in \Res_{\mathcal{O}/\ZZ}(\Gm)(R) \mapsto  \begin{pmatrix}
                                                            z & 0 & 0 \\
                                                            0 & z & 0 \\
                                                            0 & 0 & z
                                                         \end{pmatrix} \in Z_{G}(R). 
        \end{equation}
        The group $G$ is quasi-split over $\QQ$ and is split when base change from $\QQ$ to $E$: 
        \begin{equation*}
            G_{E} \cong \GL_{3, E} \times \Gm_{,E}. 
        \end{equation*}
        We define $G_{0} = \ker(\mu)$ as the unitary group. Hence, we have 
        \begin{equation*}
            G(R) = Z_{G}(R) G_{0}(R),
        \end{equation*}
        for all $\mathbb{Z}$-algebras $R$. 
        Note that $G$ is reductive over $\ZZ_{l}$ for $l \nmid D$ (even for $l = 2$). 
    \item Let $H$ be the group scheme over $\mathbb{Z}$ such that for a $\mathbb{Z}$-algebras $R$, we have 
        \[
            H(R) = \{ (g,z) \in \GL_2(R) \times (\mathcal{O} \otimes R)^{\times} | \det(g) = z\bar{z} \}. 
        \]
        The center $Z_{H}$ of the group $H$ is 
        \begin{equation}
        \label{eq: Z_{H}}
            Z_{H}(R) = \{ (\begin{pmatrix}
                               t & 0 \\
                               0 & t 
                           \end{pmatrix}, z) \in \GL_2(R) \times (\mathcal{O} \otimes R)^{\times} | t^{2} = z\bar{z}\}. 
        \end{equation}
        The group $H$ is quasi-split and is split when base change from $\QQ$ to $E$: 
        \begin{equation*}
            H_{E} \cong \GL_{2, E} \times \Gm_{, E}. 
        \end{equation*}
    \item We have the following two maps:
            \[
                \iota \colon H \hookrightarrow G,  \quad (\begin{pmatrix} 
                                                        a & b \\
                                                        c & d \\
                                                     \end{pmatrix}, z) \mapsto (\begin{pmatrix}
                                                                                    a & 0 & b \\
                                                                                    0 & z & 0 \\
                                                                                    c & 0 & d 
                                                                                 \end{pmatrix}, z\bar{z})
            \]
        and 
                    \[p \colon H \twoheadrightarrow \GL_2, \quad (\begin{pmatrix} 
                                                                a & b \\
                                                                c & d \\
                                                             \end{pmatrix}, z) \mapsto \begin{pmatrix} 
                                                                                        a & b \\
                                                                                        c & d \\
                                                                                     \end{pmatrix}.           
        \]
        The relation between the groups $\GL_2$, $H$ and $G$ can be summarized in the diagram: 
        \[
            \begin{tikzcd}
                H \ar[r, hook] \ar[d, twoheadrightarrow] & G  \\
                \GL_2  &  
            \end{tikzcd}. 
        \]
        \item 
            Via the map $\iota$, we define the group scheme $Z$ as 
            \begin{equation}
            \label{eq: Z}
                Z(R) := Z_{G}(R) \cap H(R) = \{ (\begin{pmatrix}
                                                    t & 0 \\
                                                    0 & t 
                                                \end{pmatrix}, t) \in \GL_2(R) \times R^{\times}\},
            \end{equation}
            which is the common center of $G$ and $H$. 
\end{enumerate}
\end{definition}

The definition of $H$ is natural for the following reasons. 
\par Let $J_2 \in \GL_2(E)$ be the Hermitian matrix 
            \[
                J_2 = \begin{pmatrix} 
                     0 & \delta^{-1} \\
                    -\delta^{-1} & 0
                    \end{pmatrix}, \quad \mathrm{where} \ \delta = \sqrt{-D}, 
            \]
            and let $\GU(J_2)$ be the group scheme over $\mathbb{Z}$ such that for all $\mathbb{Z}$-algebras $R$, we have 
            \[
                G(R) = \{(g, \mu) \in \GL_2(\mathcal{O} \otimes R) \times R^{\times} | \prescript{t}{}{\bar{g}}J_2g = \mu J_2\}. 
            \]
Let $\GU(J_2)^{\prime}$ be the subgroup scheme of $\GU(J_2)$ defined by  
\begin{equation*}
        \GU(J_2)^{\prime} = \{ g \in \GU(J_2)| \det(g) = \mu(g) \}. 
\end{equation*}
\begin{fact}
    We have
    \begin{equation*}
        \GU(J_2)^{\prime} = \GL_2, 
    \end{equation*}
    through which the similitude character $\mu$ is identified with the determinant character $\det$. Hence, we use the notation $\mu$ to denote the determinant character for $H$ and $\GL_2$. 
    
\end{fact}
\begin{proof}
     For $g \in \GU(J_2)^{\prime}$, we have $^{t}{\bar{g}}J_2g = J_2$, and if we left multiply both sides by $\left( \begin{matrix}
                                                                                                                            \delta & 0 \\
                                                                                                                            0 & \delta 
                                                                                                                         \end{matrix} \right)$, 
    we have 
    \begin{equation*}
        ^{t}{\bar{g}}\begin{pmatrix}
                        0 & 1 \\
                        -1 & 0
                    \end{pmatrix}g = \mu \begin{pmatrix}
                                        0 & 1 \\
                                        -1 & 0
                                     \end{pmatrix}. 
    \end{equation*}
    Hence, for $g = \left( \begin{matrix}
                              a & b \\
                              c & d
                            \end{matrix} \right)$, we have 
    \begin{align*}
        \begin{pmatrix}
            -\bar{c} & \bar{a} \\
            -\bar{d} & \bar{b}
        \end{pmatrix}=\begin{pmatrix}
                        \bar{a} & \bar{c} \\
                        \bar{b} & \bar{d} 
                      \end{pmatrix} \begin{pmatrix}
                                        0 & 1 \\
                                        -1 & 0 
                                    \end{pmatrix} & = \mu \begin{pmatrix}
                                                            0 & 1 \\
                                                            -1 & 0 
                                                        \end{pmatrix} \begin{pmatrix}
                                                                        d & -b \\
                                                                        -c & a
                                                                      \end{pmatrix} \frac{1}{ad - bc} \\ 
        & = \begin{pmatrix}
                -c & a \\
                -d & b 
            \end{pmatrix} \frac{\mu}{ad - bc} = \begin{pmatrix}
                                                    -c & a \\
                                                    -d & b 
                                                 \end{pmatrix}. 
    \end{align*}
    So we have $a = \bar{a}, b = \bar{b}, c = \bar{c}$ and $d = \bar{d}$, which means $g \in \GL_2$. 
    \par For $g \in \GL_2$, the above process can be reversed. 
\end{proof}

\begin{remark}
The natural subgroup $\GU(J_2)^{\prime} \boxtimes \Res_{\mathcal{O}/\ZZ}(\Gm)$ of $G$, which is defined by 
\begin{equation*}
    \GU(J_2)^{\prime} \boxtimes \Res_{\mathcal{O}/\ZZ}(\Gm)(R) = \{(g, z) \in \GU(J_2)^{\prime}(R) \times (\mathcal{O} \otimes R)^{\times} | \mu(g) = z \bar{z} \}
\end{equation*}
for all $\ZZ$-algebras $R$, can be identified with $H$. 
\end{remark}

\begin{convention}
    When we consider $H$ as an algebraic group over $\QQ$, we are free to use the notation
    \begin{equation*}
        H = \GU(J_2)^{\prime} \boxtimes E^{\times} = \GL_2 \boxtimes E^{\times}. 
    \end{equation*}
\end{convention}

\subsubsection{Algebraic representations}

\begin{convention}
    In this subsubsection, the algebraic groups $G$, $H$ and $\GL_{2}$ are all defined over $\QQ$. 
\end{convention}

\par We first define algebraic representations of $G$. 

\begin{definition}
\label{def:rep_G}
    \begin{enumerate}
        \item Let $\chi_i,i = 1 ,\dots, 4$, be the four characters of the diagonal torus $T/E$ of $G$ such that 
        \begin{align*}
          &  \chi_1 \colon \begin{pmatrix}
                     x & 0 & 0 \\
                     0 & z & 0 \\
                     0 & 0 & \frac{z\bar{z}}{\bar{x}}
                   \end{pmatrix} \mapsto x,\  \chi_2 \colon \begin{pmatrix}
                                                     x & 0 & 0 \\
                                                     0 & z & 0 \\
                                                     0 & 0 & \frac{z\bar{z}}{\bar{x}}
                                                   \end{pmatrix} \mapsto \bar{x}, \\
          &  \chi_3 \colon \begin{pmatrix}
                     x & 0 & 0 \\
                     0 & z & 0 \\
                     0 & 0 & \frac{z\bar{z}}{\bar{x}}
                   \end{pmatrix} \mapsto \frac{\det}{\mu},\  \chi_4 \colon \begin{pmatrix}
                                                                        x & 0 & 0 \\
                                                                        0 & z & 0 \\
                                                                        0 & 0 & \frac{z\bar{z}}{\bar{x}}
                                                                  \end{pmatrix} \mapsto \frac{\overline{\det}}{\mu}.                                           
        \end{align*}
        Here, $x, \bar{x}, z, \bar{z}$ are $4$ independent variables. 
        \item In this paper, we also use the notation $(\mu_1, \mu_2, \mu_3; d)$ to denote the character of the diagonal torus $T$ such that 
              \begin{equation*}
                   t = \begin{pmatrix}
                         a & 0 & 0 \\
                         0 & b & 0 \\
                         0 & 0 & c 
                       \end{pmatrix} \mapsto a^{\mu_1} b^{\mu_2} c^{\mu_3} \mu(t)^{d}. 
              \end{equation*}
        We denote by $\lambda(\mu_1, \mu_2, \mu_3; d)$ the irreducible algebraic representation of $G$ with highest weight $(\mu_1, \mu_2, \mu_3; d)$. When the action of $\mu$ is clear, sometimes we write $(\mu_1, \mu_2, \mu_3)$ for $\lambda(\mu_1, \mu_2, \mu_3; d)$. 
        \item For $a_1, a_2 \ge 0$, let $V^{a_1, a_2}$ denote the representation of $G$ of highest weight $a_1\chi_1 + a_2\chi_2$.
        \item For each representation $V$ of $G$, we write $V\{a_3,a_4\}$ for its twists by $\chi_3^{a_3}\chi_4^{a_4}$.
    \end{enumerate}
\end{definition}
Each irreducible representation of $G$ is $V^{a_1,a_2}\{a_3,a_4\}$ for some $a_1, a_2, a_3, a_4 \in \mathbb{Z}$ with $a_1 \ge 0, a_2 \ge 0$. 

\begin{proposition}
    \begin{enumerate}
        \item Each irreducible representation of $G$ is $V^{a_1,a_2}\{a_3,a_4\}$ for some $a_1, a_2, a_3, a_4 \in \mathbb{Z}$ with $a_1 \ge 0, a_2 \ge 0$.
        \item The highest weight of $V^{a_1,a_2}\{a_3,a_4\}$ is $(a_1 + a_3 - a_4, a_3 - a_4, a_3 - (a_2 + a_4); 2a_4 + a_2 - a_3)$. 
        \item The contragredient representation $D^{a_1, a_2}$ of $V^{a_1, a_2}$ is $V^{a_2, a_1}\{-(a_1 + a_2), -(a_1 + a_2) \}$. 
    \end{enumerate}
\end{proposition}

\begin{proof}
  \begin{enumerate}
      \item The special unitary group $SU(J)$ is the subgroup of $G$ such that $\det = \mu = 1$, which is a semisimple algebraic group. It is enough to consider the algebraic representation of $\SU(J) \cong \SL_{3, E}$. For $\SL_{3, E}$, the fundamental weights are $\omega_1 = (1, 0)$ and  $\omega_2 = (1, 1)$. 
      For any $t = \begin{pmatrix} 
                             x & 0 & 0 \\
                             0 & z & 0 \\
                             0 & 0 & \frac{1}{xz}
                          \end{pmatrix}  \in T_{E}$, we have $\chi_1(t) = x = \omega_1(t)$ and $\chi_2(t) = \bar{x} = xz = \omega_2(t)$. 
        Hence, by highest weight theory, every irreducible algebraic representation of $\SU(J)_{E}$ is determined by 
        \begin{equation*}
            a_1 \omega_1 + a_2 \omega_2 = a_1 \chi_1 + a_2 \chi_2,
        \end{equation*}
        for $a_1 \ge 0$ and $a_2 \ge 0$. Hence, every irreducible algebraic representation of $G$ is of the form $V^{a_1,a_2}\{a_3,a_4\}$ for some $a_1, a_2, a_3, a_4 \in \mathbb{Z}$ with $a_1 \ge 0, a_2 \ge 0$.
      \item The highest weight of $V^{a_1,a_2}\{a_3,a_4\}$ is 
            \begin{align*}
                t = \left( \begin{pmatrix} 
                             x & 0 & 0 \\
                             0 & z & 0 \\
                             0 & 0 & \frac{z\bar{z}}{\bar{x}}
                          \end{pmatrix}, z\bar{z}\right) \in T_{E} \mapsto \, &x^{a_1} (\bar{x})^{a_2}  (\frac{xz}{\bar{x}})^{b_1}(\frac{\bar{x}\bar{z}}{x})^{b_2} \\
                          = &x^{a_1 + b_1 - b_2}\bar{x}^{a_2 + b_2 - b_1}z^{b_1} \bar{z}^{b_2} \\
                          = & x^{a_1 + b_1 - b_2} z^{b_1 - b_2} (\frac{z \bar{z}}{\bar{x}})^{b_1 - (a_2 + b_2)}(z\bar{z})^{2b_2 + a_2 - b_1}. 
            \end{align*}
      \item First, we have 
            \begin{align*}
                a_1 \omega_1 + a_2 \omega_2 &= a_1 (\frac{2}{3}, -\frac{1}{3}, -\frac{1}{3}) + a_2 (\frac{1}{3}, \frac{1}{3}, -\frac{2}{3}) \\
                                            &= (\frac{2a_1 + a_2}{3}, \frac{-a_1 + a_2}{3}, \frac{-a_1 - 2a_2}{3}). 
            \end{align*}
            Hence, when restrict to $\SU(J)$, the highest weight of $D^{a_1, a_2} = (V^{a_1, a_2})^{\vee}$ is 
            \begin{align*}
                (\frac{a_1 + 2a_2}{3},  \frac{a_1 - a_2}{3},  \frac{-2a_1 - a_2}{3}) = a_2 \omega_1 + a_1 \omega_2. 
            \end{align*}
          Second, for the $t$ in (2), we have 
          \begin{align*}
              \chi_1(t) = x & = x^{\frac{2}{3}}z^{-\frac{1}{3}}(\frac{z\bar{z}}{\bar{x}})^{-\frac{1}{3}} \cdot (x \cdot z \cdot (\frac{z\bar{z}}{\bar{x}}))^{\frac{1}{3}} \\
                            & = \omega_1(t) \chi_3(t)^{\frac{2}{3}} \chi_4(t)^{\frac{1}{3}}.
          \end{align*}
          Similarly, we have 
          \begin{equation*}
              \chi_2(t) = \omega_2(t) \chi_3(t)^{\frac{1}{3}} \chi_4(t)^{\frac{2}{3}}. 
          \end{equation*}
          So the highest weight of $D^{a_1, a_2}$ is $\omega_1^{a_2}\omega_2^{a_1} \chi_3^{\frac{-2a_1 - a_2}{3}} \chi_4^{\frac{-a_1 - 2a_2}{3}}$ which is 
          \begin{equation*}
              \chi_1^{a_2} \chi_2^{a_1} \chi_3^{-(a_1 + a_2)} \chi_4^{-(a_1 + a_2)}. 
          \end{equation*}
  \end{enumerate}
\end{proof}
\begin{remark}
\label{remark: parameter of repn transform}
    For later use, we record the highest weight and central character in the following table. 
    \begin{equation}
        \begin{array}{|c|c|c|}
            \hline  \text{Representation}& \text{Highest weight} & \text{central character} \\
            \hline V = V^{a, b} \{r, s\} & (a + r - s, r - s, r - s - b; 2s - r + b) &  a + b + r + s\\
            \hline D = V^{\vee} & (b + s - r, s - r, s - r- a; -2s + r - b) & -a -b - r -s \\
            \hline 
            \end{array}
    \end{equation}
\end{remark}

Now we define irreducible representations of $H$. 
\begin{definition}
\label{def:repn_H}
    \begin{enumerate}
        \item For each representation $W$ of $H$, we write $W\{a_3,a_4\}$ for its twists by $\chi_3^{a_3}\chi_4^{a_4}$ through $\iota \colon H \hookrightarrow G$. 
        \item In this paper, we also use the notation $(\mu_1, \mu_2; d)$ to denote the character of the diagonal torus
              \begin{equation*}
                   t = (\begin{pmatrix}
                         a & 0 \\
                         0 & b \\
                        \end{pmatrix}; z) \in H \mapsto a^{\mu_1} b^{\mu_2} z^{d}. 
              \end{equation*}
              We denote by $\lambda^{\prime}(\mu_1, \mu_2; d)$ the irreducible algebraic representation of $H$ with highest weight $(\mu_1, \mu_2; d)$.
    \end{enumerate}
\end{definition}

\begin{convention}
    When the action of $\mathrm{Res}_{E/\QQ}\Gm$ in $H$ is clear, we sometimes write $(\mu_1, \mu_2)$ instead of $(\mu_1, \mu_2; d)$. 
\end{convention}

\begin{remark}
    Along the map $\iota$, the restriction of $\chi_3$ and $\chi_4$ is 
    \begin{align*}
        & \chi_3 \colon (\begin{pmatrix}
                    a & b \\
                    c & d
                 \end{pmatrix}, z) \in H \mapsto z \\
         & \chi_4 \colon (\begin{pmatrix}
                    a & b \\
                    c & d
                 \end{pmatrix}, z) \in H \mapsto \bar{z}. 
    \end{align*}
    So the representation $W\{a_3,a_4\}$ is the representation $W$ with a twist by the character $z \mapsto z^{a_3}\bar{z}^{a_4}$. 
\end{remark}

Hence, we can get the following proposition: 
\begin{proposition}
    \begin{enumerate}
        \item  For $b_1 \ge 0$, let $W^{b_1}$ be the representation $\mathrm{Sym}^{b_1}V_{2}$ of $H$, where $V_{2}$ denotes the pullback of the standard representation of $\GL_2$ through the map $p \colon: H \twoheadrightarrow \GL_2$. \\
        Then every irreducible representation of $H$ has the form $W^{b_1}\{b_2,b_3\}$ for some $b_1, b_2,b_3 \in \mathbb{Z}$ with $b_1 \ge 0$.
        \item The contragredient of $W^{b_1}$ is $W^{b_1}\{-b_1, -b_1\}$. 
    \end{enumerate}
\end{proposition}

\subsubsection{Branching law}

\par The following is a corollary of the classical branching law for $\GL_2 \subset \GL_3$ \cite[Lemma 8.3.1]{Goodman09}.
\begin{proposition}
    For an irreducible representation $V^{a_1,a_2}\{b_1,b_2\}$ of $G$ satisfying: 
    \begin{align*}
        & 0 \le -b_1 \le a_1, \\
        & 0 \le -b_2 \le a_2, 
    \end{align*}
    we have an embedding in the category of representations of $H$:
    \[                
        W^{n} \hookrightarrow \iota^{*} V^{a_1,a_2}\{b_1,b_2\}, 
    \]
    where $n = a_1 + a_2 + b_1 + b_2$ and $\iota^{*}$ denotes the restriction of the representation of $G$ to $H$ through $\iota$. 
\end{proposition}
\begin{proof}
    The highest weight of $W^{n}$ is $(a_1 + a_2 + b_1 + b_2, 0)$ and the highest weight of $V^{a_1,a_2}\{b_1,b_2\}$ is 
    $(a_1 + b_1 - b_2 , b_1 - b_2, b_1 - (a_2 + b_2); 2b_2 + a_2 - b_1)$. 
    \par First, the common center of $G$ and $H$ is isomorphic to $\Gm$ via 
         \begin{equation*}
             z\in \Gm \mapsto (\begin{pmatrix}
                                z & 0 \\
                                0 & z
                              \end{pmatrix},z)\in H. 
         \end{equation*}
         Its action on $W^{n}$ is via
         \begin{equation*}
              z \mapsto z^{n}
         \end{equation*}
         and its action on $\iota^{*} V^{a_1,a_2}\{b_1,b_2\}$\ is 
         \begin{equation*}
             z \mapsto z^{a_1 + a_2 + b_1 + b_2}. 
         \end{equation*}
         Hence, their central actions are compatible. 
    \par Second, the highest weight of $W^{n}\{-(2b_2 + a_2 - b_1), -(2b_2 + a_2 - b_1)\}$ is $(a_1 + 2b_1 - b_2, b_1 -2b_2 - a_2)$. Up to conjugation, by the classical branching $\GL_2 \subset \GL_3$ \cite[Lemma 8.3.1]{Goodman09}, in order to have the branching 
    \begin{equation*}
         W^{n}\{-(2b_2 + a_2 - b_1), -(2b_2 + a_2 - b_1)\} \hookrightarrow (\iota^{*} V^{a_1,a_2}\{b_1,b_2\}) \{-(2b_2 + a_2 - b_1), -(2b_2 + a_2 - b_1)\}
    \end{equation*}
    we must have 
    \begin{equation*}
        a_1 + b_1 - b_2 \ge a_1 + 2b_1 - b_2 \ge b_1 - b_2 \ge b_1 -2b_2 -a_2 \ge b_1 - (a_2 + b_2), 
    \end{equation*}
    which is equivalent to 
    \begin{align*}
        & 0 \le -b_1 \le a_1, \\
        & 0 \le -b_2 \le a_2. 
    \end{align*}
\end{proof}

\begin{remark}
In this paper, we only consider the case $W^{n} \hookrightarrow \iota^{*} V^{a, b}\{r, s\}$ when $n = a + b + r + s$ and will write $W$ instead of $W^{n}$ and $V$ instead of $V^{a,b} \{r, s\}$ to simplify notations. We will also denote by $D$ the contragredient of $V$. 
\end{remark}

\begin{notation}
    \begin{itemize}
        \item Throughout the paper, we always let $n = a + b + r + s$. 
        \item Throughout the paper, we always use $V$ to denote the algebraic representation $V = V^{a, b} \{r, s\}$ of $G$, $D$ to denote the contragredient of $V$ and $W = W^{n}$ to denote the algebraic representation of $H$ except in the situation where we will explain the meaning. 
        \item For a representation $V$ of $G$, $H$ or $\GL_2$ and an integer $d$, we use the notation $V(d)$ denote the representation $V$ with a twist of the character $\mu^{d}$. 
    \end{itemize}
   
\end{notation}

\subsection{Lie groups and representations}
\label{SS: Lie groups}
\subsubsection{Lie groups and Lie algebras}
\begin{notation}
To describe the maximal compact subgroup of $G(\RR)$ easily, when considering representations of Lie groups, we use the Hermitian form 
\begin{equation*}
    J^{\prime} = \begin{pmatrix}
                    1 & 0 & 0 \\
                    0 & 1 & 0 \\
                    0 & 0 & -1 
                 \end{pmatrix} 
\end{equation*}
to define the group $G = \GU(J^{\prime})$ and use the Hermitian form 
\begin{equation*}
    J_2^{\prime} = \begin{pmatrix}
                    1 & 0 \\
                    0 & -1 
                 \end{pmatrix} 
\end{equation*}
to define the group $\GU(J_2^{\prime})^{\prime}$, and then define the group $H = \GU(J_2^{\prime})^{\prime} \boxtimes (\Res_{E/\QQ}(\Gm))$. 
The definition of $G$ (resp., $H$) is the same as in Definition \ref{defn:G_H} except that we use $J^{\prime}$ (resp., $J^{\prime}_{2}$) instead of $J$ (resp., $J_{2}$).
\end{notation}

\begin{proposition}
    \begin{enumerate}
        \item If we let 
              \begin{equation*}
                  C = \begin{pmatrix}
                         \frac{D^{\frac{1}{4}}}{\sqrt{2}} & 0 & \frac{D^{\frac{1}{4}}}{\sqrt{2}} \\
                         0 & 1 & 0 \\
                          -\frac{D^{\frac{1}{4}}}{\sqrt{-2}} & 0 & \frac{D^{\frac{1}{4}}}{\sqrt{-2}}
                      \end{pmatrix}, 
              \end{equation*}
              then we have the following isomorphism of real Lie groups: 
              \begin{equation*}
                  \begin{tikzcd}[row sep = 0]
                   \GU(J)(\RR)  \ar[r, "\cong"] & \GU(J^{\prime})(\RR) \\
                      g \ar[r, mapsto] & C^{-1}gC
                  \end{tikzcd}. 
              \end{equation*}
        \item If we let 
              \begin{equation*}
                  C_2 = \begin{pmatrix}
                         \frac{D^{\frac{1}{4}}}{\sqrt{2}} & \frac{D^{\frac{1}{4}}}{\sqrt{2}} \\
                          -\frac{D^{\frac{1}{4}}}{\sqrt{-2}} &  \frac{D^{\frac{1}{4}}}{\sqrt{-2}}
                      \end{pmatrix}, 
              \end{equation*}
              then we have the following isomorphism 
              \begin{equation*}
                  \begin{tikzcd}[row sep = 0]
                   \GU(J_{2})^{\prime}(\RR)  \ar[r, "\cong"] & \GU(J_{2}^{\prime})^{\prime}(\RR) \\
                      g \ar[r, mapsto] & C_{2}^{-1}gC_{2}
                  \end{tikzcd}. 
              \end{equation*}
    \end{enumerate}
\end{proposition}

\begin{proof}
    The proof follows from direct computation. 
\end{proof}

\begin{convention}
    In this paper, we write $G(\RR)$ (resp., $H(\RR)$ and $\GL_{2}(\RR)$) for $\GU(J^{\prime})(\RR)$ (resp., 
    $(\GU(J_{2}^{\prime})^{\prime} \boxtimes E^{\times})(\RR)$ and $\GU(J_{2}^{\prime})^{\prime}(\RR)$). 
\end{convention}

Let $K_{G}$ be the real Lie group $A_{G}(\U(2) \times \U(1))(\RR)$,  $K_{H}$ be the real Lie group $A_{H}(\U(1) \times \U(1))(\RR)$ and $K_{\GL_2}$ be the real Lie group $A_{\GL_{2}}\U(1)(\RR)$. Let $\lieg$ be the Lie algebra of $G(\RR)$, $\lieh$ be the Lie algebra of $H(\RR)$ and $\liegl_{2}$ be the Lie algebra of $\GL_2(\RR)$. Then by the 
Iwasawa decomposition, we have the following decomposition of complexified Lie algebra: 
\begin{align*}
    & \lieg_{\CC} = \liek_{G, \CC} \oplus \liep_{\CC}, \\
    & \lieh_{\CC} = \liek_{H, \CC} \oplus \liep_{H, \CC}, \\
    & \liegl_{2, \CC} = \liek_{H, \CC} \oplus \liep_{H, \CC}, 
\end{align*}
where $\liek_{\CC}$, $\liek_{H, \CC}$ and $\liek_{\GL_2, \CC}$ are the complexified Lie algebras of $K_{G}$, $K_{H}$ and $K_{\GL_{2}}$, which is
\begin{align*}
    & \liek_{\CC} = \{\begin{pmatrix}
                        * & * & 0 \\
                        * & * & 0 \\
                        0 & 0 & * 
                      \end{pmatrix}\} \subset \lieg_{\CC}, \\ 
    & \liek_{H, \CC} = \{\begin{pmatrix}
                            * & 0 \\
                            0 & * 
                        \end{pmatrix}\} \oplus \liegl_{1, \CC} \subset \lieh_{\CC}, \\
    & \liek_{\GL_{2}, \CC} = \{\begin{pmatrix}
                            * & 0 \\
                            0 & * 
                        \end{pmatrix}\} \subset \liegl_{2, \CC}; 
\end{align*}
and $\liep_{\CC}$, $\liep_{H, \CC}$ and $\liep_{\GL_{2}, \CC}$ are the sub Lie algebras of $\lieg_{\CC}$, $\lieh_{\CC}$ and $\liegl_{2, \CC}$: 
\begin{align*}
    & \liep_{\CC} = \{\begin{pmatrix}
                    0 & 0 & * \\
                    0 & 0 & * \\
                    * & * & 0 
                  \end{pmatrix}\} \subset \lieg_{\CC}, \\ 
    & \liep_{H, \CC} = \{\begin{pmatrix}
                            0 & * \\
                            * & 0 
                        \end{pmatrix}\} \subset \lieh_{\CC}, \\
    & \liep_{\GL_{2}, \CC} = \{\begin{pmatrix}
                            0 & * \\
                            * & 0 
                        \end{pmatrix}\} \subset \liegl_{2, \CC}. 
\end{align*}
We also have a natural embedding $\iota_{*} \colon \lieh_{\CC} \hookrightarrow \lieg_{\CC}$ induced by the map $\iota \colon H \hookrightarrow G$ that maps $\liep_{H, \CC}$ to $\liep_{\CC}$. Explicitly, the map is given by 
\begin{equation*}
    \begin{pmatrix}
            0 & A \\
            B & 0 
    \end{pmatrix} \in  \liep_{H, \CC}  \mapsto \begin{pmatrix}
                                                    0 & 0 & A \\
                                                    0 & 0 & 0 \\
                                                    B & 0 & 0 
                                                \end{pmatrix} \in \liep_{\CC}. 
\end{equation*}

\subsubsection{Representation Theory of compact Lie groups}
\label{SSS: repn of compact Lie group}
We recall some basic facts of representation theory of the compact Lie group $(\U(2) \times \U(1))(\RR)$.

\par The compact torus $T_c$ is $(\U(1) \times \U(1) \times \U(1))(\RR)$. Hence, the dominant weights of $T_c$ are parametrized by the triples $(\lambda_1, \lambda_2, \lambda_3)$ such that $\lambda_1, \lambda_2, \lambda_3 \in \ZZ$ and $\lambda_1 \ge \lambda_2 \ge \lambda_3$. By highest weight theory of representations of compact Lie groups, for each dominant weight $(\lambda_1, \lambda_2, \lambda_3)$, there is an unique irreducible representation $\tau_{(\lambda_1, \lambda_2, \lambda_3)}$ of $(\U(2) \times \U(1))(\RR)$ whose highest weight is $(\lambda_1, \lambda_2, \lambda_3)$. 

By the Weyl dimension formula, we have 
\begin{equation*}
    \dim \tau_{(\lambda_1, \lambda_2, \lambda_3)} = \lambda_1 - \lambda_2 + 1. 
\end{equation*}

We denote $d = \lambda_1 - \lambda_2 + 1$. 
Actually, there exists a basis $(\nu_s)_{0 \le s \le d}$ of $\tau_{(\lambda_1, \lambda_2, \lambda_3)}$ such that 
\begin{align}
\label{eqn: action on K-type}
    & \left( \begin{matrix}
                1 & 0 \\
                0 & 0 
             \end{matrix} \right) \nu_s = (s + \lambda_2) \nu_s, \\
     & \left( \begin{matrix}
                0 & 0 \\
                0 & 1 
             \end{matrix} \right) \nu_s = (-s + \lambda_1) \nu_s, \\ 
      & \left( \begin{matrix}
                0 & 1 \\
                0 & 0 
             \end{matrix} \right) \nu_s = (s + 1) \nu_{s + 1}, \\ 
      & \left( \begin{matrix}
                0 & 0 \\
                1 & 0 
             \end{matrix} \right) \nu_s = (d - s + 1) \nu_{s - 1}.   
\end{align}

\begin{convention}
    \begin{itemize}
        \item We will use $K_{G}$ (resp., $K_{H}$ and $K_{\GL_{2}}$) to denote $(\U(2) \times \U(1))(\RR)$ (resp., $(\U(1) \times \U(1))(\RR)$ and $\U(1)(\RR)$) when the action $A_{G}$ is clear.
        \item We will use the coordinate system $(\lambda_1, \lambda_2, \lambda_3)$ in subsection \ref{SS:discrete series} and \ref{SS: Hodge decomposition} to denote the Blattner parameter and Harish-Chandra parameter of a discrete series representation of a real reductive Lie group. 
    \end{itemize}
\end{convention}

\subsection{Shimura varieties}
\label{SS:Shimura varieties}
In this subsection, we define the Shimura varieties associated to $G$, $H$ and $\GL_{2}$.

\begin{notation}
    Throughout the paper, we let $\mathbb{S} = \Res_{\CC/\RR} \Gm$ be the Deligne torus. 
\end{notation}

\begin{definition}
\label{def: Shimura variety}
    \begin{enumerate}
        \item Let $X_{G}$ be the $G(\RR)$-conjugacy class of the morphism $h \colon \mathbb{S} \rightarrow G_{\RR}$ given by 
              \begin{equation*}
                  z = x + iy \mapsto \begin{pmatrix}
                                        x & 0 & y \\
                                        0 & z & 0 \\
                                        -y & 0 & x
                                     \end{pmatrix}.
              \end{equation*}
              The pair $(G, X_{G})$ is a (pure) Shimura datum associated to $G$. 
              \par The Hodge cocharacter $\mu_{\CC}$ of $\Gm_{, \CC}$ associated to $h$ is given on $\CC$-points by 
              \begin{equation*}
                  z \mapsto \begin{pmatrix}
                                z & 0 & 0 \\
                                0 & z & 0 \\
                                0 & 0 & 1 
                            \end{pmatrix}.
              \end{equation*}
              Since it is not complex conjugation invariant, we can see that the reflex field of $(G, X_{G})$ is $E$ with a the fixed complex embedding of $E$ into $\CC$ \cite[Lemma 4.2., P15]{LR92}. 
              For any neat compact open subgroup $L$ of $G(\AAA_f)$, there is a smooth \textit{quasi-projective} $E$-scheme $\Sh_{G}(L)$ such that as complex analytic varieties, we have 
              \begin{equation*}
                  \Sh_{G}(L)_{\CC}^{an} = G(\QQ) \backslash (X_{G} \times G(\AAA_f) / L), 
              \end{equation*}
              where $\Sh_{G}(L)_{\CC}^{an}$ is the analytification of the base change of $\Sh_{G}(L)$ to $\CC$. 
              \item  Let $X_{H}$ be the $H(\RR)$-conjugacy class of the morphism $h \colon \mathbb{S} \rightarrow H_{\RR}$ given by 
              \begin{equation*}
                  z = x + iy \mapsto (\begin{pmatrix}
                                        x & y \\
                                        -y & x
                                     \end{pmatrix}, z). 
              \end{equation*}
              The pair $(H, X_H)$ is a (pure) Shimura datum associated to $H$. 
              \par The Hodge cocharacter $\mu_{\CC}$ of $\Gm_{, \CC}$ associated to $h$ is given on $\CC$-points by 
              \begin{equation*}
                  z \mapsto (\begin{pmatrix}
                                z & 0 \\
                                0 & 1 
                            \end{pmatrix}, z).
              \end{equation*}
              The reflex field of $(H, X_H)$ is $E$ with a fixed complex embedding of $E$ into $\CC$. 
              For any neat compact open subgroup $K$ of $H(\AAA_f)$, there is a smooth \textit{quasi-projective} $E$-scheme $\Sh_{H}(K)$ such that as complex analytic varieties, we have 
              \begin{equation*}
                  \Sh_{H}(K)_{\CC}^{an} = H(\QQ) \backslash (X_H \times H(\AAA_f) / K), 
              \end{equation*}
              where $\Sh_{H}(K)_{\CC}^{an}$ is the analytification of the base change of $\Sh_{H}(K)$ to $\CC$.
        \item Let $X_{\GL_2}$ be the $\GL_2(\RR)$-conjugacy class of the morphism $h \colon \mathbb{S} \rightarrow \GL_{2,\RR}$ that is given by 
              \begin{equation*}
                  z \mapsto \begin{pmatrix}
                                x & y \\
                                -y & x
                              \end{pmatrix},
              \end{equation*}
              The pair $(\GL_2, X_{\GL_2})$ is a (pure) Shimura datum associated to $\GL_2$. 
              \par The Hodge cocharacter $\mu_{\CC}$ of $\Gm_{, \CC}$ associated to $h$ is given on $\CC$-points by 
              \begin{equation*}
                  z \mapsto \begin{pmatrix}
                                z & 0 \\
                                0 & 1 
                            \end{pmatrix}.
              \end{equation*}
              The reflex field of $(\GL_2, X_{\GL_{2}})$ is $\QQ$. 
              For any neat compact open subgroup $K^{\prime}$ of $\GL_{2}(\AAA_f)$, there is a smooth \textit{quasi-projective} $\QQ$-scheme $\Sh_{\GL_{2}}(K^{\prime})$ such that as complex analytic varieties, we have 
              \begin{equation*}
                    \Sh_{\GL_2}(K^{\prime})_{\CC}^{an} = \GL_{2}(\QQ) \backslash (X_{\GL_2} \times \GL_2(\AAA_f) / K^{\prime}), 
              \end{equation*}
              where $\Sh_{\GL_2}(K^{\prime})_{\CC}^{an}$ means the analytification of the base change of $\Sh_{\GL_2}(K^{\prime})$ to $\CC$.
    \end{enumerate}
\end{definition}
\begin{remark}
    \begin{itemize}
        \item Our choice of Shimura datum is standard, bu is a little different from \cite[8.1.1]{LSZ22} (see the difference in \cite[Remark 8.1.1]{LSZ22}). 
        \item The maps $\iota \colon H \hookrightarrow G$ and $p \colon H \twoheadrightarrow \GL_2$ of algebraic groups induce the following $E$-morphisms of Shimura varieties for suitable neat compact open subgroup $L$, $K$ and $K^{\prime}$ of $G(\AAA_f)$, $H(\AAA_f)$ and $\GL_2(\AAA_f)$: 
        \begin{equation}
        \label{eq: morphism of Shimura varieties over E}
            \begin{tikzcd}
                \Sh_{H}(K) \ar[r, hook, "\iota"] \ar[d, rightarrow, "p"] & \Sh_{G}(L) \\
                \Sh_{\GL_2}(K^{\prime})_{E}  &  
            \end{tikzcd},
        \end{equation}
        where $\iota$ is a closed immersion and $\Sh_{\GL_2}(K^{\prime})_{E}$ is the base change of $\Sh_{\GL_2}(K^{\prime})$ to $E$. 
    \end{itemize}
\end{remark}

\begin{fact}
    \begin{itemize}
        \item The decomposition
              \begin{equation*}
                  \lieg_{\CC} = \liek_{G, \CC} \oplus \liep_{G, \CC}
              \end{equation*}
              can be further decomposed as 
              \begin{equation*}
                  \lieg_{\CC} = \liek_{G, \CC} \oplus \liep^{+}_{G, \CC} \oplus \liep^{-}_{G, \CC}, 
              \end{equation*}
              where $\liep^{+}_{G, \CC}$ is 
              \begin{equation*}
                  \{\begin{pmatrix}
                    0 & 0 & * \\
                    0 & 0 & * \\
                    0 & 0 & 0 
                  \end{pmatrix}\} \subset \lieg_{\CC}, 
              \end{equation*}
              on which $(z, \bar{z}) \in \mathbb{S}(\CC)$ acts by $\frac{z}{\bar{z}}$ and
              $\liep^{-}_{\CC}$ is 
              \begin{equation*}
                  \{\begin{pmatrix}
                    0 & 0 & 0 \\
                    0 & 0 & 0 \\
                    * & * & 0 
                  \end{pmatrix}\} \subset \lieg_{\CC}, 
              \end{equation*}
              on which $(z, \bar{z}) \in \mathbb{S}(\CC)$ acts by $\frac{\bar{z}}{z}$. 
        \item Similarly, we have the decomposition for $H$: 
              \begin{equation*}
                  \lieh_{\CC} = \liek_{H, \CC} \oplus \liep^{+}_{H, \CC} \oplus \liep^{-}_{H, \CC}, 
              \end{equation*}
              where $\liep^{+}_{H, \CC}$ is 
              \begin{equation*}
                   \{\begin{pmatrix}
                        0 & * \\
                        0 & 0 
                     \end{pmatrix}\} \subset \lieh_{\CC}
              \end{equation*} and $\liep^{-}_{H, \CC}$ is 
              \begin{equation*}
                   \{\begin{pmatrix}
                        0 & 0 \\
                        * & 0 
                     \end{pmatrix}\} \subset \lieh_{\CC}. 
              \end{equation*}
        \item We have a decomposition for $\GL_{2}$: 
              \begin{equation*}
                  \liegl_{2, \CC} = \liek_{\GL_{2}, \CC} \oplus \liep^{+}_{\GL_{2}, \CC} \oplus \liep^{-}_{\GL_{2}, \CC}, 
              \end{equation*}
              where $\liep^{+}_{\GL_{2}, \CC}$ is 
              \begin{equation*}
                   \{\begin{pmatrix}
                        0 & * \\
                        0 & 0 
                     \end{pmatrix}\} \subset \liegl_{2, \CC}
              \end{equation*} and $\liep^{-}_{\GL_{2}, \CC}$ is 
              \begin{equation*}
                   \{\begin{pmatrix}
                        0 & 0 \\
                        * & 0 
                     \end{pmatrix}\} \subset \liegl_{2, \CC}. 
              \end{equation*}
    \end{itemize}
\end{fact}

\begin{convention}
\label{Convention: M_{Q} and S_{Q}}
    \begin{enumerate}
        \item We use $M(K)$ (resp., $S(L)$) to denote $\Res_{E/\QQ}(\Sh_{H}(K))$ (resp., $\Res_{E/\QQ}(\Sh_{G}(L))$), where $\Res_{E/\QQ}$ is the Weil restriction. Hence, we have 
        \begin{align*}
            & S(L)_{\CC}^{an} = \Sh_{G}(L)_{\CC}^{an} \bigsqcup \overline{\Sh_{G}(L)_{\CC}^{an}}, \\
            & M(K)_{\CC}^{an} = \Sh_{H}(K)_{\CC}^{an} \bigsqcup \overline{\Sh_{H}(K)_{\CC}^{an}}, 
        \end{align*}
        where $\overline{\Sh_{G}(L)_{\CC}^{an}}$ (resp., $\overline{\Sh_{H}(K)_{\CC}^{an}}$) is the complex conjugate of $\Sh_{G}(L)_{\CC}^{an}$ (resp., $\Sh_{H}(K)_{\CC}^{an}$). 
        \item Since we compute complex regulators in this paper, an explicit choice of a level structure of a Shimura variety is not necessary. Hence, we will sometimes ignore the level structure in the notation related to Shimura varieties and write $S$, $M$ and $\Sh_{\GL_2}$ instead of $S(L)$, $M(K)$ and $\Sh_{\GL_2}(K^{\prime})$ when there is no confusion. If reader prefers, this can be explained as a Shimura variety with suitable neat level structure. 
        \item Applying the Weil restriction functor $\Res_{E/\QQ}$ to the diagram (\ref{eq: morphism of Shimura varieties over E}), we get the following $\QQ$-morphisms of Shimura varieties: 
        \begin{equation}
        \label{eq: morphism of Shimura varieties over Q}
            \begin{tikzcd}
                M \ar[r, hook, "\iota"] \ar[d, rightarrow, "p"] & S \\
                \Sh_{\GL_2}  &  
            \end{tikzcd},
        \end{equation}
        where $\iota$ is a closed immersion. 
    \end{enumerate}
\end{convention}

\subsection{Cohomology of Picard modular surfaces}
\label{SS: cohomology of Picard modular surfaces}

\begin{convention}
\label{convention: V^{Q} and W_{Q}}
    \begin{enumerate}
        \item For any $V \in \mathrm{Rep}_{E}(G)$ (resp., $W \in \mathrm{Rep}_{E} \in \mathrm{Rep}_{E}(H)$), we denote by $\overline{V}$ (resp., $\overline{W}$) the twist of $V$ (resp., $W$) by the non-trivial element of $\Gal(E/\QQ)$. 
        \item For any morphism $(f\colon V \rightarrow V^{\prime})\in \mathrm{Rep}_{E}(G)$ (resp., $(g\colon W \rightarrow W^{\prime}) \in \mathrm{Rep}_{E}(H)$), we denote by $\overline{f}\colon \overline{V} \rightarrow \overline{V^{\prime}}$ (resp., $\overline{g}\colon \overline{W} \rightarrow \overline{W^{\prime}}$) the twist of $f$ (resp., $g$) by the non-trivial element of $\Gal(E/\QQ)$. 
        \item We denote by $\mathrm{Rep}_{\QQ}(G)$ (resp., $\mathrm{Rep}_{\QQ}(H)$) the cateogry whose objects are $(V, \overline{V})$ (resp., $(W, \overline{W})$), where $V \in \mathrm{Rep}_{E}(G)$ (resp., $W \in \mathrm{Rep}_{E}(H)$); and whose morphisms are $(f, \overline{f})$ (resp., $(g, \overline{g})$), where $(f: V \rightarrow V^{\prime}) \in \mathrm{Rep}_{E}(G)$ (resp., $(g: W \rightarrow W^{\prime}) \in \mathrm{Rep}_{E}(H)$). 
        \item We denote by $V \in \mathrm{Rep}_{\QQ}(G)$ (resp., $W \in \mathrm{Rep}_{\QQ}(H)$) the pair $(V, \overline{V})$ (resp., $(W, \overline{W})$). When talking about weights of $V \in \mathrm{Rep}_{\QQ}(G)$ (resp., $W \in \mathrm{Rep}_{\QQ}(H)$), we offen mean the weights of $V \in \mathrm{Rep}_{E}(G)$ (resp., $W \in \mathrm{Rep}_{E}(H)$). 
        \item For any $V \in \mathrm{Rep}_{\QQ}(G)$ (resp., $W \in \mathrm{Rep}_{\QQ}(H)$), there is a canonical local system over $S(L)$ (resp., $M(K)$), which is also denoted by $V$ (resp., $W$). 
        \item We denote by $\mathrm{Rep}_{\CC}(G)$ (resp., $W \in \mathrm{Rep}_{\CC}(H)$ the category whose objects are $V \otimes_{\QQ} \CC$ (resp., $W \otimes_{\QQ} \CC$) such that $V \in \mathrm{Rep}_{\QQ}(G)$ (resp., $W \in \mathrm{Rep}_{\QQ}(H)$). 
    \end{enumerate}
\end{convention}

For any $V \in \mathrm{Rep}_{\CC}(G)$, let $\xi$ be the restriction of the inverse of the central character of $V$ to $A_{G}$ and let $L$ be an compact open subgroup of $G(\AAA_{f})$. Let $\mathcal{A}^{*}_{c}(S(L), V(2))$ be the de Rham complex of $C^{\infty}$-differential forms with compact support on $S(L)_{\CC}^{an}$ with values in the local system $V(2)$. Similarly, let $\mathcal{A}^{*}(S(L), V(2))$ be the usual de Rham complex and let $\mathcal{A}^{*}_{(2)}(S(L), V(2))$ be the complex of square integrable differential forms. If $\circ$ is the symbol $c$, $(2)$ or the empty symbol, we define 
\begin{equation*}
    \mathcal{A}^{*}_{\circ}(S, V(2)) = \varinjlim_{L} \mathcal{A}^{*}_{\circ}(S(L), V(2)). 
\end{equation*}

Let $\lieg$ be the Lie algebra of $G(\RR)$ and $K_{G} = A_{G}(\U(2) \times \U(1))(\RR)$ be the subgroup of $G(\RR)$ which is maximal modulo the center. For any $(\lieg_{\CC}, K_{G})$-module $M$, let $\mathcal{C}^{*}(\lieg_{\CC}, K_{G}; M)$ be the $(\lieg_{\CC}, K_{G})$-complex of $M$. Then we have $G(\AAA_{f})$-equivariant isomorphisms of complexes: \cite[VII \S 2]{Borel_Wallach}
\begin{align*}
    &  \mathcal{A}^{*}_{c}(S_{\QQ}, V(2)) \cong \mathcal{C}^{*}(\lieg_{\CC}, K_{G}; V(2) \otimes_{\CC} C_{c}^{\infty}(G(\QQ) \backslash G(\AAA), \xi)), \\
    &  \mathcal{A}^{*}(S_{\QQ}, V(2)) \cong \mathcal{C}^{*}(\lieg_{\CC}, K_{G}; V(2) \otimes_{\CC} C^{\infty}(G(\QQ) \backslash G(\AAA), \xi)) ,
\end{align*}
that are compatible with the inclusions $\mathcal{A}^{*}_{c}(S_{\QQ}, V(2)) \subset \mathcal{A}^{*}(S_{\QQ}, V(2))$ and 
\begin{equation*}
  \mathcal{C}^{*}(\lieg_{\CC}, K_{G}; V(2) \otimes_{\CC} C_{c}^{\infty}(G(\QQ) \backslash G(\AAA), \xi)) \subset \mathcal{C}^{*}(\lieg_{\CC}, K_{G}; V(2) \otimes_{\CC} C^{\infty}(G(\QQ) \backslash G(\AAA), \xi)). 
\end{equation*}

If we take cohomology on both sides, we obtain $G(\AAA_{f})$-equivariant isomorphisms 
\begin{align*}
    & \mathrm{H}^{*}_{dR, c}(S, V(2)) \cong \mathrm{H}^{*}(\lieg_{\CC}, K_{G}; V(2) \otimes_{\CC} C_{c}^{\infty}(G(\QQ) \backslash G(\AAA), \xi)), \\
    & \mathrm{H}^{*}_{dR}(S, V(2)) \cong \mathrm{H}^{*}(\lieg_{\CC}, K_{G}; V(2) \otimes_{\CC} C^{\infty}(G(\QQ) \backslash G(\AAA), \xi)), 
\end{align*}
which are compatible with the maps from cohomology with compact support to cohomology without support.  
\par Let $\mathrm{H}^{*}_{(2)}(S, V(2))$ be the cohomology of the complex $\mathcal{A}^{*}_{(2)}(S, V(2))$ which 
is the $L^{2}$ cohomology of $S$ with coefficients in $V(2)$. According to \cite{Borel_regularization83}, we have a $G(\AAA_f)$-equivariant isomorphism 
\begin{equation*}
    \mathrm{H}^{*}_{(2)}(S, V(2)) \cong \mathrm{H}^{*}(\lieg_{\CC}, K_{G}; V(2) \otimes_{\CC} C_{(2)}^{\infty}(G(\QQ) \backslash G(\AAA), \xi)), 
\end{equation*}
where $C_{(2)}^{\infty}(G(\QQ) \backslash G(\AAA), \xi)$ denotes the subspace of functions in $C^{\infty}(G(\QQ) \backslash G(\AAA), \xi)$ which are square integrable modulo the center of $G$.
We also use the notation $\mathrm{H}^{*}_{\text{cusp}}(S, V(2))$ to denote 
\begin{equation*}
     \mathrm{H}^{*}(\lieg_{\CC}, K_{G}; V(2) \otimes_{\CC} C_{\text{cusp}}^{\infty}(G(\QQ) \backslash G(\AAA), \xi)), 
\end{equation*}
where $C_{cusp}^{\infty}(G(\QQ) \backslash G(\AAA), \xi)$ denotes the subspace of functions in $C^{\infty}(G(\QQ) \backslash G(\AAA), \xi)$ which are cuspidal. 

By definition, we have the following maps: 
\begin{equation*}
    C_{\text{cusp}}^{\infty}(G(\QQ) \backslash G(\AAA), \xi) \rightarrow C_{c}^{\infty}(G(\QQ) \backslash G(\AAA), \xi) \rightarrow C_{(2)}^{\infty}(G(\QQ) \backslash G(\AAA), \xi) \rightarrow C^{\infty}(G(\QQ) \backslash G(\AAA), \xi), 
\end{equation*}
where the first map is defined by smooth truncation to a large set which is compact modulo the center, and the second and the third arrows are the natural inclusions. 
These maps induce maps of the corresponding cohomology groups:
\begin{equation*}
    \mathrm{H}^{*}_{\text{cusp}}(S, V(2)) \rightarrow \mathrm{H}^{*}_{dR, c}(S, V(2)) \rightarrow \mathrm{H}^{*}_{(2)}(S, V(2)) \rightarrow \mathrm{H}^{*}_{dR}(S, V(2)).  
\end{equation*}
Let 
\begin{equation*}
    \mathrm{H}^{*}_{dR, !}(S, V(2)) = \mathrm{Im}(\mathrm{H}^{*}_{dR, c}(S, V(2)) \rightarrow \mathrm{H}^{*}_{dR}(S, V(2)). 
\end{equation*}
be the interior de Rham cohomology. 
In general, by \cite[Theorem 7.5]{Borel_ENS74}, the map 
\begin{equation*}
    \mathrm{H}^{*}_{\text{cusp}}(S, V(2)) \rightarrow \mathrm{H}^{*}_{dR, !}(S, V(2))
\end{equation*}
is injective. 

\begin{proposition}
\label{Prop:coh_identity}
    If the algebraic representation $V$ is regular, we have 
    \begin{equation*}
        \mathrm{H}^{*}_{\text{cusp}}(S, V(2)) = \mathrm{H}^{*}_{(2)}(S, V(2)) = \mathrm{H}^{*}_{dR, !}(S, V(2)) = \mathrm{H}^{2}_{dR, !}(S, V(2)). 
    \end{equation*}
\end{proposition}

\begin{proof}
    For the equality besides the last one, by the fact that the map 
    \begin{equation*}
         \mathrm{H}^{*}_{\text{cusp}}(S, V(2)) \rightarrow \mathrm{H}^{*}_{dR, !}(S, V(2))
    \end{equation*}
    is injective, it suffices to prove 
    \begin{equation*}
        \mathrm{H}^{*}_{\text{cusp}}(S, V(2)) = \mathrm{H}^{*}_{(2)}(S, V(2)). 
    \end{equation*}
     By \cite[Theorem 4]{Borel_Bull80} (which applies for $\rank G = \rank K_{G}$), we have 
     \begin{equation*}
         \mathrm{H}^{*}_{dR, !}(S, V(2)) = \bigoplus_{\pi = \pi_{\infty} \otimes \pi_f} m(\pi) \mathrm{H}^{*}(\lieg_{\CC}, K_{G}; V(2) \otimes_{\CC} \pi_{\infty})\otimes \pi_f, 
     \end{equation*}
     where $\pi$ runs over the discrete spectrum of $L^{2}(G(\QQ) \backslash G(\AAA), \xi)$, and $\xi$ is the restriction of the central character of $V^{\vee}$ to $A_{G}$.
     Let $\pi = \pi_f \otimes \pi_{\infty}$ be such an automorphic representation. By \cite[Theorem 5.6]{VZ84}\footnote{Although $\pi_{\infty}$ is not unitary as a representation of $G(\RR)^{+}$, if 
     we restrict to the semisimple part we could make it unitary. Hence, we can use the theorem in \cite{VZ84}, because the authors assume the real Lie group is semisimple.},
     the assumption $\mathrm{H}^{*}(\lieg_{\CC}, K_{G}; V(2) \otimes_{\CC} \pi_{\infty}) \neq 0$ implies $\pi_{\infty} = A_{\lieq}(\lambda)$, which is a cohomological induction from a $\theta$-stable parabolic subalgebra $\lieq$. Since the highest weight of $V$ is regular, the only $\lieq$ that satisfies the condition in \cite[Theorem 5.6]{VZ84} is the Borel. In our case, the torus of the Borel is compact, so $A_{q}(\lambda)$ must be a discrete series. Discrete series are tempered, so by \cite[III, Corollary 5.2]{Borel_Wallach}, $\pi_{\infty}$ only contributes to middle degree, that is 
     \begin{equation*}
         \mathrm{H}^{*}(\lieg_{\CC}, K_{G}; V(2) \otimes_{\CC} \pi_{\infty}) = \mathrm{H}^{2}(\lieg_{\CC}, K_{G}; V(2) \otimes_{\CC} \pi_{\infty}). 
     \end{equation*}
     Finally, for $\pi$ appearing in the discrete spectrum, by Wallach's theorem \cite[Theorem 4.3]{Wallach84} $\pi$ must be cuspidal. 
\end{proof}

\begin{remark}
    The condition that $V$ is regular is very important for Proposition \ref{Prop:coh_identity} to hold. 
\end{remark}

\subsection{Discrete series L-packets}
\label{SS:discrete series}

\begin{convention}
    In this subsection, $V$ is any object in $\mathrm{Rep}_{\CC}(G)$. 
\end{convention}

The $((\lieg_{\CC}, K_G) \times G(\AAA_f)) $-module $C_{\text{cusp}}^{\infty}(G(\QQ) \backslash G(\AAA), \xi)$ can be decomposed into a direct sum 
\begin{equation*}
    C_{\text{cusp}}^{\infty}(G(\QQ \backslash G(\AAA), \xi) = \bigoplus_{\pi} m(\pi) \pi,
\end{equation*}
where the sum runs over the irreducible cuspidal representations $\pi$ of $G$ and $m(\pi)$ is the multiplicity of $\pi$. This induces a decomposition 
\begin{equation}
\label{eq: decomosition of Betti into (g,K)-cohomology}
    \mathrm{H}^{2}_{B, !}(S, V(2)) \cong \mathrm{H}^{2}_{dR, !}(S, V(2)) = \bigoplus_{\pi = \pi_{\infty} \otimes \pi_f} m(\pi) \mathrm{H}^{2}(\lieg_{\CC}, K_{G}; V(2) \otimes_{\CC} \pi_{\infty})\otimes \pi_f. 
\end{equation}

Recall that we have the following theorem about the multiplicity $m(\pi)$: 

\begin{theorem} \textnormal{\cite{Mok15}, \cite[\S 1.1]{CHL11}, \cite[Theorem 2.6.3]{LSZ22}}
\label{Thm: multiplicity 1}
    The multiplicity $m(\pi)$ of the cohomological $\pi$ (i.e., the $\pi$ with $\mathrm{H}^{*}(\lieg_{\CC}, K_{G}; V(2) \otimes_{\CC} \pi_{\infty}) \neq 0$, 
    for some cohomological degree $*$) is $1$. 
\end{theorem}

\begin{definition}
    The \textit{discrete series $L$-packet} $P(V(2))$ associated to $V(2)$ is the set of isomorphism classes of discrete series $\pi_{\infty}$ of $G(\RR)^{+}$ whose Harish-Chandra parameter and central character are opposed to the ones of $V(2)$. 
\end{definition}

\begin{proposition}
    The $\pi$ such that 
    \begin{equation*}
         \mathrm{H}^{2}(\lieg_{\CC}, K_{G}; V(2) \otimes_{\CC} \pi_{\infty}) \neq 0 
    \end{equation*}
    are those $\pi$ such that $\pi_{\infty}|_{G(\RR)^{+}} \in P(V)$ \footnote{The archimedean part $\pi_{\infty}$ of a smooth automorphic representation $\pi$ is a $(\lieg_{\CC}, K_{G})$-module. By the exponential map, $\pi_{\infty}$ is equivalent to an representation of $G(\RR)^{+}$.}. 
\end{proposition}

\begin{proof}
    See the proof of Proposition \ref{Prop:coh_identity}. 
\end{proof}

According to \cite[Chapter II, \S 3, Proposition 3.1]{Borel_Wallach}, for any representation $\pi_{\infty}$ of $G(\RR)^{+}$, 
the $(\lieg_{\CC}, K_{G})$-complex $\pi_{\infty} \otimes_{\CC} V(2)$ has zero differential map. Hence, we have 
\begin{equation*}
    \mathrm{H}^{2}(\lieg_{\CC}, K_{G}; V(2) \otimes_{\CC} \pi_{\infty}) = \Hom_{K_G}(\wedge^{2} \lieg_{\CC} / \liek_{G, \CC}, V(2) \otimes_{\CC} \pi_{\infty}). 
\end{equation*}

We have the decomposition: 
\begin{equation*}
    \wedge^{2} \lieg_{\CC} / \liek_{G, \CC} = \wedge^{2} \liep_{G, \CC}^{+} \oplus  \wedge^{2} \liep_{G, \CC}^{-} \oplus \liep_{G, \CC}^{+} \otimes \liep_{G, \CC}^{-}. 
\end{equation*}

And we have the following identities as representations of $K_G$\footnote{When considering $(\lieg_{\CC}, K_{G})$-cohomology, the action of $A_{G}$ is determined by the action of $K$ because the centeral character is fixed, hence, we ignore the action of $A_{G}$ in the notation of representations of $K_{G}$. } : 
\begin{align*}
     & \wedge^{2} \liep_{G, \CC}^{+} = \tau_{(1, 1, -2)},  \\
     & \wedge^{2} \liep_{G, \CC}^{-} = \tau_{(-1, -1, 2)},  \\
     & \liep_{G, \CC}^{+} \otimes \liep_{G, \CC}^{-} = \tau_{(1, -1, 0)} \oplus \tau_{(0, 0, 0)}. 
\end{align*}

We also have the following proposition:
\begin{proposition}
\label{Prop: (g, K)-cohomology = Hom}
    The dimension of the $\CC$-vector space 
    \begin{equation*}
         \mathrm{H}^{2}(\lieg_{\CC}, K_{G}; V(2) \otimes_{\CC} \pi_{\infty}) = \Hom_{K_G}(\wedge^{2} \lieg_{\CC} / \liek_{G, \CC}, V(2) \otimes_{\CC} \pi_{\infty})
    \end{equation*}
    is less than or equal to $1$.
\end{proposition}
\begin{proof}
    This is a direct consequence of \cite[II, Proposition 3.1. and Theorem 5.3.]{Borel_Wallach}. 
\end{proof}

Finally, we can compute explicitly the elements of $P(V(2))$. 
\begin{proposition}
\label{Prop: DS L-packet}
    If the highest weight of $V(2)$ is $(\lambda_1, \lambda_2, \lambda_3; d)$ such that $\lambda_1 \ge \lambda_2 \ge \lambda_3$, 
    we have 
    \begin{equation*}
        P(V(2)) = \{ \pi_1, \pi_2, \pi_3, \bar{\pi}_1, \bar{\pi}_2, \bar{\pi}_3 \},
    \end{equation*}
    where for $i = 1,2,3$, $\bar{\pi}_{i}$ is the complex conjugate of $\pi_{i}$ and the Harish-Chandra parameters and Blattner parameters\footnote{We ignore the action of similitude factor in the notation of Harish-chandra and Blattner parameters because it is determind by the fixed central character. }of $\pi_1, \pi_2$ and $\pi_3$ are listed in the following table.
     \begin{equation}
     \label{talble: DS L-packt}
            \begin{array}{|c|c|c|}
            \hline  & \text{Harish-Chandra parameters} & \text{Blattner parameters} \\
            \hline \pi_1 & (1 - \lambda_1, -\lambda_2, -1-\lambda_1) & (1 - \lambda_3, 1 - \lambda_2, -2 - \lambda_1) \\
            \hline \pi_2 & (1 - \lambda_3, -1 - \lambda_1, -\lambda_2) & (1 - \lambda_3, -1 - \lambda_1, -\lambda_2) \\
            \hline \pi_3 & (-\lambda_2, -1-\lambda_1, 1-\lambda_3) & (-1 - \lambda_2, -1 - \lambda_1, 2 - \lambda_3) \\
            \hline 
            \end{array}
    \end{equation}
\end{proposition}
\begin{proof}
    For a dominant root system with positive roots 
    \begin{equation*}
        \Delta^{+} = \{(1, -1, 0), (0, 1, -1), (1, 0, -1) \}, 
    \end{equation*}
     we have $\rho = (1, 0, -1)$. Hence, the infinitesimal character of $V(2)$ is $(\lambda_1 + 1, \lambda_2, \lambda_3 - 1)$. 
    By \cite[II, Proposition 3.1]{Borel_Wallach}, the infinitesimal character of the representation $\pi_{\infty}$ of $G(\RR)^{+}$ with 
    \begin{equation*}
        \mathrm{H}^{*}(\lieg_{\CC}, K_{G}; V(2) \otimes_{\CC} \pi_{\infty}) \neq 0
    \end{equation*}
    must be the negative of the infinitesimal character of $V(2)$. Hence, the representations in $P(V(2))$ must have infinitesimal character $(1-\lambda_3, -\lambda_2, -1-\lambda_1)$. Then by \cite[Theorem 9.20]{Knapp86}, we can write down all the discrete series with fixed infinitesimal character. 
\end{proof}

\begin{remark}
    The representations $\pi_1$, $\pi_2$ and $\pi_3$ contribute to $\mathrm{H}^{2}_{dR, !}(\Sh_{G}(L)_{\CC}^{an}, V(2))$ while the representations $\bar{\pi}_1$, $\bar{\pi}_2$ and $\bar{\pi}_3$ contribute to $\mathrm{H}^{2}_{dR, !}(\overline{\Sh_{G}(L)_{\CC}^{an}}, V(2))$
\end{remark}

\subsection{Hodge decomposition}
\label{SS: Hodge decomposition}
In this subsection, we compute the Hodge decomposition of the Betti realization of the pure motive associated to an automorphic representation $\pi$. 

\begin{convention}
    In this subsection, $V$ is any object in $\mathrm{Rep}_{\QQ}(G)$. We use $V_{\CC}$ to denote $V \otimes_{\QQ} \CC$. 
\end{convention}

Recall that we have the following theorem. 
\begin{theorem}[\cite{BHD94},Theorem 3.2.2 and (2.3.3)]
\label{Thm: rational field}
    For any cuspidal automorphic representation $\pi = \pi_{f} \otimes \pi_{\infty}$ such that $\pi_{\infty} \in P(V_{\CC}(2))$, $\pi_f$ is defined over a number field $E(\pi_f)$. 
\end{theorem}

\begin{convention}
    In this subsection, we use $\pi_{f}$ to denote its $E(\pi_f)$-model. 
\end{convention}

\begin{definition}
\label{def: M_B and M_dR}
We define
\begin{equation*}
    M_{dR}(\pi_{f}, V(2)) := \Hom_{\QQ[G(\AAA_f)]}(\Res_{E(\pi_f)/\QQ}\pi_f, \mathrm{H}_{dR, !}^{2}(S, V(2)))
\end{equation*}
and similarly 
\begin{equation*}
    M_{B}(\pi_{f}, V(2)) := \Hom_{\QQ[G(\AAA_f)]}(\Res_{E(\pi_f)/\QQ}\pi_f, \mathrm{H}_{B, !}^{2}(S, V(2))), 
\end{equation*}
where $\mathrm{H}_{dR, !}^{2}(S, V(2))$ and $\mathrm{H}_{B, !}^{2}(S, V(2))$ are the interior de Rham and Betti cohomology with coefficients in $V$. 
These are $\QQ$-vector spaces with a $\QQ$-linear action of $E(\pi_f)$. 
\end{definition}

\begin{remark}
    Since $M_{B}(\pi_{f}, V(2))$ is a $\QQ$-vector space, define its base change to $\RR$ and $\CC$: 
    \begin{align*}
        & M_{B}(\pi_{f}, V(2))_{\RR} := M_{B}(\pi_{f}, V(2)) \otimes_{\QQ} \RR, \\
        & M_{B}(\pi_{f}, V(2))_{\CC} := M_{B}(\pi_{f}, V(2)) \otimes_{\QQ} \CC. 
    \end{align*}
\end{remark}

After base change to $\CC$, we have the comparison isomorphism: 
\begin{equation*}
    \begin{tikzcd}
        I_{\infty}\colon M_{B}(\pi_{f}, V(2))_{\CC} \ar[r, "{\cong}"] &  M_{dR}(\pi_{f}, V(2))_{\CC} 
    \end{tikzcd}
\end{equation*}
We also have 
\begin{align*}
    \rank_{E(\pi_f) \otimes_{\QQ} \CC}M_{B}(\pi_{f}, V(2))_{\CC} = \rank_{E(\pi_f) \otimes_{\QQ} \CC}M_{dR}(\pi_{f}, V(2))_{\CC} = 6. 
\end{align*}

Finally, we state a proposition about the Hodge decomposition of $M_{B}(\pi_{f}, V)_{\CC}$. 

\begin{proposition}
\label{Prop: Hodge decomp for motives}
     If the highest weight of $V(2)$ is $(\lambda_1, \lambda_2, \lambda_3; d)$ such that $\lambda_1 \ge \lambda_2 \ge \lambda_3$, 
     we have 
     \begin{align*}
        & M_{B}({\pi}_f, V(2))_{\CC} \\
        \cong & M_{B}^{2-\lambda_2-\lambda_3 - d, -\lambda_1 - d} \oplus M_{B}^{1-\lambda_1 - \lambda_3 - d, 1 - \lambda_2 - d} \oplus M_{B}^{-\lambda_1 - \lambda_2 - d, 2 - \lambda_3 - d} \\
        \oplus & M_{B}^{-\lambda_1 - d, 2 - \lambda_2 - \lambda_3 - d} \oplus M_{B}^{1 - \lambda_2 -d, 1 - \lambda_1 - \lambda_3 - d} \oplus M_{B}^{2 - \lambda_3 - d, -\lambda_1 - \lambda_2 - d},  
    \end{align*}
    where $M_{B}^{s, t}$ is a $1$-dimensional $\CC$-subspace with Hodge type $(s, t)$. 
    So $M_{B}({\pi}_f, V(2))$ is a pure $\QQ$-Hodge structure with weight $2 - (\lambda_1 + \lambda_2 + \lambda_3 + 2d)$. 
    Moreover, we have 
    \begin{align*}
        & M_{B}^{2-\lambda_2-\lambda_3 - d, -\lambda_1 - d} = \bigoplus_{ \sigma: E(\pi_F) \hookrightarrow \CC } \mathrm{H}^{2}(\lieg_{\CC}, K_{G}; V_{\CC}(2) \otimes_{\CC} \pi_{1}), \\
        & M_{B}^{1-\lambda_1 - \lambda_3 - d, 1 - \lambda_2 - d} = \bigoplus_{\sigma: E(\pi_F) \hookrightarrow \CC } \mathrm{H}^{2}(\lieg_{\CC}, K_{G}; V_{\CC}(2) \otimes_{\CC} \pi_{2}), \\
        & M_{B}^{-\lambda_1 - \lambda_2 - d, 2 - \lambda_3 - d}  = \bigoplus_{\sigma: E(\pi_F) \hookrightarrow  \CC } \mathrm{H}^{2}(\lieg_{\CC}, K_{G}; V_{\CC}(2) \otimes_{\CC} \pi_{3}), 
    \end{align*}
    where 
    \begin{equation*}
        P(V_{\CC}(2)) = \{ \pi_1, \pi_2, \pi_3, \bar{\pi}_1, \bar{\pi}_2, \bar{\pi}_3 \}
    \end{equation*}
    is defined in Proposition \ref{Prop: DS L-packet}. 
\end{proposition}

\begin{proof}
    For $i = 1,2,3$, if the Hodge type of $\pi_{i}$ is $(s, t)$, then the Hodge type of $\bar{\pi}_i$ is $(t, s)$. So we only need to compute the Hodge type of $\pi_1, \pi_2, \pi_3$. 
    Since complexification of $h$ in the definition of the Shimura datum of $\Sh_{G}(L)$ is 
    \begin{equation*}
        h_{\CC}\colon (z, \bar{z}) \in \mathbb{S}(\CC) \mapsto \begin{pmatrix}
                                                    z & 0 & 0 \\
                                                    0 & z & 0 \\
                                                    0 & 0 & \bar{z} 
                                                \end{pmatrix}, 
    \end{equation*}
    the contribution from the local system $V$ only depends on $K_{G}$-types. Then we compute the Hodge type of $\pi_1, \pi_2, \pi_3$ case by case. 
    \begin{enumerate}
        \item For $\pi_1$: The positive roots in \cite[Theorem 9.20]{Knapp86} for $\pi_1$ are 
                           \begin{equation*}
                                \Delta^{+} = \{(1, -1, 0), (0, 1, -1), (1, 0, -1) \}. 
                           \end{equation*}
                           Hence the lowest weight of $V_{\CC}(2)$ is $(\lambda_3, \lambda_2, \lambda_1)$, the minimal $K_{G}$-type of $V_{\CC}(2)$ is $\tau_{(\lambda_2, \lambda_3, \lambda_1)}$. 
                           Since the minimal $K_{G}$-type of $\pi_1$ is $\tau_{(1 - \lambda_3, 1 - \lambda_2, -2 - \lambda_1)}$, and the minimal $K_{G}$-type of 
                           $V_{\CC}(2) \otimes_{\CC} \pi_1$ is $\tau_{(1, 1, -2)}$. Hence, only the $K_{G}$-type $\tau_{(1, 1, -2)}$ contributes to 
                           \begin{equation*}
                               \mathrm{H}^{2}(\lieg_{\CC}, K_{G}; V_{\CC}(2) \otimes_{\CC} \pi_{1}) = \Hom_{K_G}(\wedge^{2} \lieg_{\CC} / \liek_{G, \CC}, V_{\CC}(2) \otimes_{\CC} \pi_{1}). 
                           \end{equation*}
                           Since the contributed $K_{G}$-type of $V_{\CC}(2)$ is $\tau_{(\lambda_2, \lambda_3, \lambda_1)}$ and 
                           \begin{equation*}
                               \tau_{(1,1,-2)} \subset \wedge^{2}\liep_{G, \CC}^{+} \subset \wedge^{2} \lieg_{\CC} / \liek_{G, \CC}, 
                           \end{equation*}
                           the contributed Hodge type of $\pi_1$ is 
                           \begin{equation*}
                               (2 - \lambda_2 - \lambda_3 - d, -\lambda_1 - d).
                           \end{equation*}
        \item For $\pi_2$: The positive roots in \cite[Theorem 9.20]{Knapp86} for $\pi_2$ are 
                           \begin{equation*}
                                \Delta^{+} = \{(1, -1, 0), (0, -1, 1), (1, 0, -1) \}. 
                           \end{equation*}
                           Hence the lowest weight of $V_{\CC}(2)$ is $(\lambda_3, \lambda_1, \lambda_2)$, so the minimal $K_{G}$-type of $V_{\CC}(2)$ is $\tau_{(\lambda_1, \lambda_3, \lambda_2)}$. 
                           Since the minimal $K_{G}$-type of $\pi_2$ is $\tau_{(1 - \lambda_3, -1 - \lambda_1,  - \lambda_2)}$, the minimal $K_{G}$-type of 
                           $V_{\CC}(2) \otimes_{\CC} \pi_2$ is $\tau_{(1, -1, 0)}$. Hence, only the $K_{G}$-type $\tau_{(1, -1, 0)}$ contributes to 
                           \begin{equation*}
                               \mathrm{H}^{2}(\lieg_{\CC}, K_{G}; V_{\CC}(2) \otimes_{\CC} \pi_{2}) = \Hom_{K_G}(\wedge^{2} \lieg_{\CC} / \liek_{G, \CC}, V_{\CC}(2) \otimes_{\CC} \pi_{2}). 
                           \end{equation*}
                           Since the contributed $K_{G}$-type of $V_{\CC}(2)$ is $\tau_{(\lambda_1, \lambda_3, \lambda_2)}$ and 
                           \begin{equation*}
                               \tau_{(1,-1,0)} \subset \liep_{G, \CC}^{+} \otimes \liep_{G, \CC}^{-} \subset \wedge^{2} \lieg_{\CC} / \liek_{G, \CC}, 
                           \end{equation*}
                           the contributed Hodge type of $\pi_1$ is 
                           \begin{equation*}
                               (1 - \lambda_1 - \lambda_3 - d, 1 -\lambda_2 - d).
                           \end{equation*}
        \item For $\pi_3$: The positive roots in \cite[Theorem 9.20]{Knapp86} for $\pi_3$ are 
                           \begin{equation*}
                                \Delta^{+} = \{(-1,  0, 1), (0, -1, 1), (1, -1, 0) \}. 
                           \end{equation*}
                           Hence the lowest weight of $V_{\CC}(2)$ is $(\lambda_2, \lambda_1, \lambda_3)$, so the minimal $K_{G}$-type of $V_{\CC}(2)$ is $\tau_{(\lambda_1, \lambda_2, \lambda_3)}$. 
                           Since the minimal $K_{G}$-type of $\pi_3$ is $\tau_{(-1 - \lambda_2, -1 - \lambda_1,  2- \lambda_3)}$, the minimal $K_{G}$-type of 
                           $V_{\CC}(2) \otimes \pi_3$ is $\tau_{(-1, -1, 0)}$. Hence, only the $K_{G}$-type $\tau_{(-1, -1, 0)}$ contributes to 
                           \begin{equation*}
                               \mathrm{H}^{2}(\lieg_{\CC}, K_{G}; V_{\CC}(2) \otimes_{\CC} \pi_{3}) = \Hom_{K_G}(\wedge^{2} \lieg_{\CC} / \liek_{G, \CC}, V_{\CC}(2) \otimes_{\CC} \pi_{3}). 
                           \end{equation*}
                           Since the contributed $K_{G}$-type of $V_{\CC}(2)$ is $\tau_{(\lambda_1, \lambda_2, \lambda_3)}$ and 
                           \begin{equation*}
                               \tau_{(-1,-1,0)} \subset \wedge^{2} \liep_{G, \CC}^{-} \subset \wedge^{2} \lieg_{\CC} / \liek_{G, \CC}, 
                           \end{equation*}
                           the contributed Hodge type of $\pi_3$ is 
                           \begin{equation*}
                               ( -\lambda_1 - \lambda_2 - d, 2 -\lambda_3 - d).
                           \end{equation*}
    \end{enumerate}
\end{proof}

\begin{corollary}
\label{Corollary: Hodge decomp for motives}
If $V = V^{a, b} \{r, s\}$, we have 
\begin{align*}
        & M_{B}({\pi}_f, V(2))_{\CC} \\
        \cong & M_{B}^{-r, -2 - a - b - s} \oplus M_{B}^{-1 - a - r, -1 - b - s} \oplus M_{B}^{-2 - a - b - r, -s} \\
        \oplus & M_{B}^{ -2 - a - b - s, -r} \oplus M_{B}^{-1 - b - s, -1 - a - r} \oplus M_{B}^{-s, -2 -a -b - r}.  
\end{align*}
\end{corollary}


\clearpage 

\part{The vanishing on the boundary}

\section{Construction of the motivic classes}
\label{Sec: construction of motivic class}
\subsection{Relative motives and motivic cohomology} 
\label{SS: relative motive} 

In this subsection, we first recall some basic facts about motivic categories. We then use motivic categories to give the definition of motivic cohomology and collect the needed facts about it. Finally, we give the Gysin map of motivic cohomology corresponding to the map of the Shimura varieties that we are interested in. 

\subsubsection{Motivic categories}
Let $\Lambda$ be a commutative ring with a unit and let $\mathrm{DM}_{B,c}(S, \Lambda)$ be the \textit{triangulated category of constructible Beilinson motives over a base excellent scheme $S$ with $\Lambda$-coefficients} defined in \cite[Def.15.1.1]{Cisinski_Deglise}, which is a symmetric monoidal category and we denote by $1_{S}$ the unit of the category, which is the Tate motive \cite[Introduction B, page xxi]{Cisinski_Deglise}. We write $\mathrm{DM}_{B,c}(S)$ for $\mathrm{DM}_{B,c}(S, \QQ)$. 

\begin{fact}
We have the following properties: 
 \begin{itemize}
     \item \textnormal{\cite[Theorem 15.2.4]{Cisinski_Deglise}} The triangulated category $\mathrm{DM}_{B,c}(S)$ has the formalism of Grothendieck's six-functors 
     \begin{equation*}
        (f^{*}, f_{*}, f_{!}, f^{!},\underline{\Hom}, \otimes)
    \end{equation*}
     and dualizing functor $\mathbb{D}$. 
     \item \textnormal{\cite[Remark 11.1.14]{Cisinski_Deglise}} Over a field $F$, we have the following equivalence of triangulated categories: 
           \begin{equation}
           \label{eq: isomorphism Beilinson motive to geometric motive}
               \mathrm{DM}_{B,c}(\Spec (F)) \cong \mathrm{DM}_{gm}(F), 
           \end{equation}
           where the right hand side is the triangulated category of geometric motives \textnormal{\cite[Definition 2.1.1, Page 5]{V00}}. We use the following isomorphism: for any smooth $f \colon X \rightarrow \Spec (F)$, the motive $f_{*}1_{X} \in \mathrm{DM}_{B,c}(\Spec (F))$ is mapped to the dual of the motive $M_{gm}(X) \in \mathrm{DM}_{gm}(F)$ defined in \textnormal{\cite[Definition 2.1.1]{V00}}.
\end{itemize}
 \end{fact}

 \begin{proposition}[{\cite[Proposition 1.7]{PL15}}]
       For a regular immersion $f \colon Y \rightarrow X$ with $\dim X - \dim Y = d$ and any $M \in \mathrm{DM}_{B,c}(S)$, we have the following absolute purity isomorphism: 
       \begin{equation*}
           f^{*}(M) \cong f^{!}M(d)[2d]. 
       \end{equation*}
 \end{proposition}

 \begin{remark}
     It follows from \cite[\href{https://stacks.math.columbia.edu/tag/0E9J}{Tag 0E9J}]{stacks-project} that if $Y$ and $X$ are regular schemes, then $f$ is a regular immersion.
 \end{remark}

 \par Let $\Lambda$ be a field. Corti and Hanamura \cite[Definition 2.9]{Corti-Hanamura00}  defined the category of relative Chow motives $\mathrm{CHM}(S, \Lambda)$ with $\Lambda$-coefficients over a base scheme $S$ generalizing the category of Chow motives over a field. We write $\mathrm{CHM}(S)$ for $\mathrm{CHM}(S, \QQ)$. The relationship between $\mathrm{CHM}(S)$ and $\mathrm{DM}_{B,c}(S)$ proved by Jin is as follows: 
 
 \begin{proposition}[{\cite[Theorem 3.17]{Jin16}}]
 \label{Prop: Jin}
 Let $\mathrm{Chow}(S)$ be the heart of the Chow weight structure on $\mathrm{DM}_{B,c}(S)$ constructed in \textnormal{\cite[Th\'{e}or\`{e}me 3.3.]{Hebert10}}. Then there is a natural functor 
 \[
        F \colon \mathrm{CHM}(S) \rightarrow \mathrm{Chow}(S),
 \]
 which is an equivalence of the two categories $\mathrm{CHM}(S)$ and $\mathrm{Chow}(S)$. 
 \end{proposition}

 \begin{convention}
     This proposition tells us that for any object $V \in \mathrm{CHM}(S)$, $F(V)$ is an object of the category $\mathrm{DM}_{B,c}(S)$. In order to simplify notation, we write $V$ for $F(V)$. 
 \end{convention}

 Finally, we recall the definition of motives of Abelian type which will be used in the proof of vanishing on the boundary for motivic cohomology.
 \begin{definition}[{\cite[Definition 1.1]{Wild_15}}]
     Let $F$ be a perfect field and $\bar{F}$ be its algebraic closure. 
     \begin{enumerate}
         \item  The category of \textit{Chow motives of Abelian type} over $\bar{F}$ is defined as a strict full additive tensor sub-category $\mathrm{CHM}^{Ab}(\bar{F})$ of $\mathrm{CHM}(\bar{F})$ generated by the following: 
         \begin{enumerate}
             \item Shifts of Tate motives $\ZZ(m)[2m]$, for $m \in \ZZ$, 
             \item Chow motives of Abelian varieties over $\bar{F}$. 
         \end{enumerate}
         \item Define the category of \textit{Chow motives of Abelian type over $\bar{F}$} as the (strict) full (additive tensor) sub-category $\mathrm{CHM}^{Ab}(F)$ of $\mathrm{CHM}(F)$ whose base change to $\bar{F}$ lies in $\mathrm{CHM}^{Ab}(\bar{F})$. 
         \item The category $\mathrm{DM}_{gm}^{Ab}(F)$ is defined as the (strict) full triangulated subcategory of $\mathrm{DM}_{gm}(F)$ generated by $\mathrm{CHM}^{Ab}(F)$. 
     \end{enumerate}
 \end{definition}

 \begin{notation}
    We denote by $\mathrm{DM}_{B, c}^{Ab}(\Spec (\QQ))$ the sub-category of $\mathrm{DM}_{B, c}(\Spec (\QQ))$ that is isomorphic to $\mathrm{DM}_{gm}^{Ab}(F)$ via the isomorphim in equation (\ref{eq: isomorphism Beilinson motive to geometric motive}).
 \end{notation}

\subsubsection{The localization exact triangle} We recall the localization exact triangle of Beilinson motives that will be used in the proof of vanishing on the boundary for motivic cohomology. 

\par Let 
\begin{equation*}
    \xymatrix{
        U \ar[r]^{j} & X & \ar[l]_{i} Y 
    }
\end{equation*}
be a diagram in $\mathrm{Sch}(\QQ)$ where $j$ is an open immersion and $i$ is the complementary reduced closed immersion. We assume that $U$ has the same dimension as $X$ and that $Y$ has codimension $c$ in $X$. 

\begin{proposition}[{\cite[Proposition 2.3.3]{Cisinski_Deglise}}]
\label{Prop: motivic localization}
    Let $N$ be an object of $\mathrm{DM}_{B,c}(S)$. We have the following exact triangle
    \begin{equation*}
        \xymatrix{
            j_{!}j^{*}N \ar[r] & N \ar[r] & i_{*}i^{*}N \ar[r] & j_{!}j^{*}N[1]
        }
    \end{equation*}
    in $\mathrm{DM}_{B,c}(S)$. 
\end{proposition}

\subsubsection{Motivic cohomology}
 
 \begin{definition}[Motivic Cohomology] \label{Def: motivic cohomology}
 If $W$ is an object of $\mathrm{DM}_{B,c}(S)$ and ${1}_{S}$ is the unit object of $\mathrm{DM}_{B,c}(S)$, then the \textit{motivic cohomology} is defined as the $\Lambda$-module
 \[
        \mathrm{H}_{\mathrm{M}}^{*}(S,W) = \Hom_{\mathrm{DM}_{B,c}(S)}(1_{S},W[*]).
 \]
 \end{definition}

 \begin{remark}
    \begin{itemize}
        \item It follows from \cite[Corrollary 14.2.14]{Cisinski_Deglise} that this definition is compatible with the K-theoretical one: 
         for any regular scheme $X$, we have a canonical isomorphism
         \[
                \mathrm{H}_{\mathrm{M}}^{q}(X, \QQ(p)) \cong \mathrm{K}_{2p-q}(X)_{\QQ}^{(p)},
         \]
         where the right-hand side is the $p$-th graded piece of the $\gamma$-filtration of an algebraic $\K$-group.
         \item Recall that in \cite[Page 21]{Wild_09}, for any $M \in \mathrm{DM}_{gm}(\QQ)$, the motivic cohomology $\mathrm{H}_{M}^{i}(M, \QQ(j))$ is defined as 
         \begin{equation}
         \label{eq: def of motivic cohomology of motives}
             \mathrm{H}_{M}^{i}(M, \QQ(j)): = \Hom_{\mathrm{DM}_{gm}(\QQ)}(M, \QQ(i)[j]). 
         \end{equation}
         We can see that this definition is compatible with our definition through the isomorphism in equation (\ref{eq: isomorphism Beilinson motive to geometric motive}). 
    \end{itemize}
 \end{remark}

\begin{notation}
Let $\mathcal{G}$ temporarily denote any of the three groups $\{\GL_2, H, G \}$, so we have the associated Shimura variety $\Sh_{\calG}$. 
\end{notation}

\begin{lemma}[{\cite[Theorem 8.6]{Ancona15}}]
\label{lemma: Ancona}
Let $F$ be a number field. There is an additive functor 
\[
    \mu_{M} \colon \mathrm{Rep}_{F}(\calG) \rightarrow \mathrm{CHM}(\Sh_{\calG}, F)
\]
from the category of representations of $\calG$ over $F$ to the category of relative Chow motives over $\mathrm{CHM}(\Sh_{\calG}, F)$ with the following properties: 
\begin{enumerate}
\item The functor $\mu_{M}$ preserves tensor product and duals;
\item For the similitude character $\mu \colon \calG \rightarrow \mathbf{G}_{m}$, $\mu_{M}(\mu)$ is the Lefschetz motive 
$\QQ(1)$\footnote{The Lefschetz motive $\QQ(1)$ has weight $-2$, and its image under the Hodge realization functor is $\RR(1): = 2\pi i$. };
\item If $V$ denotes the defining representation of $\mathcal{G}$, then $\mu_{M}(V) = \mathrm{h}^{1}(A)(1)$, where $A$ is the universal $\mathrm{PEL}$ abelian variety over $\Sh_{\mathcal{G}}$;
   
\end{enumerate}
\end{lemma}

\begin{remark} We give four remarks about Lemma \ref{lemma: Ancona}. 
\begin{enumerate}
    \item The construction by Ancona in Lemma \ref{lemma: Ancona} works much more generally for arbitrary PEL Shimura varieties, but in this paper, we shall only need it for the above three groups.
    \item Our normalization is compatible with Ancona's original normalization, but dual to the normalization of \cite[Theorem 8.3.1]{LSZ22} where they send the multiplier representation to $\QQ(-1)$ and the defining representation to $\mathrm{h}^{1}(A)$.
    \item The Ancona functor $\mu_{M}$ is faithful. \cite[REMARK 5.10]{Torzewski19}. 
    \item For any $V \in \mathrm{Rep}_{F}(\calG)$ and any integer $d$, we have $\mu_{M}(V(d)) = \mu_{M}(V)(d)$. Hence, we can use the twist $(d)$ before or after applying ancona functor. 
\end{enumerate}
\end{remark}

\par By applying Torzewski's general theorem \cite[Theorem 9.7]{Torzewski19} in the setting of $G$ and $H$, we get the following result: 
\begin{proposition}[``Branching" for motivic sheaves]
For the map of Shimura varieties $\iota: M \rightarrow S$, we have the following commutative diagram of functors 
\[
    \begin{tikzcd}
        \mathrm{Rep}_{F}(G) \ar[d, "\iota^{*}"] \ar[r, "\mu_{M}"] & \mathrm{CHM}(S, F) \ar[d, "\iota^{*}"] \\
        \mathrm{Rep}_{F}(H) \ar[r, "\mu_{M}"] & \mathrm{CHM}(M, F), 
    \end{tikzcd}
\]
where the left-hand $\iota^{*}$ denotes the restriction of representations, and the right-hand $\iota^{*}$ denotes the pullback of relative Chow motives.

\end{proposition}

\begin{convention}
\begin{itemize}
    \item For any $\mathcal{G}$, we write $\mathrm{Rep}(\mathcal{G})$ for $\mathrm{Rep}_{\QQ}(\mathcal{G})$. For any $V \in \mathrm{Rep}(\mathcal{G})$, we also write $V$ for $\mu_{M}(V)$. 
    \item Recall that we define the elements of $\mathrm{Rep}(G)$ and $\mathrm{Rep}(H)$ as the Weil restriction of elements of $\mathrm{Rep}_{E}(G)$ and $\mathrm{Rep}_{E}(H)$. 
    \item Recall that we always let $M = \Sh_{H}$ and $S = \Sh_{G}$. 
\end{itemize}
\end{convention}

\par From the above discussion, for a representation of $\calG$ denoted as $V$, we get a relative Chow motive $V \in \mathrm{CHM}(\Sh_{\calG})$ over Shimura variety $\Sh_{\calG}$, which can also be viewed as an object of $\mathrm{DM}_{B,c}(\Sh_{\calG})$. So for $V$, we can define the motivic cohomology $\mathrm{H}_{\mathrm{M}}^{*}(\Sh_{\calG}, V(*))$ \footnote{In this paper, in order to apply Beilinson regulator, we define all motivic cohomology groups to be $\QQ$-vector spaces, which is the same as Weil restriction to $\QQ$ of the motivic cohomology groups defined in \cite{LSZ22}.}.
 
 \begin{proposition}[Gysin morphism for motivic cohomology]
 \label{Prop: Gysin for motivic cohomology}
For the map of Shimura varieties $\iota \colon M \rightarrow S$ defined earlier and for the representations $W \in \mathrm{Rep}(H)$ and $V \in \mathrm{Rep}(G)$ satisfying $W \hookrightarrow \iota^{*}V$ in $\mathrm{Rep}(H)$, we have the Gysin morphism 
\[
    \iota_{*} \colon \mathrm{H}^{1}_{\mathrm{M}}(M, W(1)) \rightarrow \mathrm{H}^{3}_{\mathrm{M}}(S, V(2)), 
\]
which is induced by 
\begin{equation}
\label{eq: Gysin morphism in DM}
    (\iota_{*}W(1)[1] \rightarrow V(2)[3]) \in \mathrm{DM}_{B,c}(S). 
\end{equation}
\end{proposition}
\begin{proof}
 From the above discussions, we have a map $W \hookrightarrow \iota^{*}V$ in $\mathrm{DM}_{B,c}(M)$, where $\iota^{*}$ is the pullback functor on the category of Beilinson motives. By the absolute purity isomorphism \cite[Proposition 1.7]{PL15}, the right hand side is $\iota^{*}V \cong \iota^{!}V(1)[2]$. By adjunction, we get the morphism $\iota_{!}W \rightarrow V(1)[2]$ in $\mathrm{DM}_{B,c}(S)$. Since $\iota$ is proper, we have $\iota_{*} = \iota_{!}$. So we get the morphism 
 \begin{equation*}
    \iota_{*}W(1)[1] \rightarrow V(2)[3] \in \mathrm{DM}_{B,c}(S). 
 \end{equation*}
 Then by applying the functor $N \longmapsto \Hom_{\mathrm{DM}_{B,c}(S)}(1_{S}, N)$ to the morphism, we get 
\[
\Hom_{\mathrm{DM}_{B,c}(S)}(1_{S}, \iota_{*}W(1)[1]) \rightarrow \Hom_{\mathrm{DM}_{B,c}(S)}(1_{S}, V(2)[3]). 
\]
The right-hand side is $\mathrm{H}^{3}_{\mathrm{M}}(S, V(2))$, and by adjunction and $\iota^{*} 1_{S} = 1_{M}$, the left-hand side is 
\begin{equation*}
    \Hom_{\mathrm{DM}_{B,c}(M)}(1_{M}, W(1)[1]), 
\end{equation*} which is $\mathrm{H}^{1}_{\mathrm{M}}(M, W(1))$. 
\end{proof}

 \subsection{Mixed Hodge modules and absolute Hodge cohomology}
 \label{SS:MHM and abs Hodge coh}
 
 In this subsection, we first recall some basic facts about mixed Hodge modules, which can be viewed as a relative version of mixed Hodge structures, ``archimedean" analogues of mixed $l{\textendash}\mathrm{adic}$ perverse sheaves, or the Hodge realization of the mixed motivic sheaves which have not been defined. We then use mixed Hodge modules to define absolute Hodge cohomology, which can be seen as the Hodge realization of the motivic cohomology defined in Definition \ref{Def: motivic cohomology}, and collect the needed facts about it. Finally, we give the Gysin map of absolute Hodge cohomology corresponding to the map of the Shimura varieties that we are interested in.  For this part, we mainly follow \cite[\S2.2]{LemmaI15}. 

 \subsubsection{Mixed Hodge modules}
Let $A \subset \RR$ be a subfield and $\mathrm{Sch}(\QQ)$ be the category of quasi-projective $\QQ$-schemes. For $X \in \mathrm{Sch}(\QQ)$, we have the abelian category $\mathrm{MHM}_{A}(X/\RR)$ of real algebraic mixed $A$-Hodge modules \cite[Definition A.2.4]{HW98}.  Let $\mathrm{D}_{c}^{b}(X^{\mathrm{an}}_{\CC}, A)$ be the bounded derived category of sheaves for the analytic topology of $A$-vector spaces with constructible cohomology objects and let $\mathrm{Perv}_{A}(X^{\mathrm{an}}_{\CC}) \subset \mathrm{D}_{c}^{b}(X^{\mathrm{an}}_{\CC}, A)$ be the abelian category of perverse sheaves on $X^{\mathrm{an}}_{\CC}$.   

 \begin{fact}
    We have the following properties: 
    \begin{enumerate}
        \item \textnormal{\cite{Beilinson87}} The natural functor $\mathrm{D}^{b}(\mathrm{Perv}_{A}(X^{\mathrm{an}}_{\CC})) \rightarrow \mathrm{D}_{c}^{b}(X^{\mathrm{an}}_{\CC}, A)$ is an equivalence of categories. 
        \item \textnormal{\cite[Theorem 0.1]{MHM90}} There is a faithful and exact functor 
        \begin{equation*}
            rat \colon \mathrm{MHM}_{A}(X/\RR) \rightarrow \mathrm{Perv}_{A}(X^{\mathrm{an}}_{\CC}). 
        \end{equation*}
        We also use the same notation to denote its derived functor 
         \begin{equation*}
             rat \colon \mathrm{D}^{b}(\mathrm{MHM}_{A}(X/\RR)) \rightarrow \mathrm{D}_{c}^{b}(X^{\mathrm{an}}_{\CC}, A). 
         \end{equation*}
        Also, for $M \in \mathrm{D}^{b}(\mathrm{MHM}_{A}(X/\RR))$, we have $rat(\mathcal{H}^{i}M) =  \prescript{p}{}{\mathcal{H}^{i}}rat(M)$, where $\mathcal{H}^{i}M$ is the standard cohomology functor and $\prescript{p}{}{\mathcal{H}^{i}}$ is the perverse cohomology functor.
        \item \textnormal{\cite{S2}, \cite{MHM90}, {\cite[Theorem A.2.5]{HW98}}} We have six-functors formalism with 
        \begin{equation*}
            (f^{*}, f_{*}, f_{!},f^{!}, \underline{\Hom}, \otimes)
        \end{equation*}
        and the dualizing functor $\mathbb{D}$ on the derived category $\mathrm{D}^{b}(\mathrm{MHM}_{A}(X/\RR))$, which is a symmetric monoidal category. Furthermore, these functors commute with the functor $rat$.
    \end{enumerate}
 \end{fact}

 We now collect some facts about the relationship between variations of Hodge structures and mixed Hodge modules. 
 
\par Assume that $X$ is smooth and pure of dimension $d$. Then for any local system $V$ of $A$-vector spaces on $X_{\CC}^{\mathrm{an}}$, by the definition of perverse sheaves the complex $V[d]$ concentrated in degree $-d$ is an object of $\mathrm{Perv}_{A}(X^{\mathrm{an}}_{\CC})$. 
\par Let $\mathrm{MHM}_{A}(X/\RR)^{s}$ be the full subcategory of $\mathrm{MHM}_{A}(X/\RR)$ whose objects are those $M \in \mathrm{MHM}_{A}(X/\RR)$ such that $rat(M) = V[d]$ for some local system $V$. Denote by $\mathrm{Var}_{A}(X/\RR)$ the category of real admissible polarizable variations of mixed $A$-Hodge structures over $X$\cite[Definition A.2.4 b]{HW98}.

\begin{fact}[{\cite[Definition A.2.4.b]{HW98}}] 
    There is an equivalence of categories
    \[
        \mathrm{Var}_{A}(X/\RR) \cong \mathrm{MHM}_{A}(X/\RR)^{s}.  
    \]
    From this, we have the following equivalence of abelian categories:
    \[
        \mathrm{MHM}_{A}(\Spec(\QQ)/\RR) \cong \mathrm{MHS}_{A}^{+},
    \]
    where the right hand side denotes the abelian category of mixed polarizable $A$-Hodge structures \textnormal{\cite[Lemma  A.2.2]{HW98}}. 
\end{fact}

\begin{convention}
    \begin{enumerate}
        \item 
            We use the convention that a variation of mixed Hodge structure is viewed as a mixed Hodge module but not a shift of mixed Hodge modules, which is the same as the one used in \cite{Bow04}. When $X$ is smooth and pure of dimension $d$, the embedding 
            \begin{equation*}
                \iota \colon \mathrm{Var}_{A}(X/\RR) \rightarrow \mathrm{D}^{\mathrm{b}}(\mathrm{MHM}_{A}(X/\RR))
            \end{equation*}
            is characterized by the fact: for any object $V \in \mathrm{Var}_{A}(X/\RR)$, the complex $rat(\iota(V))[-d]$ is a single non-trivial constituent in degree zero,  which means that it is a local system. When the context is clear, we always write $V$ for $\iota(V)$. 
        \item In order to simplify terminology, when $A = \RR$ we will write ``variation of Hodge structure" instead of ``real admissible polarizable variation of mixed $\RR$-Hodge structure" and write ``mixed Hodge modules" for ``real algebraic mixed $\RR$-Hodge modules". 
    \end{enumerate}
\end{convention}

\begin{proposition}[{\cite[Proposition 2.3.12, Proposition 2.3.14, Lemme 2.3.7]{Bl06}}]
\label{Prop: compare between VHS and MHM (tensor and pullback)}
    Under the embedding
    \begin{equation*}
        \iota \colon \mathrm{Var}_{\RR}(X/\RR) \rightarrow \mathrm{D}^{\mathrm{b}}(\mathrm{MHM}_{\RR}(X)/\RR), 
    \end{equation*}
    we have the following three consequences. 
    \begin{enumerate}
        \item For any $V, W \in \mathrm{Var}_{\RR}(X/\RR)$, we have 
             \begin{equation*}
                 \iota(V \otimes W)[d] = \iota(V) \otimes \iota(W), 
             \end{equation*}
             where the left tensor product is the tensor product in $\mathrm{Var}_{\RR}(X/\RR)$ and the right tensor product is the tensor product in $\mathrm{D}^{b}(\mathrm{MHM}_{\RR}(X)/\RR)$. 
        \item Let $f \colon X \rightarrow Y$ be a morphism between smooth algebraic varieties with $\dim X - \dim Y = d$. For any $V \in \mathrm{Var}_{\RR}(X/\RR)$, we have the identity: 
        \begin{equation*}
            \iota((f^{s})^{*}V) = f^{*}\iota(V)[d], 
        \end{equation*}
        where $(f^{s})^{*}$ is pullback in $\mathrm{Var}_{\RR}(X/\RR)$ and $f^{*}$ is the pullback in $\mathrm{D}^{b}(\mathrm{MHM}_{\RR}(X)/\RR)$. 
        In particular, if $s\colon X \rightarrow \Spec(\QQ)$ is a smooth scheme which is pure of dimension $d$, $A_{X}(n)$ (resp., $A(n)$) is the Tate variation on $X$ (resp., $\Spec (\QQ)$) viewed as an object of $\mathrm{MHM}_{A}(X/\RR)$ (resp., $\mathrm{MHS}_{A}^{+}$), we have the equality $A_{X}(n) = s^{*}A(n)[d]$ in $\mathrm{MHM}_{A}(X/\RR)$. 
        \item For any $V \in \mathrm{Var}_{\RR}(X/\RR)$, the following identity holds: 
        \begin{equation*}
            \mathbb{D}(\iota(V)) = \iota(V^{\vee})(d), 
        \end{equation*}
        where $(\cdot)^{\vee}$ denotes the dualizing functor in $\mathrm{Var}_{\RR}(X/\RR)$ and $\mathbb{D}$ stands for the dualizing functor in $\mathrm{D}^{b}(\mathrm{MHM}_{\RR}(X)/\RR)$. 
    \end{enumerate}
\end{proposition}

\begin{remark}
    From the proposition, we can see that the unit $1_{X}$ in the symmetric monoidal category $\mathrm{D}^{b}(\mathrm{MHM}_{\RR}(X)/\RR)$ is ${A(n)_{X}}[-d]$. 
\end{remark}

\begin{convention}
    When the scheme $X$ is clear from the context, we write $A(n)$ for $A(n)_{X}$.  
\end{convention}

\begin{theorem}[{\cite[Theorem 2]{Saito88}}]
 Let $X \in \mathrm{Sch}(\QQ)$ be smooth and pure of dimension $d$. A variation of Hodge structure of weight $w$ is a mixed Hodge module of weight $w + d$ via the above identification. 
\end{theorem}

Finally, we state the absolute purity theorem for mixed Hodge modules, which will be used in the construction of the Gysin map for absolute Hodge cohomology.

\begin{proposition}[{\cite[Proposition 2.3.16]{Bl06}}]
\label{Prop: absolute purity for MHM}
    Let $f\colon X \rightarrow Y$ be a morphism of pure relative dimension $d$ between smooth algebraic varieties. Then for any $V \in \mathrm{Var}_{\RR}(Y/\RR)$, we have the identity: 
    \begin{equation*}
        f^{!}V = f^{*}V(d)[2d]. 
    \end{equation*}
\end{proposition}

\subsubsection{The localization exact triangle} We recall the localization exact triangle of mixed Hodge modules that will be used in the proof of vanishing on the boundary for absolute Hodge cohomology. 

\par Let 
\begin{equation*}
    \xymatrix{
        U \ar[r]^{j} & X & \ar[l]_{i} Y 
    }
\end{equation*}
be a diagram in $\mathrm{Sch}(\QQ)$ where $j$ is an open immersion and $i$ is the complementary reduced closed immersion. We assume that $U$ has the same dimension as $X$ and that $Y$ has codimension $c$ in $X$. 

\begin{proposition}[{\cite[(4.4.1)]{MHM90}}]
\label{Prop: MHM localization}
    Let $N$ be an object of $\mathrm{D}^{b}(\mathrm{MHM}_{\RR}(X/\RR))$, we have the following exact triangle
    \begin{equation*}
        \xymatrix{
            j_{!}j^{*}N \ar[r] & N \ar[r] & i_{*}i^{*}N \ar[r] & j_{!}j^{*}N[1]
        }
    \end{equation*}
    in $\mathrm{D}^{b}(\mathrm{MHM}_{\RR}(X/\RR))$. 
\end{proposition}

\begin{definition}[Singular (Betti) cohomology]
    \begin{enumerate}
        \item 
        For $X \in \mathrm{Sch}(\QQ)$ with the structure map $s \colon X \rightarrow \Spec(\QQ)$, by \cite[Corollary A.1.7.c]{HW98}, $\mathcal{H}^{i}s_{*}s^{*}A(0)$ is the $i$-th singular cohomology of the underlying topological space of $X_{\CC}^{\mathrm{an}}$ with coefficients in $A$. It also has the mixed Hodge structure constructed by Deligne with the involution induced by the complex conjugation on $X(\CC)$. From this, we get that if $X$ is smooth of pure dimension $d$, and if $A(0)$ denotes the trivial Tate variation of Hodge structures on $X$, then the \textit{$i$-th singular cohomology} is $\mathcal{H}^{i-d}s_{*}A(0)$. 
        \item 
        For any $X$ and any $M \in \mathrm{MHM}_{A}(X/\RR)$, we define the \textit{$i$-th singular cohomology} of $X$ with $M$ as coefficients, $\mathrm{H}^{i}(X, M)$, as the group $\mathcal{H}^{i-d}s_{*}M$; similarly, we define the compactly supported cohomology $\mathrm{H}_{c}^{i}(X,M)$ as $\mathcal{H}^{i-d}s_{!}M$. 
    \end{enumerate}
\end{definition}

\begin{corollary}
\label{Corollary: Localization long exact seq}
    Let $M$ be an object of $\mathrm{D}^{b}(\mathrm{MHM}_{\RR}(U/\RR))$ and assume X is proper. We have the following long exact sequence of mixed Hodge structures
    \begin{equation} \label{eq: les of singular cohomology}
        \cdots \rightarrow \mathrm{H}_{c}^{i}(U,M) \rightarrow \mathrm{H}^{i}(U,M) \rightarrow \mathrm{H}^{i-c}(Y,i^{*}j_{*}M) \rightarrow \mathrm{H}_{c}^{i+1}(U,M)  \rightarrow \cdots.  
    \end{equation}
\end{corollary}
\begin{proof}
    Let $N = j_{*}M$. Then apply the above triangle to get the exact triangle:
    \[
        \xymatrix{
            j_{!}M \ar[r] & j_{*}M \ar[r] & i_{*}i^{*}j_{*}M \ar[r] & j_{!}M[1]
        }. 
    \]
    Let $s\colon X \rightarrow \Spec(\QQ)$ be the structure morphism of $X$ and apply the functor $s_{*}$ to the exact triangle and then take cohomology $\mathcal{H}^{i-d}$, where $d = \dim X$. 
\end{proof}

\subsubsection{Absolute Hodge cohomology}
We now give the definition of absolute Hodge cohomology and construct its Gysin morphism in the case that we are interested in. 

\begin{definition}[Absolute Hodge cohomology]
Let $X$ be the scheme as above. The \textit{$i$-th real absolute Hodge cohomology} $\mathrm{H}^{i}_{H}(X, M)$\footnote{The standard notation for real absolute Hodge cohomology is $\mathrm{H}^{i}_{H}(X/\RR, M)$, which is good to distinguish with complex absolute Hodge cohomology. However, in our paper, we only work with real absolute Hodge cohomology, so we use $\mathrm{H}^{i}_{H}(X, M)$ to denote real absolute Hodge cohomology.} of $X$ with coefficients in $M$ is 
\begin{align*}
       \mathrm{H}^{i}_{H}(X, M) := & \Hom_{\mathrm{D}^{b}(\mathrm{MHM}_{A}(X/\RR))}(A(0)_{X}, M[i]) \\
       = & \Hom_{\mathrm{D}^{b}(\mathrm{MHM}_{A}(X/\RR))}(s^{*}A(0)[d], M[i]) \\
       \cong & \Hom_{\mathrm{D}^{b}(\mathrm{MHS}^{+}_{A})}(A(0), s_{*}M[i-d]),     
\end{align*}
where the second isomorphism comes from adjunction. 
\end{definition}

\begin{remark}
\begin{enumerate} 
    \item This definition is the ``archimedean" analogue of the definition of motivic cohomology given in Definition \ref{Def: motivic cohomology}. 
    \item If the abelian category $\mathrm{MHS}_{\RR}^{+}$ has cohomological dimension $1$ \cite[Corollary 1.10]{Bei83}, then by the Leray spectral sequence, we have the short exact sequence 
    \begin{equation}
    \label{Eq: ses for abs Hodge cohomology}
    0 \rightarrow \Ext^{1}_{\mathrm{MHS}_{\RR}^{+}}(\RR(0), \mathrm{H}^{i-1}(X,M)) \rightarrow \mathrm{H}_{H}^{i}(X,M) \rightarrow \Hom_{\mathrm{MHS}_{\RR}^{+}}(\RR(0), \mathrm{H}^{i}(X,M)) \rightarrow 0
    \end{equation}
    for all $i$ and all $M \in \mathrm{D}^{b}(\mathrm{MHM}_{\RR}(X/\RR))$. 
\end{enumerate}
\end{remark}

Now we come to the setting of Shimura varieties. 
\begin{notation}
\begin{itemize}
  \item 
    Recall that there is a $\RR$-linear tensor functor associating an algebraic representation of the group underlying a given Shimura variety to the corresponding variation of Hodge structure on the  Shimura variety, and we denote the functor as $\mu_{H}$ (see \cite[2]{Bow04}). 
  \item Similar to the last section, we let $\mathcal{G}$ temporarily denote any of the three groups $\{\GL_2,   H, G \}$, and we have the associated Shimura variety $\Sh_{\calG}$. So if $V \in \mathrm{Rep}_{\QQ}(\mathcal{G})$\footnote{For $H$ and $G$, the meaning of $V \in \mathrm{Rep}_{\QQ}(G)$ (reps. $V \in \mathrm{Rep}_{\QQ}(H)$) is that there exists some $V^{\prime} \in \mathrm{Rep}_{E}(G)$ (resp., $V^{\prime} \in \mathrm{Rep}_{E}(H)$) such that  $V = \Res_{E/\QQ} V^{\prime}$. }, we have the variation of Hodge stucture $\mu_{H}(V) \in \mathrm{Var}_{\RR}(\Sh_{\mathcal{G}})$. 
  \item Compose with the embedding $\mathrm{Var}_{\RR}(\Sh_{\mathcal{G}}/\RR) \rightarrow \mathrm{D}^{b}(\mathrm{MHM}_{\RR}(\Sh_{\mathcal{G}})/\RR)$, we can view $\mu_{H}(V)$ as an object of $\mathrm{D}^{b}(\mathrm{MHM}_{\RR}(\Sh_{\mathcal{G}}/\RR))$ which has a single non-trivial constituent in degree $0$. We will write $V$ instead of $\mu_{H}(V)$ when there is no confusion. In other words, we can view $V \in \mathrm{Rep}_{\QQ}(\mathcal{G})$ as an object of $\mathrm{D}^{b}(\mathrm{MHM}_{\RR}(\Sh_{\mathcal{G}}/\RR))$, and we can then define the absolute Hodge cohomology $\mathrm{H}_{H}^{i}(\Sh_{\mathcal{G}}, V)$. 
  \item Recall that we always let $M = \Sh_{H}$ and $S = \Sh_{G}$. 
\end{itemize}
\end{notation}

\begin{remark}
    Unlike the motivic setting, the ``branching'' for variation of Hodge structrues is direct from the definition: 
    for the map of Shimura varieties $\iota \colon M \rightarrow S$, we have the following commutative diagram of functors 
    \[
        \begin{tikzcd}
            \mathrm{Rep}(G) \ar[d, "\iota^{*}"] \ar[r, "\mu_{H}"] & \mathrm{Var}_{\RR}(S) \ar[d, "\iota^{*}"] \\
            \mathrm{Rep}(H) \ar[r, "\mu_{H}"] & \mathrm{Var}_{\RR}(M), 
        \end{tikzcd}
    \]
where the left-hand $\iota^{*}$ denotes the restriction of representations, and the right-hand $\iota^{*}$ denotes the pullback of variation of Hodge structures.
\end{remark}

Just as with motivic cohomolgy, we have the following Gysin morphism. 

\begin{proposition}[Gysin morphism for absolute Hodge cohomology]
\label{Prop: Gysin for abs Hdg cohomology}
    For the map of Shimura varieties $\iota \colon M \rightarrow S$ defined as before and for the representations $W \in \mathrm{Rep}_{\QQ}(H)$ and $V \in \mathrm{Rep}_{\QQ}(G)$ satisfying $W \hookrightarrow \iota^{*}V$ in $\mathrm{Rep}_{\QQ}(H)$, we have the Gysin morphism 
    \[
    \iota_{*} \colon \mathrm{H}^{1}_{H}(M, W(1)) \rightarrow \mathrm{H}^{3}_{H}(S, V(2)), 
    \]
    which is induced by 
    \begin{equation}
    \label{eq: Gysin morphism in derived category of MHM}
        (\iota_{*}W(1) \rightarrow V(2)[1]) \in \mathrm{D}^{b}(\mathrm{MHM}_{\RR}(S/\RR)). 
    \end{equation}
\end{proposition}
\begin{proof}
From the above discussions, we have a map $W \hookrightarrow (\iota^{s})^{*}V$ in $\mathrm{D}^{b}(\mathrm{MHM}_{\RR}(M)/\RR)$, where $\iota^{*}$ is the pullback functor for $\mathrm{Var}_{\RR}(\Sh_{\mathcal{G}})$. It follows from Proposition \ref{Prop: compare between VHS and MHM (tensor and pullback)} that $(\iota^{s})^{*}V = \iota^{*}V[-1]$,  where $\iota^{*}$ is the pullback functor for $\mathrm{D}^{b}(\mathrm{MHM}_{\RR}(S/\RR))$. By the absolute purity isomorphism (see Proposition \ref{Prop: absolute purity for MHM}), the right-hand side is $\iota^{*}V \cong \iota^{!}V(1)[2]$. By adjunction, we get the morphism $\iota_{!}W \rightarrow V(1)[1]$ in $\mathrm{D}^{\mathrm{b}}(\mathrm{MHM}_{\RR}(S/\RR))$. Since $\iota$ is proper, we have $\iota_{*} = \iota_{!}$, and we get the morphism 
\begin{equation*}
    (\iota_{*}W(1) \rightarrow V(2)[1]) \in \mathrm{D}^{b}(\mathrm{MHM}_{\RR}(S/\RR)). 
\end{equation*}
Finally, applying the functor $N \mapsto \Hom_{\mathrm{D}^{b}(\mathrm{MHM}(S/\RR))}(\RR(0)_{S}, N[2])$ to the map, we get the Gysin map $\mathrm{H}^{1}_{H}(M, W(1)) \rightarrow \mathrm{H}^{3}_{H}(S, V(2))$. 
\end{proof}

 \subsection{Beilinson regulator and functoriality}
 \label{SS: Beilinson regulator}
 In this subsection, we give the definition of Beilinson regulator using the Hodge realization functor for smooth quasi-projective varieties\footnote{This level of generality is enough for our applications.} over $\QQ$ and state its functorial properties.

 \begin{definition}[Hodge realization]
     \label{def: Hodge realization}
      For a smooth quasi-projective variety $S$ over $\QQ$, let $S_{\CC}$ and $S_{\RR}$ be the base change of $S$ to $\CC$ and $\RR$. The \textit{Hodge realization functor} is defined as the composition of the following functors: 
      \[
            \begin{tikzcd}
                R_{H} \colon \mathrm{DM}_{B,c}(S) \ar[r, "f_{1}^{*}"] &  \mathrm{DM}_{B,c}(S_{\RR}) \ar[r, "f_{2}^{*}"] & \mathrm{DM}_{B,c}(S_{\CC}) \ar[r, "R_{H}^{Bou}"] & \mathrm{D}(\mathrm{MHM}(S_{\CC})), 
            \end{tikzcd}
      \]
      where $f_{1} \colon S_{\RR} \rightarrow S$ and $f_{2} \colon S_{\CC} \rightarrow S_{\RR}$ are the canonical map and the functor $R_{H}^{Bou}$ is the Hodge realization functor from the triangulated category of constructible Beilinson motives $\mathrm{DM}_{B,c}(S_{\CC})$ over $S_{\CC}$ to the (unbounded) derived category $\mathrm{D}(\mathrm{MHM}(S_{\CC}))$ of algebraic mixed Hodge modules over $S_{\CC}$ defined in \cite[Definition 169, Theorem 43]{Bouli22}.
      \par It follows from the \'{e}tale descent property of constructible Beilinson motives (see \cite[Theorem 14.3.4]{Cisinski_Deglise}), the definition of the cateogory $\mathrm{MHM}_{\RR}(S/\RR)$ (see \cite[Definition A.2.4]{HW98}), and the functoriality of $R_{H}$ (see \cite[Theorem 43 (ii0)]{Bouli22}) that the image of $R_{H}$ lies in $\mathrm{D}^{b}(\mathrm{MHM}_{\RR}(S/\RR))$. Hence, we get the covariant Hodge realization functor 
      \begin{equation*}
          R_{H} \colon \mathrm{DM}_{B,c}(S) \rightarrow \mathrm{D}(\mathrm{MHM}_{\RR}(S/\RR)). 
      \end{equation*}
 \end{definition}

 \begin{remark}
     \begin{itemize}
         \item As the functor $R_{H}^{Bou}$ (see \cite[Page 1]{Bouli22}) is an extension of the functor 
               \[
                    \begin{tikzcd}[row sep = 0]
                        \mathrm{Variety} / S_{\CC} \ar[r] & \mathrm{D}(\mathrm{MHM}(S_{\CC})) \\
                            (f \colon X \mapsto S) \ar[r, mapsto] & f_{*} \RR_{X}, 
                    \end{tikzcd}
               \]
               where $\mathrm{Variety} / S_{\CC}$ is the category of $\CC$-varieties over $S_{\CC}$, we see that when we restrict $R_{H}^{Bou}$ to the sub-category $\mathrm{CHM}(S_{\CC})$, the image lies in $\mathrm{D}^{b}(\mathrm{MHM}(S_{\CC}))$. 
         \item We normalize the realization functor such that for any $V$ in the category of relative Chow motives $\mathrm{CHM}(S)$ over $S$ with $\dim S = d$, we have 
         \begin{equation*}
             R_{H}(V) = V[-d] \in \mathrm{D}^{\mathrm{b}}(\mathrm{MHM}_{\RR}(S/\RR)). 
         \end{equation*}
               
     \end{itemize}
 \end{remark}

 \begin{definition}[Beilinson regulator]
 \label{def: Beilinson regulator}
     For a smooth quasi-projective variety $S$ over $\QQ$, $\{i, j\} \in \ZZ$ and $W \in \mathrm{DM}_{B,c}(S)$, the \textit{Beilinson regulator} $r_H$ is defined as follows: 
     \begin{align*}
         & r_{H} \colon \mathrm{H}_{\mathrm{M}}^{i}(S,W(j)) = \Hom_{\mathrm{DM}_{B,c}(S)}(1_{S},W(j)[i]) \\ 
         & \rightarrow \mathrm{H}^{i}_{H}(S, W(j)) = \Hom_{\mathrm{D}^{b}(\mathrm{MHM}_{\RR}(S/\RR))}(R_{H}(1_{S}), R_H(W)(j)[i]),
     \end{align*}
     where $R_{H}$ is the Hodge realization functor defined above. 
 \end{definition}

 \begin{remark}
     \begin{itemize}
         \item When $\dim S = d$, we have 
         \begin{equation*}
             R_{H}(1_{S}) = \RR(0)_{S}[-d], 
         \end{equation*}
         which is the unit of $\mathrm{D}^{b}(\mathrm{MHM}_{\RR}(S/\RR))$. 
     \end{itemize}
 \end{remark}

The following proposition tells us that our definition is compatible with Beilinson's original definition (see \cite[\S 2]{Beilinson_conjecture_original}). 
 \begin{proposition}
 \label{Prop: compatible with Beilinson regulator}
 When $W = \QQ$, if we choose a canonical isomorphism 
 \begin{equation*}
     \mathrm{H}_{\mathrm{M}}^{i}(S, \QQ(j)) \cong \mathrm{K}_{2j-i}(S)_{\QQ}^{(j)},
 \end{equation*}
 then the regulator map 
 \begin{equation*}
     r_{H} \colon \mathrm{H}_{\mathrm{M}}^{i}(S,\QQ(j)) \rightarrow  \mathrm{H}^{i}_{H}(S, \RR(j))
 \end{equation*}
 is the same as the original definition of Beilinson in \textnormal{\cite[\S 2]{Beilinson_conjecture_original}}. 
     
 \end{proposition}

 \begin{proof}
     It follows from the functoriality of $R_{H}$ that the Beilinson regulator can be defined as 
     \begin{align*}
         & r_{H} \colon \mathrm{H}_{\mathrm{M}}^{i}(S,\QQ(j)) = \Hom_{\mathrm{DM}_{B, c}(\Spec (\QQ))}(1_{\Spec (\QQ)}, a_{*} 1_{S} (j)[i]) \\ 
         & \rightarrow \mathrm{H}^{i}_{H}(S, \RR(j)) = \Hom_{\mathrm{D}^{\mathrm{b}}(\mathrm{MHS}^{+}_{\RR})}(R_{H}(1_{\Spec (\QQ)}), R_H(a_{*} 1_{S})(j)[i]),
     \end{align*}
     where $a \colon S \rightarrow \Spec (\QQ)$ is the structure morphism of $S$. It can be seen from the definition of Huber's contravariant Hodge realization functor $R_{H}^{Hu} \colon \mathrm{DM}_{gm}(\QQ) \rightarrow \mathrm{D}^{\mathrm{b}}(\mathrm{MHS}^{+}_{\RR
     })$ (see \cite[Theorem 2.3.3]{Huber00} and \cite[Theorem B.2.2]{Huber04}) that in this case $r_{H}$ is the same as 
     \begin{align*}
         & r_{H}^{Hu} \colon \mathrm{H}_{\mathrm{M}}^{i}(S,\QQ(j)) \cong \Hom_{\mathrm{DM}_{gm}(\QQ)}((a_{*} 1_{S})^{*}, 1_{\Spec (\QQ)} (j)[i]) \\ 
         & \rightarrow \mathrm{H}^{i}_{H}(S, \RR(j)) = \Hom_{\mathrm{D}^{\mathrm{b}}(\mathrm{MHS}^{+}_{\RR})}(R_{H}^{Hu}(1_{\Spec (\QQ)}), R_H^{Hu}(a_{*} 1_{S})(j)[i]),
     \end{align*}
     where the notation $(\cdot)^{*}$ means taking dual. 
     Finally, we have the commutative diagram 
     \[
        \begin{tikzcd}
            & \mathrm{K}_{2j - i}(S) \ar[d, "p"] \ar[rd,"r_{H}^{B}"] &  \\
          \mathrm{H}_{\mathrm{M}}^{i}(S,\QQ(j))  \ar[ru, hook] \ar[r, "id"] & \mathrm{H}_{\mathrm{M}}^{i}(S,\QQ(j)) \ar[r, "r_{H}^{Hu}"] &  \mathrm{H}^{i}_{H}(S, \RR(j)), 
        \end{tikzcd}
     \]
     where $r_{H}^{B}$ is Beilinson's original definition of his regulator in \textnormal{\cite[\S 2]{Beilinson_conjecture_original}}, and the map $p$ is defined in \cite[Corollary 4.2.2]{Huber00}. The commutativity of the right triangle follows from \cite[Definition 18.2.6, Proposition 18.2.8]{Huberbook} and \cite[Corollary 4.2.3]{Huber04} and the commutativity of the left triangle comes from the canonical choice of $\mathrm{H}_{\mathrm{M}}^{i}(S, \QQ(j)) \cong \mathrm{K}_{2j-i}(S)_{\QQ}^{(j)}$. Thus, the proposition is proved.
 \end{proof}

 By the definition of the Beilinson regulator and functoriality of the Hodge realization functor, we have the following functorial properties of the Beilinson regulator. 

 \begin{proposition}
 \label{Prop: functoriality of Beilinson regulator}
     \begin{enumerate}
         \item For the morphism of Shimura varieties $p \colon M \rightarrow \Sh_{\GL_2}$ induced by the projection $H \rightarrow \GL_2$, we have the commutative diagram. 
               \[
                     \begin{tikzcd}
                        \mathrm{H}^{1}_{M}(\Sh_{\GL_2}, \Sym^{n} V_{2}(1)) \ar[r, "p^{*}"] \ar[d, "r_H"] & \mathrm{H}^{1}_{M}(M, W(1)) \ar[d, "r_H"] \\
                         \mathrm{H}^{1}_{H}(\Sh_{\GL_2}, \Sym^{n} V_{2}(1)) \ar[r, "p^{*}"]  & \mathrm{H}^{1}_{H}(M, W(1))  
                     \end{tikzcd}, 
               \]
               where $p^{*}$ is the pullback of the corresponding cohomomlogy theory. 
         \item  For the closed immersion of Shimura varieties $\iota \colon M  \rightarrow S$ and $V, W$ as in Proposition \ref{Prop: Gysin for motivic cohomology} and Proposition \ref{Prop: Gysin for abs Hdg cohomology}, we have the following commutative diagram:  \[
                        \xymatrix{
                            \mathrm{H}^{i}_{M}(M, W(1)) \ar[r]^{\iota_{*}} \ar[d]_{r_H} & \mathrm{H}^{3}_{M}(S, V(2)) \ar[d]_{r_H} \\
                            \mathrm{H}^{1}_{H}(M, W(1)) \ar[r]^{\iota_{*}}  & \mathrm{H}^{3}_{H}(S, V(2))  \\
                        }.  
                    \]
     \end{enumerate}
 \end{proposition}
   
\subsection{Motivic classes and their Hodge realizations} 
\label{SS: motivic class}
 In this subsection, we give the constructions of the motivic classes that we are interested in and compute their Hodge realizations. 

 \subsubsection{Eisenstein symbols}
 \label{SSS: Eisenstein symbol}
 We first recall the construction of Eisenstein symbols. 

 Recall that the \textit{Eisenstein symbol} \cite[\S3]{Beilinson_Modular_Curve} is a $\QQ$-linear map 
 \begin{equation*}
     Eis^{n}_{M} \colon \mathcal{B}_{n} \longrightarrow \mathrm{H}_{M}^{n + 1}(E^n, \QQ(n + 1))
 \end{equation*}
where $\mathrm{H}_{M}^{n + 1}(E^n, \QQ(n + 1))$ is the direct limit over compact open subgroups $K \subset \GL_{2}(\AAA_f)$ of the motivic cohomology $\mathrm{H}_{M}^{n + 1}(E^n_K, \QQ(n + 1))$ of the $n$-th fiber product of the universal elliptic curve $E_K / M_K$ over the modular curve of level $K$. 

\begin{definition}
\label{Def: source Eisenstein symbol}
The notation $\mathcal{B}_{n}$ stands for the space of locally constant $\QQ$-valued functions $\phi_f$ on $\GL_2(\AAA_f)$ statisfying the following conditions: 
\begin{enumerate}
    \item for all $a, d \in \QQ$ such that $ad > 0$ and for all $b \in \AAA_f$, we have 
          \begin{equation}
          \label{eqn:Eis_Symbol_(1)}
              \phi_f( \begin{pmatrix}
                        a & b \\
                        0  & d \\
                      \end{pmatrix} g) = a^{-1} d^{n+1} \phi_f(g), 
          \end{equation}
    \item \begin{equation*}
             \phi_f(\begin{pmatrix}
                        1 &  0 \\
                        0   & -1 \\
                    \end{pmatrix}g) = \phi_f(g), 
          \end{equation*}
     \item for all $k \in \hat{\ZZ}^{\times}$, we have 
           \begin{equation*}
               \phi_{f}(\begin{pmatrix}
                            1 & 0\\
                            0  & k\\ 
                        \end{pmatrix}g) = \phi_f(g). 
           \end{equation*}
\end{enumerate}
\end{definition}

\begin{remark}
    \begin{itemize}
    \item 
        The source of the Eisenstein symbol can identified with a space of $\QQ$-valued functions $\mathcal{F}^{n}$ on $\GL_2(\AAA_f)$ in \cite[2.1.2, P7]{Beilinson_Modular_Curve} by the following map
        \begin{equation*}
            \Psi_f \in \mathcal{F}^{n} \mapsto (\phi_f \colon g \mapsto \phi_f(g) = \Psi_f(^{t}g)) \in \mathcal{B}^{n}. 
        \end{equation*}
    \item At present, many of the results relating special value of $L$-functions to regulators rely on Eisenstein symbols (\cite{Beilinson_Modular_Curve}, \cite{Deninger89}, \cite{Deninger90}, \cite{Kings98}, \cite{LemmaI15}, \cite{LemmaII17}, \cite{Kato04}, ...).
    \end{itemize}
\end{remark}

\begin{notation}
    By the definition of the motivic sheaf $\Sym V_2(1)$, the motivic cohomology $\mathrm{H}_{M}^{1}(M, \Sym^{n} V_{2}(1))$ is a direct summand of $\mathrm{H}_{M}^{n + 1}(E^n, \QQ(n + 1))$. Moreover, the Eisenstein symbol map factors through the inclusion $\mathrm{H}_{M}^{1}(M, \Sym^{n} V_{2}(1)) \subset \mathrm{H}_{M}^{n + 1}(E^n, \QQ(n + 1))$ and we will also use the notation $Eis_{M}^{n}$ to denote the map
    \begin{equation*}
        Eis_{M}^{n} \colon \mathcal{B}_{n} \longrightarrow \mathrm{H}_{M}^{1}(M, \Sym^{n} V_{2}(1)). 
    \end{equation*}
\end{notation}

Let $\nu$ be a finite-order Hecke character, so its archimedean component $\nu_{\infty}$ is a unitary character. 

\begin{definition}
\label{Def: In(v)}
Let $I_{n}(\nu)$ be the space of locally constant $\bar{\QQ}$-valued functions $f$ such that for all $a, d \in \QQ_{+}^{\times}$, $\alpha, \delta \in \hat{\ZZ}^{\times}$, $b \in \AAA_f$ and $g \in \GL_2(\AAA_f)$, 
          the function $f$ satisfies that 
          \begin{equation}
          \label{eq: In(v) action}
              f(\begin{pmatrix}
                  a\alpha & b \\ 
                          & d\delta \\
                \end{pmatrix}g) = a^{-1}d^{n + 1}\nu(\alpha)f(g), 
          \end{equation}
Thus, the space $I_n(\nu)$ is endowed with the action of $\GL_2(\AAA_f)$ by right translation. 
\end{definition}

\begin{lemma}
\label{Lemma: decomposition of source of Eisenstein symbol}
    There is a $\GL_2(\AAA_f)$-equivariant decomposition 
    \begin{equation*}
        \mathcal{B}_{n, \bar{\QQ}} = \bigoplus_{\nu, \mathrm{sgn}(\nu) = (-1)^{n}} I_n(\nu),
    \end{equation*}
    where the direct sum is indexed by all finite-order Hecke characters of sign $(-1)^{n}$. 
\end{lemma}

\begin{proof}
    The proof follows from the proof of \cite[Lemma 4.3]{LemmaII17}. Let $T_2$ be the diagonal maximal torus of $\GL_2$. We are interested in the action of $T_2(\AAA_f)$ on $\mathcal{B}_{n, \bar{\QQ}}$ by left translation. It follows from the decomposition $\AAA_f^{\times} = \QQ^{\times}_{+} \hat{\ZZ}^{\times}$ and equation (\ref{eqn:Eis_Symbol_(1)}) that we only need to consider the action of $T_2(\hat{\ZZ}^{\times})$. Since each function in $\mathcal{B}_{n, \bar{\QQ}}$ is locally constant, it is fixed by an compact open subgroup of $\GL_2(\hat{\ZZ}^{\times})$. Hence, $\mathcal{B}_{n, \bar{\QQ}}$ can be decomposed into a direct sum of finite-dimensional $\bar{\QQ}$-vector spaces that are stable under the action of $T_{2}(\hat{\ZZ}^{\times})$. For each of the finite-dimensional vector spaces $V$, $V$ is a direct sum of finite order Hecke characters
    \begin{equation*}
        \chi \colon \begin{pmatrix}
                \alpha & 0 \\
                0 & \delta  
              \end{pmatrix} \mapsto \nu(\alpha) \chi(\delta)
    \end{equation*}
    because the action is continuous.
    By equation (\ref{eq: In(v) action}), we can see that $\chi$ is trivial. Finally, if we plug $\begin{pmatrix}
                                                                                                -1 & 0 \\
                                                                                                0 & -1 
                                                                                           \end{pmatrix}$                                                      into  equation (\ref{eqn:Eis_Symbol_(1)}), then we get that $\nu$ has to have sign $(-1)^{n}$. 
\end{proof}

\subsubsection{Motivic classes and their image under the Beilinson regulator}
\label{SSS: motivic classes and Hodge}
Using the Eisenstein symbol, we will give the construction of  motivic classes, which has been done in \cite[Definition 9.2.3]{LSZ22} and compute 
their Hodge realizations. 

\begin{construction}
\label{Construction: motivic classes}
    We compose the Eisenstein symbol with the pullback 
    \begin{equation*}
    p^{*} \colon \mathrm{H}^{1}_{M}(\Sh_{\GL_2},\mathrm{Sym}^{n} V_2(1)) \rightarrow \mathrm{H}^{1}_{M}(M, W(1))
    \end{equation*}
    and the Gysin map 
    \begin{equation*}
        \iota_{*} \colon \mathrm{H}^{1}_{M}(M, W(1)) \rightarrow \mathrm{H}^{3}_{M}(S, V(2))
    \end{equation*}
    to get the map 
    \begin{equation*}
        \mathcal{E}is^{n}_{M} \colon \mathcal{B}_{n} \rightarrow \mathrm{H}^{3}_{M}(S, V(2)). 
    \end{equation*}
    Hence, for any $\phi_{f} \in \mathcal{B}_{n}$, we a have motivic class $\mathcal{E}is_{M}^{n}(\phi_f) \in \mathrm{H}^{3}_{M}(S, V(2))$. In summary, we have the following diagram\footnote{Our construction is the Weil restriction to $\QQ$ of the construction over $E$ in \cite[Definition 9.2.3.]{LSZ22} and it is more convenient for Beilinson's conjectures.}: 
    \[
        \begin{tikzcd}[row sep=0pt]
            \mathcal{B}_n \ar[r, "Eis^n_M"] & \mathrm{H}^{1}_{M}(\Sh_{\GL_2}, \mathrm{Sym}^{n}V_2(1)) \ar[r, "p^{*}"] & \mathrm{H}^{1}_{M}(M, W(1)) \ar[r, "\iota_{*}"] & \mathrm{H}^{3}_{M}(S,V(2)) \\
            \phi_f \ar[r, mapsto] & Eis^n_M(\phi_f) \ar[rr, mapsto] & & \mathcal{E}is^{n}_{M}(\phi_f) 
        \end{tikzcd}. 
    \]    
\end{construction}

By Proposition \ref{Prop: functoriality of Beilinson regulator}, we have the following commutative diagram: 
\[
    \begin{tikzcd}
        \mathcal{B}_n \ar[r, "Eis^n_M"] \ar[d, "r_{H}"] & \mathrm{H}^{1}_{M}(\Sh_{\GL_2}, \mathrm{Sym}^{n}V_2(1)) \ar[r, "p^{*}_{M}"] \ar[d, "r_{H}"] & \mathrm{H}^{1}_{M}(M, W(1)) \ar[r, "\iota_{M, *}"] \ar[d,"r_{H}"] & \mathrm{H}^{3}_{M}(S,V(2)) \ar[d, "r_{H}"] \\
        \mathcal{B}_{n, \RR} \ar[r, "Eis^n_H"]  & \mathrm{H}^{1}_{H}(\Sh_{\GL_2}, \mathrm{Sym}^{n}V_2(1)) \ar[r, "p^{*}_{H}"] & \mathrm{H}^{1}_{H}(M, W(1)) \ar[r, "\iota_{H, *}"]  &  \mathrm{H}^{3}_{H}(S,V(2))
    \end{tikzcd}, 
\]
where we use the subscript $M$ and $H$ to distinguish maps of motivic cohomology and absolute Hodge cohomology. 
Hence, the image of the motivic classes $\mathcal{E}is_{M}^{n}(\phi_f)$ under the Beilinson regulator $r_{H}$ is 
\begin{align*}
    \mathcal{E}is_{H}^{n}(\phi_f) &= r_{H} (\iota_{M, *} \circ p_{M}^{*} \circ Eis_{M}^{n}(\phi_{f})) \\
                                  &= \iota_{H, *} \circ p_{H}^{*} \circ Eis_{H}^{n}(r_{H} (\phi_{f})). 
\end{align*}    

\begin{convention}
    In the rest of the paper, when the context is clear, we will ignore the subscript $M$ and $H$, and only write $\iota_{*}$, $p^{*}$ and $Eis^{n}$. 
\end{convention}


\section{The vanishing on the boundary of absolute Hodge cohomolgy}
\label{Sec: vanishing of the boudary abs Hodge}
\subsection{The Baily{\textendash}Borel compactification of Picard modular surfaces}
\label{SS: Baily-Borel compactification of PMF} 
In this subsection, we recall the basics of the The Baily{\textendash}Borel compactification of Picard modular surfaces.

\par The boundary of the Baily-Borel compactification of a Shimura variety is stratified by Shimura varieties associated to standard (containing the Borel) admissible parabolic subgroups \cite[4.5 Definition]{Pink1} of the underlying reductive group $G$.

\par Let us recall the structure of standard admissible parabolic subgroups in the cases that we are interested in. 

\begin{itemize}
    \item The reductive group is $H$ (Definition \ref{defn:G_H}): Recall that $H$ is $\GL_2 \boxtimes E^{\times}$\footnote{In the following sections, we use $E^{\times}$ to denote $\mathrm{Res}_{E/\QQ} \Gm$.}, where the notation $\boxtimes$ denotes a pair with the same norm. If we denote by $B_2 = T_{2}N^{'}$ the standard Borel of $\GL_2$ together with its Levi decomposition,  then the standard admissible parabolic subgroups $Q^{\prime}$ of $H$ are isomorphic to $B_2 \boxtimes E^{\times}$, whose Levi decomposition is denoted as $Q^{\prime} = M^{\prime}N^{\prime}$. The canonical normal subgroup $P_{1}^{\prime}$ of $Q^{\prime}$ defined in \cite[4.7]{Pink1} is 
    \begin{equation*}
        P_{1}^{\prime} = \{ \begin{pmatrix}
                            * & * \\
                            0 & 1
                        \end{pmatrix} \} \boxtimes E^{\times} \subseteq H, 
    \end{equation*}
    whose Levi is 
    \begin{equation*}
         M_{1}^{\prime} = \{  \begin{pmatrix}
                                    * & 0 \\
                                    0 & 1
                             \end{pmatrix} \} \boxtimes E^{\times} \subseteq H. 
    \end{equation*}
    \item The reductive group is $G$ (Definition \ref{defn:G_H}): The standard admissible parabolic subgroups $Q$ of $G$ is isomorphic to the standard Borel $B = \begin{pmatrix}
                                             * & * & * \\
                                             0 & * & * \\
                                             0 & 0 & * 
                                          \end{pmatrix}$ of $G$,  whose Levi decomposition is denoted by $Q = MN$ (see \cite[Proposition 3.3]{Anc17}).  The canonical normal subgroup $P_{1}$ of $Q$ defined in \cite[4.7]{Pink1} is 
    \begin{equation*}
        P_1 = \{ \begin{pmatrix}
                    * & * & * \\
                    0 & * & * \\
                    0 & 0 & 1 
                  \end{pmatrix} \} \subseteq G, 
    \end{equation*}
    whose Levi is 
    \begin{equation*}
        M_1 = \{ \begin{pmatrix}
                     * & 0 & 0 \\
                     0 & * & 0 \\
                     0 & 0 & 1
                 \end{pmatrix}\} \subseteq G 
    \end{equation*} (see \cite[Lemma 3.8]{Anc17}).
\end{itemize}

 \par For each of the two Shimura varieties $X = \{M, S\}$, we denote by $X^{*}$ its Baily{\textendash}Borel compactification and let $\partial{X} = X^{*} - X$ be the cusps of the compactification. Then we have the commutative diagram.
\begin{equation}
\label{eq: diagram of compactification}
        \begin{tikzcd}
            M \ar[r, "j^{\prime}"] \ar[d, "{\iota}"] & M^{*} \ar[d, "p"] & \partial{M} \ar[l, "i^{\prime}" above] \ar[d, "q"] \\
            S \ar[r, "j"] & S^{*} & \partial{S} \ar[l, "i" above]  
        \end{tikzcd},
\end{equation}
where $j$ and $j^{\prime}$ are open immersions, $i$ and $i^{\prime}$ are closed immersions and $p$ and $q$ are closed immersions induced by the closed immersion $\iota$ \footnote{For this to hold, we need to choose level structures carefully.}. We also have 
\begin{equation*}
    \dim \partial S = 0, \, \dim \partial M = 0. 
\end{equation*}

\subsection{Degeneration of Hodge structures}
\label{SS: thm of BW and Kostant}
In this subsection, we compute the degeneration of Hodge structures over Shimura varieties $M$ and $S$, which will be used in the proof of vanishing on the boundary.

\begin{notation}
In the remaining part of this section, we use $\mu$ (instead of $\mu_{H}$) to denote the tensor functor that associates an algebraic representation of a reductive group to the corresponding variation of Hodge structure on the corresponding Shimura variety.  
\end{notation}

\subsubsection{The theorem of Burgos-Wildehaus and Kostant}
We first recall the theorem of Burgos-Wildehaus in the setting of $M$ and $S$. 

\begin{theorem}[\cite{BW04}, Theorem  2.6, 2.9]
\label{theorem: BW04}
Let $E$ (resp., $E^{\prime}$) be an algebraic representation of $G$ (resp., $H$). In the derived category $\mathrm{D}^{b}(\mathrm{MHM}_{\RR}(S))$ (resp., $\mathrm{D}^{b}(\mathrm{MHM}_{\RR}(M))$), we have 
\begin{align*}
    & i^{*}j_{*} \mu(E) \cong \bigoplus_{n} \mathcal{H}^{n}i^{*}j_{*} \mu(E)[-n] \\ 
    (\text{resp.,} \, & i^{\prime, *}j_{*}^{\prime} \mu(E^{\prime}) \cong \bigoplus_{n} \mathcal{H}^{n}i^{\prime, *}j_{*}^{\prime} \mu(E^{\prime})[-n]). 
\end{align*}
Furthermore, we have 
\begin{align*}
    & \mathcal{H}^{n}i^{*}j_{*} \mu(E) \cong  \mu(\mathrm{H}^{n + 2}(N, E)) \\
   (\text{resp.,} \,  & \mathcal{H}^{n}i^{\prime, *}j_{*}^{\prime} \mu(E^{\prime}) \cong  \mu(\mathrm{H}^{n + 1}(N^{\prime}, E^{\prime})).
\end{align*}
\end{theorem}

\begin{remark}
    \begin{enumerate}
        \item The group $G_1$ (resp., $G_{1}^{\prime}$) acts on the group cohomology of the unipotent group $\mathrm{H}^{q}(N, E)$ (resp., $\mathrm{H}^{q}(N^{\prime}, E^{\prime})$) via its action both on $N$ (resp., $N^{\prime}$) and $E$ (resp., $E^{\prime}$), so the second isomorphism can be viewed an isomorphism of representations of $G_1$ (resp., $G_{1}^{\prime}$). 
        \item The theorem of Burgos-Wildehaus and Kostant works for general Shimura varieties. 
    \end{enumerate}
\end{remark}

Now, let us introduce the necessary notations for Kostant's theorem. 

\begin{notation}
    \begin{enumerate}
        \item 
        Recall that for any unipotent group $N$ and any representation $E$ of $N$, we have $\mathrm{H}^{q}(N, E) = \mathrm{H}^{q}(\lien, E)$ where $\lien$ is the Lie algebra of $N$ and the right hand side is Lie algebra cohomology. 
        \item Let $Q$ be a parabolic  subgroup of a reductive group $G$ with Levi decomposition $Q = MN$, and $\lieq = \lien \oplus \liem$ where $\lieq = \mathrm{Lie}(Q)$, $\lien = \mathrm{Lie}(N)$ and $\liem = \mathrm{Lie}(M)$. Let $\lieh$ be the Cartan subalgebra of $\lieg = \mathrm{Lie}(G)$ corresponding to a fixed Borel and $\mathrm{R}^{+}(\lieg, \lieh)$ be the set of positive roots. We also write $\mathrm{R}(\lien, \lieh)$ as the subset of $\mathrm{R}^{+}(\lieg, \lieh)$ appearing in $\lien$. Denote by $\rho$ the half-sum of positive roots, and by $\mathrm{W}(\lieg, \lieh)$ the Weyl group. For any $w \in W(\lieg, \lieh)$, we write 
        \begin{itemize}
            \item  $\mathrm{R}^{+}(w) = \{ \alpha \in \mathrm{R}^{+}(\lieg, \lieh) | w^{-1} \alpha \notin \mathrm{R}^{+}(\lieg, \lieh) \}$, 
            \item $\mathrm{l}(w) = |\mathrm{R}^{+}(w)|$,
            \item $\mathrm{W}^{\prime} = \{w \in \mathrm{W}(\lieg, \lieh) | \mathrm{R}^{+}(w) \subset \mathrm{R}(\lien, \lieh) \}$. 
        \end{itemize}
    \end{enumerate}
\end{notation}

\begin{theorem} [\cite{Vogan81}, Theorem 3.2.3]
\label{thm: Kostant_theorem}
    Let $E_{\lambda}$ be an irreducible representation of $\lieg$ of highest weight $\lambda$. Then 
    \begin{equation*}
        \mathrm{H}^{q}(\lien, E_{\lambda}) = \bigoplus_{\{ w \in \mathrm{W}^{'} | \mathrm{l}(w) = q \}} F_{w(\lambda + \rho) - \rho}
    \end{equation*}
    as representations of $G_1$, where $F_{\mu}$ is an irreducible representation of $\liem$ of highest weight $\mu$. 
\end{theorem}

\subsubsection{Degeneration of Hodge structures on $M$ and $S$}
Now we give the explicit description of degeneration of Hodge structures on the Shimura varieties $M$ and $S$. \\

\par We first consider the case of $M$. 

\begin{notation}
Let us denote by $\lambda^{\prime}(d_1, d_2)$ the algebraic representation of 
\begin{equation*}
    G_1^{\prime} = \{(x;z) \in \Gm \boxtimes E^{\times} | x = z \bar{z} \}
\end{equation*}
with (highest) weight $(x; z) \mapsto z^{d_1}\bar{z}^{d_2}$. 
\end{notation}

\begin{lemma}
\label{lemma: BW_H}
    Consider the diagram 
    \[
        \xymatrix{
            M \ar[r]^-{j^{\prime}}  & M^{*}  & \partial{M} \ar[l]_-{i^{\prime}},  
        }. 
    \]
    As variations of mixed Hodge structures on $\partial M$, we have 
    \begin{equation}
    \label{eq: H: deg of MHS}
        i^{\prime*}j_{*}^{\prime} \mu(W) \cong \bigoplus_{n = -1}^{0} \mathcal{H}^{n}i^{\prime*}j_{*}^{\prime} \mu(W)[-n], 
    \end{equation}
    where 
    \begin{align*}
        & \mathcal{H}^{-1}i^{\prime*}j_{*}^{\prime}\mu(W) = \mu(\lambda^{\prime}(a + b + r + s, a + b + r + s)), \\
        & \mathcal{H}^{0}i^{\prime*}j_{*}^{\prime}\mu(W) = \mu(\lambda^{\prime}(-1, -1)), \\
        & \mathcal{H}^{n}i^{\prime*}j_{*}^{\prime}\mu(W) = 0, \ \text{for} \ n < -1 \ \text{or} \ n > 0. 
    \end{align*}
    The Hodge type of the corresponding variation of mixed Hodge structures are listed below: 
     \begin{equation}
     \label{eq: H: Hodge type}
            \begin{array}{|c|c|}
            \hline \text{VMHS} & \text{ Hodge type} \\
            \hline \mu(\lambda^{\prime}(a + b + r + s, a + b + r + s)) & (-(a + b + r + s), -(a + b + r + s))) \\
            \hline \mu(\lambda^{\prime}(-1, -1)) & (1, 1), \\
            \hline 
            \end{array}, 
    \end{equation}
    where ``VMHS" stands for variations of mixed Hodge structures.    
\end{lemma}

\begin{proof}
    It follows from Theorem \ref{theorem: BW04} and Theorem \ref{thm: Kostant_theorem} that 
    \begin{equation*}
        \mathcal{H}^{n}i^{\prime*}j_{*}^{\prime}\mu(W) =  \begin{cases}
                                                \mu(\mathrm{H}^{n + 1}(N^{\prime}, W)) & \text{when} \, n \ge -1, \\ 
                                                0 & \text{when} \, n < -1.
                                           \end{cases} 
    \end{equation*}
    Moreover, in our case when $\mathrm{W}^{\prime} = \mathrm{W} = S_{2}$ and $\rho = (\frac{1}{2}, -\frac{1}{2}; 0)$, we have the following table of elements $w \in \mathrm{W}^{\prime}$:
    \[
        \begin{array}{|c|c|c|}
           \hline w  & \mathrm{l}(w) & w(\lambda + \rho) - \rho \\
           \hline 1 & 0 & (\lambda_1, \lambda_2; d) \mapsto (\lambda_1, \lambda_2; d) \\
           \hline (12) & 1 & (\lambda_1, \lambda_2; d) \mapsto (\lambda_2 - 1, \lambda_1 + 1; d) \\
           \hline
        \end{array}. 
    \]
    From this,  we can see that 
     \begin{equation*}
        \mathcal{H}^{n}i^{\prime*}j_{*}^{\prime}\mu(W) = 0, \, \, \text{for} \,\,  n > 0. 
    \end{equation*}
    For $W = \Sym^{n}V_{2}$ where $n = a + b + r + s$, we get as algebraic representations of $G_1^{\prime}$, 
    \begin{align*}
       & \mathrm{H}^{0}(N^{\prime}, W) = \lambda^{\prime}(a + b + r + s, a + b + r + s), \\
       & \mathrm{H}^{1}(N^{\prime}, W) = \lambda^{\prime}(-1, -1).
    \end{align*}
\end{proof}

\begin{remark}
    This computation is essentially the same as \cite[Lemma 4.4]{LemmaI15}. 
\end{remark}

Now, let us consider the case of $S$. 
\begin{notation}
    Let us denote by $\lambda(\mu_1, \mu_2; d)$ the algebraic representation of 
    \begin{equation*}
        G_1 = \{\begin{pmatrix}
                x & & \\
                & z & \\
                &  & 1 
               \end{pmatrix} \} \subseteq G
    \end{equation*}
    with (highest) weight $x^{\mu_1}z^{\mu_2}(z\bar{z})^{d}$. 
\end{notation}
\begin{lemma}
\label{lemma: BW_G}
    Consider the diagram 
    \[         \xymatrix{             
                    S \ar[r]^-{j}  & S^{*}  & \partial{S} \ar[l]_-{i}.        
                }     
    \]
    As variations of mixed Hodge structures on $\partial S$, we have 
    \begin{equation}
    \label{eq: G: deg of MHS}
        i^{*}j_{*} \mu(V) \cong \bigoplus_{n = -2}^{1} \mathcal{H}^{n}i^{*}j_{*} \mu(V)[-n], 
    \end{equation}
    where
    \begin{align*}
         \mathcal{H}^{-2}i^{*}j_{*}\mu(V) &= \mu(\lambda(a + r - s, r - s; 2s - r + b)), \\
         \mathcal{H}^{-1}i^{*}j_{*}\mu(V) &= \mu(\lambda(r - s - 1, a + r -s; 2s - r + b)) \\ &\oplus \mu(\lambda(a + r - s, r - s - b - 1; 2s - r + b)), \\
         \mathcal{H}^{0}i^{*}j_{*}\mu(V) &=  \mu(\lambda(r - s - b - 2, a + r - s + 1; 2s - r + b)) \\ &\oplus \mu(\lambda(r - s - 1, r - s - b - 1; 2s - r + b)), \\
        \mathcal{H}^{1}i^{*}j_{*}\mu(V) &= \mu(\lambda(r - s - b - 2, r - s; 2s - r + b), \\
        \mathcal{H}^{n}i^{*}j_{*}\mu(V) &= 0, \ \text{for} \ n < -2 \ \text{or} \ n > 1. 
    \end{align*}
    The Hodge type of the corresponding variation of mixed Hodge structures are listed as follows, where ``VMHS" stands for variations of mixed Hodge structures.
    \begin{equation}
    \label{eq: G: Hodge type}
            \begin{tabular}{|c|c|}
            \hline 
            \text{VMHS} & \text{ Hodge type} \\
            \hline 
            \multirow{2}*{$\mu(\lambda(a + r - s, r - s; 2s - r + b))$} & $(-(a + b + r), -(a + b + s))$ \\
            \cline{2-2}
            ~ & $(-(a + b + s), -(a + b + r))$ \\
            \hline 
            \multirow{2}*{$\mu(\lambda(r - s - 1, a + r -s; 2s - r + b))$} & $(-(a + b + r), -(b + s - 1))$ \\
            \cline{2-2}
            ~ & $(-(b + s - 1), -(a + b + r))$ \\
            \hline 
            \multirow{2}*{$\mu(\lambda(a + r - s, r - s - b - 1; 2s - r + b))$} & $(-(a + r - 1), -(a + b + s))$ \\ 
            \cline{2-2}
            ~ & $(-(a + b + s), -(a + r - 1))$\\
            \hline 
            \multirow{2}*{$\mu(\lambda(r - s - b - 2, a + r - s + 1; 2s - r + b))$} & $(-(a + r - 1), -(s - 2))$ \\ 
            \cline{2-2}
            ~ & $(-(s - 2), -(a + r - 1))$ \\
            \hline 
            \multirow{2}*{$\mu(\lambda(r - s - 1, r - s - b - 1; 2s - r + b))$} & $(-(r - 2), -(s + b - 1))$ \\
            \cline{2-2}
            ~ & $(-(s + b - 1), -(r - 2))$ \\
            \hline 
            \multirow{2}*{$\mu(\lambda(r - s - b - 2, r - s; 2s - r + b)$} & $(-(r - 2), -(s - 2))$ \\ 
            \cline{2-2}
            ~ & $(-(s - 2), -(r - 2))$ \\
            \hline 
            \end{tabular}
    \end{equation}
\end{lemma}

\begin{proof}
It follows from Theorem \ref{theorem: BW04} that 
\begin{equation*}
    \mathcal{H}^{n}i^{*}j_{*}\mu(V) = \mu(\mathrm{H}^{n + 2}(N, V)). 
\end{equation*}
It follows from the fact $\mathrm{R}(\lien, \lieh) = \mathrm{R}^{+}(\lieg, \lieh)$ that $\mathrm{W}^{'} = \mathrm{W}$. By $\rho = (1, 0, -1; 0)$, we have the following table about elements $w \in \mathrm{W}^{'}$: 
\[
\begin{array}{|c|c|c|}
\hline w & \mathrm{l}(w) & w(\lambda + \rho) - \rho \\ 
\hline 1 & 0 & (\lambda_1, \lambda_2, \lambda_3; d) \mapsto (\lambda_1, \lambda_2, \lambda_3; d) \\
\hline (12) & 1 & (\lambda_1, \lambda_2, \lambda_3; d) \mapsto (\lambda_2 - 1, \lambda_1 + 1, \lambda_3; d) \\
\hline (13) & 3 & (\lambda_1, \lambda_2, \lambda_3; d) \mapsto (\lambda_3 - 2, \lambda_2, \lambda_1 + 2; d) \\
\hline (23) & 1 & (\lambda_1, \lambda_2, \lambda_3; d) \mapsto (\lambda_1, \lambda_3 - 1, \lambda_2 + 1; d) \\
\hline (132) & 2 & (\lambda_1, \lambda_2, \lambda_3; d) \mapsto (\lambda_3 - 2, \lambda_1 + 1, \lambda_2 + 1; d) \\
\hline (123) & 2 & (\lambda_1, \lambda_2, \lambda_3; d) \mapsto (\lambda_2 - 1, \lambda_3 - 1, \lambda_1 + 2; d) \\
\hline
\end{array}. 
\]
So if $n < -2$ or $n > 1$, we have 
\begin{equation*}
    \mathcal{H}^{n}i^{*}j_{*}\mu(V) = 0. 
\end{equation*}
Therefore, for $V = V^{a, b} \{r, s \} = V(a + r - s, r - s, r - s - b; 2s - r + b) $ and $ -2 \le n \le 1$, we get the results listed in the proposition. 
\end{proof}

\begin{remark}
    The computation is carried out in \cite[\S 4]{Anc17}. Our parametrization is different from \cite{Anc17}, so we reproduce it here. 
\end{remark}

\subsection{A vanshing theorem for Betti cohomology}
\label{SS: interior and boundary cohomologies}
We recall the following vanishing theorem.

\begin{theorem}[\cite{Sap1}, Theorem 5]
\label{thm: Sap}
    \par Let $G$ be a reductive algebraic group defined over $\QQ$ and let $D$ be the associated symmetric space, which can be written as $D = G(\RR) / K_{\infty}A_G$, where $K_{\infty}$ is a maximal compact of $G(\RR)$ and $A_{G}$ is the identity component of $\RR$-points of a maximal $\QQ$-split torus in the center of $G$. Let $\Gamma \subset G(\QQ)$ be a neat arithmetic subgroup, $X = \Gamma \backslash D$ be the associated locally symmetric space and $E$ be the local system on $X$ associated to an irreducible algebraic representation of $G$.
    \par If $D$ is a Hermitian or equal-rank symmetric space and the highest weight of $E$ is \textit{regular}, then $\mathrm{H}^{i}(X, E) = 0$ for all $i < \frac{1}{2}\dim X$, where $\dim X$ is the real dimension of the locally symmetric space $X$. 
\end{theorem}

\begin{remark}
    This result is announced without a detailed proof in \cite{Sap1} and proved in the preprint \cite{Sap2}. 
\end{remark}

\subsection{Vanishing on the boundary for absolute Hodge cohomology}
\label{SS: the proof of the Hodge vanishing}
In this subsection, we state and prove the theorem of vanishing on the boundary for absolute Hodge cohomology. 

\begin{convention}
\label{convention: the proof ov Hodge vanishing}
    \begin{itemize}
        \item In this subsection, Betti cohomology, compactly supported cohomology and interior cohomology all use $\RR$-coefficients. We will not stress this in the notation used for  cohomology.
        \item For Betti and compactly supported cohomology, we do not write the subscript $B$ but for interior cohomology, we write the subscript $B$. 
        \item In this subsection, we ignore the notation $\mu$ and write $V$ (resp., $W$) instead of $\mu(V)$ (resp. $\mu(W)$). 
    \end{itemize}
\end{convention}

\subsubsection{The theorem}
We first give the statement of the theorem. 

\begin{theorem}
\label{Thm: Hdg vanish on the boudary}
Under the conditions
\begin{enumerate}
     \item $0 \le -r \le a$ and  $0 \le -s \le b$, 
     \item $a > 0$ and $b > 0$,
     \item $r \neq 0$ or $s \neq 0$, 
\end{enumerate}
the map $\mathcal{E}is_{H}^n \colon \mathcal{B}_{n,\RR} \rightarrow \mathrm{H}^{3}_{H}(S, V(2))$ factors through the inclusion 
\[
    \Ext^{1}_{\mathrm{MHS}_{\RR}^{+}}(\mathbf{1}, \mathrm{H}^{2}_{B,!}(S, V(2))) \hookrightarrow \mathrm{H}^{3}_{H}(S, V(2)), 
\]
where $\mathrm{MHS}_{\RR}^{+}$ is the abelian category of mixed $\RR$-Hodge structures and ${\mathbf{1}}$ denotes the trivial Hodge structure that is the unit of $\mathrm{MHS}_{\RR}^{+}$. 
\end{theorem}

\subsubsection{The proof the theorem}
We give a proof of Theorem \ref{Thm: Hdg vanish on the boudary} here and the proof is divided into three steps. \\

\noindent \textbf{Step I.} In the first step, the goal is to prove the exactness of the sequence in Proposition \ref{Prop: Hodge exact sequence}. 

\begin{lemma} 
\label{Lemma:``Hom = 0" preparation1.}
    We have 
    \begin{align*}
        & \Hom_{\mathrm{MHS_{\RR}^{+}}}(\RR(0), \mathrm{H}^{0}(\partial S,i^{*}j_{*}V(2))) = 0, \\
        & \Hom_{\mathrm{MHS_{\RR}^{+}}}(\RR(0), \mathrm{H}^{1}(\partial S,i^{*}j_{*}V(2))) = 0, \\
        & \Hom_{\mathrm{MHS_{\RR}^{+}}}(\RR(0), \mathrm{H}^{3}(S,V(2))) = 0. 
    \end{align*}
\end{lemma}

\begin{proof}
    \par First, by equation (\ref{eq: G: deg of MHS}) and $\dim \partial S = 0$, we have 
        \begin{align*}
            \Hom_{\mathrm{MHS_{\RR}^{+}}}(\RR(0), \mathrm{H}^{0}(\partial S,i^{*}j_{*}V(2))) = \Hom_{\mathrm{MHS_{\RR}^{+}}}(\RR(0), \mathcal{H}^{0}i^{*}j_{*}V(2)), \\
            \Hom_{\mathrm{MHS_{\RR}^{+}}}(\RR(0), \mathrm{H}^{1}(\partial S,i^{*}j_{*}V(2))) = \Hom_{\mathrm{MHS_{\RR}^{+}}}(\RR(0), \mathcal{H}^{1}i^{*}j_{*}V(2)).
        \end{align*}
        By Table (\ref{eq: G: Hodge type}) in Lemma \ref{lemma: BW_G}, we can see that there is no part with Hodge type $(0, 0)$ contained in $\mathcal{H}^{0}i^{*}j_{*}V(2)$ and $\mathcal{H}^{1}i^{*}j_{*}V(2)$. Hence, we get
        \begin{align*}
             & \Hom_{\mathrm{MHS_{\RR}^{+}}}(\RR(0), \mathrm{H}^{0}(\partial S,i^{*}j_{*}V(2))) = 0, \\
        & \Hom_{\mathrm{MHS_{\RR}^{+}}}(\RR(0), \mathrm{H}^{1}(\partial S,i^{*}j_{*}V(2))) = 0. 
        \end{align*}
        Second, by the exact sequence in equation (\ref{eq: les of singular cohomology}) and Theorem \ref{thm: Sap}, we have 
        \begin{equation*}
            \mathrm{H}^{3}(S,V(2)) \cong \mathrm{H}^{1}(\partial S,i^{*}j_{*}V(2))). 
        \end{equation*}
        Hence, we conclude that 
        \[
            \Hom_{\mathrm{MHS_{\RR}^{+}}}(\RR(0), \mathrm{H}^{3}(S,V(2))) = 0.    
        \]
\end{proof}

\begin{proposition}
\label{Prop: Hodge exact sequence}
We have the following exact sequence
\[
    \begin{tikzcd}
        0 \ar[r] & \Ext^{1}_{\mathrm{MHS}_{\RR}^{+}}(\mathbf{1}, \mathrm{H}^{2}_{B,!}(S, V(2))) \ar[r] &  \mathrm{H}^{3}_{H}(S,V(2)) \ar[r, "{\partial_{S}^{H}}"] & \mathrm{H}^{1}_{H}(\partial{S}, i^{*}j_{*}V(2)),
    \end{tikzcd} 
\]
where $i^{*}$ and $j_{*}$ denote the pullback and pushforward of the mixed Hodge module. 

\end{proposition}

\begin{proof}
Apply the choices $U = S$, $c = 2$, $Y = \partial S$ and $M = V(2)$ in the exact sequence in equation \eqref{eq: les of singular cohomology}, to get the exact sequence 
\begin{equation} 
    0 \rightarrow \mathrm{H}_{c}^{2}(S,V(2)) \rightarrow \mathrm{H}^{2}(S,V(2)) \rightarrow \mathrm{H}^{0}(\partial S,i^{*}j_{*}V(2)) \rightarrow \mathrm{H}_{c}^{3}(S, V(2)). 
\end{equation}

By Poincar\'{e} duality, we have 
\begin{equation*}
    \mathrm{H}_{c}^{3}(S, V(2)) \cong \mathrm{H}^{1}(S, V^{\vee}(-2)) = 0,
\end{equation*}
where the last equality holds since $V^{\vee}$ is regular. So we have the short exact sequence
\begin{equation*}
     0 \rightarrow \mathrm{H}_{B, !}^{2}(S,V(2)) \rightarrow \mathrm{H}^{2}(S,V(2)) \rightarrow \mathrm{H}^{0}(\partial S,i^{*}j_{*}V(2)) \rightarrow 0. 
\end{equation*}
Apply the functor $\Hom_{\mathrm{MHS_{\RR}^{+}}}(\RR(0), -)$ to this short exact sequence to get the exact sequence
\begin{align*}
     \Hom_{\mathrm{MHS_{\RR}^{+}}}(\RR(0), \mathrm{H}^{0}(\partial S,i^{*}j_{*}V(2))) & \rightarrow  \Ext^{1}_{\mathrm{MHS_{\RR}^{+}}}(\RR(0),\mathrm{H}_{B, !}^{2}(S,V(2))) \\ & \rightarrow 
      \Ext^{1}_{\mathrm{MHS_{\RR}^{+}}}(\RR(0), \mathrm{H}^{2}(S,V(2))) \rightarrow 
     \Ext^{1}_{\mathrm{MHS_{\RR}^{+}}}(\RR(0), \mathrm{H}^{0}(\partial S,i^{*}j_{*}V(2)).  
\end{align*}

Since $\Hom_{\mathrm{MHS_{\RR}^{+}}}(\RR(0), \mathrm{H}^{0}(\partial S,i^{*}j_{*}V(2))) = 0$ by Lemma \ref{Lemma:``Hom = 0" preparation1.}, we have the following exact sequence
\begin{align*}
     0  & \rightarrow \Ext^{1}_{\mathrm{MHS_{\RR}^{+}}}(\RR(0),\mathrm{H}_{B, !}^{2}(S,V(2))) \\ & \rightarrow \Ext^{1}_{\mathrm{MHS_{\RR}^{+}}}(\RR(0), \mathrm{H}^{2}(S,V(2))) \rightarrow 
     \Ext^{1}_{\mathrm{MHS_{\RR}^{+}}}(\RR(0), \mathrm{H}^{0}(\partial S,i^{*}j_{*}V(2)).   
\end{align*}

Also by Lemma \ref{Lemma:``Hom = 0" preparation1.} and the exact sequence (\ref{Eq: ses for abs Hodge cohomology}), we get 
\begin{equation}
\mathrm{H}_{H}^{3}(S, V(2)) \cong \Ext^{1}_{\mathrm{MHS_{\RR}^{+}}}(\RR(0), \mathrm{H}^{2}(S,V(2)))
\end{equation}
and 
\begin{equation}
    \mathrm{H}_{H}^{1}(\partial S, i^{*}j_{*}V(2)) \cong \Ext^{1}_{\mathrm{MHS_{\RR}^{+}}}(\RR(0), \mathrm{H}^{0}(\partial S,i^{*}j_{*}V(2)),
\end{equation}
so we get the desired exact sequence
\[
    \xymatrix{
     0 \ar[r] &
     \Ext^{1}_{\mathrm{MHS_{\RR}^{+}}}(\RR(0),\mathrm{H}_{B, !}^{2}(S,V(2))) \ar[r] &   
     \mathrm{H}_{H}^{3}(S, V(2)) \ar[r] & 
     \mathrm{H}_{H}^{1}(\partial S, i^{*}j_{*}V(2)).  
    }
\]
    
\end{proof}

\vspace{10pt}

\noindent \textbf{Step II.} In the second step, the goal is to prove the diagram in Proposition \ref{Prop: Hodge commutative diagram} commutes. 

\begin{proposition}
\label{Prop: Hodge commutative diagram}
The following diagram is commutative,
\[
    \begin{tikzcd}
         \mathcal{B}_{n, \RR} \ar[d, "p^{*}\circ Eis_{H}^{n}"] & \\
        \mathrm{H}^{1}_{H}(M,W(1)) \ar[r, "\partial_{M}^{H}"] \ar[d, "{\iota_{*}^{H}}"] & \mathrm{H}^{0}_{H}(\partial{M}, i^{\prime, *}j^{\prime}_{*}W(1))) \ar[d, "{\theta^{H}}"] \\
        \mathrm{H}^{3}_{H}(S, V(2)) \ar[r, "\partial_{S}^{H}"] & \mathrm{H}^{1}_{H}(\partial{S}, i^{*}j_{*}V(2))
    \end{tikzcd}, 
\]
where $\iota_{*}^{H}$ is the Gysin morphism defined in Proposition \ref{Prop: Gysin for abs Hdg cohomology}, $\partial_{M}^{H}$ and $\partial_{S}^{H}$ are the corresponding boundary maps and the map $\theta^{H}$ is induced from the map
\begin{equation}
\label{eq: Hodge theta in derived category.}
    (\partial \iota_{*}^{H} \colon q_{*}i^{\prime*}j^{\prime}_{*}W(1) \rightarrow i^{*}j_{*}V(2)[1]) \in \mathrm{D}^{b}(\mathrm{MHM}_{\RR}(\partial S/\RR)),
\end{equation}
which is degenerated from $\iota^{H}_{*}$. 
    
\end{proposition}

\begin{proof}
It follows from Proposition \ref{Prop: Gysin for abs Hdg cohomology} that the Gysin morphism $\iota_{*} \colon \mathrm{H}^{1}_{H}(M,W(1)) \rightarrow \mathrm{H}^{3}_{H}(S, V(2))$ is induced by 
\begin{equation*}
    (\iota_{*}W(1) \rightarrow V(2)[1]) \in \mathrm{D}^{b}(\mathrm{MHM}_{\RR}(S/\RR)). 
\end{equation*}
Apply the functor $j_{*}$ to the morphism, and by the fact that $j_{*}\iota_{*}W(1) = p_{*}j^{\prime}_{*}W(1)$, we get the map
\begin{equation*}
    (p_{*}j^{\prime}_{*}W(1) \rightarrow j_{*}V(2)[1]) \in \mathrm{D}^{b}(\mathrm{MHM}_{\RR}(S^{*}/\RR)). 
\end{equation*}
Then we apply the functor $i_{*}i^{*}$ to the morphism and have a morphism 
\begin{equation*}
    (i_{*}i^{*}p_{*}j^{\prime}_{*}W(1) \rightarrow i_{*}i^{*}j_{*}V(2)[1]) \in \mathrm{D}^{b}(\mathrm{MHM}_{\RR}(S^{*}/\RR)). 
\end{equation*}
From the equation $i^{*}p_{*} = q_{*}i^{\prime*}$ which follows from the proper base change theorem of mixed Hodge modules for the proper map $p$, we get the following map 
\begin{equation*}
    (i_{*}q_{*}i^{\prime*}j^{\prime}_{*}W(1) \rightarrow i_{*}i^{*}j_{*}V(2)[1]) \in \mathrm{D}^{b}(\mathrm{MHM}_{\RR}(S^{*}/\RR)). 
\end{equation*}
By functoriality of the natural transformation $\mathrm{id} \rightarrow i_{*}i^{*}$ from adjunction, we can see the following commutative diagram in $\mathrm{D}^{b}(\mathrm{MHM}_{\RR}(S^{*}/\RR))$: 
\[
    \begin{tikzcd}
        p_{*}j^{\prime}_{*}W(1) \ar[r] \ar[d] & i_{*}q_{*}i^{\prime*}j^{\prime}_{*}W(1) \ar[d] \\
        j_{*}V(2)[1] \ar[r] & i_{*}i^{*}j_{*}V(2)[1]
    \end{tikzcd}. 
\]
Finally, apply the functor $\Hom_{\mathrm{D}^{\mathrm{b}}(\mathrm{MHM}_{\RR}(S^{*}/\RR))}(\RR(0)_{S^{*}}, -[2])$ to the above  commutative diagram to get the commutative diagram 
\begin{equation*}
    \begin{tikzcd}
        \mathrm{H}^{1}_{H}(M, W(1)) \ar[r] \ar[d] & \mathrm{H}^{0}_{H}(\partial M, i^{\prime*}j^{\prime }_{*}W(1))  \ar[d] \\
        \mathrm{H}^{3}_{H}(S, V(2)) \ar[r] & \mathrm{H}^{1}_{H}(\partial S, i^{*}j_{*}V(2))
    \end{tikzcd}.
\end{equation*}
\end{proof}

\begin{remark}
    From the construction, the boundary map $\theta_{S}^{H}$ in Proposition \ref{Prop: Hodge commutative diagram} is the same as the $\theta_{S}^{H}$ in Proposition \ref{Prop: Hodge exact sequence}. Hence, we use the same notation to for them. 
\end{remark}

\vspace{10pt}

\noindent \textbf{Step III.} In the third step, the goal is to prove the map $\theta^{H}_{S}$ in Proposition \ref{Prop: Hodge commutative diagram} is zero, which is stated in Proposition \ref{Prop: Hdg theta is zero}. \\

\par We first recall some notation. 
\begin{notation}
\begin{itemize}
    \item 
        We have the following commutative diagram:
        \begin{equation*}
                    \begin{tikzcd}
                        M \ar[r, "j^{\prime}"] \ar[d, "{\iota}"] & M^{*} \ar[d, "p"] & \partial{M} \ar[l, "i^{\prime}" above] \ar[d, "q"] \\
                        S \ar[r, "j"] & S^{*} & \partial{S} \ar[l, "i" above]  
                    \end{tikzcd}. 
            \end{equation*}
    \item Recall that the group homomorphism underlying the morphism between the boundary components $q \colon \partial M \rightarrow \partial S$ is      the map 
           \begin{equation}
           \label{eq: group embedding boundary}
                \begin{tikzcd}[row sep = 0]
                    {\iota} \colon M_1^{\prime} = \Gm \boxtimes E^{\times} \ar[r] & M_1 \\ (x; z) \ar[r, mapsto] & \begin{pmatrix}
                                                                                                                x & &  \\
                                                                                                                & z & \\
                                                                                                                & & 1 
                                                                                                            \end{pmatrix}
                 \end{tikzcd}. 
            \end{equation}
    \item We denote by $\lambda(\mu_1, \mu_2; d)$ the algebraic representation of $M_{1}$ with weight 
    \begin{equation}
    \label{eq: rep of G_1}
        \begin{pmatrix}
            x & &  \\
            & z & \\
            & & 1 
        \end{pmatrix} \mapsto x^{\mu_1}z^{\mu_2}(z\bar{z})^{d}, 
    \end{equation} 
    and we use $\lambda^{\prime}(d_1, d_2)$ to represent the algebraic representation of $M_1^{\prime}$ with weight $(x; z) \mapsto z^{d_1}\bar{z}^{d_2}$.
    \item We let $n = a + b + r + s$. 
\end{itemize}
\end{notation}

\begin{lemma}
\label{lemma: Hodge: source of theta}
    The source $\mathrm{H}^{0}_{H}(\partial{M}, i^{\prime*}j^{\prime}_{*}W(1))$ of the map $\theta^{H}_{S}$ is isomorphic to 
    \begin{equation*}
        \mathrm{H}^{0}_{H}(\partial{M}, \mathcal{H}^{0}i^{\prime*}j^{\prime}_{*}W(1)).
    \end{equation*}
\end{lemma}

\begin{proof}
    It follows from Lemma \ref{lemma: BW_H} that 
    \begin{equation*}
        i^{*}j_{*} \mu(W) \cong \bigoplus_{n = -1}^{0} \mathcal{H}^{n}i^{*}j_{*} \mu(W)[-n]. 
    \end{equation*}
    Hence, we have 
    \begin{equation*}
        \mathrm{H}^{0}_{H}(\partial{M}, i^{\prime*}j^{\prime}_{*}W(1)) \cong \mathrm{H}^{0}_{H}(\partial{M}, \mathcal{H}^{0}i^{\prime*}j^{\prime}_{*}W(1)) \oplus \mathrm{H}^{1}_{H}(\partial{M}, \mathcal{H}^{-1}i^{\prime*}j^{\prime}_{*}W(1)). 
    \end{equation*}
    We can see 
    \begin{equation*}
        \mathrm{H}^{1}_{H}(\partial{M}, \mathcal{H}^{-1}i^{\prime*}j^{\prime}_{*}W(1)) = 0 
    \end{equation*}
    from the fact that $\mathcal{H}^{-1}i^{\prime*}j^{\prime}_{*}W(1)$ is regular and $\dim \partial M = 0$. 
\end{proof}
\begin{remark}
    From Lemma \label{lemma: Hodge: source of theta}, we can see that it suffices to prove the map
    \begin{equation*}
        \theta_{H} \colon \mathrm{H}^{0}_{H}(\partial{M}, \mathcal{H}^{0}i^{\prime*}j^{\prime}_{*}W(1)) \rightarrow  \mathrm{H}^{1}_{H}(\partial{S}, i^{*}j_{*}V(2))
    \end{equation*}
    is zero. 
\end{remark}

\begin{lemma}
\label{lemma: pullback of deg of MHS}
    We have the following pullback formulas in $\mathrm{D}^{b}(\mathrm{MHM}_{\RR}(\partial M/\RR))$: 
    \begin{align*}
         & q^{*} \mathcal{H}^{-2}i^{*}j_{*}V(2) = \mu(\lambda^{\prime}(a + b + r + 2, a + b + s + 2))[1],  \\
         & q^{*} \mathcal{H}^{-1}i^{*}j_{*}V(2) = \mu(\lambda^{\prime}(a + b + r + 2, b + s + 1))[1] \oplus \mu(\lambda^{\prime}(a + r + 1, a + b + s + 2))[1], \\
         & q^{*} \mathcal{H}^{0}i^{*}j_{*}V(2) = \mu(\lambda^{\prime}(a + r + 1 , s))[1] \oplus \mu(\lambda^{\prime}(r, b + s + 1))[1], \\
         & q^{*} \mathcal{H}^{1}i^{*}j_{*}V(2) = \mu(\lambda^{\prime}(r, s))[1]. 
    \end{align*}
\end{lemma}

\begin{proof}
    First,  it follows from \cite[Proposition 2.3.12]{Bl06} that for $-2 \le l \le 1$, we have 
    \begin{equation*}
        q^{*} \mathcal{H}^{l}i^{*}j_{*}V(2) = (q^{s})^{*} \mathcal{H}^{l}i^{*}j_{*}V(2)[1], 
    \end{equation*}
    where $(q^{s})^{*}$ is the pullback functor in the category of variation of Hodge structures. The desired formulas are now
    direct consequences of Lemma \ref{lemma: BW_G} and equation (\ref{eq: group embedding boundary}).
\end{proof}

\begin{lemma}
\label{Lemma: Hom = 0 by cohomology degree}
    For $-2 \le l \le 0$, we have the following vanishing: 
    \begin{align*}
        \Hom_{\mathrm{D}^{b}(\mathrm{MHM}_{\RR}(\partial M/\RR))}(\mathcal{H}^{0}i^{\prime*}j_{*}^{\prime}W(1), q^{*}\mathcal{H}^{l}i^{*}j_{*}V(2)[1-l]) = 0.  
    \end{align*}
\end{lemma}

\begin{proof}
    From Lemma \ref{lemma: BW_H} and Lemma \ref{lemma: pullback of deg of MHS}, there exist $X, Y \in \mathrm{Var}_{\RR}(M)$ such that 
    \begin{align*}
        & \Hom_{\mathrm{D}^{b}(\mathrm{MHM}_{\RR}(\partial M/\RR))}(\mathcal{H}^{0}i^{\prime*}j_{*}^{\prime}W(1), q^{*}\mathcal{H}^{l}i^{*}j_{*}V(2)[1-l]) \\
        = & \Hom_{\mathrm{D}^{b}(\mathrm{MHM}_{\RR}(\partial M/\RR))}(X, Y[2-l]). 
    \end{align*}
    From the geometry of $\partial M$ and the definition of the derived category, 
    \begin{align*}
        \Hom_{\mathrm{D}^{b}(\mathrm{MHM}_{\RR}(\partial M/\RR))}(X, Y[2-l]) 
    \cong \Hom_{\mathrm{D}^{b}(\mathrm{MHS}_{\RR}^{+})}(X, Y[2-l]) 
    \cong \Ext^{2-l}_{\mathrm{MHS_{\RR}^{+}}}(X, Y). 
    \end{align*}
    Then the lemma follows from the fact that the cohomological dimension of the category $\mathrm{MHS_{\RR}^{+}}$ is $1$. 
\end{proof}

\begin{lemma}
\label{Lemma: Hom = 0 by Hodge type}
We have the following vanishing result: 
\begin{align*}
    \Hom_{\mathrm{D}^{b}(\mathrm{MHM}_{\RR}(\partial M / \RR))}(\mathcal{H}^{0}i^{\prime*}j_{*}^{\prime}W(1), q^{*}\mathcal{H}^{-1}i^{*}j_{*}V(2)) = 0. 
\end{align*}
\end{lemma}

\begin{proof}
It follows from Lemma \ref{lemma: BW_H} and Lemma \ref{lemma: pullback of deg of MHS} that
\begin{align*}
    & \Hom_{\mathrm{D}^{b}(\mathrm{MHM}_{\RR}(\partial M / \RR))}(\mathcal{H}^{0}i^{\prime*}j_{*}^{\prime}W(1), q^{*}\mathcal{H}^{1}i^{*}j_{*}V(2))  \\ 
    = & \Hom_{\mathrm{D}^{b}(\mathrm{MHM}_{\RR}(\partial M / \RR))}(\mu(\lambda^{\prime}(0, 0)), \mu(\lambda^{\prime}(r, s))[1]) \\
    = & \Hom_{\mathrm{D}^{b}(\mathrm{MHS}_{\RR}^{+})}(\RR(0), \mu(\lambda^{\prime}(r, s))[1])  \\
    = & \Ext^{1}_{\mathrm{MHS}_{\RR}^{+}}(\RR(0), \mu(\lambda^{\prime}(r, s))). 
\end{align*}

The weight of $\mu(\lambda^{\prime}(r, s))$ is greater than zero,  so by \cite[Theorem A.2.10]{HW98}, 
\begin{equation*}
    \Ext^{1}_{\mathrm{MHS}_{\RR}^{+}}(\RR(0), \mu(\lambda^{\prime}(r, s))) = 0. 
\end{equation*}
\end{proof}
        
\begin{proposition}
\label{Prop: Hdg theta is zero}
    The map $\theta^{H}$ in Proposition \ref{Prop: Hodge commutative diagram} is zero. 
\end{proposition}
\begin{proof}
    \par From Proposition \ref{Prop: Hodge commutative diagram} and Lemma \ref{lemma: Hodge: source of theta}, in order to prove $\theta^{H}_{S} = 0$ it suffices to prove 
    \begin{equation*}
        \Hom_{\mathrm{D}^{b}(\mathrm{MHM}_{\RR}(\partial S/\RR))}(q_{*}\mathcal{H}^{0}i^{\prime*}j^{\prime}_{*}W(1), i^{*}j_{*}V(2)[1]) = 0. 
    \end{equation*}
    Since $q \colon \partial M \rightarrow \partial S$ is a morphism from a $0$-dimensional scheme to a $0$-dimensional scheme, it is both an open immersion and a closed immersion. Therefore, we have
    \begin{equation*}
        q_{!} = q_{*}, q^{!} = q^{*}, 
    \end{equation*}
    from which we get the following isomorphisms:
    \begin{align*}
        & \Hom_{\mathrm{D}^{b}(\mathrm{MHM}_{\RR}(\partial S/\RR))}(q_{*}\mathcal{H}^{0}i^{\prime*}j_{*}^{\prime}W(1), i^{*}j_{*}V(2)[1]) \\ 
    \cong & \Hom_{\mathrm{D}^{b}(\mathrm{MHM}_{\RR}(\partial S/\RR))}(q_{!}\mathcal{H}^{0}i^{\prime*}j_{*}^{\prime}W(1), i^{*}j_{*}V(2)[1]) \\
    \cong & \Hom_{\mathrm{D}^{b}(\mathrm{MHM}_{\RR}(\partial M/\RR))}(\mathcal{H}^{0}i^{\prime*}j_{*}^{\prime}W(1), q^{!}i^{*}j_{*}V(2)[1]) \\
    \cong & \Hom_{\mathrm{D}^{b}(\mathrm{MHM}_{\RR}(\partial M/\RR))}(\mathcal{H}^{0}i^{\prime*}j_{*}^{\prime}W(1), q^{*}i^{*}j_{*}V(2)[1]), 
    \end{align*}
    where the second isomorphism comes from the adjunction between $q_{!}$ and $q^{!}$. 
    It follows from Lemma \ref{lemma: BW_G} that 
    \begin{equation*}
        i^{*}j_{*}V(2) = \bigoplus_{l = -2}^{1} \mathcal{H}^{l}i^{*}j_{*}V(2)[-l]. 
    \end{equation*}
    Finally, by Lemma \ref{Lemma: Hom = 0 by cohomology degree} and Lemma \ref{Lemma: Hom = 0 by Hodge type}, we complete the proof of the proposition.
\end{proof}

The following corollary follows directly from Proposition \ref{Prop: Hdg theta is zero} and Lemma \ref{lemma: Hodge: source of theta}. 

\begin{corollary}
\label{Corollary: Hodge theta is zero}
    For any 
    \begin{equation*}
        f \in \mathrm{H}^{0}_{H}(\partial M, i^{\prime*}j^{\prime }_{*}W(1)) = \Hom_{\mathrm{D}^{\mathrm{b}}(\mathrm{MHM}_{\RR}(S^{*}/\RR))}(\RR(0)_{S^{*}}, i_{*}q_{*}i^{\prime*}j^{\prime}_{*}W(1)[2]), 
    \end{equation*}
    the element 
    \begin{equation*}
        i_{*} (\partial \iota_{*}^{H}) \circ f \in \Hom_{\mathrm{D}^{\mathrm{b}}(\mathrm{MHM}_{\RR}(S^{*}/\RR))}(\RR(0)_{S^{*}}, i_{*}i^{*}j_{*}V(2)[3]) = \mathrm{H}^{1}_{H}(\partial S, i^{*}j_{*}V(2))
    \end{equation*}
    is $0$. 
\end{corollary}

\begin{remark}
    Although we have proved that the map 
    \begin{equation*}
        \theta^{H}_{S} \colon \mathrm{H}^{0}_{H}(\partial{M}, \mathcal{H}^{0}i^{\prime*}j^{\prime}_{*}W(1)) \rightarrow \mathrm{H}^{1}_{H}(\partial{S}, i^{*}j_{*}V(2)) 
    \end{equation*}
    is zero, the source and the target may not be zero for the following reasons.
    \begin{itemize}
        \item It can be seen that 
            \begin{equation*}
                \mathrm{H}^{0}_{H}(\partial{M}, \mathcal{H}^{0}i^{\prime*}j^{\prime}_{*}W(1)) = \Hom_{\mathrm{MHS}_{\RR}^{+}}(\RR(0), \mathrm{H}^{0}(\partial M, \mathcal{H}^{0}i^{\prime*}j^{\prime}_{*}W(1))). 
            \end{equation*}
            Then by Lemma \ref{lemma: BW_H}, we get
            \begin{equation*}
                \mathrm{H}^{0}_{H}(\partial{M}, \mathcal{H}^{0}i^{\prime*}j^{\prime}_{*}W(1)) = \Hom_{\mathrm{MHS}_{\RR}^{+}}(\RR(0), \RR(0)) = \RR.
            \end{equation*}
            Hence, we have shown that $\mathrm{H}^{0}_{H}(\partial{M}, i^{\prime*}j^{\prime}_{*}W(1)) \neq 0$.  
        \item It follows from Lemma \ref{lemma: BW_G} that 
        \begin{equation*}
              \mathrm{H}^{1}_{H}(\partial{S}, i^{*}j_{*}V(2)) 
             = \bigoplus_{l = - 2}^{1} \mathrm{H}^{-l+1}_{H}(\partial{S}, \mathcal{H}^{i}i^{*}j_{*}V(2)). 
        \end{equation*}
         We can see that 
          \begin{align*}
             \mathrm{H}^{1}_{H}(\partial{S}, \mathcal{H}^{0}i^{*}j_{*}V(2)) = & \Hom_{\mathrm{D}^{b}(\mathrm{MHS}_{\RR}^{+})}(\RR(0), \mathrm{H}^{0}(\partial S, \mathcal{H}^{0}i^{*}j_{*}V(2))[1]) \\
             = & \Ext^{1}_{\mathrm{MHS}_{\RR}^{+}}(\RR(0), \mathrm{H}^{0}(\partial S, \mathcal{H}^{0}i^{*}j_{*}V(2))) \\
             = & \Ext^{1}_{\mathrm{MHS}_{\RR}^{+}}(\RR(0), H_{1}) \oplus \Ext^{1}_{\mathrm{MHS}_{\RR}^{+}}(\RR(0), H_{2}), 
         \end{align*}
         where $H_{1} \in \mathrm{MHS}_{\RR}^{+}$ has Hodge type $(-(a + r + 1), -s)$ and $(-s, -(a + r + 1))$, and $H_{2} \in \mathrm{MHS}_{\RR}^{+}$ has Hodge type $(-r, -(b + s + 1))$ and $(-(b + s + 1), -r)$. 
         It follows from \cite[Theorem A.2.10.]{HW98} that if $a + r + s > 0$, then $\Ext^{1}_{\mathrm{MHS}_{\RR}^{+}}(\RR(0), H_{1}) \neq 0$, 
         and if $b + r + s > 0$, then $\Ext^{1}_{\mathrm{MHS}_{\RR}^{+}}(\RR(0), H_{2}) \neq 0$.
         Hence, we conclude that when $V$ is sufficiently regular, $\mathrm{H}^{1}_{H}(\partial{S}, \mathcal{H}^{0}i^{*}j_{*}V(2)) \neq 0$.
    \end{itemize}
\end{remark}

Finally, by combining Proposition \ref{Prop: Hodge exact sequence}, Proposition \ref{Prop: Hodge commutative diagram} and Proposition \ref{Prop: Hdg theta is zero}, we complete the proof of Theorem \ref{Thm: Hdg vanish on the boudary}.


\section{The vanishing on the boundary of motivic cohomology}
\label{Sec: vanishing on the boundary motivic}

\subsection{Weight structures}
\label{SS: weight structures}

\subsubsection{Abstract weight structure}
First we collect some definitions and facts about weight structrue on an abstract triangulated category. We mainly follow \cite[\S 1]{Wild_09}. 

\begin{definition}[Weight structure]
\label{def: weight structure}
    Let $\mathcal{C}$ be a triangulated category. A \textit{weight structure} on $\mathcal{C}$ is a pair $w = (\mathcal{C}_{w \le 0}, \mathcal{C}_{w \ge 0})$ of full-subcategories of $\mathcal{C}$ such that if we put 
    \begin{equation*}
        \mathcal{C}_{w \le n} := \mathcal{C}_{w \le 0}[n], \, \, \, \mathcal{C}_{w \ge n}:= \mathcal{C}_{w \ge 0}[n]
    \end{equation*}
    for $n \in \ZZ$, the following conditions are satisfied. 
    \begin{enumerate}
        \item The categories $\mathcal{C}_{w \le 0}$ and $\mathcal{C}_{w \ge 0}$ are Karoubi-closed: for any object $M$ of $\mathcal{C}_{w \le 0}$ (resp., $\mathcal{C}_{w \ge 0}$), any direct sum of $M$ formed in $\mathcal{C}$ is an object of $\mathcal{C}_{w \le 0}$ (resp., $\mathcal{C}_{w \ge 0}$). 
        \item (Semi-invariance with respect to shifts) \  We have the following inclusions: 
        \begin{equation*}
            \mathcal{C}_{w \le 0} \subset \mathcal{C}_{w \le 1},  \, \, \, \mathcal{C}_{w \ge 0} \supset \mathcal{C}_{w \ge 1}
        \end{equation*}
        of full sub-categories of $\mathcal{C}$. 
        \item (Orthogonality) \ For any pair of objects $M \in \mathcal{C}_{w \le 0}$ and $N \in \mathcal{C}_{w \ge 1}$, we have 
        \begin{equation*}
            \Hom_{\mathcal{C}}(M, N) = 0. 
        \end{equation*}
        \item (Weight filtration) \ For any object $M \in \mathcal{C}$, there exists an exact triangle 
        \begin{equation*}
            A \rightarrow M \rightarrow B \rightarrow A[1]
        \end{equation*}
        in $\mathcal{C}$ such that $A \in \mathcal{C}_{w \le 0}$ and $B \in \mathcal{C}_{w \ge 1}$. We shall refer to any such exact triangle as a weight filtration of $M$. 
    \end{enumerate}
\end{definition}

\begin{remark}
    \begin{itemize}
        \item By (2), we have 
        \begin{equation*}
            \mathcal{C}_{w \le n} \subset \mathcal{C}_{w \le 0}
        \end{equation*}
        for any negative integer $n$, and 
        \begin{equation*}
            \mathcal{C}_{w \ge n} \subset \mathcal{C}_{w \ge 0}
        \end{equation*}
        for any postive integer $n$. There are analogues of the other conditins for all the categories $\mathcal{C}_{w \le n}$ and $\mathcal{C}_{w \ge n}$. In particular, they are all Karoubi-closed, and for any object $M \in \mathcal{C}$, there is an exact triangle
        \begin{equation*}
            A \rightarrow M \rightarrow B \rightarrow A[1]
        \end{equation*}
        in $\mathcal{C}$ such that $A \in \mathcal{C}_{w \le n}$ and $B \in \mathcal{C}_{w \ge n + 1}$. 
        \item In (4), ``the" weight filtration is not assumed to be unique. 
        \item Recall that on a triangulated category $\mathcal{C}$, there is a notion of \textit{t-structure} \cite[D\'{e}f. 1.3.1]{BBD}. It consists of a pair $t = (\mathcal{C}^{t \le 0}, \mathcal{C}^{t \ge 0})$ of full sub-categories satisfying formal analogues of conditions (2)-(4), but putting 
        \begin{equation*}
            \mathcal{C}^{t \le n} := \mathcal{C}^{t \le 0}[-n], \, \, \, \mathcal{C}^{t \ge n} := \mathcal{C}^{t \ge 0}[-n],
        \end{equation*}
        for any $n \in \ZZ$. Note that in the context of $t$-structures, the analogues of (4) is unique up to isomorphism, and the analogue of (1) is implied by the other conditions. 
    \end{itemize}
\end{remark}

\begin{definition}[Heart] \cite[Definition 1.2.1]{Bondarko10}
    Let $w = (\mathcal{C}_{w \le 0}, \mathcal{C}_{w \ge 0})$ be a weight structure on $\mathcal{C}$. The \textit{heart} of $w$ is the full additive subcategory $\mathcal{C}_{w = 0}$ of $\mathcal{C}$ whose objects lie both in $\mathcal{C}_{w \le 0}$ and in $\mathcal{C}_{w \ge 0}$. 
\end{definition}

\begin{definition}[Weight exact] \cite[Definition 4.4.4]{Bondarko10}
    Let $F \colon \mathcal{C} \rightarrow \mathcal{D}$ be a functor between two triangulated categories with weight structures $(\mathcal{C}_{w \le 0}, \mathcal{C}_{w \ge 0})$ and $(\mathcal{D}_{w \le 0}, \mathcal{D}_{w \ge 0})$. The functor $F$ is called weight exact if $F$ maps $\mathcal{C}_{w \le 0}$
    to $\mathcal{D}_{w \le 0}$ and maps $\mathcal{C}_{w \ge 0}$ to $\mathcal{D}_{w \ge 0}$. 
\end{definition}

\begin{proposition}
    Let $w = (\mathcal{C}_{w \le 0}, \mathcal{C}_{w \ge 0})$ be a weight structure on $\mathcal{C}$. There is an exact triangle 
    \begin{equation*}
        L \rightarrow M \rightarrow N \rightarrow L[1]
    \end{equation*}
    in $\mathcal{C}$.
    \begin{enumerate}
        \item If both $L$ and $N$ belong to $\mathcal{C}_{w \le 0}$, then so does $M$. 
        \item If both $L$ and $N$ belong to $\mathcal{C}_{w \ge 0}$, then so does $M$. 
    \end{enumerate}
\end{proposition}
\begin{proof}
    See \cite[Proposition 1.3.3.3]{Bondarko10}. 
\end{proof}

From now on, we consider a fixed weight structure $w$ on a triangulated category $\mathcal{C}$. 
\begin{definition}
    Let $M \in \mathcal{C}$ and $m \le n$ be two integers. A \textit{weight filtration of $M$ avoiding weights $m, m + 1, \ldots, n - 1, n$} is an 
    exact triangle 
    \begin{equation*}
        M_{w \le m - 1} \rightarrow M \rightarrow M_{\ge n + 1} \rightarrow M_{w \le m - 1}[1]
    \end{equation*}
    in $\mathcal{C}$,  with $M_{w \le m - 1} \in \mathcal{C}_{w \le m - 1}$ and $M_{\ge n + 1} \in \mathcal{C}_{\ge n + 1}$. 
\end{definition}

\begin{proposition}
    Assume that $m \le n$, and that $M, N \in \mathcal{C}$ admit weight filtrations
    \[
        \begin{tikzcd}
             M_{w \le m - 1} \ar[r, "{x_{-}}"] & M \ar[r, "{x_{+}}"] & M_{\ge n + 1} \ar[r] & {M_{w \le m - 1}[1]}
        \end{tikzcd}
    \]
    and 
    \[
        \begin{tikzcd}
             N_{w \le m - 1} \ar[r, "{y_{-}}"] & N \ar[r, "{y_{+}}"] & N_{\ge n + 1} \ar[r] & {N_{w \le m - 1}[1]}
        \end{tikzcd}
    \]
    avoiding weights $m, \ldots, n$. Then any morphism $M \rightarrow N$ in $\mathcal{C}$ extends uniquely to a morphism of exact triangles: 
    \[
        \begin{tikzcd}
             M_{w \le m - 1} \ar[r, "{x_{-}}"] \ar[d] & M \ar[r, "{x_{+}}"] \ar[d] & M_{\ge n + 1} \ar[r] \ar[d] & {M_{w \le m - 1}[1]} \ar[d] \\
             N_{w \le m - 1} \ar[r, "{y_{-}}"] & N \ar[r, "{y_{+}}"] & N_{\ge n + 1} \ar[r] & {N_{w \le m - 1}[1]}
        \end{tikzcd}.
    \]
\end{proposition}

\begin{corollary}
    Assume that $m \le n$. If $M \in \mathcal{C}$ admits a weight filtration avoiding weights $m, \ldots, n$, then it is unique up to a unique isomorphism. 
\end{corollary}

\begin{definition}
    Assume that $m \le n$. We say that \textit{$M \in \mathcal{C}$ does not have weights $m, \cdots n$} or \textit{$M$ is without weights $m \ldots n$}, if it admits a weight filtration avoiding $m, \ldots, n$. 
\end{definition}

\subsubsection{Motivic weight structure}

We recall H\'{e}bert's construction of motivic weight structure by applying the abstract machine to the triangulated category of constructible Beilinson motives. 

\begin{theorem}[{\cite[Th\'{e}or\`{e}me 3.3]{Hebert10}}]
    Let $S \in \mathrm{Sch}(\QQ)$ be an excellent scheme. 
    Set
    \begin{equation*}
        \mathrm{DM}_{B, c}(S) \supset \mathcal{H}_{S}:= \{ f_{*} 1_{X}(n)[2n] | n \in \ZZ, f \colon X \rightarrow S \, \text{is proper with regular source $X$} \}.
    \end{equation*}
    There exists a unique weight structure $w = (\mathrm{DM}_{B, c}(S)_{w \le 0}, \mathrm{DM}_{B, c}(S)_{w \ge 0})$ on $\mathrm{DM}_{B, c}(S)$ such that 
    \begin{equation*}
        \mathcal{H}_{S} \subset \mathrm{DM}_{B, c}(S)_{w = 0}. 
    \end{equation*}
\end{theorem}

\begin{remark}
    \begin{itemize}
        \item If $S = \Spec(\QQ)$, the weight structure 
        \begin{equation*}
            w = (\mathrm{DM}_{gm}(\QQ)_{w \le 0}, \mathrm{DM}_{gm}(\QQ)_{w \ge 0}) 
        \end{equation*}
        is constructed by Bondarko in \cite[Proposition 6.5.3]{Bondarko10}. 
        \item The existence of this weight structure was also proved independently by Bondarko in \cite{Bondarko17}. 
        \item It can be seen from the definition that $\mathrm{CHM}(S) \subset \mathcal{H}_{S}$. 
    \end{itemize}
\end{remark}

\subsection{Intersection motives, interior motives and boundary motives}
\label{Sec: interior and boundary motive}
In this subsection, we recall the definition of intersection motives, interior motives and boundary motives of Picard modular surfaces and some of their properties. 

\begin{notation}
    We use the same notation as in section \ref{Sec: vanishing of the boudary abs Hodge}:
    for each of the two Shimura varieties $X = \{M, S\}$, we denote by $X^{*}$ its Baily{\textendash}Borel compactification and let $\partial{X} = X^{*} - X$ be the cusps of the compactification. Then we have the following commutative diagram: 
    \begin{equation}
    \label{eq: diagram of compactification}
            \begin{tikzcd}
                M \ar[r, "j^{\prime}"] \ar[d, "{\iota}"] & M^{*} \ar[d, "p"] & \partial{M} \ar[l, "i^{\prime}" above] \ar[d, "q"] \\
                S \ar[r, "j"] & S^{*} & \partial{S} \ar[l, "i" above]  
            \end{tikzcd},
    \end{equation}
    where $j$ and $j^{\prime}$ are open immersions, $i$ and $i^{\prime}$ are closed immersions and $p$ and $q$ are closed immersions induced by the closed immersion $\iota$. We also have 
    \begin{equation*}
        \dim \partial S = 0, \, \dim \partial M = 0. 
    \end{equation*}
\end{notation}

\subsubsection{Boundary motives}

 Let $g \colon \partial S \rightarrow \Spec (\QQ)$ be the structure morphism of $\partial S$.
\begin{definition}[Boundary motives]
    For $V(2) \in \mathrm{Rep}_{\QQ}(G)$, the boundary motive $\partial M_{gm}(V(2))$ is defined as 
    \begin{equation*}
        \partial M_{gm}(V(2)) := (g_{*}i^{*}j_{*}V(2))^{*} \in \mathrm{DM}_{gm}(\QQ), 
    \end{equation*}
    where the notation $(\cdot)^{*}$ denotes taking dual. 
\end{definition}

\begin{remark}
    It can be seen from the proof of \cite[Corollary 3.10]{Wild_ChowII_20} that the definition of the boundary motive is the same as the one given in \cite[Definition 2.1]{Wild_boundary05}. 
\end{remark}

\begin{notation}
    Denote by $\mathrm{DM}_{B, c}(S)_{w = 0, \partial w \neq 0, 1}$ the full sub-category of $\mathrm{DM}_{B, c}(S)_{w = 0}$ of objects $N$ such that $i^{*}j_{*}N$ is without weights $0$ and $1$. 
\end{notation}

\begin{proposition}
\label{Prop: boundary weight}
For $V = V^{a, b}\{r, s\}$ and $a > 0$ and $b > 0$, we have 
\begin{align*}
    V(2) \in \mathrm{DM}_{B, c}(S)_{w = 0, \partial w \neq 0, 1}. 
\end{align*}
\end{proposition}

\begin{proof}
    First, the $k = \min(k_1 - k_2, k_2)$ in \cite[Page 21]{Wild_09} is $\min(a, b)$ in our convention, and the condition $k_1 > k_2 > 0$ in \cite{Wild_09} is equivalent to our condition $a > 0$ and $b > 0$. Second, it follows from \cite[Theorem 3.8]{Wild_15} that the boundary motive $\partial M_{gm}(V(2))$ is without weights $-1$ and $0$. Finally, it can be seen from the fact that $g_{*}i^{*}j_{*}V(2) = g_{!}i^{*}j_{*}V(2) = \partial M_{gm}(V(2))^{*}$ and the fact that $g_{!}$ is weight exact \cite[Th\'{e}or\`{e}me 3.3(i)]{Hebert10} that $i^{*}j_{*}V(2)$ is without weights $0$ and $1$. 
\end{proof}

\subsubsection{Intersection motives and interior motives}
Let $a \colon S^{*} \rightarrow \Spec(\QQ)$ be the structure morphism of $S^{*}$.
\par Recall that in \cite[Definition 2.4]{Wild_ChowII_20}, the author defines the motivic intermediate extension functor 
\begin{equation*}
    j_{!*} \colon \mathrm{DM}_{B, c}(S)_{w = 0, \partial w \neq 0, 1} \rightarrow \mathrm{DM}_{B, c}(S^{*})_{w = 0}. 
\end{equation*}

\begin{definition}
\label{def: intersection motive and interior motive}
    \begin{enumerate}
        \item (\textnormal{Intersection motives}) For $V = V^{a, b}\{r, s\}$ and $a > 0$ and $b > 0$, the \textit{intersection motive} of $S$ relative to $S^{*}$ is defined as 
        \begin{equation*}
            a_{*}j_{!*}V(2) \in \mathrm{DM}_{B, c}(\Spec(\QQ))_{w = 0}. 
        \end{equation*}
        \item (\textnormal{Interior motives})  For $V = V^{a, b}\{r, s\}$ and $a > 0$ and $b > 0$, the \textit{interior motive} $\mathrm{Gr_{0}}M_{gm}(V(2))$ of $S$ is defined as 
        \begin{equation*}
            \mathrm{Gr_{0}}M_{gm}(V(2)) := (a_{*}j_{!*}V(2))^{*} \in \mathrm{DM}_{gm}(\QQ)_{w = 0},
        \end{equation*}
        where the notation $(\cdot)^{*}$ denotes taking dual. 
    \end{enumerate}
    
\end{definition}

\begin{remark}
    \begin{itemize}
        \item 
        It follows from \cite[Corollary 3.10(b)]{Wild_ChowII_20} that the definition of interior motives is the same as the one defined in \cite[Corollary 3.9]{Wild_15}
        \item By \cite[Remark 3.13(a)]{Wild_ChowII_20}, the interior cohomology $\mathrm{H}^{2}_{B, !}(S, V(2))$ is the Hodge realization of the interior motive $\mathrm{Gr_{0}}M_{gm}(V(2))$. 
    \end{itemize}
\end{remark}

Now we recall a few properties of intersection motives. 
\begin{proposition}
Let $k = \mathrm{min}(a, b)$, which is greater than $0$. 
\begin{enumerate}
    \item There are canonical exact triangles in $\mathrm{DM}_{B, c}(\Spec(\QQ))$
         \begin{equation}
          \label{eq: exact triangle: intersection motive and j_{*}} 
            \begin{tikzcd}
                 a_{*}j_{!*}V(2) \ar[r, "\iota_{1}"] & a_{*}j_{*}V(2) \ar[r] & C_{\ge (k + 1)} \ar[r] & a_{*}j_{!*}V(2)[1]
            \end{tikzcd}
         \end{equation}
         and 
         \begin{equation}
         \label{eq: exact triangle: intersection motive and j_{!}}
             \begin{tikzcd}
                 a_{*}j_{!}V(2) \ar[r, "\iota_{2}"] & a_{*}j_{!*}V(2) \ar[r] & C_{\le {-k}} \ar[r] & a_{*}j_{!}V(2)[1]
             \end{tikzcd}, 
         \end{equation}
         where $C_{\ge (k + 1)} \in \mathrm{DM}_{B, c}(\Spec(\QQ))_{w \ge k + 1}$ and $C_{\le {-k}} \in \mathrm{DM}_{B, c}(\Spec(\QQ))_{w \le {-k}}$. 
    \item We have the following commutative diagram in $\mathrm{DM}_{B, c}(\Spec(\QQ))$: 
          \begin{equation}
           \label{eq: commutative diagram factor through j_{!} to j_{*}}
             \begin{tikzcd}
                 a_{*}j_{!}V(2) \ar[rr, "m"] \ar[rd, "\iota_{2}"]& & a_{*}j_{*}V(2) \\
                 &  a_{*}j_{!*}V(2) \ar[ru, "\iota_{1}"] & 
             \end{tikzcd}, 
          \end{equation}
          where $m$ is the map induced by the first arrow in the exact triangle of Proposition \ref{Prop: motivic localization}, and $\iota_{1}$ and $\iota_{2}$ are the corresponding maps in (1). 
\end{enumerate}
    
\end{proposition}

\begin{proof}
    The two exact triangles follow from \cite[Corollary 3.9]{Wild_15} and Definition \ref{def: intersection motive and interior motive}, and the commutative diagram of (2) follows from \cite[Theorem 2.4(c)]{Wild_09} and Definition \ref{def: intersection motive and interior motive}. 
\end{proof}

\subsection{Vanishing on the boundary for motivic cohomology}
\label{SS: the proof of the vanishing}
In this subsection, we state the theorem of vanishing on the boundary for motivic cohomology. 

\begin{theorem}
\label{Thm: mot vanish on the boudary}
Under the conditions
\begin{enumerate}
     \item $0 \le -r \le a$ and  $0 \le -s \le b$, 
     \item $a > 0$ and $b > 0$,
     \item $r \neq 0$ or $s \neq 0$, 
\end{enumerate}
the map $\mathcal{E}is_{M}^n \colon \mathcal{B}_{n} \rightarrow \mathrm{H}^{3}_{M}(S, V(2))$ factors through the inclusion 
\[
    \mathrm{H}_{M}^{3}(\mathrm{Gr_{0}}M_{gm}(V(2)), \QQ(0)) \hookrightarrow \mathrm{H}^{3}_{M}(S, V(2)). 
\] 
\end{theorem}

\begin{remark}
    Using the Beilinson regulator $r_{H}$, $\mathrm{H}_{M}^{3}(\mathrm{Gr_{0}}M_{gm}(V(2)), \QQ(0))$ has image in 
    \begin{equation*}
        \Ext^{1}_{\mathrm{MHS}_{\RR}^{+}}(\mathbf{1}, \mathrm{H}^{2}_{B,!}(S, V(2))).
    \end{equation*}
    Hence, Theorem \ref{Thm: mot vanish on the boudary} can be viewed as a ``motivic lifting'' of Theorem \ref{Thm: Hdg vanish on the boudary}. 
\end{remark}
 
\subsection{The proof of motivic vanishing}
\label{Sec: motivic vanishing}
In this subsection, we prove the theorem of vanishing on the boundary for motivic cohomology. 

\subsubsection{Conversativity and weight convervativity} Before going into the steps of the proof, we first recall the conservativity and weight conservativity theorem of the Hodge realization functor over motives of Abelian type.

\begin{theorem}[Conservativity, {\cite[Theorem 1.12.]{Wild_15}}]
\label{Theorem: conservativity}
    For $M \in \mathrm{DM}_{B, c}^{Ab}(\Spec(\QQ))$, if 
    \begin{equation*}
        H^{n} R_{H}(M) = 0
    \end{equation*}
    for all $n \in \ZZ$, then $M = 0$. 
\end{theorem}

A direct corollary of the conservativity theorem is the following. 
\begin{corollary}
\label{corllaray: conservativity}
For all $M, N \in \mathrm{DM}_{B, c}^{Ab}(\Spec(\QQ))$, if 
\begin{equation*}
    \Hom_{\mathrm{D}^{b}(\mathrm{MHS}_{\RR}^{+})}(R_{H}(M), R_{H}(N)) = 0, 
\end{equation*}
then 
\begin{equation*}
    \Hom_{\mathrm{DM}_{B, c}(\Spec(\QQ))}(M, N) = 0. 
\end{equation*}
\end{corollary}

\begin{theorem}[Weight conservativity, {\cite[Theorem 1.13]{Wild_15}}]
\label{Theorem: weight conservativity}
    Let $M \in \mathrm{DM}_{B, c}^{Ab}(\Spec(\QQ))$ and $\alpha \le \beta$ be two integers. We have the following statements: 
    \begin{enumerate}
        \item $M$ lies in $\mathrm{DM}_{B, c}^{Ab}(\Spec(\QQ))_{w = 0}$ if and only if $H^{n}R_{H}(M)$ is pure of weight $n$, for all $n \in \ZZ$,
        \item $M$ lies in $\mathrm{DM}_{B, c}^{Ab}(\Spec(\QQ))_{w \le \alpha}$ if and only if $H^{n}R_{H}(M)$ is of weights $\le n + \alpha$, for all $n \in \ZZ$, 
        \item $M$ lies in $\mathrm{DM}_{B,c }^{Ab}(\Spec(\QQ))_{w \ge \beta}$ if and only if $H^{n}R_{H}(M)$ is of weights $\ge n + \beta$, for all $n \in \ZZ$,
        \item $M$ is without weights $\alpha, \alpha + 1, \ldots, \beta$ if and only if $H^{n}R_{H}(M)$ is without weights $n + \alpha, n + (\alpha + 1), \ldots, n + \beta$, for all $n \in \ZZ$. 
    \end{enumerate}
\end{theorem}

\subsubsection{The proof}
The proof is parallel to the Hodge case, so it is also divided into three steps. 
Recall that we let $a \colon S^{*} \rightarrow \Spec(\QQ)$ be the structure morphism of $S^{*}$. \\

\noindent \textbf{Step I.} In the first step, the goal is to prove the exactness of the sequence in Proposition \ref{Prop: mot exact sequence}.

\begin{lemma}
\label{Lemma: mot vanshing}
We have the following vanishing of motivic cohomology: 
\begin{equation*}
   \Hom_{\mathrm{DM}_{B, c}(\Spec(\QQ))}(1_{\Spec(\QQ)}, C_{\ge(k + 1)}[2]) = 0, 
\end{equation*}
where $C_{\ge (k + 1)} \in \mathrm{DM}_{B, c}(\Spec(\QQ))_{w \ge k + 1}$ is the one in the exact triangle (\ref{eq: exact triangle: intersection motive and j_{*}}). 
    
\end{lemma}

\begin{proof}
    We first apply the Hodge realization functor $R_{H}$ to the above $\Hom$ space and get 
    \begin{align*}
        \Hom_{\mathrm{D^{b}}(\mathrm{MHS}_{\RR}^{+})}(\RR(0), N[2]), 
    \end{align*} 
    where $N = R_{H}(C_{\ge(k + 1)})$. 
    By the fact that $\mathrm{MHS}_{\RR}^{+}$ has cohomological dimension $1$,  we have the short exact sequence 
    \begin{equation*}
        \begin{tikzcd}[column sep = 1em]
            0 \ar[r] &  \Ext_{\mathrm{MHS}_{\RR}^{+}}^{1}(\RR(0), \mathcal{H}^{1}(N)) \ar[r] & \Hom_{\mathrm{D^{b}}(\mathrm{MHS}_{\RR}^{+})}(\RR(0), N[2]) \ar[r] & \Hom_{\mathrm{MHS}_{\RR}^{+}}(\RR(0), \mathcal{H}^{2}(N)) \ar[r] & 0. 
        \end{tikzcd}
    \end{equation*}
    It follows from Theorem \ref{Theorem: weight conservativity} that $\mathcal{H}^{1}(N)$ is of weights $\ge k + 2$ and $\mathcal{H}^{2}(N)$ is of weights $\ge k + 3$. Hence, we have $\Hom_{\mathrm{MHS}_{\RR}^{+}}(\RR(0), \mathcal{H}^{2}(N)) = 0$ for weight reasons and $\Ext_{\mathrm{MHS}_{\RR}^{+}}^{1}(\RR(0), \mathcal{H}^{1}(N)) = 0$ by \cite[Theorem A.2.10]{HW98}. Therefore, the following identity holds: 
    \begin{equation*}
        \Hom_{\mathrm{D}^{b}(\mathrm{MHS}_{\RR}^{+})}(\RR(0), N[2]) = 0.
    \end{equation*}
    Finally, the vanishing 
    \begin{equation*}
         \Hom_{\mathrm{DM}_{B, c}(\Spec(\QQ))}(1_{\Spec(\QQ)}, C_{\ge(k + 1)}[2]) = 0
    \end{equation*}
    follows from Corollary \ref{corllaray: conservativity}. 
\end{proof}

\begin{proposition}
We have the following exact sequence 
\[  
    \begin{tikzcd}
        0 \ar[r] & \mathrm{H}_{M}^{3}(S^{*}, j_{!*}V(2)) \ar[r] & \mathrm{H}^{3}_{M}(S,V(2)) \ar[r, "{\partial_{S}^{M}}"] & \mathrm{H}^{3}_{M}(\partial{S}, i^{*}j_{*}V(2)). 
    \end{tikzcd}
\]
\end{proposition}

\begin{proof}
First, if we let $N = j_{*}V(2)$ in the exact triangle of Proposition \ref{Prop: motivic localization} and apply the functor $a_{*}$ to it, then we get the following exact triangle in $\mathrm{DM}_{B, c}(\Spec(\QQ))$: 
\begin{equation*}
    a_{*}j_{!}V(2) \rightarrow a_{*}j_{*}V(2) \rightarrow a_{*}i_{*}i^{*}j_{*}V(2) \rightarrow a_{*}j_{!}V(2)[1].
\end{equation*}
We then apply the functor $\Hom_{\mathrm{DM}_{B, c}(\Spec(\QQ))}(1_{\Spec(\QQ)}, -[3])$ to it and get the long exact sequence
\begin{equation*}
    \cdots \rightarrow \mathrm{H}^{3}_{M}(S^{*},j_{!}V(2))  \rightarrow \mathrm{H}^{3}_{M}(S,V(2)) \rightarrow \mathrm{H}^{3}_{M}(\partial{S}, i^{*}j_{*}V(2)) \rightarrow \cdots. 
\end{equation*}
\par Second, if we apply the functor $\Hom_{\mathrm{DM}_{B, c}(\Spec(\QQ))}(1_{\Spec(\QQ)}, -[3])$ to the commutative diagram (\ref{eq: commutative diagram factor through j_{!} to j_{*}}), then we have the following commutative diagram: 
          \begin{equation}
             \begin{tikzcd}
                 \mathrm{H}^{3}_{M}(S^{*},j_{!}V(2)) \ar[rr, "m_{*}"] \ar[rd, "\iota_{2, *}"]& & \mathrm{H}^{3}_{M}(S,V(2)) \\
                 & \mathrm{H}_{M}^{3}(S^{*}, j_{!*}V(2))  \ar[ru, "\iota_{1, *}"] & 
             \end{tikzcd}, 
          \end{equation}
where the map $\iota_{1, *}$ is induced from the map $\iota_{1}$ in the commutative diagram (\ref{eq: commutative diagram factor through j_{!} to j_{*}}). 
\par Finally, after applying the functor 
\begin{equation*}
    \Hom_{\mathrm{DM}_{B, c}(\Spec(\QQ))}(1_{\Spec(\QQ)}, -[3])
\end{equation*}
to the exact triangle (\ref{eq: exact triangle: intersection motive and j_{!}}) we get the exact sequence
\begin{equation*}
    \begin{tikzcd}
       \Hom_{\mathrm{DM}_{B, c}(\Spec(\QQ))}(1_{\Spec(\QQ)}, C_{\ge(k + 1)}[2]) \ar[r] & \mathrm{H}_{M}^{3}(S^{*}, j_{!*}V(2)) \ar[r, "\iota_{1, *}"] & \mathrm{H}^{3}_{M}(S,V(2)). 
    \end{tikzcd}
\end{equation*}
It follows from Lemma \ref{Lemma: mot vanshing} that $\iota_{1, *}$ is injective. Therefore, we get the following exact sequence: 
\begin{equation*}
     0 \rightarrow  \mathrm{H}_{M}^{3}(S^{*}, j_{!*}V(2)) \rightarrow \mathrm{H}^{3}_{M}(S,V(2)) \rightarrow \mathrm{H}^{3}_{M}(\partial{S}, i^{*}j_{*}V(2)).
\end{equation*}
\end{proof}

\begin{definition}
\label{def: motivic cohomology of interior motives}
    For the interior motive $\mathrm{Gr_{0}}M_{gm}(V(2)) \in \mathrm{DM}_{gm}(\QQ)$, we define its motivic cohomology $\mathrm{H}_{M}^{3}(\mathrm{Gr_{0}}M_{gm}(V(2)), \QQ(0))$ as 
    \begin{equation*}
        \mathrm{H}_{M}^{3}(\mathrm{Gr_{0}}M_{gm}(V(2)), \QQ(0)) := \Hom_{\mathrm{DM}_{gm}(\QQ)}(\mathrm{Gr_{0}}M_{gm}(V(2)), \QQ(0)), 
    \end{equation*}
    which is isomorphic to 
    \begin{equation*}
        \mathrm{H}_{M}^{3}(S^{*}, j_{!*}V(2)) = \Hom_{\mathrm{DM}_{B, c}(\Spec(\QQ))}(1_{\Spec(\QQ)}, a_{*}j_{!*}V(2)[3]). 
    \end{equation*}
\end{definition}

Hence, we get the following proposition. 
\begin{proposition}
\label{Prop: mot exact sequence}
\label{}
    We have the exact sequence 
    \[  
        \begin{tikzcd}
            0 \ar[r] &   \mathrm{H}_{M}^{3}(\mathrm{Gr_{0}}M_{gm}(V(2)), \QQ(0)) \ar[r] & \mathrm{H}^{3}_{M}(S,V(2)) \ar[r, "{\partial_{S}^{M}}"] & \mathrm{H}^{3}_{M}(\partial{S}, i^{*}j_{*}V(2)). 
        \end{tikzcd}
    \]
\end{proposition}

\begin{remark}
    The exact sequence in Proposition \ref{Prop: mot exact sequence} is the ``motivic lifting" of the exact sequence in Proposition \ref{Prop: Hodge exact sequence}. 
\end{remark}

\vspace{10pt}

\noindent \textbf{Step II.} In the second step, the goal is to prove the diagram in Proposition \ref{Prop: mot commutative diagram} commutes. 

\begin{proposition}
\label{Prop: mot commutative diagram}
We have the following commutative diagram:
\[
    \begin{tikzcd}
         \mathcal{B}_{n} \ar[d, "p^{*}\circ Eis_{M}^{n}"] & \\
        \mathrm{H}^{1}_{M}(M,W(1)) \ar[r, "\partial^{M}_{M}"] \ar[d, "{\iota_{*}^{M}}"] & \mathrm{H}^{1}_{M}(\partial{M}, i^{\prime, *}j^{\prime}_{*}W(1))) \ar[d, "{\theta^{M}}"] \\
        \mathrm{H}^{3}_{M}(S, V(2)) \ar[r, "\partial_{S}^{M}"] & \mathrm{H}^{3}_{M}(\partial{S}, i^{*}j_{*}V(2))
    \end{tikzcd}, 
\]
where $\iota_{*}^{M}$ is the Gysin morphism defined in Proposition \ref{Prop: Gysin for motivic cohomology}, $\partial_{M}^{M}$ and $\partial_{S}^{M}$ are the corresponding boundary maps and the map $\theta^{M}$ is induced from the map, 
\begin{equation}
\label{eq: Hodge theta in derived category.}
    (\partial \iota_{*}^{M} \colon q_{*}i^{\prime*}j^{\prime}_{*}W(1)[1] \rightarrow i^{*}j_{*}V(2)[3]) \in \mathrm{DM}_{B,c}(\partial S),
\end{equation}
which is degenerated from $\iota^{M}_{*}$.
\end{proposition}

\begin{proof}
The proof is parallel to the proof of Proposition \ref{Prop: Hodge commutative diagram}. First, recall that we have the following diagram:
\[
    \xymatrix{
        M \ar[r]^-{j^{'}} \ar[d]_{\iota} & M^{*} \ar[d]_-{p} & \partial{M} \ar[l]_-{i^{'}} \ar[d]_-{q} \\
        S \ar[r]^-{j} & S^{*} & \partial{S} \ar[l]_-{i}\\
    }. 
\]
It follows from Proposition \ref{Prop: Gysin for motivic cohomology} that the Gysin morphism $\iota_{*}: \mathrm{H}^{1}_{M}(M,W(1)) \rightarrow \mathrm{H}^{3}_{M}(S, V(2))$ is induced by 
\begin{equation*}
    (\iota_{*}W(1)[1] \rightarrow V(2)[3]) \in \mathrm{DM}_{B,c}(S). 
\end{equation*}
\par Second,  after applying the functor $j_{*}$ to the morphism and by the fact that $j_{*}\iota_{*} = p_{*}j^{\prime}_{*}$,
we have the morphism 
\begin{equation}
\label{Map: in motivic commutative1}
    (p_{*}j^{\prime}_{*}W(1)[1] \rightarrow j_{*}V(2)[3]) \in \mathrm{DM}_{B,c}(S^{*}). 
\end{equation}
The morphism 
\begin{equation*}
    (i_{*}i^{*}p_{*}j^{\prime}_{*}W(1)[1] \rightarrow i_{*}i^{*}j_{*}V(2)[3]) \in \mathrm{DM}_{B,c}(S^{*})
\end{equation*}
follows by applying the functor $i_{*}i^{*}$ to the morphism in (\ref{Map: in motivic commutative1}). It follows from the proper base change theorem for Beilinson motives \cite[Theorem 2.2.14.(4)]{Cisinski_Deglise} and the fact that $p$ is proper that $i^{*}p_{*} = q_{*}i^{\prime*}$. From it,  we get 
\begin{equation*}
    (i_{*}q_{*}i^{\prime*}j^{\prime}_{*}W(1)[1] \rightarrow i_{*}i^{*}j_{*}V(2)[3]) \in \mathrm{DM}_{B,c}(S^{*}). 
\end{equation*}
The natural transformation $\mathrm{id} \rightarrow i_{*}i^{*}$ from adjunction gives us the following commutative diagram in $\mathrm{DM}_{B,c}(S^{*})$: 
\[
    \begin{tikzcd}
        p_{*}j^{\prime}_{*}W(1)[1] \ar[r] \ar[d] & i_{*}q_{*}i^{\prime*}j^{\prime}_{*}W(1)[1] \ar[d] \\
        j_{*}V(2)[3] \ar[r] & i_{*}i^{*}j_{*}V(2)[3]. 
    \end{tikzcd}
\]
\par Finally, we apply the functor $\Hom_{\mathrm{DM}_{B,c}(S^{*})}(1_{S^{*}}, - )$ to the above  commutative diagram and get the commutative diagram: 
\[
    \begin{tikzcd}
        \mathrm{H}^{1}_{M}(M, W(1)) \ar[r] \ar[d] & \mathrm{H}^{1}_{M}(\partial M, i^{\prime*}j^{\prime}_{*}W(1))  \ar[d] \\
        \mathrm{H}^{3}_{M}(S, V(2)) \ar[r] & \mathrm{H}^{3}_{M}(\partial S, i^{*}j_{*}V(2)). 
    \end{tikzcd}
\]
\end{proof}

\begin{remark}
    \begin{itemize}
        \item  From the construction, we can see the boundary map $\partial_{S}^{M}$ in Proposition \ref{Prop: mot commutative diagram} is the same as the $\partial_{S}^{H}$ in Proposition \ref{Prop: mot exact sequence}. Hence, we use the same notation for them. 
        \item  The commutative diagram in Proposition \ref{Prop: mot commutative diagram} is the ``motivic lifting" of the commutative diagram in Proposition \ref{Prop: Hodge commutative diagram}. 
    \end{itemize}
\end{remark}

\vspace{10pt}

\noindent \textbf{Step III.} In the third step, the goal is to prove the map $\theta^{M}$ in Proposition \ref{Prop: mot commutative diagram} is zero, which is stated in Proposition \ref{Prop: mot theta is zero}. \\

\begin{proposition}
\label{Prop: mot theta is zero}
    The map $\theta^{M}$ in Proposition \ref{Prop: mot commutative diagram} is zero. 
\end{proposition}
\begin{proof}
    First, it follows from Proposition \ref{Prop: mot commutative diagram} that in order to prove $\theta^{M} = 0$, it suffices to prove the following statement: 
    for any 
    \begin{equation*}
        f \in \mathrm{H}^{1}_{M}(\partial M, i^{\prime*}j^{\prime }_{*}W(1)) = \Hom_{\mathrm{DM}_{B,c}(S^{*})}(1_{S^{*}}, i_{*}q_{*}i^{\prime*}j^{\prime}_{*}W(1)[1]), 
    \end{equation*}
    the element 
    \begin{equation*}
        i_{*} (\partial \iota_{*}^{M}) \circ f \in \Hom_{\mathrm{DM}_{B,c}(S^{*})}(1_{S^{*}}, i_{*}i^{*}j_{*}V(2)[3]) = \mathrm{H}^{3}_{M}(\partial S, i^{*}j_{*}V(2))
    \end{equation*}
    is $0$. 
    \par Second, it follows from the identity $R_{H}(i_{*} (\partial \iota_{*}^{M}) \circ f) = i_{*} (\partial \iota_{*}^{H}) \circ R_{H}(f)$ in
    \begin{equation*}
        \Hom_{\mathrm{D}^{b}(\mathrm{MHM}_{\RR}(S^{*}/\RR))}(\RR(0)_{S^{*}}, i_{*}i^{*}j_{*}V(2)[3]) = \mathrm{H}^{1}_{H}(\partial S, i^{*}j_{*}V(2))
    \end{equation*}
    where 
    \begin{equation*}
        R_{H}(f) \in \mathrm{H}^{0}_{H}(\partial M, i^{\prime*}j^{\prime }_{*}W(1)) = \Hom_{\mathrm{D}^{\mathrm{b}}(\mathrm{MHM}_{\RR}(S^{*}/\RR))}(\RR(0)_{S^{*}}, i_{*}q_{*}i^{\prime*}j^{\prime}_{*}W(1)[2])
    \end{equation*}
    and Corollary \ref{Corollary: Hodge theta is zero} that $R_{H}(i_{*} (\partial \iota_{*}^{M}) \circ f) = 0$. 
    \par Finally, by Theorem \ref{Theorem: conservativity}, we have $i_{*} (\partial \iota_{*}^{M}) \circ f = 0$. 
\end{proof}

Finally, by combining Proposition \ref{Prop: mot exact sequence}, Proposition \ref{Prop: mot commutative diagram} and Proposition \ref{Prop: mot theta is zero}, we finish the proof of Theorem \ref{Thm: mot vanish on the boudary}.


\clearpage

\part{Connection to L-values}

\section{Constructing the differential forms and pairing with the motivic classes}
\label{Sec: the differential form and the pairing}

\begin{convention}
     In this section, we let $V = V^{a,b}\{r, s\} \in \mathrm{Rep}_{\QQ}(V(2))$ and $n = a + b + r + s$, and Betti cohomology, compactly supported cohomology and interior cohomology all have $\QQ$-coefficients, which is different from subsection \ref{SS: the proof of the vanishing}.
\end{convention}

\subsection{The use of Poincar\'{e} duality}
\label{SS: Poincare duality}
In this subsection, we will explain how the Poincar\'{e} duality pairing can be used to compute the Beilinson regulator. 
The idea is due to Beilinson \cite{Beilinson_Modular_Curve} (see also \cite{Kings98}). 

\par Let us first recall a general result.
\begin{lemma}\textnormal{\cite[Lemma 4.11]{LemmaII17}}
\label{Lemma: def of Ext^{1}}
    Let $E$ be a number field and $M$ be an object of $\mathrm{MHS}_{\RR, E}^{+}$ with pure weight $w < 0$. Let $M_{dR}$ be the $E\otimes_{\QQ}{\RR}$-submodule of $M_{\CC}$ where the de Rham involution acts trivially and let $M^{-}$ be the submodule of $M$ where the infinitesimal Frobenius $F_{\infty}$ acts by multiplication by $-1$. We write $M^{-}(-1) = \frac{1}{2 \pi i} M^{-}$. Then there is an short exact sequence of $E \otimes_{\QQ}{\RR}$-modules
    \begin{equation*}
        0 \rightarrow F^{0}M_{dR} \rightarrow M^{-}(-1) \rightarrow \Ext^{1}_{\mathrm{MHS_{\RR}^{+}}}(\RR(0), M) \rightarrow 0,  
    \end{equation*}
    where the first map is the composite of the natural inclusions 
    \begin{equation*}
        F^{0}M_{dR} \rightarrow M_{dR} \rightarrow M_{\CC}
    \end{equation*}
    and of the projections 
    \begin{equation*}
        M_{\CC} \rightarrow M(-1) \rightarrow M^{-}(-1)
    \end{equation*}
    defined by $v \mapsto \frac{1}{2}(v - \bar{v})$ and $\frac{1}{2}(v - F_{\infty}(v))$, respectively. 
\end{lemma}

We apply Lemma \ref{Lemma: def of Ext^{1}} to the situation that we are interested in to get the following lemma.

\begin{lemma}
\label{Lemma: exact seq to def Ext^{1} in our situation}
     Let $\pi_{f}$ be the non-archimedean part of an irreducible cuspidal automorphic representation $\pi$ of $G$ (see Definition \ref{defn:G_H}) whose archimedean component $\pi_{\infty}$ belongs to the discrete series $L$-packet $P(V_{\CC}(2))$ and denote by $\pi_{f}$ its rational model over $E(\pi_f)$ (see Theorem \ref{Thm: rational field}). Then there is a short exact sequence:
     \begin{equation} \label{exact seq: Ext^1}
          0 \rightarrow F^{0}M_{dR}(\pi_f, V(2))_{\RR} \rightarrow M_{B}(\pi_f, V(2))^{-}_{\RR}(-1) \rightarrow \Ext^{1}_{\mathrm{MHS_{\RR}^{+}}}(\RR(0), M_{B}(\pi_f, V(2))_{\RR}) \rightarrow 0,
     \end{equation}
     where the second map is defined in Lemma \ref{Lemma: def of Ext^{1}}. 
\end{lemma}

Let $F^{0}M_{dR}(\pi_f, V(2))^{*}$ be the dual of $F^{0}M_{dR}(\pi_f, V(2))$. It follows from Lemma \ref{Lemma: exact seq to def Ext^{1} in our situation} that the one dimensional $E(\pi_f)$-vector space 
\begin{equation*}
    \mathcal{B}(\pi_f, V(2)) = \mathrm{det}_{E(\pi_f)} F^{0}M_{dR}(\pi_f, V(2))^{*} \otimes_{E(\pi_f)} \mathrm{det}_{E(\pi_f)} M_{B}(\pi_f, V(2))^{-}(-1) 
\end{equation*}
is an $E(\pi_f)$-structure of the $E(\pi_f)\otimes \RR$-module 
\begin{equation*}
    \mathrm{det}_{E(\pi_f)\otimes_{\QQ}\RR}\Ext^{1}_{\mathrm{MHS_{\RR}^{+}}}(\RR(0), M_{B}(\pi_f, V(2))_{\RR}).
\end{equation*}

Here $\mathcal{B}(\pi_f, V(2))$ is the Beilinson $E(\pi_f)$-structure defined in \cite[\S6.1]{Nekovar94}. We now define the Deligne $E(\pi_f)$-structure using the Beilinson $E(\pi_f)$-structure.

\begin{definition}[Deligne $E(\pi_f)$-structure]
\label{def: Delign-rational structure}
    Let $\delta(\pi_f, V(2)) \in (E(\pi_f) \otimes_{\QQ} \CC)^{\times}$ be the determinant of the isomorphism $I_{\infty} \colon M_{B}(\pi_f, V(2))_{\CC} \rightarrow M_{dR}(\pi_f, V(2))_{\CC}$ computed using the basis defined over $E(\pi_f)$ on both sides. Then the Deligne $E(\pi_f)$-structure of 
    \begin{equation*}
            \mathrm{det}_{E(\pi_f)\otimes_{\QQ}\RR}\Ext^{1}_{\mathrm{MHS_{\RR}^{+}}}(\RR(0), M_{B}(\pi_f, V(2))_{\RR})
    \end{equation*}
    is 
    \begin{equation*}
        \mathcal{D}(\pi_f, V(2)) = (2\pi i)^{\dim_{E(\pi_f)} M_{B}(\pi_f, W)^{-}} \delta(\pi_f, W)^{-1} \mathcal{B}(\pi_f, V(2)). 
    \end{equation*}
\end{definition}

\begin{remark}
    This definition does not depend on the choice of the bases. 
\end{remark}

\begin{lemma}
    It can be seen that $\Ext^{1}_{\mathrm{MHS_{\RR}^{+}}}(\RR(0), M_{B}(\pi_f, V(2))_{\RR})$ is a rank one $E(\pi_f) \otimes_{\QQ} \RR$-module. 
\end{lemma}

\begin{proof}
    It can be seen that 
    \begin{equation*}
        \mathrm{rank}_{E(\pi_f) \otimes_{\QQ} \RR} F^{0}M_{dR}(\pi_f, V(2))_{\RR} = 2
    \end{equation*}
    from the Hodge decomposition of $M(\pi_f, V(2))$ in Corollary \ref{Corollary: Hodge decomp for motives} and the condition of coefficients in equation (\ref{eq: Cond of coefficients}). By the exact sequence (\ref{exact seq: Ext^1}) and 
    \begin{equation*}
        \mathrm{rank}_{E(\pi_f) \otimes_{\QQ} \RR} M_{B}(\pi_f, V(2))^{-}_{\RR}(-1) = 3, 
    \end{equation*}
    we conclude that $\Ext^{1}_{\mathrm{MHS_{\RR}^{+}}}(\RR(0), M_{B}(\pi_f, V(2))_{\RR})$ is a rank one $E(\pi_f) \otimes_{\QQ} \RR$-module. 
\end{proof}

Recall that by Theorem \ref{Thm: Hdg vanish on the boudary}, the map $\mathcal{E}is_{H}^n \colon \mathcal{B}_{n,\RR} \rightarrow \mathrm{H}^{3}_{H}(S, V(2))$ factors through the inclusion 
\[
    \Ext^{1}_{\mathrm{MHS}_{\RR}^{+}}(\mathbf{1}, \mathrm{H}^{2}_{B,!}(S, V(2))_{\RR}) \hookrightarrow \mathrm{H}^{3}_{H}(S, V(2)).
\]Let $\mathcal{K}(V(2))$ be the sub $\QQ[G(\mathbb{A}_{f})]$-module of 
\begin{equation*}
    \Ext^{1}_{\mathrm{MHS_{\RR}^{+}}}(\RR(0), \mathrm{H}^{2}_{B, !}(S, V(2))_{\RR}) = \varinjlim_{L}  \Ext^{1}_{\mathrm{MHS_{\RR}^{+}}}(\RR(0), \mathrm{H}^{2}_{B, !}(S(L), V(2))_{\RR}) 
\end{equation*}
generated by the image of $\mathcal{E}is_{H}^{n}$ with domain $\mathcal{B}_{n}$ and let $\mathcal{K}(\pi_{f}, V(2))$ be defined by
\begin{equation}
\label{eq: K(pi, V)}
    \mathcal{K}(\pi_{f}, V(2)) := \Hom_{\QQ[G(\mathbb{A}_f)]}(\Res_{E(\pi_f)/\QQ}\pi_f, \mathcal{K}(V(2)). 
\end{equation}
This is an $E(\pi_f)$-submodule of $\Ext^{1}_{\mathrm{MHS_{\RR}^{+}}}(\RR(0), M_{B}(\pi_f, V(2))_{\RR})$. Therefore, now we have two $E(\pi_f)$-submodules of $\Ext^{1}_{\mathrm{MHS_{\RR}^{+}}}(\RR(0), M_{B}(\pi_f, V(2))_{\RR})$. The first one, which is called $\mathcal{D}(\pi_f, V(2))$, is defined using the comparison between Betti and De Rham cohomology, which is more elementary. The second one, called $\mathcal{K}(\pi_f, V(2))$, is defined using the Beilinson regulator $r_{H}$, which is more sophisticated. By definition, we know that 
$\mathcal{D}(\pi_f, V(2))$ is non-zero, so we hope to know whether $\mathcal{K}(\pi_f, V(2))$ is zero or not by a comparision between $\mathcal{K}(\pi_f, V(2))$ and $\mathcal{D}(\pi_f, V(2))$. 

\begin{definition}
    If $\mu$ and $\mu^{\prime}$ are two elements of $E(\pi_f)\otimes_{\QQ}{\CC}$, we write $\mu \sim \mu^{\prime}$ if there exists $\lambda \in E(\pi_f)^{\times}$ such that $\mu = \lambda \mu^{\prime}$. 
\end{definition}

\begin{lemma}
\label{Lemma: linear form and rational structure}
    Let ${v}_{D}$ be a non-zero vector in $\mathcal{D}(\pi_f, V(2))$ and $v_{K}$ be a non-zero vector in $\mathcal{K}(\pi_{f}, V(2))$. Let $\tilde{v}_{D}$ and $\tilde{v}_{K}$ be any lifting of $v_{D}$ and $v_{K}$ in $M_{B}(\pi_f, V)^{-}_{\RR}(-1)$ using the last arrow in the exact sequence (\ref{exact seq: Ext^1}). Then for any $E(\pi_{f})\otimes_{\QQ}\RR$-linear map $\psi: M_{B}(\pi_f, V(2))^{-}_{\RR}(-1) \rightarrow E(\pi_f) \otimes_{\QQ} \CC$ which is trivial on $F^{0}M_{dR}(\pi_f, V(2))_{\RR}$, we have 
    \begin{equation*}
        \mathcal{K}(\pi_{f}, V(2)) = \frac{\psi(\tilde{v}_{K})}{\psi(\tilde{v}_{D})} \mathcal{D}(\pi_f, V(2)). 
    \end{equation*}
\end{lemma}

\begin{proof}
    It follows from the exact sequence (\ref{exact seq: Ext^1}).
\end{proof}

Our goal is to compute $\psi(\tilde{v}_{K})$ and $\psi(\tilde{v}_{D})$ seperately for a well-chosen linear functional $\psi$. One natural way to choose $\psi$ is to use Poincar\'{e} duality. Hence, let us first recall properties of 
Poincar\'{e} duality for Picard modular surfaces. By previous computations, we assume that 
\begin{equation}
\label{eq: Cond of coefficients}
    \left \{
    \begin{aligned}
        & 0 \le -r \le a \\
        & 0 \le -s \le b \\
        & a > 0 \quad \text{and} \quad b > 0 \\
        &  r \neq 0 \quad \text{or} \quad s \neq 0
    \end{aligned}
    \right. 
\end{equation}
Recall for a representation $V = V^{a, b}\{r, s\}$, its contragredient representation is $D = V^{*} = V^{b, a}\{-a - b - r, -a - b -r \}$. In other words, we have a perfect pairing \footnote{Both $V$ and $D$ are in $\mathrm{Rep}_{\QQ}(G)$. }
\begin{equation*}
    V \otimes D \rightarrow \QQ(0). 
\end{equation*}
The pairing induces a $G(\AAA_{f})$-equivariant pairing 
\begin{equation*}
    \langle \cdot, \cdot \rangle_{B} \colon \mathrm{H}^{2}_{B, !}(S, V(2)) \otimes \mathrm{H}^{2}_{B, !}(S, D) \rightarrow \mathrm{H}^{4}_{B, !}(S, \QQ(2)) \rightarrow \QQ(0),
\end{equation*}
where the last map is defined by the trace. The pairing becomes perfect after restriction to the vectors which are invariant by a compact open subgroup of $G(\AAA_f)$.
\begin{fact}
    Here, $\QQ(0)$ is given by an action of $G(\AAA_f)$ by $|\mu|^{-2}$.
\end{fact}
\begin{proof}
    We have 
    \begin{equation}
    \label{eq: isom duality of Betti interior cohomology}
        \mathrm{H}^{4}_{B, !}(S, \QQ(2)) \cong  \mathrm{H}^{4}_{B, c}(S, \QQ(2)) \cong  \mathrm{H}^{0}_{B}(S, \QQ(-2)) \cong \bigoplus (\chi |\mu|^{-2}),
    \end{equation}
    where the sum is over finite Hecke characters $\chi \colon \QQ^{\times} \backslash \AAA_{f, \QQ}^{\times} \rightarrow \CC^{\times}$. The first isomorphism in (\ref{eq: isom duality of Betti interior cohomology}) comes from Corollary \ref{Corollary: Localization long exact seq} (1) and the fact $\dim \partial S = 0$; the second isomorphism in (\ref{eq: isom duality of Betti interior cohomology}) holds by Poincar\'{e} duality, and the last isomorphism in (\ref{eq: isom duality of Betti interior cohomology}) can be deduced from \cite[Theorem 5.17]{Milne17}. The trace map is a projection onto the factor with $\chi$ trivial. 
\end{proof}
\begin{remark}
    This fact is similar to \cite[p295]{Taylor93} in the Siegel 3-folds case.
\end{remark}
Hence, this induces a $G(\AAA_f)$-equivariant morphism of Hodge structures \footnote{It is a $\QQ$-morphism.}
\begin{equation}
\label{eq: M_{B} pairing}
    \langle \cdot, \cdot \rangle_{B} \colon M_{B}(\pi_f, V(2)) \otimes M_B(\tilde{\pi}_f|\mu|^{-2}, D) \rightarrow E(\pi_f)(0). 
\end{equation}

Recall by Proposition \ref{Prop: Hodge decomp for motives} that we have the Hodge decomposition
\begin{align*}
    & M_{B}(\tilde{\pi}_f|\mu|^{-2}, D)_{\CC} \\
    \cong & M_{B}^{a + b + r + 2, s} \oplus M_{B}^{a + r + 1, b + s + 1} \oplus M_{B}^{r, a + b + s + 2} \\
    \oplus & M_{B}^{s, a + b + r + 2} \oplus M_{B}^{b + s + 1, a + r + 1} \oplus M_{B}^{a + b + s + 2, r}. 
\end{align*}

\begin{lemma}
\label{lemma: abstrac paring with Omega}
    Let $\Omega \in M(\tilde{\pi}_f|\mu|^{-2}, D)_{\CC}^{+}$. Let 
    \[
        \begin{tikzcd}
            M_{B}(\pi_f, V(2))^{-}_{\RR}(-1) \ar[r, "{\langle\Omega, \cdot \rangle_{B}}"] & E(\pi_f) \otimes_{\QQ} \CC 
        \end{tikzcd}
    \]
    be the morphism defined by the composition of the inclusion $M_{B}(\pi_f, V(2))^{-}_{\RR}(-1) \rightarrow M_{B}(\pi_f, V(2))_{\CC}$ and the pairing with $\Omega$. Assume $\Omega$ belongs to $M_{B}^{a + r + 1, b + s + 1} \oplus M_{B}^{b + s + 1, a + r + 1}$. Then 
    \begin{equation}
    \label{eq:comparion_rat_struc}
        \mathcal{K}(\pi_{f}, V(2)) = \frac{\langle \Omega, \tilde{v}_{K} \rangle_{B}}{\langle \Omega, \tilde{v}_{D} \rangle_{B}} \mathcal{D}(\pi_f, V(2))
    \end{equation}
\end{lemma}

\begin{proof}
    This follows from a Hodge-type computation. 
\end{proof}

\begin{convention}
    From now on, we will use $\Sh_{G}$ (resp. $\Sh_{H}$) to denote $\Sh_{G, \CC}^{an}$ (resp., $\Sh_{H, \CC}^{an}$). 
\end{convention}

\begin{remark}
\label{remark: pairing}
    Let $\mathbf{1}$ be the multiplicative identities of $E(\pi_f) \otimes_{\QQ} \CC$. 
    \begin{enumerate}
        \item By definition, we have the following isomorphisms 
              \begin{align*}
                  & M_{B}(\pi_f, V(2))_{\CC} \cong ((\mathrm{H}^{2}_{B, !}(\Sh_{G}, V(2))_{\CC} \oplus  \mathrm{H}^{2}_{B, !}(\overline{\Sh_{G}}, V(2))_{\CC})[\pi_f]) \otimes_{\CC} (E(\pi_f) \otimes_{\QQ} \CC), \\
                  & M_B(\tilde{\pi}_f|\mu|^{-2}, D)_{\CC} \cong ((\mathrm{H}^{2}_{B, !}(\Sh_{G}, D)_{\CC} \oplus \mathrm{H}^{2}_{B, !}(\overline{\Sh_{G}}, D))_{\CC})[\tilde{\pi}_f|\mu|^{-2}]) \otimes_{\CC} (E(\pi_f) \otimes_{\QQ} \CC),
              \end{align*}
              where the notation $[\cdot]$ means the corresponding isotypic subspaces. 
        \item If we could construct a differential form 
        \begin{equation*}
            \omega_{\Psi} = \omega \otimes \Psi_f \in \mathrm{H}^{2}_{B, !}(\Sh_{G}, D)_{\CC}[\pi_f] = \mathrm{H}^{2}(\lieg_{\CC}, K_{G}; D_{\CC} \otimes_{\CC} \tilde{\pi}_{\infty})\otimes \tilde{\pi}_f|\mu|^{-2}
        \end{equation*}
        associated to a $\Psi = \Psi_{f} \otimes \Psi_{\infty}$
        that belongs to a cuspidal automorphic representation $\tilde{\pi} = \tilde{\pi}_{f}|\mu|^{-2} \otimes \tilde{\pi}_{\infty}$ of $G(\AAA)$, then it is natural to let $\Omega \in M(\tilde{\pi}_f|\mu|^{-2}, D)_{\CC}^{+}$ be 
        \begin{equation*}
            \frac{1}{2} (\omega_{\Psi} + \overline{\omega_{\Psi}}) \otimes \mathbf{1}, 
        \end{equation*}
        where $\overline{\omega_{\Psi}} \in \mathrm{H}^{2}_{B, !}(\overline{\Sh_{G}}, D)_{\CC}$ is the complex conjugate of $\omega_{\Psi}$. 
        \item Since the Eisenstein symbol is defined over $\QQ$ (see subsubsection \ref{SSS: Eisenstein symbol}),  
              $\tilde{v}_{K}$ can be represented by 
              \begin{equation*}
                (\tilde{v}_{K}^{1}, \tilde{v}_{K}^{1}) \otimes \mathbf{1}, 
              \end{equation*}
              where the first $\tilde{v}_{K}^{1}$ belongs to $\mathrm{H}^{2}_{B, !}(\Sh_{G}, V(2))_{\CC}[\pi_f]$ and the second $\tilde{v}_{K}^{1}$ belongs to 
              \begin{equation*}
                   \mathrm{H}^{2}_{B, !}(\overline{\Sh_{G}}, V(2))_{\CC}[\pi_f].
              \end{equation*}
        \item The Poincar\'{e} duality pairing 
            \begin{equation*}
                 \langle \cdot, \cdot \rangle_{B} \colon \mathrm{H}^{2}_{B, !}(S, V(2))_{\CC} \otimes \mathrm{H}^{2}_{B, !}(S, D)_{\CC} \rightarrow \CC(0),
            \end{equation*}
            can be decomposed into 
            \begin{equation*}
                 \langle \cdot, \cdot \rangle_{B} = \langle \cdot, \cdot \rangle_{B}^{\Sh_{G}} + \overline{\langle \cdot, \cdot \rangle_{B}^{\Sh_{G}}}, 
            \end{equation*}
            where $\langle \cdot, \cdot \rangle_{B}^{\Sh_{G}}$ is the Poincar\'{e} duality pairing on $\Sh_{G}$: 
            \begin{equation*}
                \langle \cdot, \cdot \rangle_{B}^{\Sh_{G}} \colon \mathrm{H}^{2}_{B, !}(\Sh_{G}, V(2))_{\CC} \otimes \mathrm{H}^{2}_{B, !}(\Sh_{G}, D)_{\CC} \rightarrow \CC(0), 
            \end{equation*}
            and its complex conjugate $\overline{\langle \cdot, \cdot \rangle_{B}^{\Sh_{G}}}$ is t Poincar\'{e} duality pairing on $\overline{\Sh_{G}}$: 
            \begin{equation*}
                \overline{\langle \cdot, \cdot \rangle_{B}^{\Sh_{G}}} \colon \mathrm{H}^{2}_{B, !}(\overline{\Sh_{G}}, V(2))_{\CC} \otimes \mathrm{H}^{2}_{B, !}(\overline{\Sh_{G}}, D)_{\CC} \rightarrow \CC(0). 
            \end{equation*}
        \item  We can see that 
                \begin{align*}
                    & \langle \Omega, \tilde{v}_{K} \rangle_{B} \\
                 =  & \langle \frac{1}{2} (\omega_{\Psi} + \overline{\omega_{\Psi}}), (\tilde{v}_{K}^{1}, \tilde{v}_{K}^{1}) \rangle_{B} \otimes \mathbf{1} \\
                 =  & \frac{1}{2} \langle \omega_{\Psi}, \tilde{v}_{K}^{1} \rangle_{B}^{\Sh_{G}} \otimes \mathbf{1} + \frac{1}{2} \overline{\langle \overline{\omega_{\Psi}}, \tilde{v}_{K}^{1} \rangle_{B}^{\Sh_{G}}}  \otimes \mathbf{1}\\
                 =  & \langle \omega_{\Psi}, \tilde{v}_{K}^{1} \rangle_{B}^{\Sh_{G}} \otimes \mathbf{1}. 
                \end{align*}
                Therefore, in order to compute the pairing $\langle \Omega, \tilde{v}_{K} \rangle_{B}$, it suffices to construct a differential form $\omega_{\Psi}$ and compute $\langle \omega_{\Psi}, \tilde{v}_{K}^{1} \rangle_{B}^{\Sh_{G}}$. 
    \end{enumerate}
\end{remark}

\begin{definition}
    If $\mu$ and $\mu^{\prime}$ are two elements of $E(\pi_f)\otimes_{\QQ}{\CC}$ (resp., $\CC$), we denote $\mu \sim \mu^{\prime}$ if there exists $\lambda \in (E(\pi_f) \otimes_{\QQ} \overline{\QQ})^{\times}$ (resp., $\overline{\QQ}^{\times}$) such that $\mu = \lambda \mu^{\prime}$. 
\end{definition}

\begin{convention}
    In this paper, because we can only compute $\langle \omega_{\Psi}, \tilde{v}_{K}^{1} \rangle_{B}^{\Sh_{G}}$ up to $\overline{\QQ}^{\times}$, we can freely choose $\tilde{v}_{K}$ up to $\overline{\QQ}^{\times}$. In other words, we can use $\mathcal{B}_{n, \overline{\QQ}}$ as the domain of $\mathcal{E}is_{H}$. 
\end{convention}

\subsection{Differential forms}
\label{SS: the test vector}

The next step is to explain how to construct a differential form $\omega_{\Psi}$ on $S$. It is based on some elementary representation-theoretic considerations. 

\par By equation (\ref{eq: decomosition of Betti into (g,K)-cohomology}), Theorem \ref{Thm: multiplicity 1} and Proposition \ref{Prop: (g, K)-cohomology = Hom}, we have the decomposition 
\begin{equation*}
     \mathrm{H}^{2}_{B, !}(S, D)_{\CC} \cong \mathrm{H}^{2}_{dR, !}(S, D)_{\CC} = \bigoplus_{\tilde{\pi} = \tilde{\pi}_{\infty} \otimes \tilde{\pi}_f} \Hom_{K_G}(\wedge^{2} \liep_{G, \CC}, D_{\CC} \otimes_{\CC} \tilde{\pi}_{\infty}) \otimes \tilde{\pi}_f. 
\end{equation*}
Hence, in order to associate a differential form $\omega_{\Psi} = \omega \otimes \Psi_{f}$ to a cusp form $\Psi = \Psi_{f} \otimes \Psi_{\infty}$ in the cuspidal automorphic representation  $\tilde{\pi} = \tilde{\pi}_{f}|\mu|^{-2} \otimes \tilde{\pi}_{\infty}$, it suffices to associate some
\begin{equation*}
    \omega \in \Hom_{K_G}(\wedge^{2} \liep_{G, \CC}, D_{\CC} \otimes_{\CC} \pi_{\infty})
\end{equation*}
to $\Psi_{\infty} \in \tilde{\pi}_{\infty}$. Here $\Psi_{f}$ is any vector in $\tilde{\pi}_{f} |\mu|^{-2}$, and the choice of $\Psi_{\infty} \in \tilde{\pi}_{\infty}$ is explained below. 

\par Now we explain the way to contruct $\omega$. Recall that the complexified Lie algebra of $G$ can be decomposed as $\lieg_{\CC} = \liek_{G, \CC} \oplus \liep^{+}_{G, \CC} \oplus \liep^{-}_{G, \CC}$. 
Let
\begin{align*}
    & X_{(1, -1, 0)} = \begin{pmatrix} 
                            0 & 1 & 0 \\
                            0 & 0 & 0 \\
                            0 & 0 & 0 
                       \end{pmatrix}, \ X_{(-1, 1, 0)} = \begin{pmatrix} 
                                                            0 & 0 & 0 \\
                                                            1 & 0 & 0 \\
                                                            0 & 0 & 0 
                                                         \end{pmatrix} \\
    & X_{(1, 0, -1)} = \begin{pmatrix} 
                            0 & 0 & 1 \\
                            0 & 0 & 0 \\
                            0 & 0 & 0 
                       \end{pmatrix}, \   X_{(-1, 0, 1)} = \begin{pmatrix} 
                                                            0 & 0 & 0 \\
                                                            0 & 0 & 0 \\
                                                            1 & 0 & 0 
                                                         \end{pmatrix} \\
     & X_{(0, 1, -1)} = \begin{pmatrix} 
                            0 & 0 & 0 \\
                            0 & 0 & 1 \\
                            0 & 0 & 0 
                       \end{pmatrix}, \ X_{(0, -1, 1)} = \begin{pmatrix} 
                                                            0 & 0 & 0 \\
                                                            0 & 0 & 0 \\
                                                            0 & 1 & 0 
                                                         \end{pmatrix} 
\end{align*}
be some fixed root vectors in $\lieg_{\CC}$. 
\par Let $\tilde{\pi}_{\infty}$ be the representation of $G(\RR)^{+}$ in the discrete series $L$-packet $P(D_{\CC})$ with Blattner parameter $(1 + a + r -s, -1 -b + r - s, r - s)$ \footnote{This is the $\pi_2$ in the talble (\ref{talble: DS L-packt}) in Proposition \ref{Prop: DS L-packet}.}. 

\begin{proposition}
\label{Prop: explicit construnction of the differential form}
    Let $v \in D_{\CC}$ be a vector of weight $(-a + s - r, b + s - r, s - r; -b - 2s + r)$ that is the lowest weight vector in the $K_G$-type $\tau_{(b + s - r, -a + s - r, s - r)}$ of the algebraic representation $D_{\CC}$ and let $\Psi_{\infty} \in \tilde{\pi}_{\infty}$ be a vector of weight $(-1 - b + r - s, 1 + a + r - s, r - s; b + 2s - r)$ that is the lowest weight vector of the minimal $K_{G}$-type $\tau_{(1 + a + r -s, -1 - b + r  + s, r - s)}$ of $\tilde{\pi}_{\infty}$. Let 
    \begin{equation*}
        X_{(1, 0, -1)} \otimes X_{(0, -1, 1)} \in \liep^{+}_{G, \CC} \otimes \liep^{-}_{G, \CC} = \tau_{(1, -1, 0)} \oplus \tau_{(0, 0, 0)}
    \end{equation*}
    be a highest weight vector of the $K_{G}$-type $\tau_{(1, -1, 0)}$. Then there exists a unique non-zero map 
    \begin{equation*}
        \omega \in \Hom_{K_G}(\liep^{+}_{G, \CC} \otimes \liep^{-}_{G, \CC}, D_{\CC} \otimes \tilde{\pi}_{\infty}) 
    \end{equation*}
    such that 
    \begin{equation}
    \label{eqn:form_satisfied}
        \omega( X_{(1, 0, -1)} \otimes X_{(0, -1, 1)}) = \sum\limits_{i = 0}^{a + b} (-1)^{i} X_{(1, -1 ,0)}^{i} v \otimes X^{2 + a + b - i}_{(1, -1, 0)} \Psi_{\infty}. 
    \end{equation}
\end{proposition}

\begin{proof}
    First, the $K_{G}$-weight of $X_{(1, -1 ,0)}^{i} v \otimes X^{2 + a + b - i}_{(1, -1, 0)} \Psi_{\infty}$ is $(1, -1, 0)$, which is the weight of $X_{(1, 0, -1)} \otimes X_{(0, -1, 1)}$. Hence, the map $\omega$ preserves $K_{G}$-weights. 
    \par Second, we have 
       \begin{align*}
           & X_{(1, -1, 0)}(\sum\limits_{i = 0}^{a + b} (-1)^{i} X_{(1, -1 ,0)}^{i} v \otimes X^{2 + a + b - i}_{(1, -1, 0)} \Psi_{\infty} ) \\
           =  & \sum\limits_{i = 0}^{a + b}(-1)^{i} X_{(1, -1 ,0)}^{i + 1} v \otimes X^{2 + a + b - i}_{(1, -1, 0)} \Psi_{\infty}  
           + \sum\limits_{i = 0}^{a + b}(-1)^{i} X_{(1, -1 ,0)}^{i} v \otimes X^{3 + a + b - i}_{(1, -1, 0)} \Psi_{\infty}  \\
           = \,  & (-1)\sum\limits_{i = 1}^{a + b}(-1)^{i} X_{(1, -1 ,0)}^{i} v \otimes X^{3 + a + b - i}_{(1, -1, 0)} \Psi_{\infty}
             + \sum\limits_{i = 1}^{a + b}(-1)^{i} X_{(1, -1 ,0)}^{i} v \otimes X^{3 + a + b - i}_{(1, -1, 0)} \Psi_{\infty} \\
           = \,  & 0. 
       \end{align*}
    Therefore, the vector $\sum\limits_{i = 0}^{a + b} (-1)^{i} X_{(1, -1 ,0)}^{i} v \otimes X^{2 + a + b - i}_{(1, -1, 0)} \Psi_{\infty}$ is a highest weight vector of $K_G$-type $\tau(1, -1, 0)$. So the map $\omega$ is non-zero and by 
    \begin{equation*}
        \dim \Hom_{K_G}(\liep^{+}_{G, \CC} \otimes \liep^{-}_{G, \CC}, D_{\CC} \otimes \tilde{\pi}_{\infty}) = 1, 
    \end{equation*}
    $\omega$ is uniquely determined. 
\end{proof}

\begin{remark}
    \begin{itemize}
        \item  Since $\omega$ is uniquely determined by the equaiton (\ref{eqn:form_satisfied}), $\omega$ must satisfy
                \begin{equation*}
                    \omega(X_{(0, 0, 0)}) = 0,
                \end{equation*}
                for any $X_{(0, 0, 0)} \in \tau_{(0, 0, 0)}$. 
        \item Since $v \in D_{\CC}$ is a vector of weight $(-a + s - r, b + s - r, s - r; -b - 2s + r)$ and the highest weight of $D_{\CC}$ is $(b + s - r, s - r, -a + s - r, -a + s - r; -b -2s + r)$, we can see that $v$ is lowest weight vector in $D_{\CC}$ with respect to the positive root system 
        \begin{equation*}
            \{ (1, -1, 0), (0, -1, 1), (-1, 0, 1) \},
        \end{equation*}
        which guarantees the existence of $v$. 
        Hence, we have 
        \begin{equation}
        \label{eq: Xv = 0}
            X_{(1, 0, -1)}^{a + b + 1}v = 0. 
        \end{equation}
    \end{itemize}
\end{remark}

\subsection{Restriction of differential forms}
\label{SS: res of the diff form}

Recall that we have the following inclusion of algebraic groups 
\[
    \iota\colon H \hookrightarrow G,  \quad (\begin{pmatrix} 
                                            a & b \\
                                            c & d \\
                                         \end{pmatrix}, z) \mapsto (\begin{pmatrix}
                                                                        a & & b \\
                                                                        & z & \\
                                                                        c & & d 
                                                                     \end{pmatrix}, z\bar{z}), 
\]
which induces the closed immersion of Shimura varieties 
\[
    \iota \colon \Sh_{H} \hookrightarrow \Sh_{G}. 
\]
Now, we compute the pullback of the differential form 
\begin{equation*}
    \iota^{*} \omega_{\Psi} = \iota^{*} (\omega \otimes \Psi_f) = \iota^{*}(\omega) \otimes \Psi_{f}
\end{equation*}
along the map $\iota$. Hence, it suffices to compute $\iota^{*}(\omega)$. 
\par Recall the complexified Lie algebra of $H$ is $\lieh_{\CC} = \liek_{H, \CC} \oplus \liep_{H, \CC}^{+} \oplus \liep_{H, \CC}^{-}$. Let 
\begin{equation}
\label{eq: def of v^{+} and v^{-}}
    v^{+} = \begin{pmatrix}
                0 & 1 \\
                0 & 0 
            \end{pmatrix} \in \liep_{H, \CC}^{+},\ v^{-} = \begin{pmatrix}
                                                                    0 & 0 \\
                                                                    1 & 0
                                                            \end{pmatrix} \in \liep_{H, \CC}^{-} . 
\end{equation}
be two root vectors in $\liegl_{2, \CC}$. So under the tangent map $\iota_{*}$ of the map $\iota$ of Shimura varieties, we have 
\begin{equation*}
    \iota_{*}(v^{+}) = X_{(1, 0, -1)}, \  \iota_{*}(v^{-}) = X_{(-1, 0, 1)}, 
\end{equation*}
so
\begin{equation*}
    \iota_{*} (v^{+} \otimes v^{-}) = \iota_{*}v^{+} \otimes \iota_{*} v^{-} = X_{(1, 0, -1)} \otimes X_{(-1, 0, 1)}.
\end{equation*}

\begin{lemma}
\label{Lemma: pushforward of v}
    If we set
    \begin{equation*}
        Y_{(1, -1, 0)}:= \mathrm{ad}_{X_{(-1, 1, 0)}}(X_{(1, 0, -1)} \otimes X_{(0, -1, 1)}). 
    \end{equation*} and $Y_{(0, 0, 0)}$ is a vector in $\tau_{(0, 0, 0)}$, then we have 
    \begin{equation*}
        \iota_{*} (v^{+} \otimes v^{-}) = -\frac{1}{2} Y_{(1, -1, 0)} + \beta Y_{(0, 0, 0)}
    \end{equation*}
    for some $\beta \in \CC$.
\end{lemma}

\begin{proof}
    First, it follows from the fact that $X_{(1, 0, -1)} \otimes X_{(0, -1, 1)}$ is a highest weight vector of $\tau_{(1, -1, 0)}$ as a $K_G$-representation that 
    $Y_{(1, -1, 0)}$ is vector in $\tau_{(1, -1, 0)}$ with weight $(0, 0, 0)$. Hence, we could let 
    \begin{equation*}
        \iota_{*} (v^{+} \otimes v^{-}) = \alpha Y_{(1, -1, 0)} + \beta Y_{(0, 0, 0)}
    \end{equation*}
    for some $\alpha, \beta \in \CC$, from which we have 
    \begin{equation}
    \label{eqn:claim_proof_pullback_form}
        \mathrm{ad_{X_{(1, -1, 0)}}}( \iota_{*} (v^{+} \otimes v^{-})) = \alpha \mathrm{ad_{X_{(1, -1, 0)}}}(\mathrm{ad}_{X_{(-1, 1, 0)}}(X_{(1, 0, -1)} \otimes X_{(0, -1, 1)})). 
    \end{equation}
    \par Second, we can compute the left hand side of the formula (\ref{eqn:claim_proof_pullback_form}) aså 
    \begin{align*}
        \mathrm{ad_{X_{(1, -1, 0)}}}( \iota_{*} (v^{+} \otimes v^{-})) & = \mathrm{ad_{X_{(1, -1, 0)}}}(X_{(1, 0, -1)} \otimes X_{(-1, 0, 1)}) \\
                                                                   & = X_{(1, 0, -1)} \otimes \mathrm{ad_{X_{(1, -1, 0)}}}(X_{(-1, 0, 1)}) \\
                                                                   & = X_{(1, 0, -1)} \otimes [X_{(1, -1, 0)}, X_{(-1, 0, 1)}] \\
                                                                   & = - X_{(1, 0, -1)} \otimes X_{(0, -1, 1)}. 
    \end{align*}
    \par Third, if we let 
    \begin{equation*}
        H := [X_{(1, -1, 0)}, X_{(-1, 1, 0)}] = \left(\begin{matrix}
                                                    1 & 0 & 0 \\
                                                    0 & -1 & 0 \\
                                                    0 & 0 & 0 
                                                \end{matrix}\right), 
    \end{equation*}
    then we have 
    \begin{align*}
        \alpha \mathrm{ad_{X_{(1, -1, 0)}}}(\mathrm{ad}_{X_{(-1, 1, 0)}}(X_{(1, 0, -1)} \otimes X_{(0, -1, 1)})) &= \alpha \mathrm{ad_{X_{(-1, 1, 0)}}}(\mathrm{ad}_{X_{(1, -1, 0)}}(X_{(1, 0, -1)} \otimes X_{(0, -1, 1)})) \\
        & + \alpha \mathrm{ad_{[X_{(1, -1, 0)}, X_{(-1, 1, 0)}]}}(X_{(1, 0, -1)} \otimes X_{(0, -1, 1)}) \\
        & = 0 + \alpha \mathrm{ad}_{H}(X_{(1, 0, -1)} \otimes X_{(0, -1, 1)}) \\
        & = \alpha X_{(1, 0, -1)} \otimes X_{(0, -1, 1)} + \alpha X_{(1, 0, -1)} \otimes X_{(0, -1, 1)} \\
        & = 2\alpha X_{(1, 0, -1)} \otimes X_{(0, -1, 1)}. 
    \end{align*}
    Finally, by equation (\ref{eqn:claim_proof_pullback_form}), we have 
    \begin{equation*}
        - X_{(1, 0, -1)} \otimes X_{(0, -1, 1)} = 2\alpha X_{(1, 0, -1)} \otimes X_{(0, -1, 1)}.
    \end{equation*}
    Hence, we conclude that $\alpha = -\frac{1}{2}$. \\
\end{proof}

\begin{lemma}
\label{Lemma: lower and up vector}
For $0 \le i \le a + b$, we have the following identities: 
\begin{enumerate} 
        \item $X_{(-1, 1, 0)}X^{2 + a + b - i}_{(1, -1, 0)} \Psi_{\infty}  = (1 + i)(2 + a + b - i)X_{(1, -1, 0)}^{1 + a + b - i} \Psi_{\infty}$, 
        \item $X_{(-1, 1, 0)} X_{(1, -1 ,0)}^{i} v = \begin{cases}
                                                            i(a + b + 1 - i) X^{i - 1}_{(1, -1, 0)} v & i > 0 \\
                                                            0 & i = 0
                                                          \end{cases}. $
\end{enumerate}    
\end{lemma}

\begin{proof}
    Since $\Psi_{\infty}$ is the lowest weight vector with weight $(-1 - b + r - s, 1 + a + r - s, r - s; b + 2s -r)$ in the minimal $K_G$-type $\tau_{(1 + a + r -s, -1-b +r -s, r-s)}$ of $\tilde{\pi}_{\infty}$, which is the vector $\nu_{0}$ in the representation $\tau_{(1 + a + r -s, -1-b +r -s, r-s)} $ in our setup of representations of $K_{G}$ (see equation (\ref{eqn: action on K-type})), 
        we have 
        \begin{align*}
            X_{(-1, 1, 0)} X^{2 + a + b - i}_{(1, -1, 0)} \Psi_{\infty}  &= X_{(-1, 1, 0)} X^{2 + a + b - i}_{(1, -1, 0)} \nu_0 \\
            & = (2 + a + b - i)! X_{(-1, 1, 0)} \nu_{2 + a + b - i}. \\
            & = (1 + i)(2 + a + b - i)! \nu_{1 + a + b - i}. 
        \end{align*}
        We also have
        \begin{equation*}
            X^{1 + a + b - i}_{(1, -1, 0)} \Psi_{\infty}  = (1 + a + b - i)! \nu_{1 + a + b - i}. 
        \end{equation*}
        Hence, we get
        \begin{equation*}
            X_{(-1, 1, 0)} X^{2 + a + b - i}_{(1, -1, 0)} \Psi_{\infty}  = (1 + i)(2 + a + b - i) X^{1 + a + b - i}_{(1, -1, 0)} \Psi_{\infty}. 
        \end{equation*}
        The computation of $X_{(-1, 1, 0)} X_{(1, -1 ,0)}^{i} v$ is similar, which we finishes the proof of the lemma. \\
\end{proof}

\begin{proposition}
\label{Prop:pullback_form}
    If we pullback the $\omega$ defined in Proposition \ref{Prop: explicit construnction of the differential form} along the map $\iota$, then we have the following formula:
    \begin{equation*}
        \iota^{*}\omega(v_{+} \otimes v_{-}) = - \sum_{i = 0}^{a + b} (-1)^{i}(i + 1) X_{(1, -1 ,0)}^{i}v \otimes X^{1 + a + b - i}_{(1, -1 , 0)} \Psi_{\infty}. 
    \end{equation*}
\end{proposition}

\begin{proof}
    The pullback functor $\iota^{*}$ preserves the type of the differential form and the type of a differential form on a Shimura variety is determined by the archimedean part, $\iota^{*}\omega$ is determined by its value on $v^{+} \otimes v^{-}$. 
    It follows from Lemma \ref{Lemma: pushforward of v} that
    \begin{align*}
        \iota^{*}\omega(v^{+} \otimes v^{-}) &= \omega(\iota_{*}(v^{+} \otimes v^{-})) \\
                                             &= \omega(-\frac{1}{2}\mathrm{ad}_{X_{(-1, 1, 0)}}(X_{(1, 0, -1)} \otimes X_{(0, -1, 1)})) \\
                                             & = -\frac{1}{2} X_{(-1, 1, 0)} (\omega(X_{(1, 0, -1)} \otimes X_{(0, -1, 1)})) \\
                                             & = -\frac{1}{2} X_{(-1, 1, 0)}(\sum\limits_{i = 0}^{a + b} (-1)^{i} X_{(1, -1 ,0)}^{i} v \otimes X^{2 + a + b - i}_{(1, -1, 0)} \Psi_{\infty}) \\
                                             & = -\frac{1}{2} \sum\limits_{i = 0}^{a + b} (-1)^{i} X_{(-1, 1, 0)} X_{(1, -1 ,0)}^{i} v \otimes X^{2 + a + b - i}_{(1, -1, 0)} \Psi_{\infty} \\
                                             & -\frac{1}{2} \sum\limits_{i = 0}^{a + b} (-1)^{i} X_{(1, -1 ,0)}^{i} v \otimes X_{(-1, 1, 0)}X^{2 + a + b - i}_{(1, -1, 0)} \Psi_{\infty}. 
    \end{align*} 
    Plugging the formula in Lemma \ref{Lemma: lower and up vector} into the formula for $\iota^{*}\omega(v^{+} \otimes v^{-})$ above, we get 
    \begin{align*}
        \iota^{*}\omega(v^{+} \otimes v^{-}) & = -\frac{1}{2} \sum\limits_{i = 1}^{a + b} (-1)^{i}i(a + b + 1 - i) X_{(1, -1 ,0)}^{i - 1} v \otimes X^{2 + a + b - i}_{(1, -1, 0)} \Psi_{\infty} \\
                                             & -\frac{1}{2} \sum\limits_{i = 0}^{a + b} (-1)^{i} (i + 1)(2 + a + b - i) X_{(1, -1 ,0)}^{i} v \otimes X^{1 + a + b - i}_{(1, -1, 0)} \Psi_{\infty} \otimes \\
                                             & = \frac{1}{2} \sum\limits_{i = 0}^{a + b - 1} (-1)^{i}(i + 1)(a + b - i) X_{(1, -1 ,0)}^{i} v \otimes X^{1 + a + b - i}_{(1, -1, 0)} \Psi_{\infty} \\
                                             & -\frac{1}{2} \sum\limits_{i = 0}^{a + b} (-1)^{i} (i + 1)(2 + a + b - i) X_{(1, -1 ,0)}^{i} v \otimes X^{1 + a + b - i}_{(1, -1, 0)} \Psi_{\infty} \\
                                             & = - \sum\limits_{i = 0}^{a + b} (-1)^{i}(i + 1) X_{(1, -1 ,0)}^{i} v \otimes X^{1 + a + b - i}_{(1, -1, 0)} \Psi_{\infty}. 
    \end{align*}
\end{proof}

\subsection{Deligne{\textendash}Beilinson cohomology and tempered currents}
 \label{SS: DB cohomology}
 In this subsection, we collect some facts about Deligne{\textendash}Beilinson cohomology and tempered currents, mainly following \cite{BCLRJ24}. 

\subsubsection{Classical Definition of Deligne{\textendash}Beilinson Cohomology}
We first recall the classical definition of Deligne{\textendash}Beilinson Cohomology (DB-cohomology for short). 

\par Let $X$ be a smooth, quasi-projective complex analytic variety of pure dimension $d$. Let $\overline{X}$ be a smooth compactification of $X$ such that $D = \overline{X} - X$ is a simple normal crossings divisor. We let $j \colon X \rightarrow \overline{X}$ be the open immersion. We will assume that $X$ is defined as the analytification of the base change to $\CC$ of a smooth, quasi-projective $\RR$-scheme. For $p \in \ZZ$, let $\RR(p)$ denote the subgroup  $(2\pi i)^{p} \RR$ of $\CC$. We will also use the same notation to denote the constant sheaf with value $\RR(p)$ on $X$. 

\par Let $\Omega^{*}_{X}$ be the sheaf of holomorphic differential forms on $X$ and let $\Omega_{\overline{X}}^{*}(\mathrm{log}\, D)$ be the sheaf of differential forms on $X$ with logarithmic singularities along $D$.  It is defined to be the sheaf of sub-$\mathcal{O}_{\overline{X}}$-module of $j_{*} \Omega_{X}$ locally generated by sections $\xi_{i}, \, 1 \le i \le d$ (see \cite[(3.1.2)]{HodgeII}). Here, we let 
\begin{equation*}
    \xi_{i} = \begin{cases}
                \frac{dz{i}}{z_{i}} &  1 \le i \le m \\
                dz_{i} & m < i \le d. 
               \end{cases}
\end{equation*}
The Hodge filtration  on $\Omega_{\overline{X}}^{*}(\mathrm{log}\, D)$ is defined as 
\begin{equation*}
    F^{p}\Omega_{\overline{X}}^{*}(\mathrm{log}\, D) = \bigoplus_{p^{\prime} \ge p}\Omega_{\overline{X}}^{p^{\prime}}(\mathrm{log}\, D). 
\end{equation*}
\begin{fact}
\label{Fact: Betti to log_de Rham}
    There are natural quasi-isomorphisms: 
    \begin{enumerate}
        \item  $Rj_{*} \CC \rightarrow Rj_{*} \Omega_{X}^{*}$,
        \item  $\Omega^{*}_{\overline{X}}(\mathrm{log}\, D) \rightarrow Rj_{*}\Omega_{X}^{*}$ \cite[(3.1.8.2)]{HodgeII}. 
    \end{enumerate}
\end{fact}

\begin{definition}
   \begin{enumerate}
       \item For each $p \in \ZZ$, the \textit{DB-cohomology group} $\mathrm{H}_{D}^{n}(X, \RR(p))$ of $X$ with coefficients in $\RR(p)$ is defined as the $n$-th hypercohomology of the complex 
       \begin{equation}
           \RR(p)_{D} := \mathrm{Cone}(Rj_{*}\RR(p) \oplus F^{p}\Omega_{\overline{X}}^{*}(\mathrm{log}\, D) \rightarrow Rj_{*} \Omega_{X}^{*})[-1], 
       \end{equation}
       where the arrow is given by the difference of the natural maps. 
       \item Let $\overline{F^{*}_{\infty}} = F_{\infty}^{*} \otimes c$ be the de Rham involution given by the complex conjugation $F_{\infty}$ on $X$ and $c$ on the coefficients. The \textit{real DB-cohomology groups} are defined as 
       \begin{equation}
           \mathrm{H}_{D}^{n}(X/\RR, \RR(p)) := \mathrm{H}_{D}^{n}(X, \RR(p))^{\overline{F_{\infty}^{*}} = 1}. 
       \end{equation}
   \end{enumerate}
\end{definition}    

\subsubsection{Tempered Currents}
We recall the definition of tempered currents in \cite{BCLRJ24} here. 

Let $\Delta = \{ z \in \CC | |z| < \frac{1}{2} \}$ be the open disc of radius $\frac{1}{2}$ and $\Delta^{*} = \{z \in \CC | 0 < |z| < \frac{1}{2} \}$ be the open punctured disc of radius $r$. 

\par For each point $x \in \overline{X}$, there is an open neighborhood $U$ of $x$ that is isomorphic to $\Delta^{d}$ with coordinates $(z_1, z_2, \ldots, z_{d})$ for which $x = (0, 0, \ldots, 0)$, and such that there exist some integers $m, n$ with $0 \le n, m \le d$ and $m + n = d$ such that 
\begin{equation*}
    X \cap U = (\Delta^{*})^{m} \times (\Delta)^{n} = \{ (z_1, z_2, \cdots, z_{d}) \in \Delta^{d} | z_1 z_2 \cdots z_{m} \neq 0 \}.
\end{equation*}
We call such a subset $U$ an open coordinate neighborhood of $x$. 

\par We denote by $\mathcal{A}_{\overline{X}}^{0}$ the sheaf of smooth functions on $\overline{X}$ and let $\mathcal{A}^{*}_{\overline{X}}$ be the complex of sheaves of smooth differential forms. The complex $\mathcal{A}^{*}_{\overline{X}}(\mathrm{log}\, D)$ of smooth differential forms on $X$ with logarithmic growth at $D$ is defined to be the $\mathcal{A}_{\overline{X}}$-algebra subsheaf of $j_{*} \mathcal{A}_{X}$ locally generated by sections $\log|z_{i}|\, (1 \le i \le m)$ and $\xi_{i}, \bar{\xi_{i}} \, (1 \le i \le d)$. \cite[\S 2]{Burgos94} 

\par We first give the definition of sheaves of differential forms with growth conditions. 
\begin{definition}
    We denote by $\mathcal{A}_{si}^{0}$ (resp., $\mathcal{A}_{rd}^{0}$) the sheaf on $\overline{X}$ whose sections on an open $U \subset \overline{X}$ is given by complex-valued functions on $U \cap X$, which are locally slowly increasing (resp. rapidly decreasing) \cite[Definition 2.3]{BCLRJ24} at each point of $U$. The graded sheaf $\mathcal{A}_{si}^{*}$ of slowly increasing differential forms is defined to be the $\mathcal{A}_{si}^{0}$-subalgebra of $j_{*} \mathcal{A}_{X}^{*}$ locally generated by $\xi_{i}, \bar{\xi_{i}}$ for $1 \le i \le d$.  Similarly, the graded sheaf $\mathcal{A}_{rd}^{*}$ of rapidly decreasing differential forms is defined to be the $\mathcal{A}_{rd}^{0}$-subalgebra of $j_{*}\mathcal{A}_{X}^{*}$ locally generated by $\xi_{i}, \bar{\xi_{i}}$ for $1 \le i \le d$. We denote by $\mathcal{A}_{si}(\overline{X})$ and $\mathcal{A}_{rd}(\overline{X})$ the corresponding complex of global sections. 
\end{definition}

\begin{remark}
    \begin{enumerate}
        \item There are natural inclusions 
        \begin{equation*}
            \mathcal{A}_{rd}^{*} \subseteq \mathcal{A}_{\overline{X}}^{*}(\log D) \subseteq \mathcal{A}_{si}^{*} \subseteq j_{*} \mathcal{A}^{*}_{X}. 
        \end{equation*}
        Moreover, all these sheaves are fine sheaves. 
        \item The complex structure on $\overline{X}$ induces compatible bigradings
        \begin{equation*}
            \mathcal{A}_{rd}^{n} = \bigoplus_{p + q = n} \mathcal{A}_{rd}^{p,q}, \,  \mathcal{A}_{si}^{n} = \bigoplus_{p + q = n} \mathcal{A}_{si}^{p,q}, \, \mathcal{A}_{\overline{X}}^{n}(\log \, D) = \bigoplus_{p + q = n} \mathcal{A}_{\overline{X}}^{p, q}(\log D), 
        \end{equation*}
        with correspoding Hodge filtrations 
        \begin{equation*}
            F^{p}\mathcal{A}_{rd}^{n} = \bigoplus_{p^{\prime} \ge p} \mathcal{A}_{rd}^{p^{\prime},q}, F^{p}\mathcal{A}_{si}^{n} = \bigoplus_{p^{\prime} \ge p} \mathcal{A}_{si}^{p^{\prime},q}, 
            F^{p}\mathcal{A}_{\overline{X}}^{n} (\log \, D) = \bigoplus_{p^{\prime} \ge p} \mathcal{A}_{\overline{X}}^{p^{\prime}, q} (\log \, D). 
        \end{equation*}
        \item We denote by 
        \begin{equation*}
            \mathcal{A}_{si, \RR}^{*} \subseteq \mathcal{A}_{si}^{*},  \mathcal{A}_{rd, \RR}^{*} \subseteq \mathcal{A}_{rd}^{*}, \mathcal{A}_{\overline{X}, \RR}^{*}(\log \, D) \subseteq \mathcal{A}_{\overline{X}}^{*}(\log \, D)
        \end{equation*}
        the subcomplexes of sheaves of $\RR$-valued differential forms. 
    \end{enumerate}
\end{remark}

    We now give the definition of sheaves of tempered currents. 

\begin{definition}
        \begin{enumerate}
            \item For any $0 \le p,q \le d$, the sheaf $\mathcal{D}_{si}^{p,q}$ of tempered currents is defined to be the sheaf on $\overline{X}$ assigning to any open coordinate neighborhood $U \subseteq \overline{X}$ the complex vector space $\mathcal{D}^{p,q}_{si}(U)$ of continuous linear forms on the compactly supported sections $\Gamma_{c}(U, \mathcal{A}_{rd}^{d - p, d - q})$. Similarly,  We denote by $\mathcal{D}_{si, \RR}^{p,q}$ the sheaf on $\overline{X}$ assigning to any open coordinate neighborhood $U \subseteq \overline{X}$ the real vector space $\mathcal{D}_{si, \RR}^{p,q}(U)$ of continuous linear forms on the compactly supported sections $\Gamma_{c}(U, \mathcal{A}_{rd, \RR}^{d - p, d -q})$. 
            \item For $n \in \ZZ_{\ge 0}$, let 
            \begin{equation*}
                \mathcal{D}_{si}^{n} = \bigoplus_{p + q = n} \mathcal{D}_{si}^{p, q},  \mathcal{D}_{si, \RR}^{n} = \bigoplus_{p + q = n} \mathcal{D}_{si, \RR}^{p, q}. 
            \end{equation*}
            The exterior differentials $\partial$ and $\overline{\partial}$ on $\mathcal{D}_{si}^{*}$ are defined as follows: for any open coordinate neighborhood $U \subseteq \overline{X}$, any $T \in \mathcal{D}^{p, q}_{si}(U)$ and any $\omega \in \Gamma_{c}(U, \mathcal{A}_{rd}^{d - p, d - q})$, we have 
            \begin{equation*}
                \partial T(w) := (-1)^{p + q + 1} T(\partial w),  \overline{\partial} T(w) := (-1)^{p + q + 1} T(\overline{\partial} w). 
            \end{equation*}
            The exterior differential $d \colon \mathcal{D}_{si}^{n} \rightarrow \mathcal{D}_{si}^{n + 1}$ is defined as $d = \partial + \overline{\partial}$. This defines complexes of sheaves of tempered currents $\mathcal{D}^{*}_{si}, \mathcal{D}_{si}^{*, *}$ (with differential $d, \partial, \overline{\partial}$). We will denote by $\mathcal{D}_{si}^{*, *}(\overline{X})$ the corresponding complex of global sections. 
            \item The complex $D^{*}_{si}$ is equipped with a Hodge filtration given by 
            \begin{equation*}
                F^{p}D^{*}_{si} = \bigoplus_{p^{\prime} \ge p} \mathcal{D}_{si}^{p^{\prime}, q}. 
            \end{equation*}
            \item For any open $U \subseteq \overline{X}$, there is a natural way to associate any form $\omega \in \mathcal{A}_{si}^{p, q}(U)$ a current $T_{\omega} \in \mathcal{D}_{si}^{p, q}(U)$ given by 
            \begin{equation*}
                T_{\omega}(\eta) = \frac{1}{(2\pi i)^{d}} \int_{U} \omega \wedge \eta , \, (\eta \in \Gamma_{c}(U, \mathcal{A}_{rd}^{d - p, d - q})). 
            \end{equation*}
        \end{enumerate}
\end{definition}

\begin{remark}
    Since currents are modules over $\mathcal{A}^{0}_{si}$, all these complexes are complexes of fine sheaves. 
\end{remark}

\begin{proposition} \textnormal{\cite[Theorem 1.1]{BCLRJ24}}
\label{Prop: BGCLRJ24 Theoreom 1.1}
    The natural inclusions 
    \begin{equation*}
        (\Omega^{*}_{\overline{X}}(\log\, D), F) \rightarrow (\mathcal{A}^{*}_{\overline{X}} (\log\, D), F) \rightarrow (\mathcal{A}_{si}^{*}, F) \rightarrow (\mathcal{D}^{*}_{si}, F)
    \end{equation*}
    are filtered quasi-isomorphisms. Futhermore, the last two quasi-isomorphisms are compatible for the corresponding real structures. 
\end{proposition}

\begin{remark}
\label{remark: Betti coh and tempered currents}
    By Proposition \ref{Prop: BGCLRJ24 Theoreom 1.1}, Fact \ref{Fact: Betti to log_de Rham} and, the fact that $\mathcal{D}_{si}^{*}$ is a fine sheaf, we can see that Betti cohomology classes of $X$ can be represented by closed tempered currents. 
\end{remark}

\subsubsection{Deligne{\textendash}Beilinson cohomology in terms of tempered currents}
In \cite{BCLRJ24}, the authors give a new definition of Deligne{\textendash}Beilinson cohomology in terms of smooth slowly increasing differential forms and tempered currents, which will be recalled here. This definition is convenient for computing Beilinson regulators. 

\begin{proposition}
\label{Prop: DB-coh_si forms}
    There is a quasi-isomorphism 
    \begin{equation*}
        \RR(p)_{D} \simeq \mathrm{Cone}(F^{p}\mathcal{A}^{*}_{si} \rightarrow \mathcal{A}_{si, \RR(p-1)}^{*})[-1],
    \end{equation*}
    where the arrow is induced by the projection $\pi_{p-1} \colon \CC \rightarrow \RR(p - 1)$ defined by $\pi_{p - 1}(z) = \frac{z + (-1)^{p-1} \overline{z}}{2}$. In particular, we have the canonical isomorphism 
    \begin{equation}
        \mathrm{H}_{D}^{n}(X, \RR(p)) \simeq \frac{\{ (\phi, \phi^{\prime}) \in  \mathcal{A}_{si, \RR(p-1)}^{n - 1}(\overline{X}) \oplus F^{p} \mathcal{A}_{si}^{n}(\overline{X}) | d \phi^{\prime} = 0, d \phi = \pi_{p - 1}(\phi^{\prime})  \}}{ \{ d(\tilde{\phi}, \tilde{\phi^{\prime}})\}}, 
    \end{equation}
    where $d(\tilde{\phi}, \tilde{\phi^{\prime}}) = (d \tilde{\phi} - \pi_{p -1}(\tilde{\phi^{\prime}}), d\tilde{\phi^{\prime}})$. 
\end{proposition}

\begin{proof}
    See \cite[Proposition 2.23]{BCLRJ24}. 
\end{proof}

\begin{proposition}
\label{Prop: DB-coh_temperd currents}
    There is a quasi-isomorphism 
    \begin{equation*}
        \RR(p)_{D} \simeq \mathrm{Cone}(F^{p}\mathcal{D}^{*}_{si} \rightarrow \mathcal{D}_{si, \RR(p-1)}^{*})[-1].
    \end{equation*}
    In particular, we have the canonical isomorphism 
    \begin{equation}
        \mathrm{H}_{D}^{n}(X, \RR(p)) \simeq \frac{\{ (S, T) \in  \mathcal{D}_{si, \RR(p-1)}^{n - 1}(\overline{X}) \oplus F^{p} \mathcal{D}_{si}^{n}(\overline{X}) | d T = 0, d S = \pi_{p - 1}(T)  \}}{ \{ d(\tilde{S}, \tilde{T})\}}, 
    \end{equation}
    where $d(\tilde{S}, \tilde{T}) = (d \tilde{S} - \pi_{p -1}(\tilde{T}), d\tilde{T})$. 
\end{proposition}

\begin{proof}
    See \cite[Theorem 2.25]{BCLRJ24}. 
\end{proof}

In what follows, we will use $[(S, T)] \in \mathrm{H}_{D}^{n}(X, \RR(p))$ to denote the cohomology class of the pair $(S, T)$.

\begin{proposition}
\label{Prop: DB-cohomology-si forms to tempered currents}
    Let $x \in  \mathrm{H}_{D}^{n}(X, \RR(p))$ be a DB-cohomology class which is represented by a pair $(\phi, \phi^{\prime})$ of smooth slowly increasing differential forms via Proposition \ref{Prop: DB-coh_si forms}. Then via the isomorphism of Proposition \ref{Prop: DB-coh_temperd currents}, the class $x$ is represented by the pairs of tempered currents $(T_{\phi}, T_{\phi^{\prime}})$. 
\end{proposition} 

\begin{proof}
    See \cite[Proposition 2.26]{BCLRJ24}. 
\end{proof}

The reinterpretation of DB-cohomology classes in terms of tempered currents allows us to describe explicitly the Gysin morphism as follows. 
Let $\iota \colon X^{\prime} \hookrightarrow X$ be a closed immersion of pure codimension $c$. Let $\overline{X^{\prime}}$ be a smooth compactification of $X^{\prime}$ such that $D^{\prime} = \overline{X^{\prime}} - X^{\prime}$ is a simple normal crossings divisor. We further assume that $\iota$ extends to a morphism $\iota \colon \overline{X^{\prime}} \rightarrow \overline{X}$ such that $\iota^{-1}(D) = D^{\prime}$. Then we define the Gysin morphism 
\begin{equation}
\label{eq: Gysin for DB-cohomology}
    \iota_{*} \colon \mathrm{H}_{D}^{n}(X^{\prime}, \RR(p)) \rightarrow \mathrm{H}_{D}^{n + 2c}(X, \RR(p + c))
\end{equation}
by $\iota_{*}[(S, T)] = [(\iota_{*}S, \iota_{*}T)]$. 
Here for a tempered current $T$, we denote by $\iota_{*}T$ the tempered current defined as: for any smooth rapidly decreasing differential form $\omega$,
\begin{equation}
\label{eq: pushforward of tempered currents}
    \iota_{*}T(\omega) = T(\iota^{*} \omega), 
\end{equation}
which makes sense because $\iota^{*} \omega$ is also rapidly decreasing. 

\begin{proposition}
    Let $n \in \ZZ_{\ge 0}$, $p \in \ZZ$ and let $\omega \in \mathcal{A}_{rd}^{2d - n}(\overline{X})$ be a smooth closed rapidly decreasing  differential form with Hodge type components inside $\{(a, b): a, b > d - p\}$. Then the assignment $(S, T) \mapsto S(\omega)$ induces a linear map 
    \begin{equation*}
        \langle -, \omega \rangle \colon \mathrm{H}_{D}^{n}(X^{\prime}, \RR(p)) \rightarrow \CC. 
    \end{equation*}
\end{proposition}

\begin{proof}
    See \cite[Proposition 2.27]{BCLRJ24}
\end{proof}

\subsection{DB-cohomology with coefficients}
\label{SS: DB-cohomology with coeffients}
In this subsection, we give the definition of DB-cohomology with coefficients using the ``Liebermann’s trick'':

Let $p \colon A \rightarrow S$ be the universal abelian $3$-fold over $S$. 

\begin{proposition}
\label{Prop:motiv_coh_abs_relative}
    We have a inclusion. 
    \begin{equation}
        \mathrm{H}^{3}_{H}(S, V(2)) \subset \mathrm{H}_{H}^{a + 5b - 3(r+s) + 3}(A^{a + 5b - 3(r+s)}, \RR(a + 3b -r - s + 2)). 
    \end{equation}
\end{proposition}

\begin{proof}
    Since $V = V^{a, b}\{a, b \} = \lambda(a + r - s, r - s,  r - s - b; 2s - r + b)$ is in $\mathrm{Rep}(G)$, we have   
    \begin{equation*}
        V \subset (\mathrm{std})^{\otimes(a + 5b - 3(r+s))}(2s + 2r -2b),
    \end{equation*}
    where $\mathrm{std}$ is the standard representation of $G$. We then get the inclusion
    \begin{equation}
    \label{eq: inclusion of sheaf in DB with coefficent}
        V \subset \mathcal{H}^{a + 5b - 3(r + s)}p_{*}\RR(a + 3b -r - s)_{A^{a + 5b - 3(r + s)}}
    \end{equation}
    in $\mathrm{D}^{b}(\mathrm{MHM}_{\RR}(S/\RR))$ by applying the functor $\mu_{H}$ to the inclusion (\ref{eq: inclusion of sheaf in DB with coefficent}). 
    Recall that for the morphism $p \colon A^{a + 5b - 3(r+s)} \rightarrow S$, we have the following isomorphism in $\mathrm{D}^{b}(\mathrm{MHM}_{\RR}(S/\RR))$ \cite[(4.5.4)]{MHM90}: 
    \begin{equation*}
        p_{*}\RR(a + 3b -r - s) \cong \bigoplus\limits_{i} \mathcal{H}^{i}p_{*} \RR(a + 3b -r - s)_{A^{a + 5b - 3(r + s)}}[-i]. 
    \end{equation*}
    Thus, we can see that
    \begin{equation*}
        \mathcal{H}^{a + 5b - 3(r + s)} p_{*}\RR(a + 3b - r - s)_{A^{a + 5b - 3(r + s)}} \subset p_{*}\RR(a + 3b - r - s)[a + 5b - 3(r+s)]. 
    \end{equation*}
    Finally, if we let $\mathcal{A} = A^{a + 5b - 3(r+s)}$, then it can be seen from the definition of absolute Hodge cohomology that  
    \begin{align*}
        \mathrm{H}^{3}_{H}(S, V(2)) &= \Hom_{\mathrm{D}^{b}(\mathrm{MHM}_{\RR}(S/\RR))}(\RR(0)_{S}, V(2)[3]) \\
                                    & \subset \Hom_{\mathrm{MHM}_{\RR}(S/\RR))}(\RR(0)_{S}, p_{*}\RR(a + 3b -r - s + 2)[a + 5b - 3(r+s) + 3]) \\ 
                                    & = \Hom_{\mathrm{MHM}_{\RR}(\mathcal{A}/\RR))}(\RR(0)_{\mathcal{A}}, \RR(a + 3b -r - s + 2)[a + 5b - 3(r+s) + 3]) \\
                                    & = \mathrm{H}_{H}^{a + 5b - 3(r+s) + 3}(\mathcal{A}, \RR(a + 3b - r -s + 2)). 
    \end{align*}
\end{proof}

\begin{definition}
\label{Def: DB-coh with coefficients}
    Let $i = a + 5b - 3(r + s)$ and $j = a + 3b -r - s + 2$. 
    \begin{enumerate}
        \item We define $\mathrm{H}^{3}_{D}(S, V(2))$ as 
              \begin{equation*}
                    \mathrm{H}^{3}_{D}(S, V(2)) := r_{H \rightarrow D} (\mathrm{H}^{3}_{H}(S, V(2))), 
              \end{equation*}
               where $r_{H \rightarrow D}$ is the natural map 
               \begin{equation*}
                   r_{H \rightarrow D} \colon \mathrm{H}_{H}^{i + 3}(A^{i}, \RR(j)) \rightarrow \mathrm{H}_{D}^{i + 3}(A^{i}, \RR(j))
               \end{equation*}
               and $A$ is the universal abelian $3$-fold over $S$;
        \item We define $\mathrm{H}^{1}_{D}(M, \iota^{*}V(1))$ as 
              \begin{equation*}
                    \mathrm{H}^{1}_{D}(M, \iota^{*}V(1)) := r_{H \rightarrow D} (\mathrm{H}^{1}_{H}(M, \iota^{*}V(1))), 
              \end{equation*}
               where $r_{H \rightarrow D}$ is the natural map 
               \begin{equation*}
                   r_{H \rightarrow D} \colon \mathrm{H}_{H}^{i + 1}(A^{i}, \RR(j - 1)) \rightarrow \mathrm{H}_{D}^{i + 1}(A^{i}, \RR(j - 1))
               \end{equation*}
               and $A$ is the pullback of the universal abelian $3$-fold over $S$ to $M$ through $\iota$;
        \item We define $\mathrm{H}^{1}_{D}(M, W(1))$ as 
              \begin{equation*}
                    \mathrm{H}^{1}_{D}(M, W(1)) := r_{H \rightarrow D}(\mathrm{H}_{H}(M, W(1))), 
              \end{equation*}
              where $r_{H \rightarrow D}$ is the map 
              \begin{equation*}
                  r_{H \rightarrow D} \colon \mathrm{H}^{n + 1}_{H} (E^{n}, \RR(n + 1)) \rightarrow \mathrm{H}^{n + 1}_{D} (E^{n}, \RR(n + 1))
              \end{equation*}
              and $E$ is the universal elliptic curve over $M$; 
        \item We define $\mathrm{H}^{1}_{D}(\Sh_{\GL_2}, W(1))$ as
              \begin{equation*}
                    \mathrm{H}^{1}_{D}(\Sh_{\GL_2}, W(1)) := r_{H \rightarrow D}(\mathrm{H}_{H}(\Sh_{\GL_2}, W(1))), 
              \end{equation*}
              where $r_{H \rightarrow D}$ is the map 
              \begin{equation*}
                  r_{H \rightarrow D} \colon \mathrm{H}^{n + 1}_{H} (E^{n}, \RR(n + 1)) \rightarrow \mathrm{H}^{n + 1}_{D} (E^{n}, \RR(n + 1))
              \end{equation*}
              and $E$ is the universal elliptic curve over $\Sh_{\GL_2}$.
    \end{enumerate}
\end{definition}

\begin{remark}
    \begin{itemize}
        \item This definition is enough for our application to Beilinson's conjectures and can be easily generalized to all Shimura varieties of PEL type. However, we do not give a definition of ``DB-cohomology with coefficients'' for general coefficients over general analytic varieties. 
        \item It follows from \cite[Page 60]{Bei83} that the maps $r_{H \rightarrow D}$ in (3) and (4) are isomorphic, but the maps $r_{H \rightarrow D}$ in (1) and (2) may not be isomorphic. 
    \end{itemize}
\end{remark}

\begin{convention}
    Since we only work with real DB-cohomology in this paper, from now on, we will write $\mathrm{H}_{D}^{n}(X, \RR(p))$ for 
    $\mathrm{H}_{D}^{n}(X/\RR, \RR(p))$. 
\end{convention}

 \subsection{Deligne-Beilinson classes}
 \label{SS: Eis symbol and Hodge realization}
 In this subsection, we compute the realization of the motivic classes constructed in Subsection \ref{SS: motivic class} in Deligne{\textendash}Beilinson cohomology. 

\subsubsection{Eisenstein symbol in DB-cohomology}

We first compute the realization of Eisenstein symbols (see Subsubsection \ref{SSS: Eisenstein symbol}) in Deligne-Beilinson cohomology. We follow \cite[6.3]{Kings98} closely but we change the setting to use the Hermitian form $J_{2}^{\prime}$.  

\begin{convention}
    Recall that when we write $\GL_2(\RR)$, we always mean $\GU(J_{2}^{\prime})^{\prime}(\RR)$. 
\end{convention}

\begin{notation}
\label{Notation: character on GL2}
    \begin{itemize}
        \item We let $\lambda^{\prime}(n, c)$ be the character of the torus $T_{c} = Z_{2}(\RR)^{+}\U(1)(\RR)$ of $\GL_2(\RR)$
        \begin{equation*}
            \begin{pmatrix}
                t_1 & \\
                    & \bar{t}_{1} \\
            \end{pmatrix} \in  T_c   \longmapsto (x + iy)^{n}(x^2 + y^2)^{\frac{c - n}{2}}, 
        \end{equation*}
        where $Z_2$ is the center of $\GL_2$, $t_1 = x + iy$, $n$ is the weight and $c$ is the central character. 
        \item Let $T_{spl} = Z_{2}(\RR) \{ \begin{pmatrix}
                                                \cosh t & \sinh t \\
                                                \sinh t & \cosh t 
                                           \end{pmatrix} \}$,         where $t \in \RR$, $\cosh t = \frac{e^{t} + e^{-t}}{2}$ is the hyperbolic cosine function and $\sinh t = \frac{e^t - e^{-t}}{2}$ is the hyperbolic sine function. 
        It can viewed as a split torus of the group $\GL_2(\RR)$ for the following reason. 
        If we let $C = \begin{pmatrix}
                                                                                                                 i  & i \\
                                                                                                                 -1 & 1 \\
                                                                                                            \end{pmatrix}$, 
                  for $\begin{pmatrix}
                        \cosh t & \sinh t \\
                        \sinh t & \cosh t \\ 
                      \end{pmatrix} \in \GL_2(\RR)$, then we have 
                \begin{equation*}
                    C \begin{pmatrix}
                        \cosh t & \sinh t \\
                        \sinh t & \cosh t \\ 
                      \end{pmatrix} C^{-1} = \begin{pmatrix}
                                                e^{t} & \\ 
                                                      & e^{-t} \\
                                             \end{pmatrix}. 
                \end{equation*}
        \item We denote by $\lambda(n, c)$ the character of $T_{spl}$ 
                \begin{equation*}
                    \lambda(n, c) \colon z \begin{pmatrix}
                                        \cosh t & \sinh t \\
                                        \sinh t & \cosh t \\
                                     \end{pmatrix} \in T_{spl} \longmapsto (e^t)^{n} z^{c},
                \end{equation*}
            where $z \in Z_{2}(\RR)$, $n$ is the weight and $c$ is the central character. 
    \end{itemize}
    
\end{notation}

\begin{notation}
    \begin{itemize}
        \item Let $(X, Y)$ be a basis of the standard representation $V_{2}$ of $\GL_{2}$ such that each $\begin{pmatrix}
                                                                                                    a & b \\
                                                                                                    c & d \\ 
                                                                                                \end{pmatrix} \in \GL_2(\RR)$ acts by 
                    \begin{align*}
                        \begin{pmatrix}
                            a & b \\
                            c & d \\
                        \end{pmatrix} X = aX + bY, \\
                        \begin{pmatrix}
                            a & b \\
                            c & d \\
                        \end{pmatrix} Y = cX + dY. 
                    \end{align*}
        \item Let us view $\Sym^{n}V_{2, \CC}$ as the $n$-dimensional vector space of homogeneous polynomials of degree $n$ in the variables $X$ and $Y$ with coefficients in $\CC$. For any integer $0 \le j \le n$, let $b_{j}^{n}$ be the vector $b_{j}^{n} = (-1)^{j}X^{j}Y^{n-j}$, so it has weight $\lambda^{\prime}(2j - n, n)$. The family $(b_{j}^{n})_{0 \le j \le n}$ forms a basis of $\Sym^{n}V_{2, \CC}$. We denote by $(a_{j}^{n})_{0 \le j \le n}$ the dual basis of $(b_{j}^{n})_{0 \le j \le n}$. Hence, the vector $a_{j}^{n}$ has weight $\lambda^{\prime}(n - 2j, -n)$.
    \end{itemize}
\end{notation}

\begin{definition}
\label{def: function phi}
    \begin{enumerate}
        \item We define the function $z_{2}$, $z_{2}^{\prime}$, $w_{2}$ and $w_{2}^{\prime}$ on $\GL_{2}(\RR)$ as: 
            \begin{align*}
                & z_{2}(\begin{pmatrix}
                    a & b \\
                    c & d \\
                \end{pmatrix}) := \frac{b + d}{-b + d}i,  \\ & z_{2}^{\prime}(\begin{pmatrix}
                                                                            a & b \\
                                                                            c & d \\
                                                                        \end{pmatrix}) := \frac{a + c}{-a + c}i; \\
                & w_{2}(\begin{pmatrix}
                            a & b \\
                            c & d \\
                        \end{pmatrix}) :=-b + d, \\ &      w_{2}^{\prime}(\begin{pmatrix}
                                                                            a & b \\
                                                                            c & d \\
                                                                        \end{pmatrix}) := a - c. 
            \end{align*}
        \item For any integer $r$ such that $r \equiv n (\text{mod} \ 2)$, we define the function $\phi_{r}^{n}$ on $\GL_2(\RR)$ as 
                \begin{equation*}
                    \phi_{r}^{n} = (z_2 - z_2^{\prime}) w_{2}^{-\frac{n - r}{2}} {w_{2}^{\prime}}^{-\frac{n + r}{2}}. 
                \end{equation*}
    \end{enumerate}
\end{definition}

\begin{remark}
    Since we identify $\GL_2(\RR)$ with $\GU(J_2^{\prime})^{\prime}(\RR)$,                                                                                                       $\begin{pmatrix}
                                                             a & b \\
                                                             c & d 
                                                            \end{pmatrix} \in \GL_2(\RR)$ if and only if $a = \bar{d}$ and $c = \bar{b}$. Hence, it is impossible that $a = c$ or $b = d$, and the functions $z_2$ and $z_2^{\prime}$ are well-defined on $\GL_{2}(\RR)$. 
\end{remark}

\begin{notation}
    Let $B_2$ be the standard Borel of $\GL_2$. Since we identify $\GL_{2}(\RR)$ with $\GU(J_2^{\prime})^{\prime}(\RR)$, we have 
    $B_{2}(\RR) = T_{spl} N$ where $N = \left\{ \begin{pmatrix} 
                                             1 - \frac{1}{2i}u & \frac{1}{2i}u \\
                                             -\frac{1}{2i}u & 1 + \frac{1}{2i}u 
                                            \end{pmatrix} | u \in \RR \right \}$. 
\end{notation}

\begin{proposition}
\label{prop: phi equivariant}
    We have 
    \begin{equation*}
        \phi_{r}^{n} \in \Ind_{B_{2}(\RR)^{+}}^{\GL_2(\RR)^{+}} \lambda(n + 2, -n). 
    \end{equation*}
   Furthermore, the compact torus $T_c$ acts on $\phi_{r}^{n}$ by right translation and $\phi_{r}^{n}$ has weight $\lambda^{\prime}(-r,-n)$. 
\end{proposition}

\begin{proof}
    \begin{enumerate}
        \item We have the following identity:
            \begin{align*}
                &\begin{pmatrix}
                    \cosh t & \sinh t \\
                    \sinh t & \cosh t 
                \end{pmatrix} \begin{pmatrix}
                                a & b \\
                                c & d 
                             \end{pmatrix} = \begin{pmatrix}
                                                a \cosh t + c \sinh t & b \cosh t + d\sinh t \\
                                                a \sinh t + c \cosh t & b \sinh t + d \cosh t
                                            \end{pmatrix}, \\
                &\begin{pmatrix}
                     1 - \frac{1}{2i}u & \frac{1}{2i}u \\
                     -\frac{1}{2i}u & 1 + \frac{1}{2i}u 
                 \end{pmatrix} \begin{pmatrix}
                                    a & b \\
                                    c & d 
                               \end{pmatrix} = \begin{pmatrix}
                                                    a + (c - a) \frac{1}{2i}u &  b + (d - b) \frac{1}{2i}u\\
                                                    c + (c - a) \frac{1}{2i}u &  d + (d - b) \frac{1}{2i}u
                                                \end{pmatrix}. 
            \end{align*}
            Then by the identity 
            \begin{align*}
               \cosh t \pm \sinh t = e^{\pm t}
            \end{align*}
            and $\phi_{r}^{n} = (z_2 - z_2^{\prime}) w_{2}^{-\frac{n - r}{2}} {w_{2}^{\prime}}^{-\frac{n + r}{2}}$, we can see that 
            \begin{equation*}
                 \phi_{r}^{n} \in \Ind_{B_{2}(\RR)^{+}}^{\GL_2(\RR)^{+}} \lambda(n + 2, -n). 
            \end{equation*}
        \item The weight of $\phi_{r}^{n}$ under $T_c$ can be seen from the identity 
        \begin{equation*}
            \begin{pmatrix}
                a & b \\
                c & d \\
            \end{pmatrix} \begin{pmatrix}
                            (x + iy) & 0 \\
                            0 & (x - iy) \\
                          \end{pmatrix} = \begin{pmatrix}
                                            (x + iy)a & (x - iy) b \\
                                            (x + iy)c & (x - iy) d 
                                         \end{pmatrix}
        \end{equation*} and the definition of $\phi_{r}^{n}$. 
    \end{enumerate}
\end{proof}

Recall that we have the decomposition 
\begin{equation*}
    \liegl_{2, \CC} = \liek_{\GL_{2}, \CC} \oplus \liep^{+}_{\GL_{2}, \CC} \oplus \liep^{-}_{\GL_{2}, \CC},
\end{equation*} and we let $v^{+} = \begin{pmatrix}
                                                            0 & 1 \\
                                                            0 & 0 
                                                       \end{pmatrix} \in \liep^{+}_{\GL_2, \CC}$ and 
                                                                                             $v^{-} = \begin{pmatrix}
                                                                                                0 & 0 \\
                                                                                                1 & 0  
                                                                                              \end{pmatrix} \in \liep^{-}_{\GL_2, \CC}$. 

Let us define 
\begin{equation*}
    w_{n}^{+}, w_{n}^{-} \in \Hom_{K_{\GL_2}}(\liep^{+}_{\GL_{2}, \CC} \oplus \liep^{-}_{\GL_{2}, \CC}, \Sym^{n}V_{2, \CC} \otimes \Ind_{B_{2}(\RR)^{+}}^{\GL_2(\RR)^{+}} \lambda(n + 2, -n))
\end{equation*}
by 
\begin{align*}
    & \omega_{n}^{+}(v^{+}) = (2\pi i)^{n+1} b^{n}_{0} \otimes \phi^{n}_{-(n + 2)}, \\ 
    & \omega_{n}^{+}(v^{-}) = 0, \\
    & \omega_{n}^{-}(v^{+}) = 0, \\
    & \omega_{n}^{-}(v^{-}) = (2 \pi i)^{n+1}b^{n}_{n} \otimes \phi^{n}_{n + 2}. 
\end{align*}
\begin{remark}
\begin{itemize}
    \item  The weight of $b^{n}_{0}$ is $\lambda^{\prime}(-n, n)$, the weight of $\phi^{n}_{-(n + 2)}$ is $\lambda^{\prime}( n + 2,-n)$ and the weight $v^{+}$ is $\lambda^{\prime}(2, 0)$; the weight of $b^{n}_{n}$ is $\lambda^{\prime}(n, n)$, the weight of $\phi^{n}_{n + 2}$ is $\lambda^{\prime}(-(n + 2),-n)$ and the weight $v^{-}$ is $\lambda^{\prime}(-2, 0)$. So $\omega_{n}^{\pm}$ preserves $K_{\GL_2}$-weight, so it is well-defined. 
    \item The vector $\phi^{n}_{-(n + 2)}$ is a minimal $K_{\GL_2}$-type vector in the discrete series $D_{n}^{+}$. 
\end{itemize}
\end{remark}

Let us state some basic properties. 

\begin{lemma}
We have the following identities of the action of $v^{\pm}$:
\begin{enumerate}
    \item 
        \begin{align*}
            & v^{+}(w_2) = -w_2^{\prime}, \  v^{+}(w_2^{\prime}) = 0, \\
            & v^{+}(z_2) = \frac{2(ad - bc)}{(-b + d)^{2}}i, \ v^{+}(z_2^{\prime}) = 0,\\
            & v^{-}(w_2) = 0, \ v^{-}(w_2^{\prime}) = -w_2, \\
            & v^{-}(z_2) = 0, \ v^{-}(z_2^{\prime}) = \frac{2(bc-ad)}{(-a+c)^{2}}i; 
        \end{align*} 
        \item 
        \begin{align*}
            & v^{+}(\phi_r^n) = \frac{n - r + 2}{2}\phi^{n}_{r - 2}, \ v^{-}(\phi_r^n) = \frac{n + 2 + r}{2}\phi^{n}_{r + 2}, 
        \end{align*}
        \item 
        \begin{equation*}
            v^{+} b^{n}_{r} = -(n - r)b^n_{r + 1},\  v^{-} b^{n}_{r} = -rb^n_{r - 1}. 
        \end{equation*}
\end{enumerate}
    
\end{lemma}

\begin{proof}
    If $f$ is a linear function on $\GL_{2}(\RR)$, then
        \begin{align*}
            (v^{+}f)(\begin{pmatrix}
                        a & b \\
                        c & d
                     \end{pmatrix}) &= \lim_{t \rightarrow 0} \frac{f(\begin{pmatrix}
                                                                        a & b \\
                                                                        c & d
                                                                     \end{pmatrix} \begin{pmatrix}
                                                                                        1 & t \\
                                                                                        0 & 1
                                                                                    \end{pmatrix}) - f(\begin{pmatrix}
                                                                                                            a & b \\
                                                                                                            c & d
                                                                                                         \end{pmatrix}) }{t} \\
                                   &= \lim_{t \rightarrow 0} \frac{f(\begin{pmatrix}
                                                                        a & b \\
                                                                        c & d
                                                                     \end{pmatrix} \begin{pmatrix}
                                                                                        0 & t \\
                                                                                        0 & 0
                                                                                    \end{pmatrix})}{t} = f(\begin{pmatrix}
                                                                                                                 a & b \\
                                                                                                                 c & d
                                                                                                            \end{pmatrix} \begin{pmatrix}
                                                                                                                                0 & 1 \\
                                                                                                                                0 & 0
                                                                                                                            \end{pmatrix}). 
        \end{align*}
        Similarly, we have 
        \begin{equation*}
             (v^{+}f)(\begin{pmatrix}
                        a & b \\
                        c & d
                     \end{pmatrix}) = f(\begin{pmatrix}
                                            a & b \\
                                            c & d
                                        \end{pmatrix} \begin{pmatrix}
                                                            0 & 0 \\
                                                            1 & 0
                                                       \end{pmatrix}). 
        \end{equation*}
        \begin{enumerate}
        \item It follows from the fact that $w_2$ and $w_2^{\prime}$ are linear functions that 
              \begin{align*}
                  & (v^{+}w_2)(\begin{pmatrix}
                            a & b \\
                            c & d
                          \end{pmatrix}) = -a + c = w_2^{\prime}(\begin{pmatrix}
                                                                    a & b \\
                                                                    c & d
                                                                  \end{pmatrix}), \\
                  & (v^{+}w_2^{\prime})(\begin{pmatrix}
                            a & b \\
                            c & d
                          \end{pmatrix}) = 0.                                                  
              \end{align*}
              The proof for $v^{-}(w_2)$ and $v^{-}(w_2^{\prime})$ is similar. \\
              It can be seen from the quotient rule that
              \begin{align*}
                  & (v^{+}z_2) = \frac{(v^{+}(b + d))(-b + d) - (v^{+}(-b + d))(b + d)}{(-b + d)^{2}}i = \frac{2(ad - bc)}{(-b + d)^{2}}i, \\
                  & (v^{+}z_2^{\prime}) = \frac{(v^{+}(a + c))(-a + c) - (v^{+}(-a + c))(a + c)}{(-a + c)^{2}}i = 0. 
              \end{align*}
              The proof for $v^{-}(z_2)$ and $v^{-}(z_2^{\prime})$ is similar. 
        \item It follows from explicit computation that
              \begin{align*}
                  z_2 - z_2^{\prime} &= \frac{2(bc - ad)}{(-b + d)(-a + c)}i \\
                                     &= (v^{+}(z_2 - z_2^{\prime}))\frac{w_2}{w_2^{\prime}}. 
              \end{align*}
              From this, we get 
              \begin{equation*}
                  v^{+}(z_2 - z_2^{\prime}) = (z_2 - z_2^{\prime}) \frac{w_2^{\prime}}{w_2}. 
              \end{equation*}
              Hence, 
              \begin{align*}
                  v^{+}\phi^{n}_{r} & =  (z_2 - z_2^{\prime})w_{2}^{-(\frac{n - r}{2} + 1)}{w_{2}^{\prime}}^{-(\frac{n + r}{2} - 1)}\\   
                  &+ (\frac{n - r}{2}) (z_2 - z_2^{\prime})w_{2}^{-(\frac{n - r}{2} + 1)}{w_{2}^{\prime}}^{-(\frac{n + r}{2} - 1)}\\
                  & = \frac{n -r + 2}{2}(z_2 - z_2^{\prime})w_{2}^{-(\frac{n - (r-2)}{2})}{w_{2}^{\prime}}^{-(\frac{n + (r-2)}{2})} \\
                  & = \frac{n -r + 2}{2} \phi^{n}_{r - 2}. 
              \end{align*}
              The computation of $v^{-}\phi^{n}_{r}$ is similar. 
        \item Since $X$ and $Y$ are linear functions, we have $v^{+}(X) = 0$ and $v^{+}(Y) = X$. 
              Hence, by the product rule, 
              \begin{equation*}
                  v^{+}(b^{n}_{r}) = v^{+}((-1)^{r}X^{r}Y^{n-r}) = (-1)^{j}(n - r) X^{r + 1} Y^{n - r - 1} = -(n - r) b^{n}_{r + 1}. 
              \end{equation*}
              Similar computations hold for $v^{-}(b^{n}_{j})$. 
        \end{enumerate}
\end{proof}

\begin{lemma}
        We have the following identities with complex conjuagtion: 
        \begin{align*}
           &  \overline{\phi^{n}_{r}} =  (-1)\phi^{n}_{-r}, \ \overline{b^n_r} = b^b_{n - r}, \overline{v^{+}} = v^{-}, \\
           &  \overline{\omega^{\pm}} =  (-1)^{n}\omega_{n}^{\mp},  
        \end{align*}
        where $\overline{(\cdot)}$ denotes the complex conjugation. 
\end{lemma}

\begin{proof}
        Recall that the Shimura data of $\Sh_{\GL_2}$ is defined on $\RR$-points as 
            \begin{equation*}
             \begin{tikzcd}[row sep = 0]
                 h \colon \mathbb{S}(\RR) \ar[r] & \GL_2(\RR) \\
                    z = x + iy \ar[r, mapsto] & \begin{pmatrix}
                                                    z & 0 \\
                                                    0 & \bar{z}
                                                \end{pmatrix}
             \end{tikzcd},
            \end{equation*}
            and the reflex field of the Shimura variety $\Sh_{\GL_2}$ is $\QQ$ and the complex conjugation on it is given by conjugation by 
                           $N = \begin{pmatrix}
                                    0 & -1 \\
                                    -1 & 0 
                                \end{pmatrix}$. From this we can see that 
        \begin{equation*}
            \overline{v^{+}} = \mathrm{Ad}_{N}(v^{+}) = v^{-}. 
        \end{equation*}
        It follows from 
        \begin{equation*}
            \mathrm{Ad}_{N}( \begin{pmatrix}
                                a & b \\
                                c & d
                             \end{pmatrix}) = \begin{pmatrix}
                                                 d & c\\
                                                 b & a 
                                              \end{pmatrix},
        \end{equation*}
        that $\overline{a} = d$ and $\overline{b} = c$.  So we have 
        \begin{align*}
            & \overline{w_2} = -\overline{b} + \overline{d} = a - c = w_2^{\prime}, \\
            & \overline{z_2} = \frac{\overline{b} + \overline{d}}{-\overline{b} + \overline{d}}(-i) = \frac{a + c}{c - a}i = z_2^{\prime}. 
        \end{align*}
        We can see from this that 
        \begin{equation*}
            \overline{\phi^{n}_{r}} = (-1) (z_2 - z_2^{\prime}) w_{2}^{-(\frac{n - (-r)}{2})}{w_{2}^{\prime}}^{-(\frac{n + (-r)}{2})} = (-1) \phi^{n}_{-r}. 
        \end{equation*}
        The complex conjugate on $X, Y$ is given by $\overline{(X, Y)} = (X, Y)N = (-Y, -X)$. Hence, we have 
        \begin{equation*}
            \overline{b^{n}_{r}} = \overline{(-1)^{r}X^{r}Y^{n - r}} = (-1)^{n+r}Y^{r}X^{n - r} = b^{n}_{n - r}. 
        \end{equation*}
        Finally, we can deduce that 
        \begin{align*}
            \overline{\omega_{n}^{-}}(v^{-}) &= \overline{(2 \pi i)^{n + 1}} \overline{b^n_n} \otimes \overline{\phi^{n}_{n + 2}} \\
                                        & = (-1)^{n + 1} (2\pi{i})^{n + 1}  b^{n}_{0} \otimes (-1) \phi^{n}_{-(n + 2)} \\
                                        & = (-1)^{n}(2\pi{i})^{n + 1} b^{n}_{0} \otimes \phi^{n}_{-(n + 2)} \\
                                        & = (-1)^{n}\omega^{+}_{n}(v^{+}). 
        \end{align*}
        Therefore, $\overline{\omega^{\pm}_{n}} = (-1)^{n}\omega_n^{\mp}$. 
\end{proof}

Given $\phi_{f} \in \mathcal{B}_{n}$, we let 
\begin{equation*}
   Eis_{B}^{n}(\phi_f) = \sum_{\gamma \in B_{2}(\QQ) \backslash \GL_2(\QQ)} \gamma^{*}(\omega_n^{+} \otimes \phi_f),
\end{equation*}
which is absolutely convergent \cite[(6.3.1)]{Kings98} and defines an element of 
\begin{equation*}
    \Hom_{K_{\GL_2}}(\liep_{\GL_{2}, \CC}^{+}\oplus \liep_{\GL_{2}, \CC}^{-}, \Sym^{n}V_{2, \CC} \otimes C^{\infty}(\GL_2(\QQ) \backslash \GL_2(\AAA)). 
\end{equation*}
Consider
\begin{equation*}
    \theta_{n} = \frac{(2 \pi i)^{n + 1}}{2(n + 1)}\sum\limits_{j = 0}^{n}  b_{n - j}^{n} \otimes \phi_{n - 2j}^{n} \in \Sym^{n}V_{2, \CC} \otimes \Ind_{B_{2}(\RR)^{+}}^{\GL_2(\RR)^{+}} \lambda(n + 2, -n),
\end{equation*}
Similar to before, for $\phi_{f} \in \mathcal{B}_{n}$ we define 
\begin{equation}
\label{eqn:Eis_symbol_Hodge}
    Eis_{H}^{n}(\phi_f) = \sum_{\gamma \in B_{2}(\QQ) \backslash \GL_2(\QQ)} \gamma^{*}(\theta_{n} \otimes \phi_f).
\end{equation}
This infinite series is absolutely convergent for $n \ge 1$. In this paper, we only treat the case $n \ge 1$. 

\begin{lemma}\textnormal{\cite[(6.3.5)]{Kings98}}
\label{Lemma: Delgine representative of Eisenstein symbol}
    We have the following identity: 
    \begin{equation*}
        \mathrm{d}(Eis_{H}^{n}(\phi_f)) = \pi_{n} (Eis_{B}^{n}(\phi_f)). 
    \end{equation*}
    In other words, for any $n \ge 1$ and $\phi_{f} \in \mathcal{B}_{n}$, the class $Eis_{D}^{n}(\phi_f)$ is represented by 
    \begin{equation*}
        (Eis_{H}^{n}(\phi_f), Eis_{B}^{n}(\phi_f)) \in \mathrm{H}_{D}^{1}(\Sh_{\GL_2}, \Sym^{n} V_{2}(1)), 
    \end{equation*}
    which is a direct factor of $\mathrm{H}_{D}^{n + 1}(E^{n}, \RR(n + 1))$. Here $E \rightarrow \Sh_{\GL_2}$ is the universal elliptic curve over the Shimura variety $\Sh_{\GL_2}$.
\end{lemma}

\begin{proof}
    We have 
    \begin{align*}
        d\theta_n(v^{+}) = v^{+}\theta_n = \frac{(2\pi{i})^{n+1}}{2(n + 1)} ( & b^{n}_{n} \otimes \phi^{n}_{n - 2} \\
                                                                       & + 2 b^{n}_{n - 1} \otimes \phi^{n}_{n - 4} - b^{n}_{n} \otimes \phi^{n}_{n - 2} \\
                                                                       & + 3 b^{n}_{n - 2} \otimes \phi^{n}_{n - 6} - 2 b^{n}_{n-1} \otimes \phi^{n}_{n - 4} \\
                                                                       & + \cdots \\
                                                                       & + (n + 1) b^{n}_{0} \otimes \phi^{n}_{-(n + 2)} - n b^{n}_{1} \otimes \phi^{n}_{-n}) \\
                                         & = \frac{(2\pi{i})^{n+1}}{2 } b^{n}_{0} \otimes \phi^{n}_{-(n + 2)}. 
    \end{align*}
    By a similar computation, we have $d\theta_{n}(v^{-}) = \frac{(2\pi{i})^{n + 1}}{2} b^{n}_{n} \otimes \phi^{n}_{n + 2}$. 
    It follows from the definitions that 
    \begin{equation*}
        \pi_n(\omega^{+}_{n}) = \frac{\omega_{n}^{+} + (-1)^{n}\overline{\omega_{n}^{+}}}{2} = \frac{\omega_{n}^{+} + {\omega_{n}^{-}}}{2}. 
    \end{equation*}
    Hence, 
    \begin{align*}
        & \pi_n(\omega^{+}_{n})(v^{+}) = \frac{\omega_{n}^{+} + {\omega_{n}^{-}}}{2}(v^{+}) = \frac{(2\pi{i})^{n+1}}{2 } b^{n}_{0} \otimes \phi^{n}_{-(n + 2)} = d\theta_n(v^{+}),  \\
        & \pi_n(\omega^{+}_{n})(v^{-}) = \frac{\omega_{n}^{+} + {\omega_{n}^{-}}}{2}(v^{-}) = \frac{(2\pi{i})^{n+1}}{2 } b^{n}_{n} \otimes \phi^{n}_{n + 2} = d\theta_{n}(v^{-}). 
    \end{align*}
    In summary, we have shown $\pi_n(\omega_n^{+}) = d\theta_n$ so $\mathrm{d}(Eis_{H}^{n}(\phi_f)) = \pi_{n} (Eis_{B}^{n}(\phi_f))$. 
\end{proof}

\begin{remark}
    This proof is essentially the same as in \cite[\S 6.3]{Kings98}, but we chose to identify $\GL_{2}(\RR)$ with $\GU(J_2^{\prime})^{\prime}$, so we reproduce the proof here. 
\end{remark}

\subsubsection{Motivic classes in DB-cohomology}
Now we give an expression of the realization of the motivic classes constructed in Construction \ref{Construction: motivic classes} in Deligne{\textendash}Beilinson cohomology using functoriality of motivic cohomology, DB-cohomology and the Beilinson regulator. 

\begin{notation}
    We use the following notations: 
    \begin{itemize}
        \item $* := a + 5b - 3 (r + s) + 1$, $** :=  a + 5b - 3 (r + s) + 3$, $\Box := a + 5b - 3 (r + s)$, $\circ := a + 3b -r - s + 1$ and $\triangle := a + 3b -r - s + 2$; 
        \item for $\diamond \in \{M, H, D\}$, $i_{\diamond}$ stands for the natural inclusion induced by $W \hookrightarrow \iota^{*}V$.
    \end{itemize}
\end{notation}

\begin{proposition}
    We have the following commutative diagram: 
    \[
        \begin{tikzcd}[column sep = 0.8em]
            \mathcal{B}_n \ar[r,"Eis^n_M"] \ar[d, "r_{H}"] & \mathrm{H}^{1}_{M}(\Sh_{\GL_2}, \mathrm{Sym}^{n}V_2(1)) \ar[r, "p^{*}_{M}"] \ar[d, "r_{H}"] & \mathrm{H}^{1}_{M}(M, W(1)) \ar[r, hook, "i_{M}"] \ar[d, "r_{H}"] & \mathrm{H}^{1}_{M}(M, \iota^{*}V(1)) \ar[r, "\iota_{M, *}"] \ar[d, "r_{H}"] & \mathrm{H}^{3}_{M}(S,V(2)) \ar[d, "r_{H}"] \\
            \mathcal{B}_{n, \RR} \ar[r, "Eis^n_H"] \ar[d, "r_{H \rightarrow D}"] & \mathrm{H}^{1}_{H}(\Sh_{\GL_2}, \mathrm{Sym}^{n}V_2(1)) \ar[r,"p^{*}_{H}"] \ar[d, "r_{H \rightarrow D}"] & \mathrm{H}^{1}_{H}(M, W(1)) \ar[r, hook, "i_{H}"] \ar[d, "r_{H \rightarrow D}"] & \mathrm{H}^{1}_{H}(M, \iota^{*}V(1)) \ar[r, "\iota_{H, *}"]  \ar[d, "r_{H \rightarrow D}"] & \mathrm{H}^{3}_{H}(S,V(2)) \ar[d, "r_{H \rightarrow D}"] \\
            \mathcal{B}_{n, \RR} \ar[r, "Eis^n_D"] \ar[d, "\cong"] & \mathrm{H}^{1}_{D}(\Sh_{\GL_2}, \mathrm{Sym}^{n}V_2(1)) \ar[r, "p^{*}_{D}"] \ar[d, hook] & \mathrm{H}^{1}_{D}(M, W(1)) \ar[r, hook,"i_D"] \ar[d, hook] & \mathrm{H}^{1}_{D}(M, \iota^{*}V(1)) \ar[r, "\iota_{D, *}"] \ar[d, hook] & \mathrm{H}^{3}_{D}(S,V(2)) \ar[d, hook] \\
            \mathcal{B}_{n, \RR} \ar[r, "Eis^n_D"]  & \mathrm{H}^{n + 1}_{D}(E^{n}, \RR(n + 1)) \ar[r, "p^{*}_{D}"] & \mathrm{H}^{n + 1}_{D}(E^{n}, \RR(n + 1)) \ar[r, hook, "i_D"] & \mathrm{H}_{D}^{*}(A^{\Box}, \RR(\circ)) \ar[r, "\iota_{D, *}"] & \mathrm{H}_{D}^{**}(A^{\Box}, \RR(\triangle)) 
        \end{tikzcd}, 
    \]
    where in the last row, $E^{n}$ in the first $\mathrm{H}^{n + 1}_{D}(E^{n}, \RR(n + 1))$ is the $n$-th fiber product of the universal elliptic curve over $\Sh_{\GL_2}$, $E^{n}$ in the second $\mathrm{H}^{n + 1}_{D}(E^{n}, \RR(n + 1))$ is the $n$-th fiber product of the universal elliptic curve over $M$, $A^{\Box}$ in $\mathrm{H}_{D}^{*}(A^{\Box}, \RR(\circ))$ is the $\Box$-th fiber product of the pullback of the universal abelian $3$-fold from $S$ to $M$ through $\iota$, and $A^{\Box}$ in $\mathrm{H}_{D}^{**}(A^{\Box}, \RR(\triangle))$ is the $\Box$-th fiber product of the universal abelian $3$-fold over $S$. 
\end{proposition}

\begin{proof}
    \par First, the commutativity of the diagram comes from the functoriality of $r_{H}$ and $r_{H \rightarrow D}$. 
    \par Second, the decomposition of the Gysin morphism 
    \begin{equation*}
        \mathrm{H}^{1}_{\mathrm{\diamond}}(M, W(1)) \rightarrow \mathrm{H}^{3}_{\mathrm{\diamond}}(S, V(2))
    \end{equation*}
    into $\iota_{\diamond, *} \circ i_{\diamond}$ is natural from the construction of Gysin morphism (Proposition \ref{Prop: Gysin for motivic cohomology}). 
    \par Finally, the embebding of the third row into the fourth row is from ``Liebermann’s trick'' as explained in Proposition \ref{Prop:motiv_coh_abs_relative}. 
\end{proof}

\begin{remark}
\label{remark: Eis_D class}
    From Construction \ref{Construction: motivic classes}, the Eisenstein classes $c \in \mathrm{H}^{3}_{M}(S, V(2))$ are 
    \begin{equation*}
        c = \iota_{M, *} \circ i_{M} \circ p_{M}^{*} \circ Eis_{M}^{n} (\phi_f), 
    \end{equation*}
    for $\phi_{f} \in \mathcal{B}_{n}$. 
    By commutativity of the above diagram, their image under the Deligne regulator $r_{D}: = r_{H \rightarrow D} \circ r_{H}$ 
    is 
    \begin{align*}
        r_{D}(c) &= r_{D} (\iota_{M, *} \circ i_{M} \circ p_{M}^{*} \circ Eis_{M}^{n} (\phi_f)) \\
                 &= \iota_{D, *} \circ i_{D} \circ p_{D}^{*} \circ Eis_{D}^{n} (r_{D}(\phi_f)) \\
                &= \iota_{D, *} \circ i_{D} \circ p_{D}^{*}(Eis_{H}^{n}(\phi_f), Eis_{B}^{n}(\phi_f)).
    \end{align*}
    By ``Liebermann’s trick'', Proposition \ref{Prop: DB-coh_temperd currents}, Proposition \ref{Prop: DB-cohomology-si forms to tempered currents} and the fact that $Eis_{H}^{n}(\phi_f)$ and $Eis_{B}^{n}(\phi_f)$ are slowly increasing as explained in \cite[page 120]{Kings98}, we can view 
    \begin{equation*}
         i_{D} \circ p_{D}^{*} \circ (Eis_{H}^{n}(\phi_f), Eis_{B}^{n}(\phi_f)). 
    \end{equation*}    
    as a pair of tempered currents 
    \begin{equation*}
        (T_{p^{*}Eis_{H}^{n}(\phi_f)}, T_{p^{*}Eis_{B}^{n}(\phi_f)}) \in \mathrm{H}^{1}_{D} (M, \iota^{*}V(1)) \subset \mathrm{H}^{a + 5b - 3 (r + s) + 1}_{D} (A^{a + 5b - 3 (r + s)} /M, \RR(a + 3b - r - s + 1)), 
    \end{equation*}
    where $A^{a + 5b - 3 (r + s)} /M$ is the $a + 5b - 3(r + s)$-th fiber product of the pullback of the universal abelian $3$-fold from $S$ to $M$ through $\iota$. 
    Hence, by the explicit description of the Gysin morphism for DB-cohomology (cf. equation (\ref{eq: Gysin for DB-cohomology})), we have 
    \begin{align*}
        r_{D}(c) = (\iota_{*}T_{p^{*}Eis_{H}^{n}(\phi_f)}, \iota_{*} T_{p^{*}Eis_{B}^{n}(\phi_f)}) & \in \mathrm{H}^{3}_{D} (S, V(2)) \\ & \subset \mathrm{H}^{a + 5b - 3 (r + s) + 3}_{D} ( A^{a + 5b - 3 (r + s)} /S, \RR(a + 3b -r - s + 2)), 
    \end{align*}
     where $A^{a + 5b - 3 (r + s)} /S$ is the $a + 5b - 3(r + s)$-th fiber product of the universal abelian $3$-fold over $S$. 
\end{remark}

\subsection{The pairing}
\label{SS the pairing}
In this subsection, we give an explicit formula for the Poincar\'{e} duality pairing with the tool of tempered currents. 

\subsubsection{From Poincar\'{e} duality pairing to an integral of differential forms} 
We first express the Poincar\'{e} duality pairing in terms of an integration of differential forms on Shimura varieties. 

\begin{lemma}
\label{Lemma: tempered current associated to Eis classes}
    Let $\widetilde{A}^{a + 5b - 3(r + s)}$ be a smooth toroidal compactification of $A^{a + 5b - 3(r + s)}$ such that the complement
    $\widetilde{A}^{a + 5b - 3(r + s)} \backslash A^{a + 5b - 3(r + s)}$ is a simple normal crossings divisor. Then there exists a closed tempered current 
    \begin{equation*}
        \rho \in \mathcal{D}_{si, \RR(a + 3b - r - s + 1)}^{a + 5b - 3(r + s) + 2}(\widetilde{A}^{a + 5b - 3(r + s)})\
    \end{equation*}
    such that: 
    \begin{enumerate}
        \item the cohomology class $[\rho]$ of $\rho$ belongs to $\mathrm{H}^{2}_{B, !}(S, V(2))^{-}(-1)$; 
        \item the class $[\rho]$ maps to $\mathcal{E}is_{H}^{n}(\phi_f) = r_{H}(\mathcal{E}is_{M}^{n}(\phi_f))$ 
            through the third map in the exact sequence (\ref{exact seq: Ext^1}):
            \begin{equation*}
                \mathrm{H}^{2}_{B, !}(S, V(2))_{\RR}^{-}(-1) \rightarrow \Ext^{1}_{\mathrm{MHS}_{\RR}^{+}}(\RR(0), \mathrm{H}^{2}_{B, !}(S, V(2))_{\RR}); 
            \end{equation*}
        \item the pair of tempered currents $(\iota_{*}T_{p^{*}Eis_{H}^{n}(\phi_f)}, \iota_{*} T_{p^{*}Eis_{B}^{n}(\phi_f)})$ and $(\rho, 0)$ represent the same cohomology class in $\mathrm{H}^{3}_{D} (S, V(2))$. 
    \end{enumerate}
\end{lemma}

\begin{proof}
   This is a direct consequence of \cite[\S 5.7.]{Bei83}.
\end{proof}

\begin{corollary}
\label{corollary: Omega pairing in terms of omega paring with rho}
    We have the following identity: 
    \begin{equation*}
        \langle \Omega, \tilde{v}_{K} \rangle_{B} = \langle \omega_{\Psi}, [\rho] \rangle_{B} \otimes \mathbf{1}, 
    \end{equation*}
    where $\mathbf{1}$ be unit of $E(\pi_f) \otimes_{\QQ} \CC$. 
\end{corollary}

\begin{proof}
    It is direct from Lemma \ref{Lemma: tempered current associated to Eis classes} (2) and Remark \ref{remark: pairing} (5). 
\end{proof}

We first take the dual of the map of representations
\begin{equation*}
    \Sym^{n}V_{2} \hookrightarrow \iota^{*} V,
\end{equation*}
 and get the map $\iota^{*} D \rightarrow \Sym^{n}V_{2}^{\vee} $. 
Composing the above map with the natural pairing $\Sym^{n}V_2^{\vee} \otimes \Sym^{n} V_2 \rightarrow \QQ(0)$ defined by 
\begin{equation*}
    a^{n}_{j} \otimes b^{n}_{k} \mapsto \begin{cases}
                                            0 & \text{if} \ j \neq k, \\
                                            \binom{n}{j} & \text{if} \ j = k,   \\
                                        \end{cases}
\end{equation*}
we have a natural pairing 
\begin{equation*}
    \langle \cdot, \cdot \rangle \colon \iota^{*}D \otimes \Sym^{n} V_2 \rightarrow \QQ(0). 
\end{equation*} 

\begin{convention}
We normalize the embedding 
\begin{equation*}
\mathrm{H}^{2}_{B, !}(S, V(2)) \hookrightarrow \mathrm{H}^{a + 5b - 3(r + s) + 2}_{B, !}(A^{a + 5b - 3(r + s)}, \RR(a + 3b -r -s + 2))
\end{equation*}
and 
\begin{equation*}
\mathrm{H}^{2}_{B, !}(S, D) \hookrightarrow \mathrm{H}^{a + 5b - 3(r + s) + 2}_{B, !}(A^{a + b - 3(r + s)}, \RR(2b - 2r - 2s))
\end{equation*}
such that we have the following commutative diagram: 
\[
    \begin{tikzcd}
        \mathrm{H}^{2}_{B, !}(S, V(2)) \otimes \mathrm{H}^{2}_{B, !}(S, D)  \ar[d, hook] \ar[r,"{\langle \cdot , \cdot \rangle_B}"] & \QQ(0) \ar[d, "="] \\
        \mathrm{H}^{*}_{B, !}(A^{\Box}, \RR(a + 3b -r - s + 2)) \otimes \mathrm{H}^{*}_{B, !}(A^{\Box}, \RR(2b - 2r -2s )) \ar[r,"{\langle \cdot, \cdot \rangle_B}"] & \QQ(0)
    \end{tikzcd}, 
\]
where $* = a + 5b - 3(r + s) + 2$ and $\Box = a + 5b - 3(r + s)$. 
\end{convention}

Let $\Psi = \Psi_{f} \otimes \Psi_{\infty}$ be a factorizable cusp form in the irreducible cuspidal automorphic representation $\tilde{\pi} = \tilde{\pi}_f \otimes \tilde{\pi}_{\infty}$ of $G(\AAA)$ with central character\footnote{Here $\tilde{\pi}_{\infty}$ is a $(\lieg_{\CC}, K_G)$-module, but $A_{G} \subset K_{G}$, so the word ``central character of $\tilde{\pi}$ makes sense.}
\[
    \begin{tikzcd}[row sep = 0]
       \omega_{\pi} \colon Z_{G}(\QQ) \backslash Z_{G}(\AAA) \ar[r] & \CC^{\times} \\
                      z \ar[r, mapsto] & {|z|^{n} \nu(z)^{-1}}, 
    \end{tikzcd}
\]
where $\nu$ is a finite order Hecke character of sign $(-1)^{n}$. Let 
\begin{equation*}
    \omega_{\Psi} := \omega \otimes \Psi_{f}
\end{equation*}
be a differential form associated to $\Psi$, where $\omega$ is the Lie algebra cohomology class constructed in Proposition \ref{Prop: explicit construnction of the differential form}. 

Recall that we let $n = a + b + r + s$ and we use $\Sh_{G}(L)$ (resp., $\Sh_{H}(K)$) to denote $\Sh_{G}(L)_{\CC}^{an}$ (resp., $\Sh_{H}(K)_{\CC}^{an}$). 
\begin{proposition}
\label{Prop: pairing to integration of diff form}
 Let $\phi_f \in I_{n}(\nu) \subset \mathcal{B}_{n, \overline{\QQ}}$ defined in Definition \ref{Def: In(v)} and let $\rho$ be the closed tempered current associated to the class $\mathcal{E}is_{D}^{n}(\phi_f) = r_{D}(\mathcal{E}is_{M}^{n}(\phi_f))$ by Lemma \ref{Lemma: tempered current associated to Eis classes}. We have 
 \begin{equation*}
      \langle \omega_{\Psi}, [\rho] \rangle_{B} = \frac{1}{(2\pi i)^{n + 1}} \int_{\Sh_H(K)} p^{*} Eis_{H}(\phi_f) \iota^{*} \omega_{\Psi},
 \end{equation*}
 where $K$ is the level of $\phi_f$. 
    
\end{proposition}

\begin{proof}
    For the pair of tempered currents $(\iota_{*}T_{p^{*}Eis_{H}^{n}(\phi_f)}, \iota_{*} T_{p^{*}Eis_{B}^{n}(\phi_f)})$ and $(\rho, 0)$ that represent the same cohomology class in $\mathrm{H}^{3}_{D} (\Sh_{G}, V(2))$, there exists 
    \begin{equation*}
        (S^{\prime}, T^{\prime}) \in \mathcal{D}_{si, \RR(a + 3b -r - s + 1)}^{a + 5b - 3(r + s) + 1}(\widetilde{A}^{a + 5b - 3(r + s)}) \oplus F^{a + 3b -r - s + 2} \mathcal{D}_{si}^{a + 5b - 3(r + s) + 2}(\widetilde{A}^{a + 5b - 3(r + s)})
    \end{equation*}
    such that 
    \begin{equation*}
        \rho = \iota_{*}T_{p^{*}Eis_{H}^{n}(\phi_f)} + dS^{\prime} + \pi_{a + 3b -r -s + 1} (T^{\prime}). 
    \end{equation*}
    It follows from the fact that $\omega_{\Psi}$ is a diffferential form on $\Sh_{G}$ that 
    \begin{equation*}
        \langle \omega_{\Psi}, [\rho] \rangle_{B} = \langle \omega_{\Psi}, [\rho] \rangle_{B}^{\Sh_{G}}. 
    \end{equation*}
    Then we can see that 
    \begin{align*}
        \langle \omega_{\Psi}, [\rho] \rangle_{B}         &= \rho(\omega_{\Psi}) \\
                                                          &= \iota_{*}T_{p^{*}Eis_{H}^{n}(\phi_f)} (\omega_{\Psi}) + dS^{\prime} (\omega_{\Psi}) + \pi_{a + 2r -s + 1} (T^{\prime}) (\omega_{\Psi}) \\
                                                          &= \iota_{*}T_{p^{*}Eis_{H}^{n}(\phi_f)} (\omega_{\Psi}) + \pi_{a + 2r -s + 1} (T^{\prime}) (\omega_{\Psi}) \\
                                                          &= \iota_{*}T_{p^{*}Eis_{H}^{n}(\phi_f)} (\omega_{\Psi}) \\
                                                          &= T_{p^{*}Eis_{H}^{n}(\phi_f)} (\iota^{*}\omega_{\Psi}) \\
                                                          &= \frac{1}{(2\pi i)^{n + 1}} \int_{\Sh_H(K)} p^{*} Eis_{H}(\phi_f) \iota^{*} \omega_{\Psi}. 
    \end{align*}   
    The first equality holds by the compatibility between Poincar\'{e} duality pairings and integrations (see \cite[Prop.1.4.4.(c)]{Harris90}). 
    The third equality follows from the fact that $\omega_{\Psi}$ is closed. 
    The fourth equality follows from the fact that the Hodge type of $\omega_{\Psi}$ is $(a + r + 1, b + s + 1)$ and the fact that $T^{\prime} \in F^{a + 3b -r - s + 2} \mathcal{D}_{si}^{a + 5b - 3(r + s) + 2}(\widetilde{A}^{a + 5b - 3(r + s)})$. 
    The fifth  equality follows from the definition of the pushforward of tempered currents (see equation (\ref{eq: pushforward of tempered currents})). The last equality comes from the definition of tempered currents. This completes the proof of the proposition. 
\end{proof}

\subsubsection{From the integral of differential forms to an adelic integral}
We now translate the above integral into an adelic integral. In order to do so, we need to make the following choices of measures. 

\begin{convention}
\label{convention: measure}
    \begin{enumerate}
        \item We normalize the Haar measure for a reductive or unipotent group $G^{\prime}$ over $\Zp$ such that $\mathrm{Vol}(G^{\prime}(\Zp)) = 1$. 
        \item At the archimedean place, we use the Iwasawa decomposition to determine the measure. More precisely, we define the measure on the unipotent radical of the Borels\footnote{The algebraic group $G$ is quasi-split over $\QQ$.} of $G(\RR)$  and $H(\RR)$ so that the integral points have covolume $1$. We choose the Haar measure on $K_{G}$ (resp., $K_{H}$) such that the volume of $K_{G} / Z_{G}(\RR)$ (resp., $K_{H} / Z_{H}(\RR)$) is $1$. On the maximal split torus $\RR^{\times, n}$, we choose the multiplicative Haar measure $\frac{dx_1dx_{2} \cdots dx_{n}}{|x_1 x_2 \cdots x_{n}|}$. 
        \item The measures on $G(\AAA)$ and $H(\AAA)$ are the product of the above nonarchimedean measures and archimedean measures. 
        \item For quotients of groups, we use the unique measure compatible with these choices. 
        \item Once we fix $\mathbf{1} = v^{+} \otimes v^{-}$, we have fixed an equivalence of top diffferential forms on $\Sh_{H}(K)$ and invariant measures on $\Sh_{H}(K)$ (see \cite[Page 83]{Harris97}). 
    \end{enumerate}
\end{convention}

\begin{proposition}
\label{Prop: integration of diff form into adelic integral}
For $\phi_f$ and $\rho$ be as in Proposition \ref{Prop: pairing to integration of diff form}. We have 
\begin{equation*}
    \langle \omega_{\Psi}, [\rho] \rangle_{B} =  \frac{[H(\Zp): K]^{-1}}{(2\pi{i})^{n + 1}} \int_{H(\QQ) Z(\AAA) \backslash H(\AAA) } (p^{*} Eis_{H}(\phi_f) \iota^{*} \omega_{\Psi} (\mathbf{1}))(g) dg, 
\end{equation*}
where $\mathbf{1}$ is the vector $v^{+} \otimes v^{-}$ and the vectors $v^{+}$ and $v^{-}$ are defined in equation (\ref{eq: def of v^{+} and v^{-}}).
\end{proposition}

\begin{proof}
    By Proposition \ref{Prop: pairing to integration of diff form}, we have 
    \begin{equation*}
         \langle \omega_{\Psi}, [\rho] \rangle_{B} = \frac{1}{(2\pi i)^{n + 1}} \int_{\Sh_{H}(K)} p^{*} Eis_{H}(\phi_f) \iota^{*} \omega_{\Psi}. 
    \end{equation*}
    Thanks to the equivalence of top differential forms on $\Sh_{H}(K)$ and invariant measures on $\Sh_{H}(K)$ explained above, we can write the integral as an adelic integral. More precisely, we have 
    \begin{align*}
        \langle \omega_{\Psi}, [\rho] \rangle_{B} &= \frac{1}{(2\pi i)^{n + 1}} \int_{H(\QQ) \backslash H(\AAA) / K_{H} K}(p^{*} Eis_{H}(\phi_f) \iota^{*} \omega_{\Psi}(\mathbf{1}))(g) dg \\
                                                  &= \frac{[H(\Zp): K]^{-1}}{(2\pi i)^{n + 1}} \int_{H(\QQ) Z_{H}(\RR) \backslash H(\AAA) }(p^{*} Eis_{H}(\phi_f) \iota^{*} \omega_{\Psi}(\mathbf{1}))(g) dg \\    
                                                  &= \frac{[H(\Zp): K]^{-1}}{(2\pi i)^{n + 1}} \int_{H(\QQ) Z(\RR) \backslash H(\AAA) }(p^{*} Eis_{H}(\phi_f) \iota^{*} \omega_{\Psi}(\mathbf{1}))(g) dg \\
                                                   &= \frac{[H(\Zp): K]^{-1}}{(2\pi i)^{n + 1}} \int_{H(\QQ) Z(\AAA) \backslash H(\AAA) }(p^{*} Eis_{H}(\phi_f) \iota^{*} \omega_{\Psi}(\mathbf{1}))(g) dg. 
    \end{align*}
    The second equality follows from $\mathrm{Vol}(K) = [H(\Zp): K]^{-1}$ and $\mathrm{Vol}(K_H / Z_{H} (\RR)) = 1$. The third equality is from the fact that $\mathrm{Vol}(\mathbb{S}^{1}) = 1$ and the fourth comes from $\mathrm{Vol}(\Zp^{\times}) = 1$. 
\end{proof}

\subsubsection{An explicit form of the adelic integral}

\par Recall that we let $\Psi_{\infty} \in \tilde{\pi}_{\infty}$ be a vector of weight $(-1 - b + r - s, 1 + a + r - s, r - s; b + 2s -r)$ and $v \in D$ be a vector of weight $(-a + s - r, b + s - r, s - r; -b - 2s + r)$. We also have the differential form $\omega_{\Psi}$ associated to $\Psi$ by Proposition \ref{Prop: explicit construnction of the differential form}. The following proposition gives a precise formula for $\langle \omega_{\Psi}, [\rho] \rangle_{B}$.

\begin{notation}
\label{notation: Xi}
    \begin{itemize}
        \item   We let 
                \begin{equation*}
                    C_{j} := [H(\Zp): K]^{-1} \langle X^{2j -b -r -s}_{(1, -1, 0)}v , b^{a + b + r + s}_{n - j} \rangle. 
                \end{equation*}
        \item  We denote by $\Xi_{m, t}(\phi_f)$ the functions on $H(\AAA)$ defined by 
                \begin{equation*}
                    \Xi_{m, t}(\phi_f) := X^{m}_{(1, -1, 0)} \Psi \sum\limits_{\gamma \in B_{2}(\QQ) \backslash \GU(J_{2}^{'})^{'}(\QQ)} \gamma^{*} (\phi_{t}^{a + b + r + s} \otimes \phi_{f}). 
                \end{equation*}
        \item We let 
             \begin{align*}
                & A = \max \{0, \frac{b + r + s}{2} \}, \\
                & B = \min\{ \frac{a + 2b + r + s}{2}, a + b + r + s\}. 
            \end{align*}
    \end{itemize}
\end{notation}

\begin{proposition}
\label{prop: explicit_pairing}
The pairing $\langle \omega_{\Psi},  [\rho] \rangle_{B} $ is equal to 
\begin{equation}
\label{eq:pairing}
    -\frac{1}{2(a + b + r + s + 1)} \sum\limits_{j = A}^{B} (-1)^{b + r + s} (2j - b -r - s + 1) C_j\int \Xi_{1 + a + 2b -2j + r + s, a + b + r + s - 2j}(\phi_f),
\end{equation}
where the integrals are over $H(\QQ) Z(\AAA) \backslash H(\AAA)$. 
\end{proposition}

\begin{proof}
    By Proposition \ref{Prop: integration of diff form into adelic integral}, we have 
    \begin{equation*}
        \langle \omega_{\Psi}, [\rho] \rangle_{B} =  \frac{[H(\Zp): K]^{-1}}{(2\pi{i})^{n + 1}} \int_{H(\QQ) Z(\AAA) \backslash H(\AAA) } (p^{*} Eis_{H}(\phi_f) \iota^{*} \omega_{\Psi} (\mathbf{1}))(g) dg,  
    \end{equation*}
    where $\mathbf{1}$ is the vector $v^{+} \otimes v^{-}$.  
    \par It follows from the formula for $Eis_{H}(\phi_f)$ in equation (\ref{eqn:Eis_symbol_Hodge}) and the formula for $\iota^{*} \omega_{\Psi}$ in Proposition \ref{Prop:pullback_form} that for $g \in H(\AAA)$, 
    \begin{align*}
        & (p^{*} Eis_{H}(\phi_f) \iota^{*} \omega_{\Psi})(\mathbf{1})(g) \\
        = & -\frac{(2\pi{i})^{n + 1}}{2(n + 1)} (\sum\limits_{i = 0}^{a + b} (-1)^{i} (i + 1) X^{1 + a + b - i}_{(1, -1, 0)} \Psi_{\infty} \otimes \Psi_f) (g)\\
         \times & (\sum\limits_{\gamma \in B_{2}(\QQ) \backslash \GL_2(\QQ)} \sum\limits_{j = 0}^{n} \gamma^{*} (\phi_{n - 2j}^{n} \otimes \phi_f))(g) \langle X^{i}_{(1, -1, 0)}v, b^{n}_{n-j} \rangle.
    \end{align*}
    The weight of $X^{i}_{(1, -1, 0)}v$ is $\lambda^{\prime}(i - a, -a -b -r -s)$ when restricting to $\GL_{2}$ since its weight  is $(i - a + s - r, -i + b + s - r, s - r; -b - 2s + r)$. It can be seen from the definition of the pairing and the fact that the weight of $b^{n}_{n - j}$ is $\lambda^{\prime}(a + b + r + s - 2j, a + b + r + s)$ 
    that if $a + b + r + s - 2j \neq -(i - a)$, then 
    \begin{equation*}
        \langle X^{i}_{(1, -1, 0)}v, b^{n}_{n-j} \rangle = 0. 
    \end{equation*}
    Thus only the term satisfying $i = 2j - b - r - s$ remains. 
    By the inequality $0 \le i \le a + b$, we have 
    \begin{equation*}
        \frac{b + r + s}{2} \le j \le \frac{a + 2b + r + s}{2}. 
    \end{equation*}
    By $0 \le j \le a + b + r + s$ and the definition of $A$ and $B$, we have $A \le j \le B$. Hence, we get 
    \begin{align*}
        & (p^{*} Eis_{H}(\phi_f) \iota^{*} \omega_{\Psi})(\mathbf{1})(g) \\ 
        =& -\frac{(2\pi{i})^{n + 1}}{2(n + 1)} \sum\limits_{j = A}^{B} (-1)^{b + r + s}(2j -b -r -s + 1) [H(\Zp): K]C_j \\
        \times & (X^{1 + a + 2b -2j + r + s}_{(1, -1, 0)} \Psi_{\infty} \otimes \Psi_f ) \times (\sum\limits_{\gamma \in B_{2}(\QQ) \backslash \GL_2(\QQ)} \gamma^{*} (\phi^{a + b + r + s}_{a + b + r + s -2j} \otimes \phi_f)). 
    \end{align*}
    Plugging this into the integral, we get
    \begin{equation*}
        \langle \omega_{\Psi}, [\rho] \rangle_{B} = -\frac{1}{2(n + 1)} \sum\limits_{j = A}^{B} (-1)^{b + r + s}(2j - b - r -s  + 1)C_j \int \Xi_{1 + a + 2b -2j + r + s, a + b + r + s - 2j}(\phi_f). 
    \end{equation*}
\end{proof}

\begin{remark}
    \begin{enumerate}
        \item  By the inequality $A \le j \le B$, we can see that $2j - b -r - s + 1 > 0$, which is important for later applications. 
        \item  The $K_{\GL_2}$-weight of $\Xi_{1 + a + 2b -2j + r + s, a + b + r + s - 2j}(\phi_f)$ is $0$, so the integral 
        \begin{equation*}
            \int \Xi_{1 + a + 2b -2j + r + s, a + b + r + s - 2j}(\phi_f)
        \end{equation*}
        could be non-zero. 
    \end{enumerate}
\end{remark}

\begin{lemma}
\label{Lemma:C_nonvanish}
    The constant $C_{b + s}$ is non-zero. 
\end{lemma}

\begin{proof}
    First, recall that
    \begin{equation*}
        C_{b + s} = [H(\Zp): K]^{-1} \langle X^{b + s - r}_{(1, -1, 0)}v, b^{a + b + r + s}_{a + r} \rangle. 
    \end{equation*}
    The weight of $X^{b + s - r}_{(1, -1, 0)}v$ is $(b - a + 2(s - r), 0, s - r; -b -2s + r)$, so the weight of it is $\lambda^{\prime}(b + s - (a + r), -(a + b + r + s))$ when restricting to $\GL_2$. By the fact that the weight of the representation $\Sym^{n}V_2^{\vee}$ of $\GL_2$ is multiplicity -free, we have 
    \begin{equation*}
        X^{b + s - r}_{(1, -1, 0)}v|_{\GL_2} = \lambda^{b+s}(v) a^{a + b + r + s}_{a + r},
    \end{equation*}
    where $a^{a + b + r + s}_{a + r}$ is $(b^{a + b + r +s}_{a + r})^{\vee}$ and $\lambda^{b+s}(v) \in \ZZ_{\ge 0}$.
    By the definition of the pairing $\langle \cdot, \cdot \rangle$, in order to show that $C_{b + s} \neq 0$, it suffices to show that 
    \begin{equation*}
         \rho(X^{b + s - r}_{(1, -1, 0)}v) = X^{b + s - r}_{(1, -1, 0)}v|_{\GL_2} \neq 0. 
    \end{equation*}
    Second, if we untwist by the character $\mu^{-b - 2s + r}$,  the weight of $X^{a + r}_{(1, 0, -1)}X^{b + s - r}_{(1, -1, 0)}v$ is $(b + 2s - r, 0, -a + s - 2r)$. After the untwist, the highest weight of 
    \begin{equation*}
        \Sym^{n}V_{2}^{\vee} \hookrightarrow \iota^{*}D
    \end{equation*}
    is also $(b + 2s - r, 0, -a + s - 2r)$, which is the same as the weight of $X^{a + r}_{(1, 0, -1)}X^{b + s - r}_{(1, -1, 0)}v$. \\
    Thirdly, we have the following two formulas: for a representation $V$ of the Lie algebra $\lieg_{\CC}$, for any vector $v \in V$, 
    we have
    \begin{align*}
        X_{(1, 0, -1)}X_{(1, -1, 0)}v &= [X_{(1, 0, -1)}, X_{(1, -1, 0)}]v + X_{(1, -1, 0)}X_{(1, 0, -1)}v \\
                                      &= X_{(1, -1, 0)}X_{(1, 0, -1)}v.
    \end{align*}
    Finally, by the above formula, we have 
    \begin{equation*}
        X_{(1, 0, -1)}(X^{a + r}_{(1, 0, -1)}X^{b + s - r}_{(1, -1, 0)}v) = X^{b + s - r}_{(1, -1, 0)}(X_{(1, 0, -1)}^{a + r + 1} v) = 0,
    \end{equation*}
    where the last equality follows from equation (\ref{eq: Xv = 0}).
    Hence, $X^{a + r}_{(1, 0, -1)}X^{b + s - r}_{(1, -1, 0)}v$ is a highest weight vector of 
    \begin{equation*}
        \Sym^{n}V_{2}^{\vee}  \hookrightarrow \iota^{*}D . 
    \end{equation*}
    By the branching law of $\GL_2 \subset \GL_3$, the highest weight vector is unique up to a non-zero constant, so $X^{a + r}_{(1, 0, -1)}X^{b + s - r}_{(1, -1, 0)}v \neq 0$. Thus
    \begin{equation*}
        \rho(X^{a + r}_{(1, 0, -1)}X^{b + s - r}_{(1, -1, 0)})v = X^{a + r}_{(1, 0, -1)} \rho(X^{b + s - r}_{(1, -1, 0)}v) \neq 0. 
    \end{equation*}
    Hence, we have 
    \begin{equation*}
        \rho(X^{b + s - r}_{(1, -1, 0)}v) \neq 0. 
    \end{equation*}
\end{proof}

\begin{remark}
    \begin{enumerate}
        \item From the computation of an archimedean zeta integral in the next section, we can see that the only non-zero summand in the summation of Lemma \ref{prop: explicit_pairing} is associated to $j = b + s$, so we only need to verify that $C_{b + s}$ is non-zero. 
        \item We can replace $v$ by $C_{b + s}^{-1}v$ so that we replace $\omega_{\Psi}$ by $C_{b + s}^{-1}\omega_{\Psi}$. This will not affect the result because our goal is compute the quotient 
        \begin{equation*}
            \frac{\langle \omega_{\Psi}, \tilde{\nu}_{K} \rangle_{B}}{\langle \omega_{\Psi},  \tilde{\nu}_{D} \rangle_{B}}. 
        \end{equation*}
    \end{enumerate}
\end{remark}

\section{Computation of the integral}
\label{Sec: zeta integral}

\begin{convention}
    \begin{itemize}
        \item In this section, we do not identify $\GL_2(\RR)$ with $\GU(J_2)^{\prime}(\RR)$ automatically. 
        \item In this section, we choose measures as in Convention \ref{convention: measure}. 
        \item Let $n = a + b + r + s$.
    \end{itemize}
\end{convention}

\subsection{The zeta integral of Gelbart and Piatetski-Shapiro}
\label{SS: global zeta integral}
In this subsection, we recall the zeta integral constructed in \cite{Gelbart&PS84} and modified in \cite{PS18}.

\begin{convention}
\begin{itemize}
    \item Let $dt_{\infty}$ be the Lebesgue measure on the additive group $\RR$. If $v$ is a nonarchimedean place of $\QQ$, let $dt_{v}$ be the Haar measure on $\QQ_{v}$ such that $\ZZ_{v}$ has volume one.
    \item 
    Let $d^{\times}t_{v}$ be the measure on $\QQ_{v}^{\times}$ by 
    \begin{equation*}
        d^{\times}t_{v} = \begin{cases}
                             & \frac{dt_{v}}{|t_v|} \quad \text{if} \ v \ \text{is archimedean}, \\
                             & \frac{p}{p - 1}\frac{dt_{v}}{|t_v|} \quad \text{if} \ v \ \text{is nonarchimedean}.  
                          \end{cases}
    \end{equation*}
    \item 
    Let $dt$ (resp., $d^{\times}t$) denote the restricted product measure $\prod_{v}dt_{v}$ on $\AAA$ (resp., the restricted product measure $\prod_{v}d^{\times}t_{v}$ on $\AAA^{\times}$). 
\end{itemize}
\end{convention}

\begin{definition}
    Let $V_{2}$ be the standard representation of $\GL_2$ on row vectors and let $B_2 = T_2U_2$ be the standard Borel of $\GL_2$. Let $\Phi$ be a Schwartz-Bruhat function on $V_{2}(\AAA)$ and let $\nu = (\nu_1, \nu_2) \colon T(\AAA)/T(\QQ) \rightarrow \CC^{\times}$ be a character of the diagonal torus of $\GL_2$. We now define
    \begin{equation*}
        f(g, \Phi, \nu, s) = \nu_{1}(\det(g))|\det(g)|^{s} \int_{\GL_1(\AAA)} \Phi((0, t)g) \nu_1\nu_2^{-1}(t) |t|^{2s} dt
    \end{equation*}
    and set the Eisenstein series on $\GL_2(\AAA)$ to be 
    \begin{equation*}
        E(g, \Phi, \nu, s) = \sum\limits_{\gamma \in B(\QQ)\backslash GL_2(\QQ)} f(\gamma g, \Phi, \nu,s). 
    \end{equation*}
    We write $f(g_1, \Phi, \nu, s)$ and $E(g_1, \Phi, \nu, s)$ as functions on $H(\AAA)$ via the projection $H(\AAA) \twoheadrightarrow \GL_2(\AAA)$. 
\end{definition}
\begin{remark}
    \begin{itemize}
        \item The general theorem \cite[Lemma 4.1]{Langlands76} implies that the Eisenstein series 
        \begin{equation*}
            E(g, \Phi, \nu, s) = \sum\limits_{\gamma \in B_{2}(\QQ)\backslash \GL_2(\QQ)} f(\gamma g, \Phi, \nu, s)
        \end{equation*}
        is absolutely convergent for $\mathrm{Re}(s)$ big enough and satisfies a functional equation. 
        \item By the explicit formula for $f(g, \Phi, \nu, s)$, we can see that the central character is $\nu_1\nu_2$.
    \end{itemize}
\end{remark}

\begin{definition}[Global zeta integral]
\label{def: Zeta_integral}
    For an irreducible cuspidal automorphic representation $\pi$ of $G(\AAA)$ and a cusp form $\varphi$ in the space of $\pi$, we define the global zeta integral as follows: 
    \begin{equation*}
        I(\varphi, \Phi, \nu, s) = \int_{H(\QQ)Z(\AAA) \backslash H(\AAA)} \varphi(g) E(g_1, \Phi, \nu, s) dg. 
    \end{equation*}
\end{definition}

\subsection{Comparison of the integral and the zeta integral}
\label{SS: comparison of integral with zeta integral}

In this subsection, we express the integral 
\begin{equation*}
    \int_{H(\QQ) Z(\AAA) \backslash H(\AAA)} \Xi_{m, t}(\phi_f)(g) dg
\end{equation*}
in Proposition \ref{prop: explicit_pairing} in terms of the zeta integral in Definition \ref{def: Zeta_integral}. 

\begin{proposition}
\label{Proposition_comp_pair_w_zeta_int}
Let $\nu_1, \nu_2 \colon \QQ^{\times} \backslash \AAA^{\times} \rightarrow \CC^{\times}$ be continuous characters and let $s \in \CC$. Let $\chi_{\nu_1, \nu_2,s}$ be the character of $B_{2}(\AAA)$ defined by 
\begin{equation*}
    \chi_{\nu_1, \nu_2, s}(\begin{pmatrix}
                                a & b \\
                                  & d 
                           \end{pmatrix}) = \nu_1(a) \nu_2(d) \left|\frac{a}{d}\right|^{s}. 
\end{equation*}
Then for any Schwartz-Bruhat function $\Phi$ on $\AAA^{2}$, the function $f(g, \Phi, \nu, s)$ on $\GL_2(\AAA)$ defined by 
    \begin{equation}
    \label{eq: def of f}
        f(g, \Phi, \nu, s) = \nu_1(\det(g)) |\det(g)|^{s} \int_{\GL_1(\AAA)} \Phi((0, t)g) \nu_1\nu_2^{-1}(t) |t|^{2s} dt
    \end{equation}
    belongs to $\Ind_{B_2(\AAA)}^{\GL_2(\AAA)}\chi_{\nu_1, \nu_2, s}$, where $\Ind$ means unnormalized induction. 
\end{proposition}

\begin{proof}
    The proposition follows from the following computations:
    \begin{equation*}
         f(\begin{pmatrix}
            1 & b \\
              & 1 \\
          \end{pmatrix}g, \Phi, \nu, s) = f(g, \Phi, \nu, s), 
    \end{equation*}
    \begin{align*}
        f(\begin{pmatrix}
              a_1 &     \\
                  & d_1 \\
           \end{pmatrix}g, \Phi, \nu, s) &= \nu_1(a_1d_1) |a_1d_1|^{s} \nu_1\nu_2^{-1}(d_1^{-1}) |d_1|^{-2s}  f(g, \Phi, \nu, s) \\
                                         &= \nu_1(a_1) \nu_2(d_1)  \left|\frac{a_1}{d_1}\right|^{s} f(g, \Phi, \nu, s). 
    \end{align*}
\end{proof}

\begin{proposition}
\label{Prop: choice of v1,v2}
    Let $\nu_1$ be the Hecke character  $||^{-n}\nu$ and $\nu_2$ be the trivial Hecke character, where $\nu$ is a finite-order Hecke character of sign $(-1)^n$. Then the following statements hold.
    \begin{itemize}
        \item The nonarchimedean part $\chi_{f}$ of $\chi_{\nu_1, \nu_2, 1 + n}$ satisfies 
        \begin{equation*}
            \chi_{f}(\begin{pmatrix}
                        a\alpha & b \\
                                & d\delta \\
                    \end{pmatrix}) = a^{-1}d^{n + 1} \nu(\alpha), 
        \end{equation*}
        for any $a, d \in \QQ_{+}^{\times}$ and $\alpha, \delta \in \widehat{\ZZ}^{\times}$. 
        \item The restriction of the archimedean part $\chi_{\infty}$ of $\chi_{\nu_1, \nu_2, 1 + n}$ to the identity component of the diagonal maximal torus of $\GL_2(\RR)^{+}$ is identified with $\lambda(n + 2, -n)$ (see notation \ref{Notation: character on GL2}) through the isomorphism $\GL_2(\RR)^{+} \cong \GU({J_2})^{\prime}(\RR)^{+}$. 
    \end{itemize}
\end{proposition}

\begin{proof}
    For the nonarchimedean part, we have 
    \begin{equation*}
         \chi_{f}(\begin{pmatrix}
                        a\alpha & b \\
                                & d\delta \\
                    \end{pmatrix}) = |a|_{f}^{-n}\nu(\alpha) |\frac{a}{d}|_{f}^{1 + n} = a^{-1}d^{n + 1} \nu(\alpha).
    \end{equation*}
    For the archimedean part, for $\begin{pmatrix}
                                     a & \\
                                       & a^{-1}v \\
                                  \end{pmatrix} \in \GL_2(\RR)^{+}$, 
    we have 
    \begin{equation*}
        \chi_{\infty}(\begin{pmatrix}
                        a &  \\
                                & a^{-1}v \\
                    \end{pmatrix}) = |a|_{\infty}^{-n}|a^{2}v^{-1}|_{\infty}^{1 + n} = a^{n + 2} v^{\frac{-n - (n + 2)}{2}}.
    \end{equation*}
    Hence, the action of $\chi_{\infty}$ is $\lambda(n + 2, -n)$ via the isomorphism $\GL_2(\RR)^{+} \cong \GU({J_2})^{\prime}(\RR)^{+}$. 
\end{proof}

\begin{proposition}
\label{Prop: identify f with phi}
    Let $\nu_1$ be the Hecke character  $||^{-n}\nu$ and $\nu_2$ be the trivial Hecke character, where $\nu$ is a finite-order Hecke character of sign $(-1)^n$.
    Let $\Phi$ be a factorizable Schwartz-Bruhat function $\Phi = \otimes_{v}^{\prime} \Phi_{v}$ on $V_{2}(\AAA)$ such that 
    \begin{enumerate}
        \item The function $\Phi_{f} = \otimes_{v < \infty}^{\prime} \Phi_{v}$ is $\overline{\QQ}$-valued, 
        \item The function $\Phi_{\infty}(x, y)$ is defined by
        \begin{equation*}
                  \Phi_{\infty}(x, y) := (ix - y)^{\frac{n - k}{2}}(ix + y)^{\frac{n + k}{2}} e^{-\pi(x^2 + y^2)}. 
        \end{equation*}
    \end{enumerate}
    Then there exists a $\phi_f \in \mathcal{B}_{n, \overline{\QQ}}$ such that for any $g \in \GL_{2}(\AAA)$, we have 
    \begin{equation*}
        (\phi_{k}^{n} \otimes \phi_f)(g) = (-1)^{\frac{a + b + r + s - k}{2}} 2i \pi^{1 +a + b + r + s} \Gamma(1 + a + b + r + s)^{-1} f(g, \Phi, \nu, 1 + a + b + r +s), 
    \end{equation*}
    where the $\phi_{k}^{n}$ on the left-hand side the pullback of the function $\phi_{k}^{n}$ defined in Definition \ref{def: function phi} through the isomorphism $\GL_2(\RR) \cong \GU(J_2)^{\prime}(\RR)$ and $f$ is defined in equation (\ref{eq: def of f}). 
\end{proposition}

\begin{proof}
    \par First, we can factor $f(g, \Phi, \nu, s)$ into an Euler product of local Tate integrals 
    \begin{equation*}
        f(g, \Phi, \nu, s) = \otimes_{v}^{\prime} f(g_{v}, \Phi_{v}, \nu_{v}, s), 
    \end{equation*}
    where 
    \begin{equation*}
        f(g_{v}, \Phi_{v}, \nu_{v}, s) = \nu_{1, v}(\det(g_{v}))|\det(g_{v})|^{s} \int_{\GL_1(\QQ_{v})} \Phi((0, t_v)g_{v}) (\nu_{1, v}\nu_{2, v})^{-1}(t_v) |t_v|_{v}^{2s} dt_v. 
    \end{equation*}
    \par Second, it follows from Proposition \ref{Proposition_comp_pair_w_zeta_int}, Proposition \ref{Prop: choice of v1,v2} and Defintion \ref{Def: In(v)} that 
         \begin{equation*}
             f^{\infty}(g, \Phi, \nu, 1 + n) = \otimes_{v < \infty}^{\prime} f(g_{v}, \Phi_{v}, \nu_{v}, 1 + n)
         \end{equation*}
         belongs to $I_n(\nu) \subset \mathcal{B}_{n, \overline{\QQ}}$ for $g = (g_{v})_{v < \infty} \in G(\AAA_f)$. Hence, there exists a $\phi_f \in \mathcal{B}_{n, \overline{\QQ}}$ such that 
        \begin{equation*}
            \phi_f = f^{\infty}(g, \Phi, \nu, 1 + n)
        \end{equation*}
        for $g \in G(\AAA_f)$. 
    \par Third, it can seen from the fact that $\phi^{n}_{k}$ (see Proposition \ref{prop: phi equivariant}) and $\Phi_{\infty}$ (see Lemma \ref{Lemma: arch_weight of Phi}) has $\U(1)(\RR)$-weight $-k$ and $\Ind_{B_2(\RR)^{+}}^{\GL_2(\RR)^{+}} \lambda(2 + n, -n)$ is $\U(1)(\RR)$-multiplicity-free that there is a $c \in \CC$ such that for any $g \in G(\RR)$, 
    \begin{equation*}
        \phi^{n}_{k}(g) = c \cdot f(g, \Phi_{\infty}, \nu_{\infty}, 1 + n),
    \end{equation*}
    where $\nu_{\infty} = (\nu_{1, \infty} = ||^{-1}\mathrm{sgn}^{n}, \nu_{2, \infty} = 1)$ is the archimedean component of $\nu = (\nu_1, \nu_2)$. 
    Since $\phi_{k}^{n}(1) = 2i$ (see Proposition \ref{def: function phi}) and
    \begin{equation*}
        f(1, \Phi_{\infty}, \nu_{\infty}, 1+n) = (-1)^{\frac{a + b + r + s - k}{2}} \pi^{-1 -a - b - r - s} \Gamma (1 + a + b + s)
    \end{equation*}
    (see Lemma \ref{Lemma: arch f formula}), 
    we get
    \begin{equation*}
        c = (-1)^{\frac{a + b + r + s - k}{2}} 2i \pi^{1 +a + b + r + s} \Gamma(1 + a + b + r + s)^{-1}. 
    \end{equation*}
\end{proof}

\begin{corollary}
\label{corollary: identify integral of Xi and I}
    There exists a $\phi_f \in \mathcal{B}_{n, \overline{\QQ}}$ such that the integral
    \begin{equation*}
        \int_{H(\QQ) Z(\AAA) \backslash H(\AAA)} \Xi_{m, k}(\phi_f)(g) dg
    \end{equation*}
    defined in notation \ref{notation: Xi} is equal to 
    \begin{equation*}
        (-1)^{\frac{a + b + r + s - k}{2}} 2i \pi^{1 +a + b + r + s} \Gamma(1 + a + b + r + s)^{-1} I(\varphi, \Phi, \nu, 1 + a + b + r + s),
    \end{equation*}
    where $\varphi = X^{m}_{(1, -1, 0)} \Psi$ and the choice of $\Phi$ and $\nu$ is explained in Proposition \ref{Prop: identify f with phi}. 
\end{corollary}

\begin{proof}
    This is direct from Proposition \ref{Prop: identify f with phi}. 
\end{proof}

\subsection{The unfolding}
\label{SS: unfolding}
In this subsection, we unfold the zeta integral defined in Subsection \ref{SS: global zeta integral} and factor it into the product of the standard automorphic $L$-function and a local integral of Whittaker functions. 

\begin{definition}
\label{def: Whittaker function}
    \begin{enumerate}
        \item (Global Whittaker functionals) We define a non-zero global Whittaker functional $\Lambda \colon \pi \rightarrow \CC$ on an irreducible cuspidal automorphic representation $\pi$ by the following integral: for $\varphi \in \pi$, 
        \begin{equation*}
            \Lambda(\varphi) := \int_{U_{B}(\QQ)\backslash U_B(\AAA)} \chi^{-1}(u) \varphi(u) du. 
        \end{equation*}
        Here, $U_{B}$ is the unipotent radical of the upper-triangular Borel subgroup of $G$ and define the character $\chi \colon U_{B}(\QQ) \backslash U_{B}(\AAA) \rightarrow \CC^{\times}$ by $\chi(u) = \psi(\Tr_{E/\QQ}(\delta^{-1}u_{23}))$, where $\psi$ is the standard additive character $\psi \colon \QQ\backslash \AAA \rightarrow \CC^{\times}$ and $u_{23}$ is the entry of $u$ in the position of the second row and the third column. 
        \item (Global Whittaker functions) For a cusp form $\varphi$ in the space of $\pi$, denote by
        \begin{equation*}
             W_{\varphi}(g) = \Lambda(\pi(g) \varphi) = \int_{U_{B}(\QQ)\backslash U_B(\AAA)} \chi^{-1}(u) \varphi(ug) du
        \end{equation*}
        the global Whittaker function associated to $\varphi$. 
        \item (Local Whittaker functions)
        We assume that $\varphi = \otimes_{v}^{\prime} \varphi_{v}$ is a pure tensor in the decomposition $\pi = \otimes_{v}^{\prime} \pi_{v}$ and such that $\Lambda(\varphi) = W_{\varphi}(1) \neq 0$. Then for each place $v$ of $\QQ$, we define a local Whittaker functional $W_{\varphi, v}\colon G(\QQ_{v}) \rightarrow \CC$ by setting $W_{\varphi, v}(g_{v}) = \Lambda(\pi(g_{v}) \varphi) / \Lambda(\varphi)$ for $g_v \in G(\QQ_v)$ and $\varphi \in \pi$. We then have 
        \begin{equation*}
            W_{\varphi}(g) = W_{\varphi}(1) \prod_{v} W_{\varphi, v}(g_v), 
        \end{equation*}
        where the product runs over all the places ${v}$ of $\QQ$. 
    \end{enumerate}
\end{definition}

\begin{convention}
    \begin{itemize}
        \item Let $\omega_{\pi}\colon Z_{G}(\QQ) \backslash Z_{G}(\AAA) \rightarrow \CC^{\times}$ be the central character of $\pi$. For later sections, we assume that $\omega_{\pi}|_{Z(\AAA)} = (\nu_1 \nu_2)^{-1}$. 
        \item When $\varphi$ is clear, we use $W_v$ to denote $W_{\varphi, v}$. 
    \end{itemize}
\end{convention}

\begin{definition}[Local zeta integral]
\label{def: local zeta integral}
    Let $U_2$ be the upper-triangular unipotent subgroup $\{\begin{pmatrix}
                                                               1 & * \\
                                                                 & 1
                                                            \end{pmatrix}\}$ of $\GL_2$. For a Whittaker function $W_{v}$ on $G(\QQ_{v})$ and a Schwartz-Bruhat function $\Phi_{v}$ on $V_{2}(\QQ_v)$, 
    we define the local zeta integral as
    \begin{align}
    \label{eqn: def of local I}
        I_{v}(W_{v}, \Phi_{v}, \nu_v, s) &:= \int_{Z(\QQ_{v})U_{2}(\QQ_v) \backslash H(\QQ_v)} f(g_{1, v}, \Phi_{v}, \nu_{v}, s) W_{v}(g_v) dg_{v}, \\
                                         &= \int_{U_{2}(\QQ_v) \backslash H(\QQ_v)} \nu_1(\det(g_{1, v})) \Phi_{v}((0,1)g_{1,v}) W_{v}(g_v) |\det(g_{1,v})|_{v}^{s}dg_{v}, 
    \end{align}
    where $g_{1,v}$ is projection of $g_v$ onto $\GL_2(\QQ_v)$ and 
    \begin{equation*}
        f(g_{1,v}, \Phi_{v}, \nu_{v}, s) = \nu_{1}(\det(g_{1,v}))|\det(g_{1,v})|^{s} \int_{\GL_1(\QQ_{v})} \Phi((0, t)g_{1,v}) (\nu_1\nu_2)^{-1}(t) |t|_{v}^{2s} dt. 
    \end{equation*}
\end{definition}

We now recall the definition of the standard automorphic $L$-function attached to $\pi$. 

\begin{definition}[Standard automorphic $L$-function]
\label{def: L-function}
\begin{itemize}
    \item The Langlands dual group of $G$ is $\widehat{G} = \Gm(\CC) \times \GL_3(\CC)$, with an action of the non-trivial element $\sigma \in \Gal(E/\QQ)$ given by 
    \begin{equation*}
        \sigma \colon (x, g) \in \Gm(\CC) \times \GL_3(\CC) \mapsto (x \det g, \Phi_{3}^{-1} {^{t}g^{-1}} \Phi_{3}) \in \Gm(\CC) \times \GL_3(\CC),
    \end{equation*}
    where 
    \begin{equation*}
        \Phi_{3} = \begin{pmatrix}
                        0 & 0 & 1 \\
                        0 & -1 & 0 \\
                        1 & 0 & 0
                   \end{pmatrix}. 
    \end{equation*}
    The $L$-group of $G$ is ${^{L}G = \widehat{G} \rtimes \Gal(E/\QQ)}$. 
    \item Define $G^{\prime} = \Res_{E/\QQ}(\Gm \times \GL_3)$. The Langlands dual group of $G^{\prime}$ is $\widehat{G^{\prime}} = \widehat{G} \times \widehat{G}$ with an action of the non-trivial element $\sigma \in \Gal(E/\QQ)$ given by 
    \begin{equation*}
        \sigma \colon (g, h) \in \widehat{G^{\prime}} \mapsto (\sigma(h), \sigma(g)) \in \widehat{G^{\prime}}, 
    \end{equation*}
    where $\sigma(g)$ and $\sigma(h)$ is the action of $\sigma$ on $\widehat{G}$ defined before. Define the $L$-group of $G^{\prime}$ as ${^{L}G^{\prime}} = \widehat{G^{\prime}} \rtimes \Gal(E/\QQ)$. 
    \item It can be seen from the definition that we have the diagonal embedding of Langlands dual groups $\widehat{G} \rightarrow \widehat{G^{\prime}}$, and we can extend the map to the $L$-groups ${^{L}G} \rightarrow {^{L}G^{\prime}}$. We define the standard representation $\mathrm{std} \colon {^{L}G^{\prime}} \rightarrow \GL_6(\CC)$ as 
    \begin{align*}
        & ((x_1, g_1), (x_2, g_2)) \rtimes 1 \mapsto \begin{pmatrix}
                                             x_1 g_1 & \\
                                                     & x_2 \det(g_2) \Phi_3 {^{t}g_{2}^{-1} \Phi_{3}^{-1}}
                                           \end{pmatrix}, \\
        & 1 \rtimes \sigma \mapsto \begin{pmatrix}
                                     & 1_{3} \\
                                     1_{3} &       \\
                                   \end{pmatrix}. 
    \end{align*}
    \item For unramified $\pi_{p}$, using the Satake transform of Cartier \cite{Cartier79} or Haines\textendash Rostami \cite{Haines-Rostami10}, the well-defined local standard $L$-function $L_{p}(s, \pi_{p}, \mathrm{std})$ is a meromorphic function of $s \in \CC$. 
    \item For a Hecke character $\nu \colon \QQ^{\times} \backslash \AAA^{\times} \rightarrow \CC^{\times}$, we define the twisted standard $L$-function $L_{p}(s, \pi_{p} \times \nu_{p}, \mathrm{std})$ of $\pi \times \nu$ as the standard $L$-function of $(\nu \circ \mu) \pi$, where $(\nu \circ \mu) \pi$ is the product of the representation $\pi$ by the automorphic character $\nu \circ \mu$ of $G$, where $\mu$ is the similitude character of $G$. 
    \end{itemize}
\end{definition}

We now state a result regarding twisted base change that can be used to define $L$-factors at all places. 

\begin{proposition}[\cite{Roga90}, \S 13]
\label{Prop: twisted base change}
    Let $\pi$ be a cuspidal automorphic representation of $G$ over $\QQ$ such that $\pi_{\infty}$ is a discrete series. Then there exists a unique automorphic representation $\tau$ on $\Res_{E/\QQ}(\Gm \times \GL_3)$ that is compatible with $\pi$ at every place $v$ of $\QQ$ such that 
    \begin{itemize}
        \item $v$ splits in $E$, or
        \item $v$ is inert in $E$ and $\pi_v$ is unramified, or 
        \item $v$ is ramified in $E$ and $\pi_ v$ has a vector fixed by $G(\Zp)$. 
    \end{itemize}
\end{proposition}
\begin{remark}
    \begin{itemize}
        \item  The results in \cite{Roga90} are only for the ordinary unitary group. Results for the similitude setting are also constructed in \cite{Morel_Coh10} and \cite{Shin_BC14}. 
        \item Since $\tau$ is uniquely determined by $\pi$, one can use $\tau$ to make a well-defined $L$-factor at \textit{all} places by using the $L$-group representation in Definition \ref{def: L-function} and the known local $L$-parameters of representations of $\Res_{E/\QQ}(\Gm \times \GL_3)$. 
    \end{itemize}
\end{remark}

\begin{proposition} \textnormal{\cite{Gelbart&PS84, Baruch97} \cite[Proposition 3.11.]{PS18}}
\label{Prop:factorize}
With the factorizable data and notation as above,  we have 
\begin{equation*}
    I(\varphi, \Phi, \nu, s) = W_{\varphi}(1) \prod_{v} I_{v}(W_{v}, \Phi_{v}, \nu_v, s). 
\end{equation*}
For all finite places $p$ of $\QQ$, the local integral is absolutely convergent for $\mathrm{Re}(s) \gg 0$ and has meromorphic continuation to a rational function of $p^{-s}$.
If the finite place $p$ is not $2$, $p$ is unramified in $E$, $\pi_{p}$ and $\nu_{p}$ are unramified, $\varphi$ is fixed by the maximal compact subgroup $G(\ZZ_{p})$ and $\Phi_{p}$ is the characteristic function of $V_2(\ZZ_{p})$, then 
\begin{equation*}
    I_{p}(W_{p}, \Phi_{p}, \nu_{p}, s) = L_{p}(s, \pi_{p} \times \nu_{1, p}, \mathrm{std}).
\end{equation*}
If we let $S$ be the completement set of the above finite places, then we have 
\begin{align*}
    I(\varphi, \Phi, \nu, s) &= W_{\varphi}(1) \prod_{p \in S}  I_{p}(W_{p}, \Phi_{p}, \nu_{p}, s) \times I_{\infty}(W_{\infty}, \Phi_{\infty}, \nu_{\infty}, s) \times \prod_{p \notin S, p < \infty} L_{p}(s, \pi_{p} \times \nu_{1, p}, \mathrm{std}) \\
                             &= W_{\varphi}(1) \prod_{p \in S} \frac{I_{p}(W_{p}, \Phi_{p}, \nu_{p}, s)}{L_{p}(s, \pi_{p} \times \nu_{1, p}, \mathrm{std})} 
                             \times I_{\infty}(W_{\infty}, \Phi_{\infty}, \nu_{\infty}, s) \times L(s, \pi \times \nu_{1}, \mathrm{std}).
\end{align*}

\end{proposition}

\subsection{The nonarchimedean local zeta integral}
\label{SS: the nonarch integral}
In this subsection, we recall results about algebraicity of the nonarchimedean local zeta-integral $I_{p}(W_{p}, \Phi_{p}, \nu_{p}, s)$ and the nonarchimedean local $L$-factor $L_{p}(s, \pi_{p} \times \nu_{1, p}, \mathrm{std})$.

\begin{proposition}[\cite{PS18}, Proposition 7.1]
\label{prop: PS18 proposition 7.1}
    Suppose that the Schwartz-Bruhat function $\Phi_p$ on $\Qp^{2}$ and the Whittaker function $W_p$ are both $\overline{\QQ}$-valued. Then the local integral $I_{p}(W_{p}, \Phi_{p}, \nu_{p}, s)$ has a meromorphic continuation to a rational function $\frac{P(X)}{Q(X)}$ of $X = p^{-s}$, where $P$ and $Q$ have algebraic coefficients. \\ 
    In particular, the meromorphic continuation of $I_{p}(W_{p}, \Phi_{p}, \nu_{p}, s)$ to any $s \in \QQ$ is $\overline{\QQ}$-valued if it is finite. 
\end{proposition}

We now address the hypothesis that $W_{p}$ is $\overline{\QQ}$-valued in the previous proposition. In order to use it in the global setting later, we work in the general setting. 

\begin{definition}
\label{def: field of def of a rep}
    Assume that $\pi$ is a representation of a group $G$ on a $\CC$-vector space $V$. We say that $\pi$ is defined over $L \subset \CC$ if there exists a representation $\pi_{0}$ on a $L$-vector space $V_{0}$ such that $V_{0} \otimes_{L} \CC$ is isomorphic to $\pi$. We call $V_{0}$ a model of $V$ over $L$. 
\end{definition}

\begin{proposition}[\cite{HS18}, Proposition 3.3.2] 
\label{prop: rational whittaker}
    Suppose that a representation $\pi$ of a group $G$ is defined over $L$ and let $V$ and $V_{0}$ be as in Definition \ref{def: field of def of a rep}. Assume that $\Lambda\colon V \rightarrow \CC$ is a non-zero functional on $V$ that has a left transformation property with respect to a subgroup $H$ and a character $\psi \colon H \rightarrow L^{\times}$ that characterizes it uniquely up an element of $\CC^{\times}$. (For instance, $\Lambda$ can be a Whittaker functional.) Then there exists a functional $\Lambda_{0} \colon  V \rightarrow \CC$ with $\mathrm{Im}(\Lambda_{0}|_{V_{0}}) \subseteq L$. 
\end{proposition}

\begin{notation}
    Let $\Lambda: \pi_{p} \rightarrow \CC$ be a fixed Whittaker functional of $\pi_p$. If $\pi_{p}$ is defined over a number field, we can assume that $\Lambda$ is $\overline{\QQ}$-valued  by Proposition \ref{prop: rational whittaker}. For any $\varphi \in \pi_{p}$, we define the Whittaker function $W_{\varphi}$ by 
    \begin{equation*}
        W_{\varphi}(g) = \Lambda(g \varphi),
    \end{equation*}
    for any $g \in G(\QQ_{p})$. 
    (Note that this is slightly different from Definition \ref{def: Whittaker function}, which we use since we will vary our choice of $\varphi$.)
\end{notation}

\begin{proposition}
\label{Prop:algebracity}
Let $\pi_{p}$ be a representation of $G(\Qp)$ defined over a number field $L$. For any $\overline{\QQ}$-valued Schwartz-Bruhat function $\Phi_{p}$ on $V_{2}(\Qp)$ and $\varphi_{p}$ in a model of $\pi_p$ over $L$ such that $\Lambda(\varphi_{p}) \neq 0$,  there exists a function $\eta$ in the Hecke algebra $C_{c}^{\infty}(G(\Qp), \QQ)$ such that 
\begin{equation*}
     I_{p}(W_{\pi_{p}(\eta)\varphi_p}, \Phi_p, \nu_p, s) \in \overline{\QQ}^{\times},
\end{equation*}
for any $s \in \QQ$. 
\end{proposition}

\begin{proof}
The algebraicity follows from Proposition \ref{prop: PS18 proposition 7.1} and the nonvanishing is from \cite[Lemma 7.4]{PS18}. 
\end{proof}

\begin{definition}
    For any finite place $p$, we could also define the local $L$-factor $L_{p}(s, \pi_{p} \times \nu_{1, p}, \mathrm{std})$ to be the greatest common divisor of the zeta integrals $I_{p}(W_{p}, \Phi_{p}, \nu_{p}, s)$ as $\Phi_{p}$ and $W_{p}$ vary, whose existence has been proved in \cite{Baruch97}. 
\end{definition}

The following corollary is a direct consequence of Proposition \ref{Prop:algebracity}. 

\begin{corollary}
\label{corollary: local L-factor algebraicity}
The local $L$-factor $L_{p}(s, \pi_{p} \times \nu_{1, p}, \mathrm{std})$ is equal to 
\begin{equation*}
    L_{p}(s, \pi_{p} \times \nu_{1, p}, \mathrm{std}) = \frac{1}{Q(p^{-s})},
\end{equation*}
where $Q(X)$ is a polynomial with $\overline{\QQ}$-coefficients and constant $1$. 
Hence, for any $s \in \QQ$, we have 
\begin{equation*}
    L_{p}(s, \pi_{p} \times \nu_{1, p}, \mathrm{std}) \in \overline{\QQ}^{\times}. 
\end{equation*}
\end{corollary}

\subsection{The archimedean local zeta integral}
\label{SS: arch local integral}
In this subsection, we compute the archimedean zeta integral for some ``nice'' test vectors. 

\begin{convention}
    \begin{itemize}
        \item  Let $\nu_{1, \infty}, \nu_{2, \infty} \colon \RR^{\times} \rightarrow \CC^{\times}$ be the characters defined by 
        \begin{align*}
            & \nu_{1, \infty}(x) = |x|^{-n} \sgn(x)^{n}, \\
            & \nu_{2, \infty}(x) = 1,
        \end{align*}
        and let $\nu_{\infty} = (\nu_{1, \infty}, \nu_{2, \infty})$. 
        \item For $(x, y) \in V_{2}(\RR)$, we let  
              \begin{equation}
                  \Phi_{\infty}(x, y) := (ix - y)^{\frac{n - k}{2}}(ix + y)^{\frac{n + k}{2}} e^{-\pi(x^2 + y^2)}. 
              \end{equation}
        \item Let $\pi_{\infty}$ be a discrete series representation of $G(\RR)^{+}$ with Blattner parameter 
            \begin{equation*}
                (1+s-r+b, -1-a+s-r, s-r). 
            \end{equation*}
        \item Let $\tilde{\pi}_{\infty}$ be the contragredient of $\pi_{\infty}$, which is a discrete series representation with Blattner parameter $\mu = (1 + a + r -s, -1 -b + r - s, r - s)$. 
        \item Let $\Psi_{\infty} \in \tilde{\pi}_{\infty}$ be a vector of weight $(-1 -b + r - s, 1 + a + r - s, r - s; b + 2s -r -2)$, which is the lowest weight vector of the minimal $K_{G}$-type $\tau_{(1 + a + r -s, -1 - b + r  + s, r - s)}$ of $\tilde{\pi}_{\infty}$. 
    \end{itemize}
\end{convention}

\begin{lemma}
\label{Lemma: arch_weight of Phi}
    The function $\Phi_{\infty}(x, y)$ has $\U(1)(\RR)$-weight $-k$. 
\end{lemma}

\begin{proof}
    For any $g = \begin{pmatrix}
                    a & b \\
                    -b & a 
                  \end{pmatrix} \in \U(1)(\RR) \hookrightarrow \GL_2(\RR)$, 
    we have 
    \begin{equation*}
        (x, y) g = (ax - by, bx + ay). 
    \end{equation*}
    We can see from this that: 
    \begin{equation*}
        \Phi_{\infty}((x, y) \begin{pmatrix}
                                a & b \\
                                -b & a
                             \end{pmatrix}) = (a + ib)^{-k} \Phi_{\infty}(x, y). 
    \end{equation*}
    Hence, the function $\Phi_{\infty}(x, y)$ has $\U(1)(\RR)$-weight $-k$. 
\end{proof}

\begin{lemma}
\label{Lemma: arch f formula}
    For any $g = \begin{pmatrix}
                    a & b \\
                    c & d 
                \end{pmatrix} \in \GL_2(\RR)$, the section $f(g, \Phi_{\infty}, \nu_{\infty}, z)$ is equal to 
    \begin{equation*}
        \sgn(\det(g))^{n}|\det(g)|^{z - n} (ic - d)^{\frac{n - k}{2}}(ic + d)^{\frac{n + k}{2}}\pi^{-z}(c^2 + d^2)^{-z} \Gamma(z)
    \end{equation*}
    for $z \in \CC$. 
\end{lemma}

\begin{proof}
     It follows from the identity  $(0, t) \begin{pmatrix}
                                                a & b \\
                                                c & d 
                                          \end{pmatrix} = (ct, dt)$ that 
                            \begin{align*}
                                \Phi_{\infty}(ct, dt) &= (ict - dt)^{\frac{n - k}{2}}(ict + dt)^{\frac{n + k}{2}}e^{-\pi(c^2 + d^2)t^2}\\
                                                      &= (ic - d)^{\frac{n - k}{2}}(ic + d)^{\frac{n + k}{2}}t^ne^{-\pi(c^2 + d^2)t^2},
                            \end{align*}
            so we have 
            \begin{align*}
                f(g, \Phi_{\infty}, \nu_{\infty}, z) &= \sgn(\det(g))^{n}|\det(g)|^{z - n}(ic - d)^{\frac{n - k}{2}}(ic + d)^{\frac{n + k}{2}} \int_{\RR^{\times}}
                e^{-\pi(c^2 + d^2)t^2}  t^{n} |t|^{2z - n} \sgn(t)^{n} \frac{dt}{t} \\
                                                &= 2\sgn(\det(g))^{n}|\det(g)|^{z - n}(ic - d)^{\frac{n - k}{2}}(ic + d)^{\frac{n + k}{2}} \int_{0}^{\infty} e^{-\pi(c^2 + d^2)t^2} t^{2z} \frac{dt}{t}. 
            \end{align*}
            It follows from a change of variables that 
            \begin{equation*}
                2  \int_{0}^{\infty} e^{-\pi(c^2 + d^2)t^2} t^{2z} \frac{dt}{t} = \left(\frac{1}{\pi(c^2 + d^2)}\right)^{z} \Gamma(z), 
            \end{equation*}
            where $\Gamma(z)$ is the Gamma function. 
            So we get
            \begin{equation*}
                f(g, \Phi_{\infty}, \nu_{\infty}, z) = \sgn(\det(g))^{n} |\det(g)|^{z - n} (ic - d)^{\frac{n - k}{2}}(ic + d)^{\frac{n + k}{2}}\pi^{-z}(c^2 + d^2)^{-z} \Gamma(z). 
            \end{equation*}
\end{proof}

\begin{notation}
    \begin{itemize}
        \item If we use the convention in \cite[page 965]{Koseki-Oda95}, we could write the Blattner parameter of $\tilde{\pi}_{\infty}$ as 
             \begin{equation*}
                 \mu = (1 + a, -1 - b). 
             \end{equation*}
        \item Let $d = 2 + a + b$. 
        \item It follows from the Iwasawa decomposition $G(\RR) = U_{B}(\RR)T_{B}(\RR) K_{G}$ that the Whittaker function with a fixed $K_{G}$-type is determined by its value on 
        \begin{equation*}
           T_{B}(\RR) =  \begin{pmatrix}
                            t & 0 & 0 \\
                            0 & 1 & 0 \\
                            0 & 0 & t^{-1}
                        \end{pmatrix}, 
        \end{equation*}
        for all $t \in \RR^{\times}$. 
        \item Recall that we have uniqueness of Whittaker models for $G(\RR)$ (see \cite[\S 3.2]{Koseki-Oda95}). Let $C_{m}^{(1+a, -1-b)}$ be the Whittaker function defined  on $\RR^{\times}$ in \textnormal{\cite[P974]{Koseki-Oda95}} for $0 \le m \le d$, which is associated to the vector $\nu_{d - m}$ in the $K_{G}$-type $\tau_{\mu}$ (see subsubsection \ref{SSS: repn of compact Lie group}). 
    \end{itemize}
\end{notation}

\begin{definition}
\label{Def: classical Whittaker function}
    For an integer $l$, we define the classical Whittaker function $W_{0, l}$ on $\RR^{\times}$ by 
             \begin{equation}
             \label{def:Whittaker_function}
                 W_{0, l}(x) = \frac{e^{-\frac{1}{2}x}}{\Gamma(\frac{1}{2} + |l|)} \int_{0}^{\infty} t^{-\frac{1}{2} + |l|}(1 + \frac{t}{x})^{-\frac{1}{2} + |l|} e^{-t} dt. 
             \end{equation}
\end{definition}

\begin{remark}
    We have the following identity: 
    \begin{equation*}
        W_{0, l}(x) = \pi^{-\frac{1}{2}} x^{1/2} K_{l}(\frac{x}{2}), 
    \end{equation*}
    where $K_{l}(\cdot)$ is the modified $K$-Bessel function. 
\end{remark}

\begin{lemma}
\label{Lemma: relation C and W}
    The function $C_{m}^{(1+a, -1-b)}$ is related to a classical Whittaker function by
             \begin{equation*}
                 C_m^{(1+a, -1-b)}(t) = (-i)^{2 + a + b -m}t^{a + b + \frac{7}{2}} W_{0, m-(1+a)}(8\sqrt{2} \pi D^{-\frac{3}{4}}t), t \in \RR^{\times}
             \end{equation*}
\end{lemma}

\begin{proof}
    In this section, we use the Hermitian form $J$ to define the unitary group, but in \cite{Koseki-Oda95} the Hermitian form $J^{\prime}$ is used. The corresponding transform is 
    \begin{equation*}
        g \in \GU(J^{\prime})(\RR) \mapsto CgC^{-1} \in \GU(J)(\RR), 
    \end{equation*}
    where $C = \begin{pmatrix}
                  \frac{D^{\frac{1}{4}}}{\sqrt{2}} & 0 & \frac{D^{\frac{1}{4}}}{\sqrt{2}} \\
                   0 & 1 & 0 \\
                   -\frac{D^{\frac{1}{4}}}{\sqrt{-2}} & 0 & \frac{D^{\frac{1}{4}}}{\sqrt{-2}}
              \end{pmatrix}$.
    \par Recall that in the setup of \cite{Koseki-Oda95}, $E_{2, +} = \begin{pmatrix}
                                                                     0 & -1 & 0 \\
                                                                     1 & 0 & -1 \\
                                                                     0 & -1 & 0
                                                                    \end{pmatrix}$ and $E_{2, -} = \begin{pmatrix}
                                                                                                 0 & -i & 0 \\
                                                                                                 -i & 0 & i \\
                                                                                                 0 & -i & 0
                                                                                                \end{pmatrix}$.
    So in our setting, we have 
    \begin{align*}
        & E_{2, +} = C \begin{pmatrix}
                          0 & -1 & 0 \\
                          1 & 0 & -1 \\
                          0 & -1 & 0
                     \end{pmatrix} C^{-1} = \begin{pmatrix}
                                                0 & -\sqrt{2}D^{\frac{1}{4}} & 0 \\
                                                0 & 0 & -\sqrt{-2}D^{-\frac{1}{4}} \\
                                                0 & 0 & 0
                                             \end{pmatrix}, \\
         & E_{2, -} = C \begin{pmatrix}
                          0 & -i & 0 \\
                          -i & 0 & i \\
                          0 & -i & 0
                     \end{pmatrix} C^{-1} = \begin{pmatrix}
                                                0 & -\sqrt{2}D^{\frac{1}{4}} & 0 \\
                                                0 & 0 & -\sqrt{2}D^{-\frac{1}{4}} \\
                                                0 & -\sqrt{2}D^{\frac{1}{4}} & 0
                                             \end{pmatrix}.                                    
    \end{align*}
    It follows from Definition \ref{def: Whittaker function} that 
    \begin{align*}
        & \chi(E_{2, +}) = \psi(\Tr_{E/\QQ}(\delta^{-1}(-\sqrt{-2}D^{-\frac{1}{4}}))) = -4\pi{i}\sqrt{2}D^{-\frac{3}{4}}, \\
        & \chi(E_{2, +}) = \psi(\Tr_{E/\QQ}(\delta^{-1}(-\sqrt{2}D^{-\frac{1}{4}}))) = 0. 
    \end{align*}
    \par Plugging these into the formulas in \cite[Theorem (4.5), Page 977]{Koseki-Oda95}, we have 
    \begin{align*}
        & \eta_{\pm} = \chi(E_{2, +}) \pm \sqrt{-1} \chi(E_{2, -}) = -4\sqrt{2} \pi{i} D^{-\frac{3}{4}}, \\
        & b = -\eta_{+}\eta_{-} = 32\pi^{2} D^{-\frac{3}{2}}, \gamma_m = (-i)^{2 + a + b - m},
    \end{align*}
    and 
    \begin{equation*}
           C_m^{(1+a, -1-b)}(t) = (-i)^{2 + a + b -m}t^{a + b + \frac{7}{2}} W_{0, m-(1+a)}(8\sqrt{2} \pi D^{-\frac{3}{4}}t). 
    \end{equation*}
\end{proof}

\begin{lemma}
\label{Lemma:Mellin_Trans_Whittaker}
        The Mellin transform of the Whittaker function $C_{m}^{(1+a, -1-b)}$ at $l \in \CC$ is 
              \begin{align*}
                  \int_{0}^{\infty} C_{m}^{(1+a, -1-b)}(t) t^{l} d^{\times}t &= (-i)^{2 + a + b - m} \frac{1}{2} (\frac{1}{2\sqrt{2}D^{-\frac{3}{4}}})^{l + a + b + \frac{7}{2}}\pi^{-(l + a + b + 4)} \Gamma(\frac{1}{2}(l + m + b + 3) ) \\
                  & \times \Gamma(\frac{1}{2}(l + b - m + 5) + a). 
              \end{align*}
\end{lemma}

\begin{proof}
     It follows from Lemma \ref{Lemma: relation C and W} that 
     \begin{equation*}
         \int_{0}^{\infty} C_{m}^{(1 + a, -1 - b)}(t) t^{l} d^{\times}t = (-i)^{2 + a + b - m} \int_{0}^{\infty} t^{l + a + b + \frac{7}{2}} W_{0, m-(1 + a)}(8\sqrt{2} \pi D^{-\frac{3}{4}}t) d^{\times}t.  
     \end{equation*}
     If we change variables by $t \mapsto \frac{t}{8\sqrt{2}\pi D^{-\frac{3}{4}}}$, then this integral is 
     \begin{equation*}
         (-i)^{2 + a + b - m} (\frac{1}{8\sqrt{2}\pi D^{-\frac{3}{4}}})^{l + a + b + \frac{7}{2}}  \int_{0}^{\infty} t^{l + a + b + \frac{7}{2}} W_{0, m-(1 + a)}(t) d^{\times}t. 
     \end{equation*}
     Using the formula in \cite[\S 5.2, P979]{Koseki-Oda95}, we find it is equal to 
     \begin{equation*}
         (-i)^{2 + a + b - m} \frac{1}{2} (\frac{1}{2\sqrt{2}D^{-\frac{3}{4}}})^{l + a + b + \frac{7}{2}}\pi^{-(l + a + b + 4)} \Gamma(\frac{1}{2}(l + m + b + 3) ) \Gamma(\frac{1}{2}(l + b - m + 5) + a). 
     \end{equation*}
\end{proof}

Now we compute the archimedean zeta integral with specific chosen test vectors. 

\begin{convention}
    \begin{itemize}
        \item Let $\Psi_{\infty} \in \tilde{\pi}_{\infty}$ be a vector of weight $(-1 - b + r - s, 1 + a + r - s, r - s; b + 2s - r)$, which is the lowest weight vector of the minimal $K_{G}$-type $\tau_{\mu}$.
        \item We write the basis $\nu_0, \nu_1, \ldots, \nu_d$ in the $K_{G}$-type $\mu$ defined in subsubsection \ref{SSS: repn of compact Lie group} as $\nu_0^{\mu}, \nu_1^{\mu}, \ldots, \nu_d^{\mu}$. 
        \item We identify $\Psi_{\infty}$ with the vector $\nu_{0}^{\mu}$, so $X^{m}_{(1, -1, 0)}\Psi_{\infty}$ is $m! \nu_{m}^{\mu}$. 
        \item By the uniqueness of Whittaker models (see \cite[\S 3.2]{Koseki-Oda95}), we choose $W_{\infty}$ to be the Whittaker function corresponding to $X_{(1, -1, 0)}^{m} \Psi_{\infty}$ such that $W_{\infty}(1) = 1$.
        \item Because $W_{\infty}(1) = 1$ and $C_{d - m}^{(1 + a, -1-b)}(1) = (-i)^{m}W_{0, 1 + b - m}(8\sqrt{2}\pi D^{-\frac{3}{4}})$ (see Lemma \ref{Lemma: relation C and W}), it can be seen from uniqueness of Whittaker models that if we let 
        \begin{equation*}
            C = (-i)^{m}W_{0, 1 + b - m}(8\sqrt{2}\pi D^{-\frac{3}{4}}),
        \end{equation*}
        then we have 
        \begin{equation}
        \label{eq: infinite Whittaker and C}
            W_{\infty}(t) = C^{-1} C_{d - m}^{(1+a, -1-b)}(t),
        \end{equation}
        for all $t \in \RR^{\times}$\footnote{Here we view $W_{\infty}$ as a function on $\RR^{\times}$.}. 
        \item Let $K_{\infty} = \U(1)(\RR)$.
    \end{itemize}
    
\end{convention}

\begin{lemma}
\label{Lemma: integral I1}
If $m = 1+ a + r - s$, then the weight of $W_{\infty}$ is 
\begin{equation*}
(a - b + 2(r - s), 0, r-s).
\end{equation*}
In this case, the integral 
\begin{equation*}
    I_1 = \int_{\GL_1(\RR)} t^{2z - n - 2} W_{\infty}( \begin{pmatrix}
                                                                        t & 0 & 0 \\
                                                                        0 & 1 & 0 \\
                                                                        0 & 0 & t^{-1} 
                                                                      \end{pmatrix}) dt
\end{equation*}
is equal to 
\begin{align*}
    C^{-1} (-i)^{1 + a - r - s} (\frac{1}{2\sqrt{2}D^{-\frac{3}{4}}})^{(2z -r - s + \frac{3}{2})}\pi^{-(2z - r - s + 2)}\Gamma(z - r + 1 + \frac{1}{2}(b - a)) \Gamma(z - s + 1 - \frac{1}{2} (a + b)). 
\end{align*}
    
\end{lemma}

\begin{proof}
    First, we have 
            \begin{align*}
                I_1 &= \int_{\GL_1(\RR)} t^{2z - n - 2} W_{\infty}( \begin{pmatrix}
                                                                        t & 0 & 0 \\
                                                                        0 & 1 & 0 \\
                                                                        0 & 0 & t^{-1} 
                                                                      \end{pmatrix}) dt \\
                    &= \int_{\RR^{\times}_{>0}} t^{2z - n - 2} W_{\infty}( \begin{pmatrix}
                                                                        t & 0 & 0 \\
                                                                        0 & 1 & 0 \\
                                                                        0 & 0 & t^{-1} 
                                                                      \end{pmatrix}) dt + \int_{\RR^{\times}_{<0}} t^{2z - n - 2} W_{\infty}( \begin{pmatrix}
                                                                        t & 0 & 0 \\
                                                                        0 & 1 & 0 \\
                                                                        0 & 0 & t^{-1} 
                                                                      \end{pmatrix}) dt. 
            \end{align*}
        By a change of variables and the fact that $W_{\infty}$ has weight $(a - b + 2(r - s), 0, r-s)$, we get 
        \begin{equation*}
            \int_{\RR^{\times}_{>0}} t^{2z - n - 2} W_{\infty}( \begin{pmatrix}
                                                                        t & 0 & 0 \\
                                                                        0 & 1 & 0 \\
                                                                        0 & 0 & t^{-1} 
                                                                            \end{pmatrix}) dt =  \int_{\RR^{\times}_{<0}} t^{2z - n - 2} W_{\infty}( \begin{pmatrix}
                                                                                    t & 0 & 0 \\
                                                                                    0 & 1 & 0 \\
                                                                                    0 & 0 & t^{-1} 
                                                                                  \end{pmatrix}) dt. 
        \end{equation*}
        Hence, we get 
        \begin{equation*}
            I_1 = 2  \int_{\RR^{\times}_{>0}} t^{2z - n - 2} W_{\infty}( \begin{pmatrix}
                                                                        t & 0 & 0 \\
                                                                        0 & 1 & 0 \\
                                                                        0 & 0 & t^{-1} 
                                                                        \end{pmatrix}) dt, 
        \end{equation*}
        which is a Mellin transform of a Whittaker function. 
        \par Second, it follows from equation (\ref{eq: infinite Whittaker and C}) that 
        \begin{equation*}
            I_1 = 2 C^{-1} \int_{0}^{\infty} t^{2z - n - 2} C_{1 + b - r + s}^{(1+a, -1-b)}(t) d^{\times}t. 
        \end{equation*}
        \par Finally, using the formula in Lemma \ref{Lemma:Mellin_Trans_Whittaker}, we have 
        \begin{align*}
             I_1 = C^{-1} (-i)^{1 + a - r - s} (\frac{1}{2\sqrt{2}D^{-\frac{3}{4}}})^{(2z -r - s + \frac{3}{2})}\pi^{-(2z - r - s + 2)}\Gamma(z - r + 1 + \frac{1}{2}(b - a)) \Gamma(z - s + 1 - \frac{1}{2} (a + b)). 
        \end{align*}
\end{proof}

\begin{proposition}
\label{Prop:Arch_Zeta}
        \begin{enumerate}
            \item If $m \neq 1 + a + r - s$,  then the archimedean local zeta integral 
                    \begin{equation*}
                        I_{\infty}(W_{\infty}, \Phi_{\infty}, \nu_{\infty}, z) = \int_{Z(\RR)U_{2}(\RR) \backslash H(\RR)} f(g_{1, \infty}, \Phi_{\infty}, \nu_{\infty}, z) W_{\infty}(g_{\infty}) dg_{\infty}
                    \end{equation*} is equal to $0$. 
            \item If $m = 1 + a + r - s$, then the zeta integral $I_{\infty}(W_{\infty}, \Phi_{\infty}, \nu_{\infty}, 1 + a + b + r + s)$ is only non-zero when the $K_{\infty}$-type of $\Phi_{\infty}$ is $-(a + r -b - s)$.  In this case,  $I_{\infty}(W_{\infty}, \Phi_{\infty}, \nu_{\infty}, 1 + a + b + r + s)$ is equal to
            \begin{align*}
                &  (-1)^{a} \pi^{-(3a + 3b + 2r + 2s + 5)} W_{0, b + s - a - r}(8\sqrt{2}\pi D^{-\frac{3}{4}})^{-1}\left(\frac{1}{2\sqrt{2}D^{-\frac{3}{4}}}\right)^{2a + 2b + r + s + \frac{7}{2}} \\
                & \times \Gamma(a + b + r + s + 1) \Gamma\left( \frac{1}{2}a + \frac{3}{2}b + s + 2\right) \Gamma\left(\frac{1}{2}a + \frac{1}{2} b + r + 2\right), 
            \end{align*}
            where $W_{0,b + s - a - r}$ is the classical Whittaker function on $\RR$ defined in Definition \ref{Def: classical Whittaker function} and $\Gamma$ is the Gamma function. 
        \end{enumerate}
\end{proposition}
\begin{proof}
   Recall that $H = \GL_2 \boxtimes E^{\times} $ and $Z(\RR) = \left \{(\begin{pmatrix}
                                                        z & 0 \\
                                                        0 & z \\ 
                                                    \end{pmatrix}, z) | z \in \RR \right \}$, so it follows from the Iwasawa decomposition that  
    \begin{equation*}
        Z(\RR) U_2(\RR) \backslash H(\RR) / K_{\infty} \cong \GL_1(\RR) \times \mathbb{S}^{1} = \left \{ (\begin{pmatrix}
                                                                                                    t & 0 \\
                                                                                                    0 & t^{-1} 
                                                                                                \end{pmatrix}, c) | t \in \GL_1(\RR), c \in \mathbb{S}^{1} \right \}, 
    \end{equation*}
    where $\mathbb{S}^{1}$ is the unit circle $\{ z \in \CC | z\bar{z} = 1 \}$. There is an embedding 
   $\iota \colon \GL_1(\RR) \times \mathbb{S}^{1} \hookrightarrow G(\RR)$ that maps $(\begin{pmatrix}
                                                                                    t & 0 \\
                                                                                    0 & t^{-1} 
                                                                                 \end{pmatrix}, c) $ to $\left(\begin{matrix}
                                                                                                            t & 0 & 0 \\
                                                                                                            0 & c & 0 \\
                                                                                                            0 & 0 & t^{-1} 
                                                                                                            \end{matrix}   \right)$. 
    \par Since the $K_{\infty}$-weight of $\Phi_{\infty}$ is $-k$ and the $K_{\infty}$-weight of $W_{\infty}$ is $m - 1 - b$, we have
    \begin{equation*}
        I_{\infty}(W_{\infty}, \Phi_{\infty}, \nu_{\infty}, z) = 0
    \end{equation*}
    unless $k = m - 1 -b$. In this case, the action of $K_{\infty}$ on $f(g_{1, \infty}, \Phi_{\infty}, \nu_{\infty}, z)W_{\infty}(g_{\infty})$ is trivial, so
    \begin{align}
        & I_{\infty}(W_{\infty}, \Phi_{\infty}, \nu_{\infty}, z) \\ = & \int_{Z(\RR)U_{2}(\RR) \backslash H(\RR)} f(g_{1, \infty}, \Phi_{\infty}, \nu_{\infty}, z) W_{\infty}(g_{\infty}) dg_{\infty} \\ 
        = & \pi^{-z}\Gamma(z) \int_{\GL_1(\RR) \times \mathbb{S}^{1}} (-t^{-1})^{\frac{n-k}{2}} (t^{-1})^{\frac{n+k}{2}} t^{2z} W_{\infty}(\iota( \begin{pmatrix}
            t & 0 \\
            0 & t^{-1} 
        \end{pmatrix}, c)) |t|^{-2} dt dc  \\
        = & \pi^{-z}\Gamma(z) (-1)^{\frac{n - k}{2}}  \int_{\GL_1(\RR) \times \mathbb{S}^{1}}t^{2z - n - 2} W_{\infty}( \begin{pmatrix}
                                                                                                                            t & 0 & 0 \\
                                                                                                                            0 & c & 0 \\
                                                                                                                            0 & 0 & t^{-1} 
                                                                                                                        \end{pmatrix}) dt dc. \\ 
        \label{eqn: Arch_Zeta_Int} = & \pi^{-z}\Gamma(z) (-1)^{\frac{n - k}{2}} \int_{\GL_1(\RR)} t^{2z - n - 2} W_{\infty}( \begin{pmatrix}
                                                                                                    t & 0 & 0 \\
                                                                                                    0 & 1 & 0 \\
                                                                                                    0 & 0 & t^{-1} 
                                                                                                  \end{pmatrix}) dt, 
         \times \int_{\mathbb{S}^{1}} c^{1 + a - m + r - s} dc                                                                                                                                                                      
    \end{align}
    where the last identity holds since
    \begin{equation*}
        W_{\infty}(\begin{pmatrix}
                     t & 0 & 0 \\
                     0 & c & 0 \\
                     0 & 0 & t^{-1} 
                    \end{pmatrix}) =  W_{\infty}(\begin{pmatrix}
                                                     t & 0 & 0 \\
                                                     0 & 1 & 0 \\
                                                     0 & 0 & t^{-1} 
                                                  \end{pmatrix} \begin{pmatrix}
                                                                     1 & 0 & 0 \\
                                                                     0 & c & 0 \\
                                                                     0 & 0 & 1
                                                                \end{pmatrix}) = c^{1 + a - m + r - s} W_{\infty}(\begin{pmatrix}
                                                                                                                     t & 0 & 0 \\
                                                                                                                     0 & 1 & 0 \\
                                                                                                                     0 & 0 & t^{-1} 
                                                                                                                    \end{pmatrix}). 
    \end{equation*}
    Now we set $I_1 = \int_{\GL_1(\RR)} t^{2z - n - 2} W_{\infty}( \begin{pmatrix}
                                                                        t & 0 & 0 \\
                                                                        0 & 1 & 0 \\
                                                                        0 & 0 & t^{-1} 
                                                                      \end{pmatrix}) dt$ and $I_2 = \int_{\mathbb{S}^{1}} c^{1 + a - m + r - s} dc$, 
    and will compute the integral $I_1$ and $I_2$, respectively. 

    \par The integral $I_2$ is computed as follows. 
         We have 
         \begin{align*}
             I_2 & = \int_{\mathbb{S}^{1}} c^{1 + a - m + r - s} dc \\
                 & = \int_{0}^{1} e^{2\pi{i}(1 + a - m + r - s)x} dx \\
                 & = \begin{cases}
                        1 & \text{If}\ m = 1+ a + r - s, \\
                        0 & \text{If}\ m \neq 1+ a + r - s. 
                     \end{cases}
         \end{align*}
    When $m = 1 + a + r - s$, the integral $I_1$ is computed in Lemma \ref{Lemma: integral I1}. 
    
        \par Finally, plugging the formulas for $I_1$ and $I_2$ into equation (\ref{eqn: Arch_Zeta_Int}), we have: 
        \begin{itemize}
            \item if $m \neq 1 + a + r - s$, then
                  \begin{equation*}
                      I_{\infty}(W_{\infty}, \Phi_{\infty}, \nu_{\infty}, z) = 0; 
                  \end{equation*}
            \item if $m = 1 + a + r - s$, then 
            \begin{align*}
                & I_{\infty}(W_{\infty}, \Phi_{\infty}, \nu_{\infty}, z) \\ = & (-1)^{a + r} C^{-1} (-i)^{1 + a - r - s} (\frac{1}{2\sqrt{2}D^{-\frac{3}{4}}})^{(2z -r - s + \frac{3}{2})}\pi^{-(3z - r - s + 2)} \\ & \times \Gamma(z - r + 1 + \frac{1}{2}(b - a)) \Gamma(z - s + 1 - \frac{1}{2} (a + b)) \Gamma(z).
            \end{align*}
            Plugging $z = 1 + a + b + r + s$ and $C = (-i)^{1 + a + r - s}W_{0, b + s - a - r}(8\sqrt{2}\pi D^{-\frac{3}{4}})$ into the formula, we have
            \begin{align*}
                &  I_{\infty}(W_{\infty}, \Phi_{\infty}, \nu_{\infty}, 1 + a + b + r + s) \\
                = (-1)^{a}& \pi^{-(3a + 3b + 2r + 2s + 5)} W_{0, b + s - a - r}(8\sqrt{2}\pi D^{-\frac{3}{4}})^{-1}\left(\frac{1}{2\sqrt{2}D^{-\frac{3}{4}}}\right)^{2a + 2b + r + s + \frac{7}{2}} \\
                \times &\Gamma(a + b + r + s + 1) \Gamma( \frac{1}{2}a + \frac{3}{2}b + s + 2) \Gamma(\frac{1}{2}a + \frac{1}{2} b + r + 2). 
            \end{align*}
        \end{itemize}
    \end{proof}

\begin{remark}
\label{remark: arch_zeta_integral}
    \begin{enumerate}
        \item If we let $m = 1 + a + r - s = 1 + a + 2b -2j + r + s$ as in Proposition \ref{prop: explicit_pairing}, then $j = b + s$. So we can see from Corollary \ref{corollary: identify integral of Xi and I} only one term in the summation of equation (\ref{eq:pairing}) is left. So the pairing $\langle \omega_{\Psi},  [\rho] \rangle_{B}$ is Eulerian by our previous discussion of zeta integral. 
        \item The integer $j = b + s$ satisfies the condition $A \le j \le B$ in Proposition \ref{prop: explicit_pairing}. 
    \end{enumerate}
 
\end{remark}

\section{The periods}
\label{Sec: periods}
\subsection{The motivic periods}
\label{SS: motivic periods}
In this subsection, we compute the pairing $\langle \Omega, \tilde{v}_{D} \rangle_{B}$ defined in Lemma \ref{lemma: abstrac paring with Omega}. 

Recall that $V = V^{a, b}\{r, s\}$ and we have the following exact sequence: 
\begin{equation}
\label{eq: def of Ext1 in section 9}
    0 \rightarrow F^{0}M_{dR}(\pi_f, V(2))_{\RR} \rightarrow M_{B}(\pi_f, V(2))^{-}_{\RR}(-1) \rightarrow \Ext^{1}_{\mathrm{MHS_{\RR}^{+}}}(\RR(0), M_{B}(\pi_f, V(2))_{\RR}) \rightarrow 0.
\end{equation}
The Beilinson $E(\pi_f)$-structure of $\mathrm{det}_{E(\pi_f)\otimes_{\QQ}\RR}\Ext^{1}_{\mathrm{MHS_{\RR}^{+}}}(\RR(0), M_{B}(\pi_f, V(2))_{\RR})$ is defined as  
\begin{equation*}
    \mathcal{B}(\pi_f, V(2)) = \mathrm{det}_{E(\pi_f)} F^{0}M_{dR}(\pi_f, V(2))^{*} \otimes_{E(\pi_f)} \mathrm{det}_{E(\pi_f)} M_{B}(\pi_f, V(2))^{-}_{\RR}(-1). 
\end{equation*}
Let $\delta(\pi_f, V(2)) \in (E(\pi_f) \otimes_{\QQ} \CC)^{\times}$ be the determinant of the isomorphism 
\begin{equation*}
I_{\infty} \colon M_{B}(\pi_f, V(2))_{\CC} \rightarrow M_{dR}(\pi_f, V(2))_{\CC}
\end{equation*}
computed using the basis defined over $E(\pi_f)$ on both sides. The Deligne $E(\pi_f)$-structure of $\mathrm{det}_{E(\pi_f)\otimes_{\QQ}\RR}\Ext^{1}_{\mathrm{MHS_{\RR}^{+}}}(\RR(0), M_{B}(\pi_f, V(2))_{\RR})$ is 
\begin{equation*}
    \mathcal{D}(\pi_f, V(2)) := (2\pi i)^{3} \delta(\pi_f, V(2))^{-1} \mathcal{B}(\pi_f, V(2)). 
\end{equation*}
Let $v_D$ be a generator of $\mathcal{D}(\pi_f, V(2))$, $\tilde{v}_{D}$ be a lifting of  $v_{D}$ via the last arrow in the exact sequence (\ref{eq: def of Ext1 in section 9})
and $\Omega = \frac{1}{2} (\omega_{\Psi} + \overline{\omega_{\Psi}}) \cdot \mathbf{1}$, where $\omega_{\Psi}$ is constructed in Subsection \ref{SS: the test vector} and $\mathbf{1}$ the multiplicative identities of $E(\pi_f) \otimes_{\QQ} \CC$. 

\par It can be seen from Remark \ref{remark: parameter of repn transform} and Proposition \ref{Prop: Hodge decomp for motives} that we have the following Hodge decompositions: 
     \begin{align*}
        & M_{B}({\pi}_f, V(2))_{\CC} \\
        \cong & M_{B}^{-r, -2-a-b-s} \oplus M_{B}^{-1-a-r, -1-b-s} \oplus M_{B}^{-2-a-b-r, -s} \\
        \oplus & M_{B}^{-2-a-b-s, -r} \oplus M_{B}^{-1-b-s, -1-a-r} \oplus M_{B}^{-s, -2-a-b-r}, 
    \end{align*} and 
    \begin{align*}
        & M_{B}(\tilde{\pi}_f|\mu|^{-2}, D)_{\CC} \\
        \cong & M_{B}^{a + b + r + 2, s} \oplus M_{B}^{a + r + 1, b + s + 1} \oplus M_{B}^{r, a + b + s + 2} \\
        \oplus & M_{B}^{s, a + b + r + 2} \oplus M_{B}^{b + s + 1, a + r + 1} \oplus M_{B}^{a + b + s + 2, r}, 
    \end{align*}
    from which we can see that 
    $M_{B}(\pi_f, V(2))^{-}_{\RR}(-1)$ is $3$-dimensional.
    By the condition of the parameters stated in (\ref{eq: Cond of coefficients}), we get that $F^{0}M_{dR}(\pi_f, V(2))_{\RR}$ is $2$-dimensional. 

    \begin{notation}
        We use $M_{dR}^{s, t}$ to denote $I_{\infty}^{-1} (M_{B}^{s, t})$ for $s, t \in \ZZ$.  
    \end{notation}

    We now recall the definition of Deligne periods as in \cite[\S 1.7]{Deligne79}. 
    \begin{definition}[Deligne period]
    \label{def: Deligne period}
    Assume that $a + r \neq b + s$\footnote{This assumption implies that condition that the infinite Frobenius $F_{\infty}$ acts on the component with Hodge type $(n, n)$ for all $n \in \ZZ$ by a scalar $1$ or $-1$(see \cite[\S 1.7]{Deligne79}) is satisfied. } and let 
    \begin{align*}
           & F^{+}(\tilde{\pi}_{f}|\mu|^{-2}, D) \\
         = & F^{-}(\tilde{\pi}_{f}|\mu|^{-2}, D) \\
         = & M_{dR}^{a + b + r + 2, s} \oplus M_{dR}^{a + b + s + 2, r} \oplus M_{dR}^{A, B},
    \end{align*}
    where $A = \max\{a + r + 1, b + s + 1\}$, $B = \min\{a + r + 1, b + s + 1\}$. 
    It is shown in \cite[\S 1.7]{Deligne79} that the map 
    \begin{equation*}
        I_{\infty}^{\pm} \colon M_{B}(\tilde{\pi}_{f}|\mu|^{-2}, D)_{\CC}^{\pm} \longrightarrow M_{dR}(\tilde{\pi}_{f}|\mu|^{-2}, D)_{\CC} / F^{\pm}(\tilde{\pi}_{f}|\mu|^{-2}, D)
    \end{equation*}
    is an isomorphism.
        \par The Deligne period $c^{\pm}(\tilde{\pi}_{f}|\mu|^{-2}, D)$ is defined as $\det(I_{\infty}^{\pm}) \in (E(\pi_f) \otimes_{\QQ} \CC)^{\times}$. 
\end{definition}

\begin{remark}
    We can define the Deligne period $c^{\pm}(\pi_f, V(2))$ similarly. This is left to the reader.
\end{remark}

 We have the following commutative diagram: 
    \begin{equation}{\label{commutative diagram_pairing}}
        \begin{tikzcd}
            M_B(\pi_f, V(2))_{\CC} \otimes M_{B}(\tilde{\pi}_{f}|\mu|^{-2}, D)_{\CC}  \ar[r, "{\langle \cdot, \cdot \rangle_{B}}"] \ar[d, "I_{\infty}"] &  E(\pi_f) \otimes_{\QQ} {\CC} \ar[d, "{I_{\infty}}"]  \\
             M_{dR}(\pi_f, V(2))_{\CC} \otimes M_{dR}(\tilde{\pi}_{f}|\mu|^{-2}, D)_{\CC} \ar[r, "{\langle \cdot, \cdot \rangle_{dR}}"] & E(\pi_f) \otimes_{\QQ} \CC 
        \end{tikzcd}, 
    \end{equation}
where the vertical arrows are isomorphisms. 

\begin{construction}
\label{construnction: basis for periods}
\par    We contruct a ``nice" basis of the spaces in the commutative diagram (\ref{commutative diagram_pairing}) as follows. 
         \begin{itemize}
             \item Let $v_1, v_2, v_3$ be an $E(\pi_f)$-basis of $M_{B}(\pi_f, V(2))^{-}_{\RR}(-1)$ such that $v_1 \in M_{B}^{-r, -2-a-b-s}$ and $v_2 \in M_{B}^{-s, -2-a-b-r}$. Then $2\pi{i}v_1, 2\pi{i}v_2, 2\pi{i}v_3$ is an $E(\pi_f)$-basis of $M_B(\pi_f, V(2))^{+}$.
             \item Let $w_1, w_2, w_3$ be an $E(\pi_f)$-basis of $M_{B}(\tilde{\pi}_{f}|\mu|^{-2}, D)^{+}_{\CC}$ such that 
                    \begin{equation*}
                        \langle 2\pi{i}v_i, w_j \rangle_{B} = \delta_{ij},
                    \end{equation*}
                    for $i,j \in \{1, 2, 3\}$.
            \item Let $\theta_1, \theta_2$ be the $E(\pi_f)$-basis of  $F^{0}M_{dR}(\pi_f, V(2))$ such that the image of $\theta_{i}$ for $i = 1,2$ under the map
            \begin{equation*}
                 F^{0}M_{dR}(\pi_f, V(2))_{\RR} \rightarrow M_{B}(\pi_f, V(2))^{-}_{\RR}(-1) 
            \end{equation*}
            is $v_i$.
            \item Let $\theta_{1}^{\prime}, \theta_{2}^{\prime} \in M_{dR}(\tilde{\pi}_{f}|\mu|^{-2}, D)_{\CC}$ be the dual vectors of $\theta_1, \theta_2$, which means that 
            \begin{align*}
                & \langle \theta_1, \theta_1^{\prime} \rangle_{dR} = \langle \theta_2, \theta_2^{\prime} \rangle_{dR} = 1_{dR}, \\
                & \langle \theta_1, \theta_2^{\prime} \rangle_{dR} = \langle \theta_2, \theta_1^{\prime} \rangle_{dR} = 0. 
            \end{align*}
            It follows that $\theta_1^{\prime} \in M_{dR}^{r, a + b + s + 2}$ and $\theta_2^{\prime} \in M_{dR}^{s, a + b + r + 2}$. If we let $\Omega^{\prime} = I_{\infty}(\Omega)$, then it follows from the fact that 
            \begin{equation*}
            \Omega \in M_{B}^{a + r + 1, b + s + 1} \oplus M_{B}^{b + s + 1, a + r + 1}.
            \end{equation*}
            that 
            \begin{equation*}
                \Omega^{\prime} \in M_{dR}^{a + r + 1, b + s + 1} \oplus M_{dR}^{b + s + 1, a + r + 1}. 
            \end{equation*}
            Hence, $\Omega^{\prime}, \theta_1^{\prime}, \theta_{2}^{\prime}$ is a basis of $M_{dR}(\tilde{\pi}_{f}|\mu|^{-2}, D)_{\CC} / F^{\pm}(\tilde{\pi}_{f}|\mu|^{-2}, D)$. 
            \item It follows from $\theta_1 \in M_{dR}^{-r, -2-a-b-s}$, $\theta_2 \in M_{dR}^{-s, -2-a-b-r}$ and 
            \begin{equation*}
            \Omega^{\prime} \in M_{dR}^{a + r + 1, b + s + 1} \oplus M_{dR}^{b + s + 1, a + r + 1},
            \end{equation*}
            that 
            \begin{equation*}
            \langle \Omega^{\prime} , \theta_1 \rangle_{dR} =  \langle \Omega^{\prime} , \theta_2 \rangle_{dR} = 0. 
            \end{equation*}
    \end{itemize}
    
\end{construction}

\begin{lemma}
\label{lemma: inverse matric for periods}
    Using the basis $w_1, w_2, w_3$ of $M_{B}(\tilde{\pi}_{f}|\mu|^{-2}, D)^{+}_{\CC}$ and basis $\Omega^{\prime}, \theta_1^{\prime}, \theta_2^{\prime}$ of 
    \begin{equation*}
        M_{dR}(\tilde{\pi}_{f}|\mu|^{-2}, D)_{\CC} / F^{\pm}(\tilde{\pi}_{f}|\mu|^{-2}, D), 
    \end{equation*}
    the inverse matrix of 
    \begin{equation*}
         I_{\infty}^{+} \colon M_{B}(\tilde{\pi}_{f}|\mu|^{-2}, D)_{\CC}^{+} \longrightarrow M_{dR}(\tilde{\pi}_{f}|\mu|^{-2}, D)_{\CC} / F^{+}(\tilde{\pi}_{f}|\mu|^{-2}, D)
    \end{equation*}
    can be written as 
    \begin{equation*}
         c^{+}(\tilde{\pi}_{f}|\mu|^{-2}, D)^{-1}\begin{pmatrix}
                                                C & * & * \\
                                                B & * & * \\
                                                (2\pi{i})^{-2} & * & * 
         \end{pmatrix}. 
    \end{equation*}
\end{lemma}

\begin{proof}
    \par Using the basis $w_1, w_2, w_3$ and $\Omega^{\prime}, \theta_1^{\prime}, \theta_2^{\prime}$, if we write the map 
    \begin{equation*}
         I_{\infty}^{+} \colon M_{B}(\tilde{\pi}_{f}|\mu|^{-2}, D)_{\CC}^{+} \longrightarrow M_{dR}(\tilde{\pi}_{f}|\mu|^{-2}, D)_{\CC} / F^{+}(\tilde{\pi}_{f}|\mu|^{-2}, D)
    \end{equation*}
    explicitly as 
    \begin{equation*}
        I_{\infty}^{+} \colon (w_1, w_2, w_3) \mapsto (\Omega^{\prime}, \theta_1^{\prime}, \theta_2^{\prime}) \begin{pmatrix}
                                                                                                            \alpha_1 & \alpha_2 & \alpha_3 \\ 
                                                                                                            \beta_1 & \beta_2 & \beta_3 \\
                                                                                                            \gamma_1 & \gamma_2 & \gamma_3 \\
                                                                                                        \end{pmatrix}, 
    \end{equation*}
    then $(I_{\infty}^{+})^{-1}$ can be written as 
    \begin{equation*}
        (I_{\infty}^{+})^{-1} \colon (\Omega^{\prime}, \theta_1^{\prime}, \theta_2^{\prime}) \mapsto (w_1, w_2, w_3) (\det(I_{\infty}^{+}))^{-1}   \begin{pmatrix}
                                                                                                                    C & * & * \\
                                                                                                                    B & * & * \\
                                                                                                                    A & * & * 
                                                                                                                \end{pmatrix}, 
    \end{equation*}
    where $A = \det \begin{pmatrix}
                  \beta_1 & \beta_2 \\
                  \gamma_1 & \gamma_2 \\
               \end{pmatrix} = \beta_1\gamma_2 - \beta_2 \gamma_1$ and $\det(I_{\infty}^{+}) = c^{+}(\tilde{\pi}_{f}|\mu|^{-2}, D)$. 
    \par Now we compute $A$. We have  
    \begin{align*}
        & I_{\infty}^{+} (w_1) = \alpha_1 \Omega^{\prime} + \beta_1 \theta_1^{\prime} + \gamma_1 \theta_2^{\prime}, \\
        & I_{\infty}^{+} (w_2) = \alpha_2 \Omega^{\prime} + \beta_2 \theta_1^{\prime} + \gamma_2 \theta_2^{\prime}. 
    \end{align*}
    By the definition of $w_i$ and commutativity of the diagram (\ref{commutative diagram_pairing}), we have the following identities: 
    \begin{align*}
        & \beta_1 = \langle \theta_1, I_{\infty}^{+}(w_1) \rangle_{dR} = \langle v_1, w_1 \rangle_{B} = (2\pi{i})^{-1}, \\
        & \gamma_1 = \langle \theta_2, I_{\infty}^{+}(w_1) \rangle_{dR} = \langle v_2, w_1 \rangle_{B} = 0, \\
        & \beta_2 = \langle \theta_1, I_{\infty}^{+}(w_2) \rangle_{dR} = \langle v_1, w_2 \rangle_{B} = 0, \\
        & \gamma_2 = \langle \theta_2, I_{\infty}^{+}(w_2) \rangle_{dR} = \langle v_2, w_2 \rangle_{B} = (2\pi{i})^{-1}. 
    \end{align*}
    So we have $\begin{pmatrix}
                    \beta_1 & \beta_2 \\
                    \gamma_1 & \gamma_2 
                \end{pmatrix} = \begin{pmatrix}
                                    (2\pi{i})^{-1} & 0 \\
                                    0 &  (2\pi{i})^{-1}
                                 \end{pmatrix}$ and $A = (2\pi{i})^{-2}$. 
\end{proof}

\begin{notation}
    Recall that we had introduced the following notation: 
    if $\mu$ and $\mu^{\prime}$ are two elements of $E(\pi_f)\otimes_{\QQ}{\CC}$, we write $\mu \sim \mu^{\prime}$ if there exists $\lambda \in E(\pi_f)^{\times}$ such that $\mu = \lambda \mu^{\prime}$. 
\end{notation}

\begin{proposition}
\label{prop: periods}
Assuming that $a + r \neq b + s$, we have the following equality:
\begin{equation*}
    \langle \Omega, \tilde{v}_D \rangle_{B} \sim  c^{-}(\pi_f, V(2))^{-1}. 
\end{equation*}
    
\end{proposition}

\begin{proof}
    \par First, it can be seen from the construction that $v_3$ is a generator of the Beilinson $E(\pi_f)$-structure $\mathcal{B}(\pi_f, V(2))$. Therefore, we have 
    \begin{equation*}
        \langle \Omega, \tilde{v}_D \rangle_{B} \sim (2\pi i)^{3} \delta(\pi_f, V(2))^{-1} \langle \Omega, v_3 \rangle_{B}. 
    \end{equation*}
    \par Second, we compute $\langle \Omega, v_3 \rangle_{B}$. It can seen from Lemma \ref{lemma: inverse matric for periods} that 
    \begin{align*}
        \langle \Omega, v_3 \rangle_{B} &= c^{+}(\tilde{\pi}_{f}|\mu|^{-2}, D)^{-1}\langle Cw_1 + Bw_2 + (2\pi i)^{-2}w_3, v_3 \rangle_{B} \\
                                    &= c^{+}(\tilde{\pi}_{f}|\mu|^{-2}, D)^{-1} (2\pi i)^{-2} \langle w_3, v_3 \rangle_{B} \\
                                    &= c^{+}(\tilde{\pi}_{f}|\mu|^{-2}, D)^{-1} (2\pi{i})^{-3}. 
    \end{align*}
    \par Finally, we can see that 
    \begin{align*}
         \langle \Omega, \tilde{v}_D \rangle_{B} & \sim (2\pi i)^{3} \delta(\pi_f, V(2))^{-1} \langle \Omega, v_3 \rangle_{B} \\
                                           & = c^{+}(\tilde{\pi}_{f}|\mu|^{-2}, D)^{-1} \delta(\pi_f, V(2))^{-1}. 
    \end{align*}
    It follows from \cite[(5.1.7)]{Deligne79} that 
    \begin{equation*}
        c^{+}(\tilde{\pi}_{f}|\mu|^{-2}, D)^{-1} \delta(\pi_f, V(2))^{-1} \sim c^{-}(\pi_f, V(2))^{-1}. 
    \end{equation*}
    So we conclude that 
    \begin{equation*}
        \langle \Omega, \tilde{v}_d \rangle_{B} \sim c^{-}(\pi_f, V(2))^{-1}. 
    \end{equation*}
\end{proof}

\begin{remark}
    The condition $a + r \neq b + s$ guarantees that the weight of $M_B(\pi_f, V(2))$ is $\le -3$. 
\end{remark}

\subsection{The Whittaker period}
\label{SS: Whittaker period}
In this subsection, we define a special automorphic period called the Whittaker period.  

\par For any cuspidal automorphic representation $\pi = \pi_f \otimes \pi_{\infty}$ of $G(\AAA)$ such that $\pi_{\infty}$ is a discrete series representation of $G(\RR)^{+}$, $\pi_{f}$ is defined over a number field $E(\pi_f) \subseteq \overline{\QQ}$ by Theorem \ref{Thm: rational field}. Hence, for the vector space $V_{\pi_f}$ under $\pi_f$, there is a model $V_{\overline{\QQ}}$ of $V_{\pi_f}$ such that 
\begin{equation*}
    V_{\pi_f} = V_{\overline{\QQ}} \otimes_{\overline{\QQ}} \CC.
\end{equation*}

Therefore, by Proposition \ref{prop: rational whittaker}, for the Whittaker functional $\Lambda$ associated to $\pi_f$ by Definition \ref{def: Whittaker function}, there exists a $\Lambda_f \colon V_{\pi_f} \rightarrow \CC$ sending $V_{\overline{\QQ}}$ to $\overline{\QQ}$, which is unique up to multiplication by an element of $\overline{\QQ}^{\times}$. 

\begin{definition}[Whittaker period]
\label{def: Whittaker period}
    By uniqueness of Whittaker models, there exists a $c \in \CC^{\times}$ such that 
    \begin{equation*}
        \Lambda = c \cdot \Lambda_f. 
    \end{equation*}
    The \textit{Whittaker period} $W(\pi) \in \CC^{\times} / \QQ^{\times}$ is defined to be the class of $c$ in $\CC^{\times} / \QQ^{\times}$. 
\end{definition}

\section{The main results on connection to L-values}
\label{Sec: main result}
\begin{notation}
    We use the same notations as Section \ref{Sec: zeta integral}. 
\end{notation}

\subsection{Connection to automorphic $L$-functions}
\label{SS: connection to auto L-function}
In this subsection, we state the main result in terms of automorphic $L$-functions. 

\par Recall that we let $n = a + b + r + s$. Let $\pi = \pi_{f} \otimes \pi_{\infty}$ be a \textit{generic} cuspidal automorphic representation of $G(\AAA)$ whose central character is
\begin{align*}
    \omega_{\pi} \colon Z_{G}(\QQ) & \backslash Z_{G}(\AAA) \rightarrow \CC^{\times} \\
                             & z \longmapsto |z|^{-n} \nu(z), 
\end{align*}
where $\nu$ is a finite-order Hecke character of sign $(-1)^{n}$. Moreover, let the archimedean component $\pi_{\infty}$ of $\pi$ be a discrete series with Blattner parameter $(1+s-r+b, -1-a+s-r, s-r)$. The cuspidal automorphic representation $\pi$ can be viewed as a direct summand of
\begin{equation*}
    \mathrm{H}_{B, !}^{2}(\Sh_G, V(2))_{\CC} := \mathrm{Im}(\mathrm{H}^{2}_{\dR,c}(\Sh, V(2))_{\CC} \rightarrow \mathrm{H}^{2}_{\dR}(\Sh_G,V(2))_{\CC}, 
\end{equation*}
which is the interior cohomology of the Shimura variety of infinite level with coefficients in the local system associated to the representation $V(2)$. Let $E(\pi_f)$ be the rational field of $\pi$, which is a number field. 

\par Let $M_{B}(\pi_f, V(2))$ (see Definition \ref{def: M_B and M_dR}) be the $\pi_f$-isotypic subspace of $\mathrm{H}_{B, !}^{2}(\Sh_G, V(2))$. Similarly, let $M_{\dR}(\pi_f, V(2))$ (see Definition \ref{def: M_B and M_dR}) be the $\pi_f$-isotypic subspace of 
\begin{equation*}
    \mathrm{H}_{\dR, !}^{2}(\Sh_G, V(2)) := \mathrm{Im}(\mathrm{H}^{2}_{\dR,c}(\Sh, V(2)) \rightarrow \mathrm{H}^{2}_{\dR}(\Sh_G,V(2)).
\end{equation*}

\par Let $\mathcal{K}(V(2))$ (see equation (\ref{eq: K(pi, V)})) be the $\QQ[G(\AAA_f)]$-submodule of $\Ext^{1}_{\mathrm{MHS}_{\RR}^{+}}(1, \mathrm{H}^{2}_{B,!}(\Sh_G,V(2))_{\RR})$ generated by   the image of $\mathcal{E}is^{n}_H$ with domain $\mathcal{B}_{n, \QQ}$. Let $\mathcal{K}(\pi_f, V(2))$ be the $\pi_f$-isotypic subspace of $\mathcal{K}(V(2))$ which is a $1$-dimensional $E(\pi_f)$-subspace of the rank one $E(\pi_f)\otimes_{\QQ}\RR$-module $\Ext^{1}_{\mathrm{MHS}_{\RR}^{+}}(1, M_{B}(\pi_f, V(2))_{\RR})$. In $\Ext^{1}_{\mathrm{MHS}_{\RR}^{+}}(1, M_{B}(\pi_f, V(2))_{\RR})$, there is another $E(\pi_f)$-subspace $\mathcal{D}(\pi_f, V(2))$ called the Deligne $E(\pi_f)$-structure (see Definition \ref{def: Delign-rational structure}) coming from the comparison between Betti and de Rham realizations of the motive. The main result in this section gives a measure of the difference between $\mathcal{K}(\pi_f,V(2))$ and $\mathcal{D}(\pi_f,V(2))$ in terms of a non-critical automorphic $L$-value. 

\begin{notation}
    We use the notation $|\cdot|_{f}$ to denote the product over all finite absolute values:
    \begin{equation*}
        \prod_{v < \infty} |\cdot|_{v}. 
    \end{equation*}
\end{notation}

\begin{lemma}
\label{Lemma: relate L-function of contra of pi to L-function of pi}
    Let $S$ be a finite set of places of $\QQ$ (excluding $p = 2$ and $p = \infty$) such that $\pi_{p}$ is unramified for $p \notin S$ and let 
    \begin{equation*}
            L^{S}(s, \pi |\mu|_{f}^{-2}, \mathrm{std}) = \prod_{p \notin S} L_{p}(s, \pi_p |\mu|_{p}^{-2}, \mathrm{std})
    \end{equation*}
    be the partial standard $L$-function of $\pi|\mu|_{f}^{-1}$. Then we have for $s \in \CC$
    \begin{equation*}
            L^{S}(s, (\tilde{\pi} \times \nu_1 )|\mu|_{f}^{-2}, \mathrm{std}) = L^{S}(s, \pi, \mathrm{std}). 
    \end{equation*}
\end{lemma}

\begin{proof}
    First, let $\eta$ be the endomorphism on $E_{p}^{3}$ that maps $(x_1, x_2, x_3) \in E_p^{3}$ to $(\overline{x_1}, \overline{x_2}, -\overline{x_3})$.
    It induces an endomorphism on $G(\Qp)$ and $H(\Qp) = \GL_2(\Qp) \boxtimes E_{p}^{\times}$ by 
    \begin{align*}
        & \eta \colon \begin{pmatrix}
                    a_{11} & a_{12} & a_{13} \\
                    a_{21} & a_{22} & a_{23} \\
                    a_{31} & a_{32} & a_{33} 
                \end{pmatrix} \in G(\Qp) \mapsto \begin{pmatrix}
                                                    \overline{a_{11}} & \overline{a_{12}} & -\overline{a_{13}} \\
                                                        \overline{a_{21}} & \overline{a_{22}} & -\overline{a_{23}} \\
                                                        -\overline{a_{31}} & -\overline{a_{32}} & \overline{a_{33}} 
                                                   \end{pmatrix}, \\
        & \eta \colon (\begin{pmatrix}
                    a & b \\
                    c & d
                \end{pmatrix}, z) \in \GL_2(\Qp) \boxtimes E_{p}^{\times} \mapsto (\begin{pmatrix}
                                                                                        a & -b \\
                                                                                        -c & d
                                                                                    \end{pmatrix}, \overline{z}).
    \end{align*}
    Hence, the action of $\eta$ on $\GL_{2}(\Qp)$ is given by conjugating by $t = \begin{pmatrix}
                                                                                        1 & 0 \\
                                                                                        0 & -1
                                                                                    \end{pmatrix}$. 
    \par Second, it follows from \cite[Proposition 7.9]{PS18} that if we define $\pi_p^{\eta}(g) = \pi_p(\eta g \eta^{-1})$ and let $\omega_{\pi_p}$ be the central character of $\pi_p$, then we have 
        \begin{equation*}
            \tilde{\pi}_p \cong \pi_p^{\eta} \otimes (\omega_{\pi_p}^{-1}|_{Z(\Qp)} \circ \mu).
        \end{equation*}
        This shows that 
    \begin{align*}
        (\tilde{\pi}_p \times \nu_1) |\mu|_{f}^{-2} &\cong \pi_p^{\eta} \otimes (\nu_{1} \circ \mu)  \otimes (\omega_{\pi_p}^{-1}|_{Z(\Qp)} \circ \mu) |\mu|_{f}^{-2} \\  
            &\cong \pi_p^{\eta}, 
    \end{align*}
    where the last isomorphism holds for $\nu_1 \circ \mu = (\omega_{\pi_p}|_{Z(\Qp)} \circ \mu) |\mu|^{2}$. 
    \par Third, by the unramified computation carried out in \cite[\S 3.4 and \S 3.6]{PS18}, we have that for unramified $W_p, \Phi_p$ and $\nu_p$, 
    \begin{align*}
        L_{p}(s, \pi_p, \mathrm{std}) &= I_{p}(W_p, \Phi_p, \nu_p, s) \\ 
        &= \int_{U_{2}(\QQ_p) \backslash H(\QQ_p)}  \Phi_{p}((0,1)g_{1,p}) W_{p}(g_p) |\det(g_{1,p})|_{p}^{s}dg_{p}. 
    \end{align*}
    Hence, for $W_{p}$ associated to $\pi_{p}$, we have 
    \begin{align*}
        L_{p}(s, (\tilde{\pi}_p \times \nu_1) |\mu|_{f}^{-2}, \mathrm{std}) & = L_{p}(s, \pi_p^{\eta}, \mathrm{std}) \\
        & = \int_{U_{2}(\QQ_p) \backslash H(\QQ_p)}  \Phi_{p}((0,1)g_{1,p}) W_{p}(tg_pt^{-1}) |\det(g_{1,p})|_{p}^{s}dg_{p} \\ 
        & = \int_{U_{2}(\QQ_p) \backslash H(\QQ_p)}  \Phi_{p}((0,1)tg_{1,p}t^{-1}) W_{p}(g_p) |\det(g_{1,p})|_{p}^{s}dg_{p} \\
        & = \int_{U_{2}(\QQ_p) \backslash H(\QQ_p)}  \Phi_{p}((0,1)g_{1,p}) W_{p}(g_p) |\det(g_{1,p})|_{p}^{s}dg_{p} \\
        & = L_{p}(s, \pi_{p}, \mathrm{std}), 
    \end{align*}
    where the fourth identity holds for unramified $\Phi_{p}$. 
\end{proof}

\begin{theorem}
\label{Thm: auto_L-value}
Let $n = a + b + r + s$. Under the conditions
\begin{enumerate}
     \item $0 \le -r \le a$ and  $0 \le -s \le b$, 
     \item $a + r \neq b + s$, 
     \item $a > 0$ and $b > 0$,
     \item $r \neq 0$ or $s \neq 0$, 
\end{enumerate}
the relation between $\mathcal{K}(\pi_f,V(2))$ and $\mathcal{D}(\pi_f,V(2))$ is as follows: 
\begin{itemize}
    \item if $a \equiv b + 1\, (\mathrm{mod} \, 2)$, then
        \[
            \mathcal{K}(\pi_f,V(2)) =  W(\pi)  c^{-}(\pi_f, V(2)) \pi^{-(2a + 2b + r + s + 3)} W_{0, b + s - a - r}(8\sqrt{2}\pi D^{-\frac{3}{4}})^{-1} L(n + 1, \pi, \mathrm{std}) \mathcal{D}(\pi_f, V(2));
        \]
    \item if $a \equiv b \, (\mathrm{mod} \, 2)$, then 
        \[
            \mathcal{K}(\pi_f,V(2)) =  W(\pi)  c^{-}(\pi_f, V(2)) \pi^{-(2a + 2b + r + s + 4)} W_{0, b + s - a - r}(8\sqrt{2}\pi D^{-\frac{3}{4}})^{-1} L(n + 1, \pi, \mathrm{std}) \mathcal{D}(\pi_f, V(2)).
        \]
\end{itemize}
The notations in the above formulas are as follows: 
\begin{itemize}
    \item $W(\pi) \in \CC^{\times}/\overline{\QQ}^{\times}$ is a non-zero constant up to non-zero algebraic numbers depending on the normalization of the Whittaker{\textendash}Fourier coeffcients of automorphic forms belongs to $\pi$, 
    \item $W_{0,b + s -a - r}$ is the classical Whittaker function on $\RR$, 
    \item $\mathrm{c}^{-}(\pi_f, V(2))$ is the Deligne period \cite[\S 1.7]{Deligne79} of the motive $M(\pi_f, V(2))$, which is non-zero, 
    \item $L(s, \pi, \mathrm{std})$ is the standard $L$-function of the automorphic representation $\pi$. 
\end{itemize}
\end{theorem}

\begin{proof}
    \par First, it follows from Corollary \ref{corollary: Omega pairing in terms of omega paring with rho} that 
    \begin{equation*}
         \langle \Omega, \tilde{v}_{K} \rangle_{B} = \langle \omega_{\Psi}, [\rho] \rangle_{B} \otimes \mathbf{1}. 
    \end{equation*}
    By equation (\ref{eq:pairing}), $\langle \omega_{\Psi}, [\rho] \rangle_{B}$ is equal to 
    \begin{equation}
    \label{fomrula: summation}
         -\frac{1}{2(a + b + r + s + 1)} \sum\limits_{j = A}^{B} (-1)^{b + r + s} (2j - b -r - s + 1) C_j\int \Xi_{1 + a + 2b - 2j + r + s, a + b + r + s - 2j}(\phi_f).
    \end{equation}
     Corollary \ref{corollary: identify integral of Xi and I} implies that 
    \begin{equation*}
     \int \Xi_{1 + a + 2b - 2j + r + s, a + b + r + s - 2j}(\phi_f) =  c \cdot I(\varphi, \Phi, \nu, 1 + a + b + r + s), 
    \end{equation*}
    for $\Phi = \Phi_f \otimes \Phi_{\infty}, \varphi = \Psi_{f} \otimes X^{1 + a + 2b - 2j + r + s}_{(1, -1, 0)}{\Psi}_{\infty}, \nu = (\nu_1 = |\cdot|^{-n} \nu, \nu_2 = 1)$ and 
    \begin{equation}
    \label{eq: chap10 formula for c}
        c = (-1)^{j} 2i \pi^{1 +a + b + r + s} \Gamma(1 + a + b + r + s)^{-1}, 
    \end{equation}
    in which 
    \begin{equation*}
        \Phi_{\infty}(x, y) := (ix - y)^{j}(ix + y)^{a + b + r + s + j} e^{-\pi(x^2 + y^2)}. 
    \end{equation*}
    By Remark \ref{remark: arch_zeta_integral}, only the term $j = b + s$ is left in the summation of formula (\ref{fomrula: summation}). Hence, the pairing $\langle \omega_{\Psi}, [\rho] \rangle_{B}$ is equal to
    \begin{equation*}
        \frac{(-1)^{b + s + r + 1}(b + s - r + 1)}{2(a + b + r + s + 1)} C_{b + s} \cdot c \cdot I(\varphi, \Phi, \nu, 1 + a + b + r + s),
    \end{equation*}
    for $\Phi = \Phi_f \otimes \Phi_{\infty}, \varphi = \Psi_{f} \otimes X^{1 + a + r - s}_{(1, -1, 0)}{\Psi}_{\infty}, \nu = (\nu_1 = |\cdot|^{-n} \nu, \nu_2 = 1)$ 
    in which 
    \begin{equation*}
        \Phi_{\infty}(x, y) := (ix - y)^{b + s}(ix + y)^{a + r + 2b + 2s} e^{-\pi(x^2 + y^2)}. 
    \end{equation*}

    \par Second, it follows from Proposition \ref{Prop:factorize} that the pairing $\langle \omega, [\rho] \rangle_{B}$ equals 
    \begin{align*}
         & \frac{(-1)^{b + r + s + 1}(b + s - r + 1)}{2(a + b + r + s + 1)}  c \cdot C_{b + s} W_{\varphi}(1)  \prod_{p \in S} I_{p}(W_{\varphi, p}, \Phi_p, \nu_p, 1 + a + b + r + s) \\ 
        &\times I_{\infty}(W_{\varphi, \infty}, \Phi_\infty, \nu_\infty, 1 + a + b + r + s) L^{S}(1 + a + b + r + s, (\tilde{\pi} \times \nu_1 )|\mu|_{f}^{-2}, \mathrm{std}). 
    \end{align*} 
    By the definition of the Whittaker period (see Definition \ref{def: Whittaker period}), we can see that 
    \begin{equation*}
        W_{\varphi}(1) \sim W(\pi). 
    \end{equation*}
    It follows from Lemma \ref{Lemma:C_nonvanish} and Lemma \ref{Lemma: relate L-function of contra of pi to L-function of pi} that 
    \begin{align*}
        \langle \omega_{\Psi}, [\rho] \rangle_{B}  \sim c \cdot W(\pi)  \prod_{p \in S} I_{p}(W_{\varphi, p}, \Phi_p, \nu_p, 1 + a + b + r + s) & I_{\infty}(W_{\varphi, \infty}, \Phi_\infty, \nu_\infty, 1 + a + b + r + s)  \\
        & \times L^{S}(a + b + r + s + 1, \pi, \mathrm{std}). 
    \end{align*}
    Then by Proposition \ref{Prop:algebracity} and Corollary \ref{corollary: local L-factor algebraicity}, if we make a careful choice of $W_{\varphi, p}$ and $\Phi_p,$ according to the proposition, we have 
    \begin{equation*}
        \langle \omega_{\Psi}, [\rho] \rangle_{B}  \sim c \cdot W(\pi) I_{\infty}(W_{\varphi, \infty}, \Phi_\infty, \nu_\infty, 1 + a + b + r + s)L(a + b + r + s + 1, \pi, \mathrm{std}). 
    \end{equation*}
    \par Third, we can see from the archimedean computation in Proposition \ref{Prop:Arch_Zeta} that 
    \begin{align*}
        \langle \omega_{\Psi}, [\rho] \rangle_{B}  \sim & c \cdot W(\pi)  \pi^{-(3a + 3b + 2r + 2s + 5)} W_{0, b + s - a - r}(8\sqrt{2}\pi D^{-\frac{3}{4}})^{-1} \Gamma(a + b + r + s + 1)\\
                & \times \Gamma( \frac{1}{2}a + \frac{3}{2}b + s + 2) \Gamma(\frac{1}{2}a + \frac{1}{2} b + r + 2) L(a + b + r + s + 1, \pi, \mathrm{std}). 
    \end{align*}
    By the formula \ref{eq: chap10 formula for c} for $c$, we have 
    \begin{align*}
         \langle \omega_{\Psi}, [\rho] \rangle_{B}  \sim & W(\pi)  \pi^{-(2a + 2b + r + s + 4)} W_{0, b + s - a - r}(8\sqrt{2}\pi D^{-\frac{3}{4}})^{-1} \\
                        & \times \Gamma( \frac{1}{2}a + \frac{3}{2}b + s + 2) \Gamma(\frac{1}{2}a + \frac{1}{2} b + r + 2) L(a + b + r + s + 1, \pi, \mathrm{std}). 
    \end{align*}
    Recall that we have the following two formulas for $\Gamma$-values: for $n \in \ZZ_{\ge 0}$, 
    \begin{align*}
        & \Gamma(\frac{1}{2} + n) = \frac{(2n)!}{4^{n}n!} \sqrt{\pi}, \\
        & \Gamma(n) = n!. 
    \end{align*}
    So it can be seen that if $a \equiv b + 1\, (\mathrm{mod} \, 2)$,
    \begin{align*}
        &  \Gamma\left( \frac{1}{2}a + \frac{3}{2}b + s + 2\right) \sim \pi^{\frac{1}{2}}, \\ 
        &  \Gamma\left(\frac{1}{2}a + \frac{1}{2} b + r + 2\right) \sim \pi^{\frac{1}{2}}; 
    \end{align*}
    and if $a \equiv b \, (\mathrm{mod} \, 2)$,  
     \begin{align*}
        & \Gamma\left( \frac{1}{2}a + \frac{3}{2}b + s + 2\right)  \in \ZZ,  \\ 
        & \Gamma\left(\frac{1}{2}a + \frac{1}{2} b + r + 2\right)  \in \ZZ.  
    \end{align*}
    Hence, we can get that if $a \equiv b + 1\, (\mathrm{mod} \, 2)$
    \begin{equation*}
         \langle \omega_{\Psi}, [\rho] \rangle_{B}    \sim W(\pi)  \pi^{-(2a + 2b + r + s + 3)} W_{0, b + s - a - r}(8\sqrt{2}\pi D^{-\frac{3}{4}})^{-1} L(a + b + r + s + 1, \pi, \mathrm{std});
    \end{equation*}
    and if 
    $a \equiv b \, (\mathrm{mod} \, 2)$
    \begin{equation*}
         \langle \omega_{\Psi}, [\rho] \rangle_{B}    \sim W(\pi)  \pi^{-(2a + 2b + r + s + 4)} W_{0, b + s - a - r}(8\sqrt{2}\pi D^{-\frac{3}{4}})^{-1} L(a + b + r + s + 1, \pi, \mathrm{std}). 
    \end{equation*}
    \par Finally, combine it with the formula 
    \begin{equation*}
        \mathcal{K}(\pi_{f}, V(2)) = \frac{\langle \Omega, \tilde{v}_{K} \rangle_{B}}{\langle \Omega, \tilde{v}_{D} \rangle_{B}} \mathcal{D}(\pi_f, V(2)). 
    \end{equation*}
    and Proposition \ref{prop: periods}, 
    we complete the proof. 
\end{proof}

\begin{remark}
\label{remark: after auto L-value}
    \begin{itemize}
        \item 
        The condition $(1)$ guarantees that we have the branching from $G$ to $H$. The condition $(2)$ makes sure that we are in the weight $\le -3$ situation of Beilinson's conjectures (see \cite[Conjecture (6.1)]{Nekovar94}). The condition $(3)$ implies the local system associated to $V$ be regular so that we have the vanishing theorem. Condition $(4)$ guarantees that we have the vanishing on the boundary theorem. 
        \item Under the condition of the theroem, it follows from Proposition \ref{Prop: twisted base change} and \cite[Page 151]{GJ72} that the point $s = 1 + n = 1 + a + b + r + s$ lies in the region of absolute convergent of $L(s, \pi, \mathrm{std})$. 
    \end{itemize}
\end{remark}

\subsection{Connection to motivic $L$-functions}
\label{SS: conenction to motivic L-function}
In this subsection, we relate the automorphic $L$-functions to motivic $L$-functions and compare Theorem \ref{Thm: auto_L-value} with Beilinson's conjectures. 

\begin{notation}
    \begin{itemize}
        \item Recall that for any $V \in \mathrm{Rep}_{\QQ}(G)$, there is a canonical $l$-adic local system $\mu_{l}(V)$ on $S$ (with infinite level structure). We use $V$ to denote the $l$-adic local system. 
        \item We can define $l$-adic cohomology $\mathrm{H}_{\text{\'et}}^{2}(S_{\overline{\QQ}}, V(2))$, compactly supported $l$-adic cohomology $\mathrm{H}_{\text{\'et}, c}^{2}(S_{\overline{\QQ}}, V(2))$ and interior $l$-adic cohomology 
        \begin{equation*}
            \mathrm{H}_{\text{\'et}, !}^{2}(S_{\overline{\QQ}}, V(2)) := \mathrm{Im}(\mathrm{H}_{\text{\'et}, c}^{2}(S_{\overline{\QQ}}, V(2)) \rightarrow \mathrm{H}_{\text{\'et}}^{2}(S_{\overline{\QQ}}, V(2))). 
        \end{equation*}
    \end{itemize}
\end{notation}

\begin{remark}
    The interior cohomology $\mathrm{H}_{!,\text{\'et}}^{2}(S_{\overline{\QQ}}, V(2))$ is isomorphic to the intersection cohomology $\mathrm{IH}^{2}_{\text{\'et}}(S^{*}_{\overline{\QQ}}, V(2))$, where $S^{*}$ is the Baily-Borel compactification of $S$. (see \cite[Corollary 2.3.3.4]{CavicchiThesis}.) 
\end{remark}

\begin{definition}
       The comparison isomorphism 
                \begin{equation*}
                    \mathrm{H}_{\text{\'et}, !}^{2}(S_{\overline{\QQ}}, V(2)) \cong \mathrm{H}_{B, !}^{2}(S(\CC), V(2)) \otimes_{\QQ} \QQ_{l}, 
                \end{equation*}
                suggests we define the following $E(\pi_f) \otimes_{\QQ} \QQ_{l}$-module as in Definition \ref{def: M_B and M_dR}:
                \begin{equation*}
                     M_{\text{\'{e}t}}(\pi_{f}, V(2)) := \Hom_{\QQ[G(\AAA_f)]}(\Res_{E(\pi_f)/\QQ}\pi_f, \mathrm{H}_{\text{\'{e}t}, !}^{2}(S_{\overline{\QQ}}, V(2))). 
                \end{equation*}
                This has an action of the absolute Galois group $\Gal(\overline{\QQ}/\QQ)$ induced by the action of $\Gal(\overline{\QQ}/\QQ)$ on $S_{\overline{\QQ}}$. 
\end{definition}

\begin{remark}
    The $\Gal(\overline{\QQ}/\QQ)$-module $M_{\text{\'{e}t}}(\pi_{f}, V(2))$ is the $l$-adic \'{e}tale realization of the Grothendieck motive associated to $\pi$ constructed in \cite[Theorem 5.6]{Wild_15}. 
\end{remark}

\begin{definition}[Motivic $L$-function]
\label{Def: mot L-function}
    Define $L(M_{\text{\'{e}t}}(\pi_{f}, V(2)), s)$, the motivic $L$-function for $s \in \CC$, as the following partial Euler product (without archimedean factor): 
    \begin{equation*}
        L(M_{\text{\'{e}t}}(\pi_{f}, V(2)), s) := \prod_{v < \infty} L_{v}(M_{\text{\'{e}t}}(\pi_{f}, V(2)), s), 
    \end{equation*}
    where 
    \begin{itemize}
        \item if $v = p \neq l$, we define the local $L$-factor as 
              \begin{equation*}
                  L_{v}(M_{\text{\'{e}t}}(\pi_{f}, V(2)), s) := \det(1 - Fr_{p} p^{-s} | M_{\text{\'{e}t}}(\pi_{f}, V(2))^{I_{p}})^{-1}, 
              \end{equation*}
              and $I_{p}$ is the inertia group at the prime $p$ and $Fr_{p}$ is geometric Frobenius at $p$; 
        \item if $v = l$, we choose a prime $l^{\prime} \neq l$ and use $l^{\prime}$ to define $M_{\text{\'{e}t}}(\pi_{f}, V(2))$ and the motivic $L$-factor $L_{v}(M_{\text{\'{e}t}}(\pi_{f}, V(2)), s)$. It makes sense by the \textit{independence of $l$} in Proposition \ref{Prop: indep of l}. 
    \end{itemize}
\end{definition}

\begin{convention}
    As explained in Remark \ref{remark: pairing}, we can view the motivic $L$-function as 
    \begin{equation*}
        L_{v}(M_{\text{\'{e}t}}(\pi_{f}, V(2)), s) \cdot \mathbf{1}
    \end{equation*}
    where $\mathbf{1}$ is unit of $E(\pi_f) \otimes_{\QQ} \overline{\QQ_{l}}$. Hence, we now view $L_{v}(M_{\text{\'{e}t}}(\pi_{f}, V(2)), s)$ as scalar-valued. 
\end{convention}

\begin{proposition}
\label{Prop: indep of l}
    For any finite place $v$, the $L$-factor $L_{v}(M_{\text{\'{e}t}}(\pi_{f}, V(2)), s)$ is a polynomial in $\frac{1}{(Nv)^{s}}$ with coefficients in $\overline{\QQ}$.
    In other words, the definition of the $L$-factor is independent of the choice of $l$. 
\end{proposition}

\begin{proof}
    This follows from a theorem of Gabber (see \cite[Theorem 1]{Fujiwara00}). 
\end{proof}

Now if we fix an embedding $\overline{\QQ} \hookrightarrow \CC$, then the $L$-factor $L_{v}(M_{\text{\'{e}t}}(\pi_{f}, V(2)), s)$ makes sense as a $\CC$-valued function. 

\begin{proposition}
\label{Prop: convergence region of motivic L-function}
    The infinite product in the definition of $L(M_{\text{\'{e}t}}(\pi_{f}, V(2)), s)$ is absolutely convergent when $\mathrm{Res}(s) > \frac{-a - b - r - s}{2}$. 
\end{proposition}

\begin{proof}
    It follows from the fact the weight of $M_{\text{\'{e}t}}(\pi_{f}, V(2))$ is $-2 - a - b - r - s$ (see Corollary \ref{Corollary: Hodge decomp for motives}) and \cite[(1.5)]{Nekovar94}. 
\end{proof}

\begin{proposition}
\label{Prop:rel_L_ftn}
        If we let $L^{S}(M_{\text{\'{e}t}}(\pi_f, V(2)), s)$ be the partial motivic $L$-function away from finitely many bad primes in $S$,  then we have 
        \begin{equation*}
            L^{S}(M_{\text{\'{e}t}}(\pi_f, V(2)), s) = L^{S}(s + n + 1, \pi, std),
        \end{equation*}
        for $s \in \CC$. 
\end{proposition}

\begin{proof}
     In our setting, we use the ``homological normalization" of $V(2)$, but in \cite{LR92}, ``cohomological normalization" is used. Hence, the motivic $L$-function in \cite{LR92} is the 
     \begin{equation*}
        L^{S}(M_{\text{\'{e}t}}(\pi_f, V(2))(-2-n), s)
    \end{equation*}
     in our setting. It is proved in \cite[Theorem A, P291]{LR92} that 
     \begin{equation*}
          L^{S}(M_{\text{\'{e}t}}(\pi_f, V(2))(-2-n), s) = L^{S}(s - 1, \pi, \mathrm{std}).
     \end{equation*}
     Hence, we have 
     \begin{equation*}
         L^{S}(M_{\text{\'{e}t}}(\pi_f, V(2)), s - n - 2) =  L^{S}(M_{\text{\'{e}t}}(\pi_f, V(2))(-2-n), s) = L^{S}(s - 1, \pi, \mathrm{std}). 
     \end{equation*}
     This implies
     \begin{equation*}
          L^{S}(M_{\text{\'{e}t}}(\pi_f, V(2)), s) = L^{S}(s + n + 1, \pi, \mathrm{std}). 
     \end{equation*}
\end{proof}

\begin{remark}
    \par Let us check the compatibility of motivic normalization with automorphic normalization. 
    \par First, we check from the motivic side. It follows from the weight of $M_{B}(\pi_f, V(2))$ is $w = -2 - n$ that\footnote{The notation ``$\leftrightarrow$'' means that equality holds after multiply by the remaining Euler factors.}
    \begin{equation*}
        L^{S}(M_{\text{\'{e}t}}(\pi_f, V(2)), s) \leftrightarrow L^{S}(M_{\text{\'{e}t}}(\pi_f, V(2)), -1 - n - s). 
    \end{equation*}
    Then by Proposition \ref{Prop:rel_L_ftn}, we can see that  
    \begin{align*}
        & L^{S}(M_{\text{\'{e}t}}(\pi_f, V(2)), s) = L^{S}(s + n + 1, \pi, \mathrm{std}), \\ 
        & L^{S}(M_{\text{\'{e}t}}(\pi_f, V(2)), -1 - n - s) = L^{S}(-s, \pi, \mathrm{std}). 
    \end{align*}
    So by the  argument in Lemma \ref{Lemma: relate L-function of contra of pi to L-function of pi}, we get
    \begin{equation}
    \label{eq: ftn_motiv}
        L^{S}(s + n + 1, \pi, \mathrm{std}) = L^{S}(1 + s, \tilde{\pi}, \mathrm{std}) \leftrightarrow L^{S}(-s, \pi, \mathrm{std}).
    \end{equation}
    Hence, the center of $L(s, \pi, \mathrm{std})$ is at $s = \frac{1}{2}$. 
    
    \par Second, we check from the automorphic side. 
    The central character of $\pi$ is 
    \begin{equation*}
        z \in Z(\AAA_\QQ) \mapsto |z|^{-n} \nu, 
    \end{equation*}
    where $\nu$ is a finite order Hecke character of sign $(-1)^{n}$. Hence, the automorphic representation $\pi |\mu|^{\frac{n}{2}}$ is unitary. 
    Therefore, by the argument in Lemma \ref{Lemma: relate L-function of contra of pi to L-function of pi}, we have 
    \begin{equation*}
        L^{S}(s, \pi, \mathrm{std}) = L^{S}(s, (\pi|\mu|^{\frac{n}{2}})|\mu|^{-\frac{n}{2}}, \mathrm{std}) = L^{S}\left(s - \frac{n}{2}, \pi|\mu|^{\frac{n}{2}}, \mathrm{std}\right). 
    \end{equation*}
    The automorphic representation $\pi|\mu|^{\frac{n}{2}}$ is unitary, so by \cite[(5.1), P163]{Gelbart&PS84}, we have 
    \begin{equation*}
        L^{S}\left(s - \frac{n}{2}, \pi|\mu|^{\frac{n}{2}}, \mathrm{std}\right) \leftrightarrow L^{S}\left(1 + \frac{n}{2} - s, \tilde{\pi}|\mu|^{-\frac{n}{2}}, \mathrm{std}\right). 
    \end{equation*}
    It can be seen from the  argument in Lemma \ref{Lemma: relate L-function of contra of pi to L-function of pi} that 
    \begin{equation*}
        L^{S}\left(1 + \frac{n}{2} - s, \tilde{\pi}|\mu|^{-\frac{n}{2}}, \mathrm{std}\right) = L^{S}(1  - s, \pi|\mu|^{\frac{n}{2}}, \mathrm{std}).
    \end{equation*}
    Thus
    \begin{equation}
         L^{S}(s, \pi, \mathrm{std}) \leftrightarrow  L^{S}(1 - s, \pi, \mathrm{std}),
    \end{equation}
    which means the center of $L(s, \pi, \mathrm{std})$ is at $s = \frac{1}{2}$. Hence, the automorphic normalization is compatible with the motivic normalization. 
\end{remark}

Using the previous proposition, we can state our main theorem in terms of motivic $L$-functions. 
\begin{theorem}
\label{Them_mot_L_function}
    Let $n = a + b + r + s$. Under the conditions 
    \begin{enumerate}
         \item $0 \le -r \le a$ and  $0 \le -s \le b$,     
         \item $a + r \neq b + s$, 
         \item $a > 0$ and $b > 0$,
         \item $r \neq 0$ or $s \neq 0$, 
    \end{enumerate}
    the relation between $\mathcal{K}(\pi_f,V(2))$ and $\mathcal{D}(\pi_f,V(2))$ is as follows: 
\begin{itemize}
    \item if $a \equiv b + 1\, (\mathrm{mod} \, 2)$, then
        \[
            \mathcal{K}(\pi_f,V(2)) =  W(\pi)  c^{-}(\pi_f, V(2)) \pi^{-(2a + 2b + r + s + 3)} W_{0, b + s - a - r}(8\sqrt{2}\pi D^{-\frac{3}{4}})^{-1} L(M_{\text{\'{e}t}}(\pi_f, V(2)), 0) \mathcal{D}(\pi_f, V(2));
        \]
    \item if $a \equiv b \, (\mathrm{mod} \, 2)$, then 
        \[
            \mathcal{K}(\pi_f,V(2)) =  W(\pi)  c^{-}(\pi_f, V(2)) \pi^{-(2a + 2b + r + s + 4)} W_{0, b + s - a - r}(8\sqrt{2}\pi D^{-\frac{3}{4}})^{-1}  L(M_{\text{\'{e}t}}(\pi_f, V(2)), 0) \mathcal{D}(\pi_f, V(2)).
        \]
\end{itemize}
The notations in the above formulas are as follows: 
\begin{itemize}
    \item $W(\pi) \in \CC^{\times}/\overline{\QQ}^{\times}$ is a non-zero constant up to nonzero algebraic numbers depending on the normalization of the Whittaker{\textendash}Fourier coefficients of automorphic forms belonging to $\pi$, 
    \item $W_{0,b + s -a - r}$ is the classical Whittaker function on $\RR$, 
    \item $\mathrm{c}^{-}(\pi_f, V(2))$ is the Deligne period \cite[\S 1.7]{Deligne79} of the motive $M(\pi_f, V(2))$ which is non-zero, 
    \item $L(M_{\text{\'{e}t}}(\pi_f, V(2)), s)$ is the motivic $L$-function of the automorphic representation $\pi$. 
\end{itemize}
\end{theorem}

\begin{proof}
    Since for a finite place $v$ of $\QQ$, the value of local $L$-factor $L_{v}(M_{\text{\'{e}t}}(\pi_f, V(2), s)$ at $s = 0$ and the value of $L_{v}(s, \pi, \mathrm{std})$ at $s = 1 + a + b + r + s$ belongs to $\overline{\QQ}^{\times}$, we only need to consider the partial $L$-functions. By Proposition \ref{Prop:rel_L_ftn}, we have 
    \begin{equation*}
        L^{S}(M_{\text{\'{e}t}}(\pi_f, V(2)), 0) = L^{S}(1 + n, \pi, std).
    \end{equation*}
\end{proof}

\begin{remark}
    The point $s = 0$ of the motivic $L$-function $L(M(\pi_f, V(2)), s)$ is exactly what Beilinson conjectured in his conjectures \textnormal{\cite[Conjecture (6.1)]{Nekovar94}}, and the shape of the theorem is similar to the rank $1$ case of the weak Beilinson's conjectures \cite[(6.6)]{Nekovar94}. It only remains to prove 
    \begin{itemize}
        \item if $a \equiv b + 1\, (\mathrm{mod} \, 2)$, then
        \begin{equation*}
            W(\pi)  c^{-}(\pi_f, V(2)) \pi^{-(2a + 2b + r + s + 3)} W_{0, b + s - a - r}(8\sqrt{2}\pi D^{-\frac{3}{4}})^{-1} \in \overline{\QQ}^{\times};
        \end{equation*}
        \item if $a \equiv b \, (\mathrm{mod} \, 2)$, then 
        \begin{equation*}
            W(\pi)  c^{-}(\pi_f, V(2)) \pi^{-(2a + 2b + r + s + 4)} W_{0, b + s - a - r}(8\sqrt{2}\pi D^{-\frac{3}{4}})^{-1}  \in \overline{\QQ}^{\times}.
        \end{equation*}
    \end{itemize}
    We will leave this to future work. 
\end{remark}

\begin{corollary}
\label{corollary: nontrivial_class}
        The $E(\pi_f)$-rational structure $\mathcal{K}(\pi_f, V(2))$ is non-trivial. In other words, the morphism in Construction \ref{Construction: motivic classes}
              \begin{equation*}
                  \mathcal{E}is_{M}^{n} \colon \mathcal{B}_{n} \rightarrow \mathrm{H}_{M}^{3}(S, V(2))
              \end{equation*}
        is non-trivial.
\end{corollary}

\begin{proof}
        It follows from Proposition \ref{Prop: convergence region of motivic L-function} (or Remark \ref{remark: after auto L-value}) that 
        \begin{equation*}
            L(M(\pi_f, V(2)), 0) \neq 0. 
        \end{equation*}
        It follows from the fact that $W(\pi) \neq 0$, $ W_{0,r-s}(8\sqrt{2}D^{-\frac{3}{4}})^{-1} \neq 0$, $\mathrm{c}^{-}(\pi_f,V(2)) \neq 0$ and $\mathcal{D}(\pi_f, V(2))$ is non-trivial and Theorem \ref{Them_mot_L_function} that $\mathcal{K}(\pi_f,V(2))$ is non-trivial.
\end{proof}

\begin{remark} 
    \begin{itemize}
        \item Theorem \ref{Them_mot_L_function} is a $w \le -3$ counterpart of what is proved in \cite[Theorem 8.18]{PS18}, where they proved a case that the motivic sheaf $V$ is trivial and the weight of the motive is $-2$. 
        \item In \cite{LSZ22}, the authors constructed an Euler system for $\GU(2,1)$ based on the motivic class $\mathcal{E}is_{M}^n(\phi_f)$ for $\phi_f \in \mathcal{B}_{n}$ being non-trivial, so our theorem of non-triviality of $\mathcal{K}(\pi_f, V(2))$ answers a question raised in their paper, which gives a necessary condition for their Euler system to work. 
    \end{itemize}
\end{remark}




\bigskip

\bibliographystyle{alpha}
\bibliography{reference}

\end{document}